\documentclass[12pt,lot, lof]{puthesis}
\newcommand{\proquestmode}{}
\title{Optimization over Nonnegative and Convex Polynomials with and without Semidefinite Programming}

\submitted{June 2018}  % degree conferral date (January, April, June, September, or November)
\copyrightyear{2018}  % year in which the copyright is secured by publication of the dissertation.
\author{Georgina Hall}
\adviser{Professor Amir Ali Ahmadi}  %replace with the full name of your adviser
\department{Operations Research and Financial Engineering}

\newcommand\scalemath[2]{\scalebox{#1}{\mbox{\ensuremath{\displaystyle #2}}}} 
\usepackage{graphicx}
\usepackage{verbatim}

\usepackage{multirow}
\usepackage{longtable}

\usepackage{booktabs}

\usepackage{latexsym,times}
\usepackage{amsthm,amsfonts,amsmath,amstext,amssymb,amsopn}
\usepackage{graphicx,epsfig}
\usepackage{xspace}
\usepackage{color}
\usepackage{tikz}
\usepackage{nicefrac}
\usepackage{enumerate}
\usepackage{subfigure}
\usepackage{float}
\usepackage{scalefnt}
\usepackage{tablefootnote}
\usepackage{algorithm}
\usepackage{algpseudocode}

\newcommand{\coefx}{\textup{coef}}

\newcommand{\conex}{\textup{cone}}

\newcommand{\coef}[1]{\coefx (#1)}

\newcommand{\cone}[1]{\conex(#1)}

\newcommand{\R}{\mathbb{R}}
\newcommand{\Z}{\mathbb{Z}}
\renewcommand{\S}{\mathcal{S}}

\newcommand{\U}{\mathcal{U}}

\newtheorem{theorem}{Theorem}[section]

\newtheorem{lemma}[theorem]{Lemma}
\newtheorem{prop}[theorem]{Proposition}
\newtheorem{corollary}[theorem]{Corollary}

\newtheorem{definition}[theorem]{Definition}
\newtheorem{remark}[theorem]{Remark}
\newtheorem{example}[theorem]{Example}

\newtheorem{proposition}[theorem]{Proposition}
\newtheorem{openprob}[theorem]{Open Problem}

\def\S{\mathcal{S}}

\newcommand{\vv}[1] {\mathbf{#1}}

\def\rot#1{\rotatebox{90}{#1}}

\ifdefined\printmode

\usepackage{url}

\else

\ifdefined\proquestmode
\usepackage{hyperref}
\hypersetup{bookmarksnumbered}

\makeatletter
\hypersetup{pdftitle=\@title,pdfauthor=\@author}
\makeatother

\else
\usepackage{hyperref}
\hypersetup{colorlinks,bookmarksnumbered}

\makeatletter
\hypersetup{pdftitle=\@title,pdfauthor=\@author}
\makeatother

\fi % proquest or online formatting
\fi % printed or online formatting

\ifodd 0
\renewcommand{\maketitlepage}{}
\renewcommand*{\makecopyrightpage}{}
\renewcommand*{\makeabstract}{}

\else

\abstract{
The problem of optimizing over the cone of nonnegative polynomials is a fundamental problem in computational mathematics, with applications to polynomial optimization, control, machine learning, game theory, and combinatorics, among others. A number of breakthrough papers in the early 2000s showed that this problem, long thought to be out of reach, could be tackled by using sum of squares programming. This technique however has proved to be expensive for large-scale problems, as it involves solving large semidefinite programs (SDPs).

In the first part of this thesis, we present two methods for approximately solving large-scale sum of squares programs that dispense altogether with semidefinite programming and only involve solving a sequence of \emph{linear or second order cone programs} generated in an adaptive fashion. We then focus on the problem of finding tight lower bounds on polynomial optimization problems (POPs), a fundamental task in this area that is most commonly handled through the use of SDP-based sum of squares hierarchies (e.g., due to Lasserre and Parrilo). In contrast to previous approaches, we provide the first theoretical framework for constructing converging hierarchies of lower bounds on POPs whose computation simply requires the ability to multiply certain fixed polynomials together and to check nonnegativity of the coefficients of their product.

In the second part of this thesis, we focus on the theory and applications of the problem of optimizing over \emph{convex} polynomials, a subcase of the problem of optimizing over nonnegative polynomials. On the theoretical side, we show that the problem of testing whether a cubic polynomial is convex over a box is NP-hard. This result is minimal in the degree of the polynomial and complements previously-known results on checking convexity of a polynomial globally. We also study norms generated by convex forms and provide an SDP hierarchy for optimizing over them. This requires an extension of a result of Reznick on sum of squares representation of positive definite forms to positive definite biforms. On the application side, we study a problem of interest to robotics and motion planning, which involves modeling complex environments with simpler representations. In this setup, we are interested in containing 3D-point clouds within polynomial sublevel sets of minimum volume. We also study two applications in machine learning: the first is multivariate monotone regression, which is motivated by some applications in pricing; the second concerns a specific subclass of optimization problems called difference of convex (DC) programs, which appear naturally in machine learning problems. We show how our techniques can be used to optimally reformulate DC programs in order to speed up some of the best-known algorithms used for solving them.
}

\acknowledgements{
I would first like to thank my adviser, Amir Ali Ahmadi. I can safely say that I would never have written a PhD thesis, nor pursued a career in academia, if it had not been for my great good luck in having Amir Ali as an adviser.  Amir Ali, thank you for being a wonderful person and mentor. Thank you for believing in me when I did not believe in myself, for not giving up on me when I was giving up on the program, and for making me feel comfortable and confident in an environment where I never expected to feel that way. Thank you for giving me the opportunity to see how a truly brilliant mind works---if I have only taken away  from this PhD a fragment of your clarity of thought, your attention to detail, and your unbridled creativity, then it will have been a success. Thank you for investing so much time in helping me grow, reading over my papers until they were perfect, listening to my practice talks for hours on end, working with me to perfect the way I answer questions during talks, and giving me feedback on my proofs and my ideas. I know of no other adviser like you, so kind, altruistic, and gifted. Our 4-year collaboration has been a source of great inspiration and pleasure to me and I hope that it will continue past this PhD, until we are both old and grey (a time that may come earlier to some than others!). 

I would also like to thank my thesis readers and committee members, Anirudha Majumdar, Bob Vanderbei, and Mengdi Wang. It has been a great honor to interact with you over the years within Princeton and without. Thank you to my co-authors, Emmanuel Abbe, Afonso Bandeira, Mihaela Curmei, Sanjeeb Dash, Etienne de Klerk, Ameesh Makadia, Antonis Papachristodoulou, James Saunderson, Vikas Sindhwani, and Yang Zheng. It was a real pleasure to work with you all. I learnt so much from my interactions with you. A special thank you to Afonso Bandeira for introducing me to Amir Ali and to S\'ebastien Bubeck for advising me during my second year at Princeton University.

I am extremely grateful to the people who wrote letters of recommendations for me, for job applications, fellowships, and otherwise: Emmanuel Abbe, Rene Carmona, Erhan Cinlar, Sanjeeb Dash, Etienne de Klerk, and Leo Liberti. I would also like to acknowledge Sanjeeb Dash, Jean-B. Lasserre, Leo Liberti, Daniel Kuhn, Antonis Papachristodoulou, Pablo Parrilo, and Bernd Sturmfels, for inviting me to give seminars and offering help and advice when I was faced with important decisions.

Je remercie mes professeurs de math\'ematiques, M. Aiudi, M. T\'ecourt, et M. Koen, ainsi que mon professeur de physique de classes preparatoires, M. Cervera, pour leur soutien. Sans eux, je n'aurais jamais eu la passion que j'ai pour les sciences, ni la rigueur que j'ai acquise avec eux pour les \'etudier.

I would like to thank my incredible friends back home for the Skypes over the years, for visiting me, and for finding the time to see me whenever I went home: Alice, Brice, Florian, JG, Joachim, Jorge, Mathieu, Marie B., Marie G., Martin, No\'emie, Paul, Sacha, Timtim, Valentine. Elise (and Sean!), merci de m'avoir soutenue pendant mes deuxi\`eme et troisi\`eme annees --- tu es une amie formidable. H\^ate de vous voir tous beaucoup plus maintenant que je rentre!

Thank you to my fantastic ORFE officemates and teammates, as well as all the people I was lucky enough to TA with over the years: Bachir, Cagin, Cemil, Chenyi, Donghwa, Elahe, Firdevs, Han, Jeff (who will remind me to register to conferences and book tickets now?), Joane, Junwei, Kaizheng, Kevin W., Sinem, Tianqi, Thomas P., Yutong, Yuyan, Zongxi. I would also like to acknowledge the wonderful students I have had the opportunity to teach over the years and particularly those with whom I worked on senior theses: Ellie, Mihaela, Salena, and Tin. Thank you to Kim for being a great department coordinator, and to Carol and Melissa for helping me with my receipts! To my other friends in Princeton, you made the toughest times not only bearable but fun: Adam, Adrianna, Amanda, Alex, Carly, Chamsi, Daniel J., Genna, Hiba, Ingrid, Jessica, Kelsey, Lili, Kobkob, Marte, Matteo, Roger, Sachin, Sara, Thomas F.

Last but not least, I am lucky enough to have the most supportive family anyone could ever ask for. Kevin, ta capacit\'e a voir le c\^ot\'e humouristique dans toutes les situations et a d\'edramatiser mes probl\`emes les plus compliqu\'es m'ont permis de rester saine. Merci pour tout. Luey, que dire de plus si ce n'est que tu es all\'ee m'acheter tout un pack de survie avant mon entretien a l’INSEAD. Chaque fois qu'on se voit (\#Disney) ou qu'on s'appelle, ma journ\'ee devient that much better. Tom, merci d'avoir \'et\'e Team Georgi jusqu'au bout. Mutts and Dads, thanks for being the best parents ever. Your unconditional support and the fact that you were willing to learn about a whole new area just to keep up with me meant so much to me. Thank you.
}

\dedication{To Mutts and Dads.}

\fi  % disable frontmatter

\begin{document}

\makefrontmatter
%\singlespacing

% If you've disabled frontmatter, you can insert the toc manually
%\tableofcontents\clearpage

% \include lets us split up the document (and each include starts a new page):

%\singlespacing

\chapter{Introduction\label{ch:intro}}

This thesis concerns itself broadly with the problem of \emph{optimizing over nonnegative polynomials}. In its simplest form, this problem involves (i) decision variables that are the coefficients of a multivariate polynomial of a given degree, (ii) an objective function that is linear in the coefficients, (iii) constraints that are affine in the coefficients, and (iv) a constraint that the multivariate polynomial be nonnegative over a \emph{closed basic semialgebraic set}, i.e., a set defined by a finite number of polynomial inequalities. We write:
\begin{equation}\label{eq:opt.nonneg}
\begin{aligned}
&\min_{\overrightarrow{p}} &&c(\overrightarrow{p})\\
&\text{subject to } &&A(\overrightarrow{p})=b\\
& &&p(x)\geq 0, \forall x \text{ such that } g_i(x)\geq 0, i=1,\ldots,s,
\end{aligned}
\end{equation}
where $\overrightarrow{p}$ here denotes the coefficients of a multivariate polynomial $p:\mathbb{R}^n \rightarrow \mathbb{R}$ of some degree $d$, $c$ is a linear functional over the coefficients of $p$, $A$ is a linear map that maps the coefficients of $p$ to $\mathbb{R}^m$, $b$ is a vector in $\mathbb{R}^m$, and $g_i, i=1,\ldots,s$, are multivariate polynomials.

This problem appears under different forms in a wide range of applications. One such application is \emph{polynomial optimization}, which is the problem of minimizing a polynomial function over a closed basic semialgebraic set: 
\begin{equation*}
\begin{aligned}
&\min_{x\in \mathbb{R}^n} &&p(x)\\
&\text{subject to } &&g_i(x)\geq 0, i=1,\ldots,s.
\end{aligned}
\end{equation*}
Indeed, the optimal value of this problem is equivalent to the largest lower bound on $p$ over the set $\{x\in \mathbb{R}^n|~ g_i(x)\leq 0,i=1,\ldots,s\}$. In other words, we can find the optimal value of the problem above by solving the following ``dual'' problem:
\begin{equation*}
\begin{aligned}
&\max_{\gamma} &&\gamma\\
&\text{subject to } &&p(x)-\gamma \geq 0, \forall x \text{ such that } g_i(x)\geq 0, i=1,\ldots,s.
\end{aligned}
\end{equation*}
This is exactly a problem of optimizing over nonnegative polynomials. Polynomial optimization problems, or POPs, feature in different areas: in power engineering via the optimal power flow problem~\cite{OPF_survey}, in discrete and combinatorial optimization~\cite{laurent2009sums, Stability_number_SOS}, in economics and game theory \cite{sturmfels2002solving}, and in distance geometry~\cite{nie2009sum}, just to name a few. Other applications of the problem of optimizing over nonnegative polynomials appear in control, in particular for searching for Lyapunov functions for dynamical systems~\cite{PabloAliHSCC,PhD:Parrilo, PositivePolyInControlBook, AAA_PhD}, robotics~\cite{ahmadiOR_letters}, and machine learning and statistics \cite{lasserre2009moments}, among other areas.

All these applications motivate the question as to whether (\ref{eq:opt.nonneg}) can be solved efficiently. The answer is unfortunately negative in general. In fact, simply testing whether a given polynomial of degree-4 is nonnegative over $\mathbb{R}^n$ is NP-hard~\cite{nonnegativity_NP_hard}. Past work has hence focused on replacing the nonnegativity condition in (\ref{eq:opt.nonneg}) with a stronger, but more tractable, condition. The idea is that the optimization problem thus obtained can be efficiently solved and upper bounds on the optimal value of (\ref{eq:opt.nonneg}) can be obtained (note that the set over which we would be optimizing would be an inner approximation of the initial feasible set). 

%We remark that lower bounds are of particular interest as they are not always easy to obtain, in contrast to upper bounds, which can be easily found if one has access to a feasible solution to the problem. 

A well-known sufficient condition for (global) nonnegativity of a polynomial $p$ is that it be a \emph{sum of squares} (sos), i.e., that it have a decomposition of the form
$$p(x)=\sum_i q_i(x)^2,$$ where $q_i$ are polynomials. Sum of squares polynomials have a long history that dates back at least to the end of the $19^{th}$ century. In 1888, Hilbert showed that not all nonnegative polynomials are sums of squares by proving that these two notions are only equivalent when some conditions on the degree of the polynomial at hand and the number of its variables are met \cite{Hilbert_1888}. His proof was not constructive and it would be another 80 years before the first example of a nonnegative but non-sum of squares polynomial would be presented by Motzkin \cite{MotzkinSOS}. Hilbert's research on sum of squares polynomials led him to include a related question in the list of so-called ``Hilbert problems'', a famous list of 23 open questions, that Hilbert put forward in the year 1900. His $17^{th}$ problem poses the question as to whether every nonnegative polynomial can be written as the ratio of two sums of squares polynomials. This was answered affirmatively by Artin~\cite{artin1927} in 1927. 

The beginning of the $21^{st}$ century brought with it a renewed interest in sum of squares polynomials, but from the optimization community this time, rather than the pure mathematics community. This was largely due to the discovery that sum of squares polynomials and \emph{semidefinite programming} are intimately related \cite{NesterovSquared,sdprelax,lasserre2001}. We remind the reader that semidefinite programming (SDP) is a class of optimization problems where one optimizes a linear objective function over the intersection of the cone of positive semidefinite matrices and an affine subspace, i.e., a problem of the type
\begin{equation}\label{eq:SDP}
\begin{aligned}
&\min_{X \in S^{n \times n}} &&tr(CX)\\
&\text{s.t. } &&tr(A_iX)=b_i, i=1,\ldots,m\\
& &&X\succeq 0,
\end{aligned}
\end{equation}
where $S^{n \times n}$ denotes the set of $n \times n$ symmetric matrices, $tr$ denotes the trace of a matrix, and $C,A_i,b_i$ are input matrices of size respectively $n \times n, n \times n, 1\times 1.$ Semidefinite programming comprises a large class of problems (including, e.g., all linear programs), and can be solved to arbitrary accuracy in polynomial time using interior point methods. (For a more detailed description of semidefinite programming and its applications, we refer the reader to \cite{vandenberghe1996semidefinite}.) The key result linking semidefinite programming and sos polynomials is the following: a polynomial $p$ of degree $2d$ is sos if and only if it can be written as$$p(x)=z(x)^TQz(x),$$ for some positive semidefinite matrix $Q$. Here, $z(x)=(1,x_1,\ldots,x_n,x_1x_2,\ldots,x_n^d)$ is the vector of standard monomials of degree $ \leq d$. Such a matrix $Q$ is sometimes called the \emph{Gram} matrix of the polynomial $p$ and it is of size $\binom{n+d}{d}$ if $p$ is of degree $2d$ and has $n$ variables. This result implies that one can optimize over the cone of sos polynomials of fixed degree in polynomial time to arbitrary accuracy. Indeed, searching for the coefficients of a polynomial $p$ subject to the constraint that $p$ be sos can be rewritten as the problem of searching for a positive semidefinite matrix $Q$ whose entries can be expressed as linear combinations of the coefficients of $p$ (this is a consequence of the fact that two polynomials are equal everywhere if and only if their coefficients are equal). In other words, any \emph{sos program}, i.e., a linear optimization problem over the intersection of the cone of sos polynomials with an affine subspace, can be recast as an SDP. (We remark that it is also true that any SDP can be written as an sos program---in fact, this sos program need only involve quadratic polynomials.)

How can sos programs be used to solve problems like (\ref{eq:opt.nonneg})? It turns out that one can produce certificates of nonnegativity of a polynomial $p$ over a closed basic semialgebraic set $$S\mathrel{\mathop{:}}=\{x\in \mathbb{R}^n|~g_i(x)\geq 0, i=1,\ldots,s\}$$ via sum of squares polynomials. Such certificates are called Positivstellens\"atze. We briefly mention one such Positivstellensatz here to illustrate the point we aim to make. Other Positivstellens\"atze as well as additional context is given in Chapter \ref{ch:positiv} of this thesis. The following Positivstellensatz is due to Putinar \cite{putinar1993positive}: under a technical assumption slightly stronger than compactness of $S$ (see Theorem \ref{th:putinar} for the exact statement), if $p$ is positive on $S$, then there exist sos polynomials $\sigma_0,\ldots,\sigma_s$ such that $$p(x)=\sigma_0(x)+\sum_{i=1}^s \sigma_i(x)g_i(x).$$ (Conversely, it is clear that if such a representation holds, then $p$ must be nonnegative on $S$.) Hence, one can replace the condition that $p$ be nonnegative over $S$ in (\ref{eq:opt.nonneg}) by a ``Putinar certificate'' and obtain the following optimization problem:
\begin{equation}\label{eq:sos.prog}
\begin{aligned}
&\min_{\overrightarrow{p},\overrightarrow{\sigma_0},\ldots,\overrightarrow{\sigma_s}} &&C(\overrightarrow{p})\\
&\text{s.t. } &&A(\overrightarrow{p})=b\\
& &&p(x)=\sigma_0(x)+\sum_{i=1}^{s} \sigma_i(x)\cdot g_i(x)\\
& &&\sigma_i(x) \text{ sos for } i=0,\ldots,s.
\end{aligned}
\end{equation}
Note that when the degrees of the polynomials $\sigma_i$ are fixed, this problem is an sos program, which can be recast as an SDP and solved in polynomial time to arbitrary accuracy. This provides an upper bound on the optimal value of (\ref{eq:opt.nonneg}). As the degree of the sos polynomials increases, one obtains a sequence of upperbounds on the optimal value of (\ref{eq:opt.nonneg}) that is nonincreasing. Putinar's Positivstellensatz tells us that if one keeps increasing the degrees of the sos polynomials, one will eventually (and maybe asymptotically) recover the optimal value of (\ref{eq:opt.nonneg}), the caveat being that the degrees needed to recover this optimal value are not known a priori. 
%The sequence of semidefinite programs generated thus is termed a \emph{converging hierarchy of sos programs}. 

We remark that the semidefinite programs arising in this hierarchy can be quite large, particularly if the number of variables and the degrees of $\sigma_i$ are high. Hence, they can be quite slow to solve as semidefinite programs are arguably the most expensive class of convex optimization problems to solve, with a running time that grows quickly with the dimension of the problem \cite{vandenberghe1996semidefinite}.

As a consequence, recent research has focused on making sum of squares optimization more scalable. One research direction has focused on exploiting structure in SDPs \cite{cifuentes2016exploiting,de2010exploiting, Symmetry_groups_Gatermann_Pablo,riener2013exploiting,vallentin2009symmetry,zheng2017fast} or developing new solvers that scale more favorably compared to interior point methods \cite{bertsimas2013accelerated,li2015qsdpnal, nie2012regularization,zhao2010newton}. Another direction involves finding cheaper alternatives to semidefinite programming that rely, e.g., on linear programming or second order cone programming. This, as well as methods to derive certificates of positivity over closed basic semialgebraic sets from certificates of global positivity, is the focus of the first part of this thesis. The second part of this thesis focuses on a special case of the problem of optimizing over nonnegative polynomials: that of optimizing over \emph{convex} polynomials, and applications thereof. In the next section, we describe the contents and contributions of each part of this thesis more precisely.

\section{Outline of this thesis}

\paragraph{Part I: LP, SOCP, and Optimization-Free Approaches to Semidefinite and Sum of Squares Programming.} The first part of this thesis focuses on linear programming, second order cone programming, and optimization-free alternatives to sums of squares (and semidefinite) programming. 

Chapter \ref{ch:bp} and Chapter \ref{ch:cholesky} are computational in nature and propose new algorithms for approximately solving semidefinite programs. These rely on generating and solving adaptive and improving sequences of linear programs and second order cone programs. 

%In this paper, we show that such hierarchies could in fact be designed from much more limited Positivstellens\"atze dating back to the early 20th century that only certify positivity of a polynomial globally. More precisely, we show that any inner approximation to the cone of positive homogeneous polynomials that is arbitrarily tight can be turned into a converging hierarchy for general polynomial optimization problems with compact feasible sets. This in particular leads to a semidefinite programming-based hierarchy that relies solely on Artin's solution to Hilbert's 17th problem. We also use a classical result of Poly\'a on global positivity of even forms to construct an ``optimization-free'' converging hierarchy for general {\gh polynomial optimization problems (POPs)} with compact feasible sets. This hierarchy only requires polynomial multiplication and checking nonnegativity of coefficients of certain fixed polynomials. As a corollary, we obtain new linear programming and second-order cone programming-based hierarchies for POPs that rely on the recently introduced concepts of dsos (diagonally dominant sum of squares) and  sdsos (scaled diagonally dominant sum of squares) polynomials. {\gh At this stage, we would like to emphasize that the scope of the methods presented here is purely theoretical.}

Chapter \ref{ch:positiv} is theoretical in nature: we show that any inner approximation to the cone of nonnegative homogeneous polynomials that is arbitrarily tight can be turned into a converging hierarchy for general polynomial optimization problems with compact feasible sets. We also use a classical result of Poly\'a on global positivity of even forms to construct an ``optimization-free'' converging hierarchy for general polynomial optimization problems (POPs) with compact feasible sets. This hierarchy only requires polynomial multiplication and checking nonnegativity of coefficients of certain fixed polynomials that arise as products.

%and focuses on understanding the minimal conditions under which a hierarchy of optimization problems that converge to the optimal value of (\ref{eq:opt.nonneg}) (such as the one presented above) can be produced. In particular, we propose new certificates of nonnegativity of polynomials over semialgebraic sets which do not involve sum of squares polynomials. Hence, certifying nonnegativity of a polynomial over a basic semialgebraic set does not necessarily require solving a semidefinite program anymore --- we show in fact that one can produce certificates that dispense with optimization altogether. 

We emphasize that the goals in Chapters \ref{ch:bp}, \ref{ch:cholesky}, and \ref{ch:positiv} are different, though they both work with more tractable (but smaller) subclasses of nonnegative polynomials than sum of squares polynomials. For the first two chapters, the goal is to solve in a fast and more efficient manner the sos program given in (\ref{eq:sos.prog}) approximately. In the third chapter, the goal is to provide new converging hierarchies for POPs with compact feasible sets that rely on simpler certificates.

\paragraph{Part II: Optimizing over convex polynomials.} The second part of the thesis focuses on an important subcase of the problem of optimizing over the cone of nonnegative polynomials: that of optimizing over the cone of convex polynomials. The relationship between nonnegative polynomials and convex polynomials may not be obvious at first sight but it be seen easily, e.g., as a consequence of the second-order characterization of convexity: a polynomial $p$ is convex if and only if its Hessian matrix $H(x)$ is positive semidefinite for all $x \in \mathbb{R}^n$. This is in turn equivalent to requiring that the polynomial  $y^TH(x)y$ in $2n$ variables $(x,y)$ be nonnegative. Hence, just as the notion of sum of squares was used as a surrogate for nonnegativity, one can define the notion of \emph{sum of squares-convexity} (\emph{sos-convexity}), i.e., $y^TH(x)y$ be sos, as a surrogate for convexity. One can then replace any constraint requiring that a polynomial be convex, by a requirement that it be sos-convex. The program thus obtained will be an sos program which can be recast as an SDP. Chapters \ref{ch:dc}, \ref{ch:polynorms}, \ref{ch:vikas}, and \ref{ch:mihaela} all present different theoretical and applied questions around this problem. 

In Chapter \ref{ch:dc}, this framework is used for a theoretical study of optimization problems known as \emph{difference of convex (dc) programs}, i.e., optimization problems where both the objective and the constraints are given as a difference of convex functions. Restricting ourselves to polynomial functions, we are able to show that any polynomial can be written as the difference of two convex polynomial functions and that such a decomposition can be found efficiently. As this decomposition is non-unique, we then consider the problem of finding a decomposition that is optimized for the performance of the most-widely used heuristic for solving dc programs. 

In Chapter \ref{ch:polynorms}, we are interested in understanding when a homogeneous polynomial $p$ of degree $2d$ generates a norm. We show that the $2d^{th}$ root of any strictly convex polynomial is a norm. Such norms are termed \emph{polynomial norms}. We show that they can approximate any norm to arbitrary accuracy. We also show that the problem of testing whether a polynomial of degree $4$ gives rise to a polynomial norm is NP-hard. We consequently provide SDP-based hierarchies to test membership to and optimize over the set of polynomial norms. Some applications in statistics and dynamical systems are also discussed. 

 In Chapter \ref{ch:vikas}, we consider a problem that arises frequently in motion planning and robotics: that of modeling complex objects in an environment with simpler representations. The goal here is to contain a cloud of 3D-points within shapes of minimum volume described by polynomial sublevel sets. A new heuristic for minimizing the volume of these sets is introduced, and by appropriately parametrizing these sublevel sets, one is also able to control their convexity.
 
 In Chapter \ref{ch:mihaela}, we consider an important application in statistics and machine learning: that of \emph{shape-constrained regression}. In this setup, unlike unconstrained regression, we are not solely interested in fitting a (polynomial) regressor to data so as to minimize a convex loss function such as least-squares error. We are also interested in imposing shape constraints, such as monotonicity and convexity to our regressor over a certain region. Motivated by this problem, we study the computational complexity of testing convexity or monotonicity of a polynomial over a box and show that this is NP-hard already for cubic polynomials. The NP-hardness results presented in this chapter are of independent interest, and in the case of convexity are a follow-up result to the main theorem in \cite{ahmadi2013complete} which shows that it is NP-hard to test whether a quartic polynomial is convex \emph{globally}. These computational complexity considerations motivate us to further study semidefinite approximations of the notions of monotonicity and convexity.  We prove that any $C^1$ (resp. $C^2$) function with given monotonicity (resp. convexity) properties can be approximated arbitrarily well by a polynomial function with the same properties, whose monotonicity (resp. convexity) are moreover certified via sum of squares proofs.
 
 Finally, we remark that for the convenience of the reader, each chapter is written to be completely self-contained.

\section{Related publications}
The material presented in this thesis is based on the following papers.

\paragraph{Chapter 2.} A. A. Ahmadi, S. Dash, and G. Hall. {\it Optimization over structured subsets of positive semidefinite matrices via column generation} (2017). In Discrete Optimization, 24, pp. 129-151.

\paragraph{Chapter 3.} A. A. Ahmadi and G. Hall. \emph{Sum of squares basis pursuit with linear and second order cone programming} (2016). In Algebraic and Geometric Methods in Discrete Mathematics, Contemporary Mathematics.

\paragraph{Chapter 4.} A. A. Ahmadi and G. Hall. \emph{On the construction of converging hierarchies for polynomial optimization based on certificates of global positivity} (2017). Under second round of review in Mathematics of Operations Research.

\paragraph{Chapter 5.} A. A. Ahmadi and G. Hall. \emph{DC decomposition of nonconvex polynomials with algebraic techniques} (2015). In Mathematical Programming, 6, pp.1-26.

\paragraph{Chapter 6.} A. A. Ahmadi, E. de Klerk, and G. Hall. \emph{Polynomial norms} (2017). Under review. Available at ArXiv:1704.07462.

\paragraph{Chapter 7.} A. A. Ahmadi, G. Hall, A. Makadia, A., and V. Sindhwani. \emph{Geometry of 3D environments and sum of squares polynomials} (2017). In the proceedings of Robotics: Science and Systems. 

\paragraph{Chapter 8.} A. A. Ahmadi, M. Curmei, G. Hall. \emph{Nonnegative polynomials and shape-constrained regression} (2018). In preparation.\\

\noindent In addition to these papers, the following papers were written during the graduate studies of the author but are not included in this thesis.\\

A. A. Ahmadi, G. Hall, A. Papachristodoulou, J. Saunderson, and Y. Zheng. \emph{Improving efficiency and scalability of sum of squares optimization:recent advances and limitations} (2017). In the proceedings of  the 56th Conference on Decision and Control. \\

E. Abbe, A. Bandeira, and G. Hall. \emph{Exact recovery in the stochastic block model} (2016). In IEEE Transactions on Information Theory, vol. 62, no. 1.

\part{LP, SOCP, and Optimization-Free Approaches to Semidefinite\\ and Sum of Squares Programming}

\chapter{Optimization over Structured Subsets of Positive Semidefinite Matrices via Column Generation} \label{ch:bp}

%\keywords{Column generation, conic optimization, sum of squares programming }

%============================================================
\section{Introduction}
%============================================================

Semidefinite programming is a powerful tool in optimization that is used in many different contexts, perhaps most notably to obtain strong bounds on discrete optimization problems or nonconvex polynomial programs. 
% including polynomial optimization, and obtaining bounds for combinatorial optimization problems, which we discuss in this chapter.
One difficulty in applying semidefinite programming is that state-of-the-art general-purpose solvers often cannot solve very large instances reliably and in a reasonable amount of time. As a result, at relatively large scales, one has to resort either to specialized solution techniques and algorithms that employ problem structure, or to easier optimization problems that lead to weaker bounds.
We will focus on the latter approach in this chapter.

%for problems where the associated semidefinite program (SDP) is simply intractable or unacceptably slow for current solvers.

At a high level, our goal is to not solve semidefinite programs (SDPs) to optimality, but rather replace them with cheaper conic relaxations---\emph{linear and second order cone relaxations} to be precise---that return useful bounds quickly. Throughout the chapter, we will aim to find lower bounds (for minimization problems); i.e., bounds that certify the distance of a candidate solution to optimality. Fast, good-quality lower bounds are especially important in the context of branch-and-bound schemes, where one needs to strike a delicate balance between the time spent on bounding and the time spent on branching, in order to keep the overall solution time low. Currently, in commercial integer programming solvers, almost all lower bounding approaches using branch-and-bound schemes exclusively produce linear inequalities.
%With the current state of affairs in branch-and-bound technology, almost all cutting-plane approaches that produce lower bounds on industrially-sized integer programs are exclusively based on linear programming (LP).
%, or occasionally second order cone programming (SOCP). 
Even though semidefinite cuts are known to be stronger, they are often too expensive to be used even at the root node of a branch-and-bound tree.
%at the root node of branch-and-bound techniques for integer programming. Because of this, 
%many practitioners are wary of using SDP and as a result
%
Because of this, many high-performance solvers, e.g., IBM ILOG CPLEX ~\cite{cplex} and Gurobi \cite{gurobi}, do not even provide an SDP solver and instead solely work with LP and SOCP relaxations. Our goal in this chapter is to offer some tools that exploit the power of SDP-based cuts, while staying entirely in the realm of LP and SOCP. We apply these tools to classical problems in both nonconvex polynomial optimization and discrete optimization. %[*Sanjeeb, confirm what I say about state of BB is OK.*]

%In the field of integer programming for instance, the cutting-plane approaches used on industrial problems are almost exclusively based on linear programming (LP) or second order cone programming (SOCP).

%In the field of sum of squares optimization, however, a sound alternative to sos programming that can avoid SDP and take advantage of the existing mature and high-performance LP/SOCP solvers is lacking. This is precisely what we propose to achieve here.

Techniques that provide lower bounds on minimization problems are precisely those that certify nonnegativity of objective functions on feasible sets. To see this, note that a scalar $\gamma$ is a lower bound on the minimum value of a function $f:\mathbb{R}^n\rightarrow\mathbb{R}$ on a set $K\subseteq\mathbb{R}^n,$ if and only if $f(x)-\gamma\geq 0$ for all $x\in K$. As most discrete optimization problems (including those in the complexity class NP) can be written as polynomial optimization problems, the problem of certifying nonnegativity of polynomial functions, either globally or on basic semialgebraic sets, is a fundamental one. A polynomial $p(x)\mathrel{\mathop:}=p(x_1,\ldots,x_n)$ is said to be \emph{nonnegative}, if $p(x)\geq0$ for all $x\in\mathbb{R}^n$. Unfortunately, even in this unconstrained setting, the problem of testing nonnegativity of a polynomial $p$ is NP-hard even when its degree equals four. This is an immediate corollary of the fact that checking if a symmetric matrix $M$ is copositive---i.e., if $x^TMx\geq 0, \forall x\geq 0$---is NP-hard.\footnote{Weak NP-hardness of testing matrix copositivity is originally proven by Murty and Kabadi~\cite{nonnegativity_NP_hard}; its strong NP-hardness is apparent from the work of de Klerk and Pasechnik~\cite{dp}.} Indeed, $M$ is copositive if and only if the homogeneous quartic polynomial $p(x)=\sum_{i,j}M_{ij}x_i^2x_j^2$ is nonnegative.

%A powerful sufficient condition for a polynomial to be nonnegative is for it to be a sum of squares of other

%The problem of minimizing a polynomial $p(x)$ where $x = (x_1, \ldots, x_n)$ is a vector of $n$ variables can be represented as the problem of finding the largest number $\lambda$ such that $p(x) - \lambda \geq 0$ for all $x \in \R^n$.
%The problem of verifying if a polynomial $p(x)$ in $n$ variables is nonnegative -- i.e., $p(x)$ satisfies $p(x) \geq 0$ for all $x \in \R^n$ -- is NP-complete for quartic polynomials (this follows from the result of Murty and Kabadi \cite{mk} showing NP-completeness of matrix copositivity verification), and can be hard to solve in practice when the degree is four or more.

Despite this computational complexity barrier, there has been great success in using sum of squares (SOS) programming~\cite{sdprelax},~\cite{lasserre_moment},~\cite{NesterovSquared} to obtain certificates of nonnegativity of polynomials in practical settings. It is known from Artin's solution~\cite{Artin_Hilbert17} to Hilbert's 17th problem that a polynomial $p(x)$ is nonnegative if and only if
\begin{equation}\label{eq:sos} p(x) = \frac{\sum_{i=1}^t q_i^2(x)}{\sum_{i=1}^r g_i^2(x)} \Leftrightarrow (\sum_{i=1}^r g_i^2(x))p(x) = \sum_{i=1}^t q_i^2(x)
\end{equation} for some polynomials $q_1, \ldots, q_t, g_1,\ldots,g_r$.
When $p$ is a quadratic polynomial, then the polynomials $g_i$ are not needed and the polynomials $ q_i$ can be assumed to be linear functions. In this case, by
writing $p(x)$ as
\[ p(x) = \left(\begin{array}{c} 1 \\ x \end{array}\right)^T Q \left(\begin{array}{c} 1 \\ x \end{array}\right), \]
where $Q$ is an $(n+1) \times (n+1)$ symmetric matrix, checking nonnegativity of $p(x)$
reduces to checking the nonnegativity of the eigenvalues of $Q$; i.e., checking if $Q$ is positive semidefinite.

More generally, if the degrees of $q_i$ and $ g_i$ are fixed in (\ref{eq:sos}), then checking for a representation of $p$ of the form in (\ref{eq:sos}) reduces to solving an SDP, whose size depends on the dimension of $x$, and the degrees of $p, q_i$ and $ g_i$ \cite{sdprelax}.
This insight has led to significant progress in certifying nonnegativity of polynomials arising in many areas. In practice, the ``first level'' of the SOS hierarchy is often the one used, where the polynomials $g_i$ are left out and one simply checks if $p$ is a sum of squares of other polynomials. In this case already, because of the numerical difficulty of solving large SDPs, the polynomials that can be certified to be nonnegative usually do not have very high degrees or very many variables. For example, finding a sum of squares certificate that a given quartic polynomial over $n$ variables is nonnegative requires solving an SDP involving roughly $O(n^4)$ constraints and a positive semidefinite matrix variable of size $O(n^2) \times O(n^2)$. Even for a handful of or a dozen variables, the underlying semidefinite constraints prove to be expensive. Indeed, in the absence of additional structure, most examples in the literature have less than 10 variables.

Recently other systematic approaches to certifying nonnegativity of polynomials have been proposed which lead to less expensive optimization problems than semidefinite programming problems. In particular, Ahmadi and Majumdar~\cite{isos_journal},~\cite{dsos_ciss14} introduce ``DSOS and SDSOS'' optimization techniques, which replace semidefinite programs arising in the nonnegativity certification problem by linear programs and second-order cone programs. Instead of optimizing over the cone of sum of squares polynomials, the authors optimize over two subsets which they call ``diagonally dominant sum of squares'' and ``scaled diagonally dominant sum of squares'' polynomials (see Section~\ref{subsec:dsos.sdsos} for formal definitions). In the language of semidefinite programming, this translates to solving optimization problems over the cone of diagonally dominant matrices and scaled diagonally dominant matrices. These can be done by LP and SOCP respectively. The authors have had notable success with these techniques in different applications. For instance, they are able to run these relaxations for polynomial optimization problems of degree 4 in 70 variables in the order of a few minutes. They have also used their techniques to push the size limits of some SOS problems in controls; examples include stabilizing a model of a humanoid robot with 30 state variables and 14 control inputs~\cite{majumdar2014control}, or exploring the real-time applications of SOS techniques in problems such as collision-free autonomous motion planning~\cite{ahmadiOR_letters}.

Motivated by these results, our goal in this chapter is to start with DSOS and SDSOS techniques and improve on them. By exploiting ideas from column generation in large-scale linear programming, and by appropriately interpreting the DSOS and SDSOS constraints, we produce several iterative LP and SOCP-based algorithms that improve the quality of the bounds obtained from the DSOS and SDSOS relaxations. Geometrically, this amounts to optimizing over structured subsets of sum of squares polynomials that are larger than the sets of diagonally dominant or scaled diagonally dominant sum of squares polynomials. For semidefinite programming, this is equivalent to optimizing over structured subsets of the cone of positive semidefinite matrices. An important distinction to make between the DSOS/SDSOS/SOS approaches and our approach, is that our approximations iteratively get larger in the direction of the given objective function, unlike the DSOS, SDSOS, and SOS approaches which all try to inner approximate the set of nonnegative polynomials \emph{irrespective} of any particular direction.

{ In related literature, Krishnan and Mitchell use linear programming techniques to approximately solve SDPs by taking a semi-infinite LP representation of the SDP and applying column generation \cite{krishnan2006semidefinite}. In addition, Kim and Kojima solve second order cone relaxations of SDPs which are closely related to the dual of an SDSOS program in the case of quadratic programming \cite{kim2003exact}; see Section \ref{sec:cg.overview} for further discussion of these two papers.}
	
The organization of the rest of the chapter is as follows.
In the next section, we review relevant notation, and discuss the prior literature on { DSOS and SDSOS programming.}
% two techniques which motivate our column generation ideas.
In Section~\ref{sec:cg.overview}, we give a high-level overview of our column generation approaches in the context of a general SDP.
In Section~\ref{sec:poly.opt}, we describe an application of our ideas to nonconvex polynomial optimization and {present} computational experiments with certain column generation implementations.
{In Section~\ref{sec:stable.set}, we apply our column generation approach to approximate a copositive program arising from a specific discrete optimization application (namely the stable set problem).}
All the work in these sections can be viewed as providing techniques to optimize over subsets of positive semidefinite matrices.
We then conclude in Section~\ref{sec:future} with some future directions, and discuss ideas for column generation which allow one to go beyond subsets of positive semidefinite matrices in the case of polynomial {and copositive } optimization.

%Instead of optimizing a linear function of a matrix variable subject to linear constraints and the restriction that the matrix is positive semidefinite, they replace the positive semidefiniteness condition by either linear constraints (DSOS programming) or second-order cone constraints (SDSOS programming) which are sufficient for positive semidefinitess but are not necessary.
%In other words, they optimize over subsets of the semidefinite cone which are computationally easier to handle. One subset is the set of diagonally dominant matrices, which can be specified via linear constraints on the matrix components, and another is the scaled diagonally dominant matrices, which need second-order cone constraints.
%They thereby solve much simpler optimization problems than SDPs but get weaker lower bounds on the minimum value of a polynomial.

%In this paper, we use the general technique of column generation used to solve very large LPs to extend the DSOS programming and SDSOS programming ideas in Ahmadi and Majumdar \cite{am}.
%In particular, we consider optimization problems defined over different classes of matrices which form strict subsets of all positive semidefinite matrices but contain the class of diagonally dominant matrices.
%We only solve linear programs, but these are harder than the ones solved in \cite{am} but give better bounds for polynomial optimization.
%We also study whether our ideas lead to better bounds compared to DSOS and SDSOS programming in the context of combinatorial optimization problems.

%============================================================
\section{Preliminaries}\label{sec:prelims}
%============================================================
Let us first introduce some notation on matrices. We denote the set of real symmetric $n\times n$ matrices by $S_n$. Given two matrices $A$ and $B$ in $S_n$, we denote their matrix inner product by $A\cdot B := \sum_{i,j}A_{ij}B_{ij} = \mbox{ Trace}(AB)$. The set of symmetric matrices with nonnegative entries is denoted by $N_n$. A symmetric matrix $A$ is \emph{positive semidefinite} (psd) if $x^TAx\geq 0$ for all $x\in\mathbb{R}^n$; this will be denoted by the standard notation $A\succeq 0$, and our notation for the set of $n\times n$ psd matrices is $P_n$. A matrix $A$ is \emph{copositive} if $x^TAx\geq 0$ for all $x\geq 0$. The set of copositive matrices is denoted by $C_n$. All three sets $N_n,P_n,C_n$ are convex cones and we have the obvious inclusion $N_n+P_n\subseteq C_n$. This inclusion is strict if $n\geq 5$~\cite{burer2009difference},~\cite{burer2012copositive}. For a cone $\mathcal{K}$ of matrices in $S_n$, we define its dual cone $\mathcal{K}^*$ as $\{Y \in \S_n: Y\cdot X \geq 0, \ \forall X \in \mathcal{K}\}$.

For a vector variable $x \in \R^n$ and a vector $q \in \Z^n_+$, let a monomial in $x$ be denoted as $x^q\mathrel{\mathop:}= \Pi_{i=1}^n x_i^{q_i}$, and let its degree be $\sum_{i=1}^n q_i$.
A polynomial is said to be \emph{homogeneous} or a \emph{form} if all of its monomials have the same degree. A form $p(x)$ in $n$ variables is nonnegative if $p(x)\geq 0$ for all $x\in\mathbb{R}^n$, or equivalently for all $x$ on the unit sphere in $\mathbb{R}^n$. The set of {\emph{nonnegative}} (or positive semidefinite) forms in $n$ variables and degree $d$ is denoted by $PSD_{n,d}$. A form $p(x)$ is a \emph{sum of squares} (sos) if it can be written as $p(x)=\sum_{i=1}^r q_i^2(x)$ for some forms $q_1,\ldots,q_r$. The set of sos forms in $n$ variables and degree $d$ is {a cone} denoted by $SOS_{n,d}$. We have the obvious inclusion $SOS_{n,d}\subseteq PSD_{n,d}$, which is strict unless $d=2$, or $n=2$, or $(n,d)=(3,4)$~\cite{Hilbert_1888}. Let $z(x,d)$ be the vector of all monomials of degree exactly $d$; it is well known that a form $p$ of degree $2d$ is sos if and only if it can be written as {$p(x)=z^T(x,d)Qz(x,d)$,} for some psd matrix $Q$~\cite{sdprelax},~\cite{PhD:Parrilo}. The size of the matrix $Q$, which is often called the \emph{Gram matrix}, is ${n+d-1 \choose d} \times {n+d-1 \choose d}$. At the price of imposing a semidefinite constraint of this size, one obtains the very useful ability to search and optimize over the convex cone of sos forms via semidefinite programming.

%Let $N_n, S_n, C_n$, stand for the cones of $n\times n$ symmetric nonnegative matrices, diagonally dominant matrices, positive semidefinite matrices (PSD cone) and copositive matrices, respectively.
%
%$D\in D_n$ is a diagonally dominant matrix if $D_{ii} \geq \sum_{j\neq i}|D_{ij}|$ for $i=1, \ldots, n$.
%For a cone of matrices $\mathcal{C} \subseteq \R^{n \times n}$, the dual cone $\mathcal{C}^*$ is defined as $\{Y \in \R^{n\times n}: Y\cdot X \geq 0 \ \forall X \in \mathcal{C}\}$.
%It is well-known that (i) $D_n \subseteq S_n \subseteq C_n$; (ii) $N_n \subseteq C_n$ which also implies that $S_n + N_n \subseteq C_n$; (iii) 
%$S_n^* = S_n$.
%Let $e_1, \ldots e_n$ be the unit vectors in $\R^n$. Let $U_k \subset \R^{n \times n}$ be the set of matrices defined as
%\[ U_k = \{uu^T : u \mbox{ has at most } k \mbox{ nonzero components, each equal to} \pm 1\}. \]
%Clearly $U_k$ is a finite set for each $k=1, \ldots, n$ and it is not hard to see that $D_n = \cone{U_2}$ and therefore $D_n^* = \{X \in \R^{n\times n}: X \cdot V \geq 0 \ \forall V \in U_2\}$.
%Furthermore as $D_n \subseteq S_n$, we have $S_n \subseteq D_n^*$.

\subsection{DSOS and SDSOS optimization}
\label{subsec:dsos.sdsos}

In order to alleviate the problem of scalability posed by the SDPs arising from sum of squares programs, Ahmadi and Majumdar~\cite{isos_journal},~\cite{dsos_ciss14}\footnote{The work in~\cite{isos_journal} is currently in preparation for submission; the one in~\cite{dsos_ciss14} is a shorter conference version of~\cite{isos_journal} which has already appeared. The presentation of the current chapter is meant to be self-contained.} recently introduced similar-purpose LP and SOCP-based optimization problems that they refer to as \emph{DSOS and SDSOS programs}. Since we will be building on these concepts, we briefly review their relevant aspects to make our chapter self-contained.

%In order to address the problem of scalability posed by SDP, we have recently introduced~\cite{isos_journal},~\cite{dsos_ciss14} alternatives to SOS programming that lead to linear programs (LPs) and second order cone programs (SOCPs). 

The idea in~\cite{isos_journal},~\cite{dsos_ciss14}  is to replace the condition that the Gram matrix $Q$ be positive semidefinite with stronger but cheaper conditions in the hope of obtaining more efficient inner approximations to the cone $SOS_{n,d}$. Two such conditions come from the concepts of \emph{diagonally dominant} and \emph{scaled diagonally dominant} matrices in linear algebra. We recall these definitions below.

\begin{definition}\label{def:dd.sdd}
A symmetric matrix $A=(a_{ij})$ is \emph{diagonally dominant} (dd) if $a_{ii} \geq \sum_{j \neq i} |a_{ij}|$ for all $i$. We say that $A$ is \emph{scaled diagonally dominant} (sdd) if there exists a diagonal matrix $D$, with positive diagonal entries, such that $DAD$ is diagonally dominant.
\end{definition}

We refer to the set of $n \times n$ dd (resp. sdd) matrices as $DD_n$ (resp. $SDD_n$). The following inclusions are a consequence of Gershgorin's circle theorem:
$$DD_n\subseteq SDD_n\subseteq P_n.$$

We now use these matrices to introduce the cones of ``dsos'' and ``sdsos'' forms and some of their generalizations, which all constitute special subsets of the cone of nonnegative forms. We remark that in the interest of brevity, we do not give the original definitions of dsos and sdsos polynomials as they appear in~\cite{isos_journal} (as sos polynomials of a particular structure), but rather an equivalent characterization of them that is more useful for our purposes. The equivalence is proven in~\cite{isos_journal}. 

%We now introduce some cones that are inner approximations of the cone of nonnegative polynomials and that lend themselves to LP and SOCP. In analogy with the representation of sos polynomials in terms of psd matrices (Theorem \ref{thm:sos.sdp}), we define the \emph{dsos} and \emph{sdsos} polynomials in terms of dd and sdd matrices respectively.

 \begin{definition}[\cite{isos_journal,dsos_ciss14}] \label{def:dsos.sdsos.rdsos.rsdsos}
	Recall that $z(x,d)$ denotes the vector of all monomials of degree exactly $d$. A form $p(x)$ of degree $2d$ is said to be
	
	\begin{enumerate}[(i)]
		\item  \emph{diagonally-dominant-sum-of-squares} (dsos) if it admits a representation as\\ $p(x)=z^T(x,d)Qz(x,d)$, where $Q$ is a dd matrix,
		\item  \emph{scaled-diagonally-dominant-sum-of-squares} (sdsos) if it admits a representation as\\ $p(x)=z^T(x,d)Qz(x,d)$, where $Q$ is an sdd matrix,
		\item \emph{$r$-diagonally-dominant-sum-of-squares} ($r$-dsos) if there exists a positive integer $r$ such that \\$p(x) (\sum_{i=1}^n x_i^2)^r$ is dsos,
			\item \emph{$r$-scaled diagonally-dominant-sum-of-squares} ($r$-sdsos) if there exists a positive integer $r$ such that \\$p(x)  (\sum_{i=1}^n x_i^2)^r$ is sdsos.
	\end{enumerate}
\end{definition}
We denote the cone of forms in $n$ variables and degree $d$ that are dsos, sdsos, $r$-dsos, and $r$-sdsos by $DSOS_{n,d}$, $SDSOS_{n,d}$, $rDSOS_{n,d}$, and $rSDSOS_{n,d}$ respectively. The following inclusion relations are straightforward: $$DSOS_{n,d}\subseteq SDSOS_{n,d}\subseteq SOS_{n,d}\subseteq PSD_{n,d},$$
$$rDSOS_{n,d}\subseteq rSDSOS_{n,d}\subseteq PSD_{n,d}, \forall r.$$

The multiplier $(\sum_{i=1}^n x_i^2)^r$ should be thought of as a special denominator in the Artin-type representation in (\ref{eq:sos}). By appealing to some theorems of real algebraic geometry, it is shown in~\cite{isos_journal} that under some conditions, as the power $r$ increases, the sets $rDSOS_{n,d}$ (and hence $rSDSOS_{n,d}$) fill up the entire cone $PSD_{n,d}.$ We will mostly be concerned with the cones $DSOS_{n,d}$ and $SDSOS_{n,d}$, which correspond to the case where $r=0$. From the point of view of optimization, our interest in all of these algebraic notions stems from the following {theorem.}

\begin{theorem}[\cite{isos_journal,dsos_ciss14}]	For any integer $r\geq 0$, the cone $rDSOS_{n,d}$ is polyhedral and the cone $rSDSOS_{n,d}$ has a second order cone representation. Moreover, for any fixed $d$ and $r$, {one can optimize a linear function over $rDSOS_{n,d}$ (resp. $rSDSOS_{n,d}$) by solving a linear program (resp. second order cone program) of size polynomial in $n$.}
% can be done with linear programming (resp. second order cone programming) of size polynomial in $n$.
\end{theorem}

The ``LP part'' of this theorem is not hard to see. The equality $p(x)(\sum_{i=1}^n x_i^2)^r={ z^T(x,d)}Qz(x,d)$ gives rise to linear equality constraints between the coefficients of $p$ and the entries of the matrix $Q$ (whose size is polynomial in $n$ for fixed $d$ and $r$). The requirement of diagonal dominance on the matrix $Q$ can also be described by linear inequality constraints on $Q$. The ``SOCP part'' of the statement comes from the fact, shown in~\cite{isos_journal}, that a matrix $A$ is sdd if and only if it can be expressed as 	$$A = \sum_{i< j} M_{2 \times 2}^{ij},$$
where each $ M_{2 \times 2}^{ij}$ is an $n\times n$ symmetric matrix with zeros everywhere except for four entries $M_{ii}, M_{ij}, M_{ji}, M_{jj}$, which must make the $2\times 2$ matrix $\begin{bmatrix} M_{ii} & M_{ij}  \\ M_{ji} & M_{jj} \end{bmatrix}$ symmetric and positive semidefinite. These constraints are \emph{rotated quadratic cone} constraints and can be imposed using SOCP~\cite{alizadeh},~\cite{socp_boyd}:
$$M_{ii}\geq 0, ~\Bigl\lvert\Bigl\lvert\begin{pmatrix}
2M_{ij}\\M_{ii}-M_{jj}
\end{pmatrix}\Bigl\lvert\Bigl\lvert \leq M_{ii}+M_{jj}.$$

We refer to optimization problems with a linear objective posed over the convex cones $DSOS_{n,d}$, $SDSOS_{n,d}$, and $SOS_{n,d}$ as DSOS programs, SDSOS programs, and SOS programs respectively. In general, quality of approximation decreases, while scalability increases, as we go from SOS to SDSOS to DSOS programs. Depending on the size of the application at hand, one may choose one approach over the other. \\
{ In related work, Ben-Tal and Nemirovski \cite{ben2001polyhedral} and Vielma, Ahmed and Nemhauser \cite{vielma2010mixed} approximate SOCPs by LPs and produce approximation guarantees.}
%{\ghtwo We remark that in related work, the idea of approximating more complicated convex programs by computationally cheaper ones has been considered by other authors. See, e.g., the work of Ben-Tal and Nemirovski \cite{ben2001polyhedral} and the work of Vielma, Ahmed and Nemhauser \cite{vielma2010mixed}.}

%In this chapter, we will be using SOS optimization (Section~\ref{sec:jamming}) and SDSOS optimization (Sections~\ref{sec:barriers} and~\ref{sec:quadrotor}) in our numerical experiments. The reader is referred to~\cite{dsos_cdc14,dsos_ciss14,Ahmadi14} for many numerical examples involving DSOS optimization. We also remark in passing that SDSOS or even DSOS programming enjoy many of the same theoretical (asymptotic) guarantees of SOS programming---results of this nature are proven in~\cite{Ahmadi14}.

%============================================================
\section{Column generation for inner approximation of positive semidefinite cones}\label{sec:cg.overview}
%============================================================

%Column generation~\cite{barnhart1998branch},~\cite{desaulniers2006column} is a technique employed extensively in large-scale, industrial linear and integer programming problems. The idea is to use only a subset of the optimization variables (or dual constraints) and bring in new ones only if they improve the objective function. These algorithms produce a sequence of iterative LPs, often with very smart warm-start and resolve strategies. But what does this have to with nonnegative polynomials?

In this section, we describe a natural approach to apply techniques from the theory of column generation~\cite{barnhart1998branch},~\cite{desaulniers2006column} in large-scale { optimization} %programming 
to the problem of optimizing over nonnegative polynomials.  Here is the rough idea: We can think of all SOS/SDSOS/DSOS approaches as ways of proving that a polynomial is nonnegative by writing it as a nonnegative linear combination of certain ``atom'' polynomials that are already known to be nonnegative. For SOS, these atoms are all the squares (there are infinitely many). For DSOS, there is actually a finite number of atoms {corresponding to the extreme rays of the cone of diagonally} dominant matrices (see Theorem~\ref{thm:dsos.corners} below). For SDSOS, once again we have infinitely many atoms, but with a specific structure which is amenable to an SOCP representation. Now the column generation idea is to start with a certain ``cheap'' subset of atoms (columns) and only add new ones---one or a limited number in each iteration---if they improve our desired objective function. This results in a sequence of monotonically improving bounds; we stop the column generation procedure when we are happy with the quality of the bound, or when we have consumed a predetermined budget on time.

In the LP case, after the addition of one or a few new atoms, one can obtain the new optimal solution from the previous solution in much less time than required to solve the new problem from scratch. However, as we show with some examples in this chapter, even if one were to resolve the problems from scratch after each iteration (as we do for all of our SOCPs and some of our LPs), the overall procedure is still relatively fast. This is because in each iteration, with the introduction of a constant number $k$ of new atoms, the problem size essentially increases only by $k$ new variables and/or $k$ new constraints. This is in contrast to other types of hierarchies---such as the rDSOS and rSDSOS hierarchies of Definition~\ref{def:dsos.sdsos.rdsos.rsdsos}---that blow up in size by a factor that depends on the dimension in each iteration. 

In the next two subsections we make this general idea more precise. While our focus in this section is on column generation for general SDPs, the next two sections show how the techniques are useful for approximation of SOS programs for polynomial optimization (Section~\ref{sec:poly.opt}), and copositive programs for discrete optimization (Section~\ref{sec:stable.set}). 

\subsection{LP-based column generation}\label{subsec:LP.based.CG}
Consider a general SDP
\begin{equation}\label{eq:sdp}
\begin{aligned} 
\max_{y\in\mathbb{R}^m} &  \quad  b^Ty\\
\text{s.t. } & \quad C-\sum_{i=1}^m y_iA_i\succeq 0,
\end{aligned}
\end{equation} 
with $b\in\mathbb{R}^m, C,A_i\in S_n$ as input, and its dual

\begin{equation}\label{eq:dual.sdp}
\begin{aligned} 
\min_{X\in S_n} &  \quad  C\cdot X\\
\text{s.t. } & \quad A_i\cdot X=b_i, \ i=1,\ldots, m,\\
\  & \quad X\succeq 0.\\
\end{aligned}
\end{equation} 

Our goal is to inner approximate the feasible set of (\ref{eq:sdp}) by increasingly larger polyhedral sets. We consider LPs of the form

\begin{equation}\label{eq:LP.inner}
\begin{aligned} 
\max_{y,\alpha} &  \quad  b^Ty\\
\text{s.t. } & \quad C-\sum_{i=1}^m y_iA_i=\sum_{i=1}^t \alpha_i B_i,\\
\ & \quad \alpha_i\geq 0,\ i=1,\ldots,t.
\end{aligned}
\end{equation} 

Here, the matrices $B_1,\ldots, B_t\in P_n$ are some fixed set of positive semidefinite matrices (our psd ``atoms''). To expand our inner approximation, we will continually add to this list of matrices. This is done by considering the dual LP

\begin{equation}\label{eq:dual.of.LP.inner}
\begin{aligned} 
\min_{X\in S_n} &  \quad  C\cdot X\\
\text{s.t. } & \quad A_i\cdot X=b_i, \ i=1,\ldots, m,\\
\  & \quad X\cdot B_i \geq 0, \ i=1,\ldots, t,
\end{aligned}
\end{equation} 
which in fact gives a polyhedral outer approximation (i.e., relaxation) of the spectrahedral feasible set of the SDP in (\ref{eq:dual.sdp}). If the optimal solution $X^*$ of the LP in (\ref{eq:dual.of.LP.inner}) is already psd, then we are done and have found the optimal value of our SDP. If not, we can use the violation of positive semidefiniteness to extract one (or more) new psd atoms $B_{j}$. Adding such atoms to (\ref{eq:LP.inner}) is called {\em column generation}, and the problem of finding such atoms is called the {\em pricing subproblem}. (On the other hand, if one starts off with an LP of the form (\ref{eq:dual.of.LP.inner}) as an approximation of (\ref{eq:dual.sdp}), then the approach of adding inequalities to the LP iteratively that are violated by the current solution is called a {\em cutting plane} approach, and the associated problem of finding violated constraints is called the {\em separation subproblem}.)
The simplest idea for pricing is to look at the eigenvectors $v_j$ of $X^*$ that correspond to negative eigenvalues. From each of them, one can generate a rank-one psd atom $B_j=v_jv_j^T$, which can be added with a new variable (``column'') $\alpha_j$ to the primal LP in (\ref{eq:LP.inner}), and as a new constraint (``cut'') to the dual LP in (\ref{eq:dual.of.LP.inner}). The subproblem can then be defined as getting the most negative eigenvector, which is equivalent to minimizing the quadratic form $x^TX^*x$ over the unit sphere $\{x|\ ||x||=1\}$. Other possible strategies are discussed later in the chapter.

This LP-based column generation idea is rather straightforward, but what does it have to do with DSOS optimization? The connection comes from the extreme-ray description of the cone of diagonally dominant matrices, which allows us to interpret a DSOS program as a particular and effective way of obtaining $n^2$ initial psd atoms.
%To make these ideas more concrete, we start by giving the extreme-ray description of the cone of diagonally dominant matrices. 

Let $\mathcal{U}_{n,k}$ denote the set of vectors in $\mathbb{R}^n$ which have at most $k$ nonzero components, each equal to $\pm 1$, and define $U_{n,k} \subset S_n$ to be the set of matrices 
%\[ U_{n,k} = \{uu^T : u\in\mathbb{R}^n  \mbox{ has at most } k \mbox{ nonzero components, each %equal to} \pm 1\}. \]
\[ U_{n,k} \mathrel{\mathop:}= \{uu^T : u\in\mathcal{U}_{n,k}\}. \]

For a finite set of matrices $T = \{T_1, \ldots, T_t\}$, let $$\cone{T} \mathrel{\mathop:}= \{\sum_{i=1}^t \alpha_i T_i : \alpha_1, \ldots, \alpha_t \geq 0\}.$$ 

%Clearly $U_k$ is a finite set for each $k=1, \ldots, n$ and it is not hard to see that $D_n = \cone{U_2}$ and therefore $D_n^* = \{X \in \R^{n\times n}: X \cdot V \geq 0 \ \forall V \in U_2\}$.

\begin{theorem}[Barker and Carlson~\cite{dd_extreme_rays}]\label{thm:dsos.corners}
$DD_n = \cone{U_{n,2}}.$
\end{theorem}
This theorem tells us that $DD_n$ has exactly $n^2$ extreme rays. It also leads to a convenient representation of the dual cone: $$DD_n^* = \{X \in S_n: v_i^TXv_i \geq 0, \ \mbox{for all vectors $v_i$ with at most 2 nonzero components, each equal to $\pm$1} \}.$$

Throughout the chapter, we will be initializing our LPs with the DSOS bound; i.e., our initial set of psd atoms $B_i$ will be the $n^2$ rank-one matrices $u_iu_i^T$ in $U_{n,2}$. This is because this bound is often cheap and effective. Moreover, it guarantees feasibility of our initial LPs (see Theorems~\ref{thm:polyopt.dsos.bound.finite} and \ref{thm:DSOSfinite}), which is {important} for starting column generation. One also readily sees that the DSOS bound can be improved if we were to instead optimize over the cone $U_{n,3}$, which has $n^3$ atoms. However, in settings that we are interested in, we cannot afford to include all these atoms; instead, we will have pricing subproblems that try to pick a useful subset (see Section~\ref{sec:poly.opt}).

We remark that an LP-based column generation idea similar to the one in this section is described in \cite{krishnan2006semidefinite}, where it is used as a subroutine for solving the maxcut problem. The method is comparable to ours inasmuch { as some columns are generated using the eigenvalue pricing subproblem.  However, contrary to us, additional columns specific to max cut are also added to the primal.} The initialization step is also differently done, as the matrices $B_i$ in (\ref{eq:LP.inner}) are initially taken to be in $U_{n,1}$ and not in $U_{n,2}$. (This is equivalent to requiring the matrix $C-\sum_{i=1}^m y_iA_i$ to be diagonal instead of diagonally dominant in (\ref{eq:LP.inner}).)

{ Another related work is \cite{sherali2002enhancing}. In this chapter, the initial LP relaxation is obtained via RLT (Reformulation-Linearization Techniques) as opposed to our diagonally dominant relaxation. The cuts are then generated by taking vectors which violate positive semidefiniteness of the optimal solution as in (\ref{eq:dual.of.LP.inner}). The separation subproblem that is solved though is different than the ones discussed here and relies on an $LU$ decomposition of the solution matrix. }

\subsection{SOCP-based column generation}\label{subsec:socp.based.CG}

In a similar vein, we present an SOCP-based column generation algorithm that in our experience often does much better than the LP-based approach. The idea is once again to optimize over structured subsets of the positive semidefinite cone that are SOCP representable and that are larger than the set $SDD_n$ of scaled diagonally dominant matrices. This will be achieved by working with the following SOCP

\begin{equation}\label{eq:SOCP.inner} {
\begin{aligned} 
\max_{y\in\mathbb{R}^m,a_i^j} &  \quad  b^Ty\\
\text{s.t. } & \quad C-\sum_{i=1}^m y_iA_i=\sum_{i=1}^t V_i \begin{pmatrix} a_i^1 & a_i^2 \\ a_i^2 & a_i^3 \end{pmatrix} V_i^T,\\
\ & \quad \begin{pmatrix} a_i^1 & a_i^2 \\ a_i^2 & a_i^3 \end{pmatrix} \succeq 0,\ i=1,\ldots,t.
\end{aligned}}
\end{equation} 

Here, the positive semidefiniteness constraints on the $2 \times 2$ matrices can be imposed via rotated quadratic cone constraints as explained in Section~\ref{subsec:dsos.sdsos}. {The $n\times 2$ matrices $V_i$ are fixed for all $i=1, \ldots, t$.} Note that this is a direct generalization of the LP in (\ref{eq:LP.inner}), in the case where the atoms $B_i$ are rank-one. To generate a new SOCP atom, we work with the dual of (\ref{eq:SOCP.inner}):

\begin{equation}\label{eq:dual.of.SOCP.inner}
\begin{aligned} 
\min_{X\in S_n} &  \quad  C\cdot X\\
\text{s.t. } & \quad A_i\cdot X=b_i, \ i=1,\ldots, m,\\
\  & \quad V_i^TXV_i \succeq 0, \ i=1,\ldots, t.
\end{aligned}
\end{equation} 

Once again, if the optimal solution $X^*$ is psd, we have solved our SDP exactly; if not, we can use $X^*$ to produce new SOCP-based cuts. For example, by placing the two eigenvectors of $X^*$ corresponding to its two most negative eigenvalues as the columns of an $n\times 2$ matrix $V_{t+1}$, we have produced a new useful atom. (Of course, we can also choose to add more pairs of eigenvectors and add multiple atoms.) As in the LP case, by construction, our bound can only improve in every iteration. 

%{\ghprime In fact, in some cases, the optimal value obtained via the dual SOCP relaxation given in (\ref{eq:dual.of.SOCP.inner}) exactly coincides with the optimal value obtained via the SDP dual given in (\ref{eq:dual.sdp}). See, e.g., \cite{kim2003exact} where it is shown that for a particular class of QCQPs (Quadratically Constrained Quadratic Programs), both the SDP relaxation and the SOCP relaxation are exact.}

We will always be initializing our SOCP iterations with the SDSOS bound. It is not hard to see that this corresponds to the case where we have ${n \choose 2}$ initial $n \times 2$ atoms $V_i$, which have zeros everywhere, except for a 1 in the first column in position $j$ and a 1 in the second column in position { $k> j$}. We denote the set of all such $n\times 2$ matrices by $\mathcal{V}_{n,2}$.

The {first} step of our procedure is carried out already in  \cite{kim2003exact} for approximating solutions to QCQPs. Furthermore, the work in \cite{kim2003exact} shows that for a particular class of QCQPs, its SDP relaxation and its SOCP relaxation (written respectively in the form of (\ref{eq:dual.sdp}) and (\ref{eq:dual.of.SOCP.inner})) are exact.

% $n(n-1)$ initial $n \times 2$ atoms $V_i$, with each column containing exactly one 1, in any position.  (The matrices where both 1s are positioned on the same row are redundant as they can be obtained as a linear combination of other matrices in the set.) 

\begin{figure}[h]
	\begin{center}
		\mbox{
			\subfigure[LP starting with DSOS and adding 5 atoms.]
			{\label{subfig:dsos.iters}\scalebox{0.5}{\includegraphics{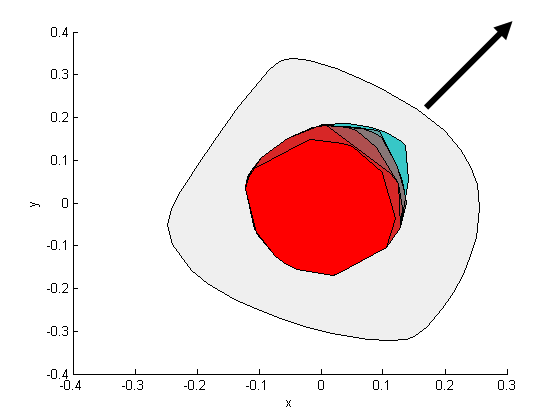}}}}
		\mbox{
			\subfigure[SOCP starting with SDSOS and adding 5 atoms.]
			{\label{subfig:sdsos.iters}\scalebox{0.5}{\includegraphics{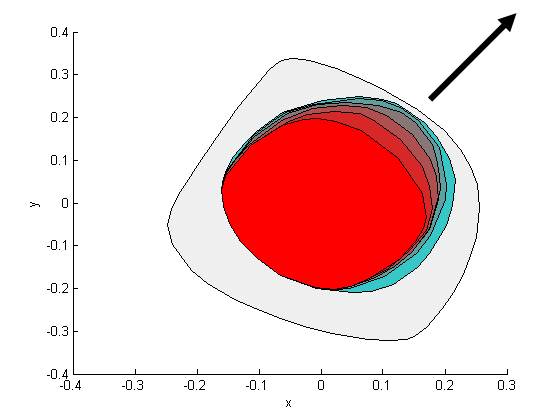}}}
		}
		
		\caption{LP and SOCP-based column generation for inner approximation of a spectrahedron.}
		\label{fig:dsos.sdsos.5atoms}
	\end{center}
	\vspace{-20pt}
\end{figure}

Figure~\ref{fig:dsos.sdsos.5atoms} shows an example of both the LP and SOCP column generation procedures. We produced two $10\times 10$ random symmetric matrices $E$ and $F$. The outer most set is the feasible set of an SDP with the constraint $I+xE+yF\succeq 0.$ (Here, $I$ is the $10\times 10$ identity matrix.) The SDP wishes to maximize $x+y$ over this set. The innermost set in Figure~\ref{subfig:dsos.iters} is the polyhedral set where $I+xE+yF$ is dd. The innermost set in Figure~\ref{subfig:sdsos.iters} is the SOCP-representable set where $I+xE+yF$ is sdd. In both cases, we do 5 iterations of column generation that expand these sets by introducing one new atom at a time. These atoms come from the most negative eigenvector (resp. the two most negative eigenvectors) of the dual optimal solution as explained above. Note that in both cases, we are growing our approximation of the positive semidefinite cone in the direction that we care about (the northeast). This is in contrast to algebraic hierarchies based on ``positive multipliers'' (see the rDSOS and rSDSOS hierarchies in Definition~\ref{def:dsos.sdsos.rdsos.rsdsos} for example), which completely ignore the objective function.

\section{Nonconvex polynomial optimization}\label{sec:poly.opt}

In this section, we apply the ideas described in the previous section to sum of squares algorithms for nonconvex polynomial optimization. In particular, we consider the NP-hard problem of minimizing a form (of degree $\geq 4$) on the sphere. Recall that $z(x,d)$ is the vector of all monomials in $n$ variables with degree $d$.
Let $p(x)$ be a form with $n$ variables and even degree $2d$, and let $\coef{p}$ be the vector of its coefficients with the monomial ordering given by $z(x,2d)$.
% the vector with the same dimension as $z(x,2d)$ defined as follows: if $p(x)$ contains a monomial, then the corresponding component of $\coef{p}$ has a value equal to the coefficient of the monomial.
Thus $p(x)$ can be viewed as $\coef{p}^T z(x,2d)$. Let $s(x) \mathrel{\mathop:}= (\sum_{i=1}^n x_i^2)^d$. With this notation, the problem of minimizing a form $p$ on the unit sphere can be written as
% in this section that we are trying to optimize a form $p(x)$ (with $n$ variables and even degree $2d$) over the unit sphere (minimizing a general polynomial over $\R^n$ is equivalent to this problem) which is equivalent to 

\begin{eqnarray} &\underset{\lambda}{\max} &\lambda \nonumber \\ 
	&\textup{s.t.} & p(x) - \lambda s(x) \geq 0, \forall x\in\mathbb{R}^n. \label{eq:max.lambda}
\end{eqnarray}
%This problem is equivalent to minimizing a general (non-homogeneous) polynomial over $\mathbb{R}^n$.
With the SOS programming approach, the following  SDP is solved to get the largest scalar $\lambda$ and an SOS certificate proving that $p(x) -\lambda s(x)$ is nonnegative: 
\begin{eqnarray} &\underset{\lambda,Y}{\max} & \lambda \nonumber\\
	& \textup{s.t.} &  p(x) - \lambda s(x) = z^T(x,d) Y z(x,d), \label{eq:SDPform}\\
	&&  Y \succeq 0. \nonumber
\end{eqnarray}
The sum of squares certificate is directly read from an eigenvalue decomposition of the solution $Y$ to the SDP above and has the form $$p(x)-\lambda s(x) \geq \sum_{i} (z^T(x,d)u_i)^2,$$ where
$Y = \sum_{i} u_iu_i^T$. 
%If $Y \succeq 0$ does not hold, we do not have a valid SOS certificate.
Since all sos polynomials are nonnegative, the optimal value of the SDP in (\ref{eq:SDPform}) is a lower bound to the optimal value of the optimization problem in (\ref{eq:max.lambda}). Unfortunately, before solving the SDP, we do not have access to the vectors $u_i$ in the decomposition of the optimal matrix $Y$. However, the fact that such vectors exist hints at how we should go about replacing $P_n$ by a polyhedral restriction in (\ref{eq:SDPform}): If the constraint $Y \succeq 0$ is changed to %On the other hand, if we replace the constraint $Y \succeq 0$ by 
\begin{equation}\label{eq:Y=sum.alpha.u.u'}
{ Y = \sum_{u \in \U} \alpha_u uu^T, \alpha_u \geq 0},
\end{equation}
where {$\U$ is a finite set}, then (\ref{eq:SDPform}) becomes an LP.
%then for any finite $\U$, we replace $P_n$ by a polyhedral restriction. 
This is one interpretation of Ahmadi and Majumdar's work in \cite{isos_journal,dsos_ciss14} where they replace $P_n$ by $DD_n$. Indeed, this is equivalent to taking $\U=\U_{n,2}$ in (\ref{eq:Y=sum.alpha.u.u'}), as shown in Theorem \ref{thm:dsos.corners}. 
We are interested in extending their results by replacing $P_n$ by larger restrictions than $DD_n$. A natural candidate for example would be obtained by changing $\U_{n,2}$ to $\U_{n,3}$. However, although $\U_{n,3}$ is finite, it contains a very large set of vectors even for small values of $n$ and $d$. For instance, when $n=30$ and $d = 4$, $\U_{n,3}$ has over 66 million elements.
Therefore we use column generation ideas to iteratively expand $\U$ in a manageable fashion. To initialize our procedure, {we would like to start} with good enough atoms to have a feasible LP. The following result guarantees that replacing $Y \succeq 0$ with $Y \in DD_n$ always yields an initial feasible LP in the setting that we are interested in.
\begin{theorem}\label{thm:polyopt.dsos.bound.finite}
For any form $p$ of degree $2d$, there exists $\lambda\in\mathbb{R}$ such that $p(x)-\lambda (\sum_{i=1}^n x_i^2)^d$ is dsos.
\end{theorem}
\begin{proof}
  As before, let $s(x) = (\sum_{i=1}^n x_i^2)^d$.
We observe that the form $s(x)$ is strictly in the interior of $DSOS_{n,2d}$. Indeed, by expanding out the expression we see that we can write $s(x)$ as $z^T(x,d)Qz(x,d)$, where $Q$ is a diagonal matrix with all diagonal entries positive. So $Q$ is in the interior of $DD_{{n+d-1 \choose d}}$, and hence $s(x)$ is in the interior of $DSOS_{n,2d}$. { This implies that for $\alpha>0$ small enough, the form $$(1-\alpha)s(x)+\alpha p(x)$$ will be dsos. Since $DSOS_{n,2d}$ is a cone, the form $$\frac{(1-\alpha)}{\alpha}s(x)+ p(x)$$ will also be dsos. By taking $\lambda$ to be smaller than or equal to} $-\frac{1-\alpha}{\alpha}$, the claim is established. 
\end{proof}

As $DD_n\subseteq SDD_n$, the theorem above implies that replacing $Y \succeq 0$ with $Y \in SDD_n$ also yields an initial feasible SOCP. Motivated in part by this theorem, we will always start our LP-based iterative process with the restriction that $Y \in DD_n$. Let us now explain how we improve on this approximation via column generation.

Suppose we have a set $\U$ of vectors in $\mathbb{R}^n$, whose outerproducts form all of the rank-one psd atoms that we want to consider. This set could be finite but very large, or even infinite. For our purposes $\U$ always includes $\U_{n,2}$, as we initialize our algorithm with the dsos relaxation. Let us consider first the case where $\U$ is finite: $\U = \{u_1, \ldots, u_t\}.$ Then the problem that we are interested in solving is
 %and add elements to $\U$  in each iteration (which in the first iteration corresponds to the set of all vectors such that their outerproduct is in $U_{n,2}$), via column generation. 
%Let us describe this last step more precisely. 
%$\cone{U_{n,3}}$ (and other larger cones) via column generation.
%Thus by solving a series of LPs, we can get an SOS certificate that $p(x) - \lambda s(x)$ for some $\lambda$. We now make this idea more precise.

%Let $z(x,2d)$ have $m$ monomials and consider the problem
\begin{eqnarray*} &\underset{\lambda,\alpha_j}{\max} & \lambda \\
	& \textup{s.t.} &  p(x) - \lambda s(x) = z^T(x,d) Y z(x,d), \\
	&&  Y = \sum_{j=1}^t\alpha_j u_ju_j^T, \  \alpha_j \geq 0 \textup{ for } j=1, \ldots, t.
\end{eqnarray*}
Suppose $z(x,2d)$ has $m$ monomials and let the $i$th monomial in $p(x)$ have coefficient $b_i$, i.e., $\coef{p}= (b_1, \ldots, b_m)^T$.
%
% Let $A_i$ be the matrix such that $A_i \cdot Y = b_i$ implies that $
%z^T(x,d) Y z(x,d)$ has the coefficient $b_i$ for the $i$th monomial in $z(x,2d)$.
Also let $s_i$ be the $i$th entry in $\coef{s(x)}$. We rewrite the previous problem as
%To get the largest value of $\lambda$ such that $p(x) - \lambda$ has an SOS representation, we can solve the problem:
\begin{eqnarray*} &\underset{\lambda,\alpha_j}{\max} & \lambda \\
	& \textup{s.t.} &  A_i \cdot Y + \lambda s_i = b_i \textup{ for } i =1, \ldots, m,\\
	&&  Y = \sum_{j=1}^t \alpha_j u_ju_j^T, \  \alpha_j \geq 0 \textup{ for } j=1, \ldots, t.
\end{eqnarray*}
where $A_i$ is a matrix that collects entries of $Y$ {that contribute to the $i^{th}$ monomial in $z(x,2d)$,} when $z^T(x,d) Y z(x,d)$ is expanded out. 
%Letting $Y_j = u_ju_j^T$, the above is equivalent to
{ The above is equivalent to}
\begin{eqnarray} &\underset{\lambda,\alpha_j}{\max} & \lambda  \nonumber \\
	& \textup{s.t.}&  \sum_j \alpha_j(A_i \cdot {u_ju_j^T}) + \lambda s_i = b_i \ \textup{ for } i=1, \ldots, m, \label{eq:polycol}\\ 
	&& \alpha_j \geq 0 \textup{ for } j=1, \ldots, t. \nonumber
\end{eqnarray}
The dual problem is 
\begin{eqnarray*} &\underset{\mu}{\min} & \sum_{i=1}^m \mu_ib_i  \\
	& \textup{s.t.}&  (\sum_{i=1}^m \mu_iA_i) \cdot{ u_ju_j^T} \geq 0,\ j=1,\ldots,t, \\
	&&  \sum_{i=1}^m \mu_is_i = 1.
\end{eqnarray*}
In the column generation framework, suppose we consider only a subset of the primal LP variables corresponding to  the matrices ${u_1u_1^T}, \ldots, {u_ku_k^T}$ for some $k < t$ (call this the reduced primal problem).
Let $(\bar\alpha_1, \ldots, \bar\alpha_k)$ stand for an optimal solution of the reduced primal problem and let $\bar \mu = (\bar\mu_1, \ldots, \bar\mu_m)$ stand for an optimal dual solution. If we have 
\begin{equation}\label{eq:sepsub}  (\sum_{i=1}^m \bar{\mu}_iA_i) \cdot {u_ju_j^T} \geq 0 \textup{ for } j=k+1, \ldots, t,\end{equation}
then $\bar\mu$ is an optimal dual solution for the original larger primal problem with columns $1,\ldots,t$. In other words, if we simply set $\alpha_{k+1} = \cdots = \alpha_t = 0$, then the solution of the reduced primal problem becomes a solution of the original primal problem.
On the other hand, if (\ref{eq:sepsub}) is not true, then suppose the condition is violated for some ${u_lu_l^T}$.
We can augment the reduced primal problem by adding the variable $\alpha_l$, and repeat this process.

% under some nondegeneracy assumptions, the solution of this augmented problem will differ from the solution of the previous reduced primal problem and we can repeat this process till (\ref{eq:sepsub}) is true.

Let $B = \sum_{i=1}^m \bar{\mu}_iA_i$.
We can test if (\ref{eq:sepsub}) is false by solving the {\em pricing subproblem}:
\begin{equation}\label{eq:sepsubu}
	\min_{u \in \U} u^TBu.
\end{equation}
If $u^TBu < 0$, then there is an element $u$ in $\U$ such that the matrix $uu^T$ violates the dual constraint written in (\ref{eq:sepsub}). Problem (\ref{eq:sepsubu}) may or may not be easy to solve depending on the set $\U.$ 
For example, an ambitious column generation strategy to improve on dsos (i.e., $\U=\U_{n,2}$), would be to take $\U=\U_{n,n}$; i.e., the set all vectors in $\mathbb{R}^n$ consisting of zeros, ones, and minus ones.
%Therefore, another way of generalizing DSOS programming is by letting $\U$ be a structured set of vectors containing $\U_2$. 
In this case, the pricing problem (\ref{eq:sepsubu})
becomes
\[ \min_{u \in \{0,\pm 1\}^n} u^T B u. \]
Unfortunately, the above problem generalizes the quadratic unconstrained boolean optimization problem (QUBO) and is NP-hard. Nevertheless, there are good heuristics for this problem (see e.g.,~\cite{boros2007local},\cite{dash2013note}) that can be used to find near optimal solutions very fast.
%However, it is very closely related to the Ising Model problem, and it is well-known that though it is hard to obtain a provably optimal
%solution to the above two NP-hard problems, one can often employ heuristics to obtain near optimal solutions very fast~\cite{boros2007local},\cite{dash2013note}.
%Thus one can try to heuristically solve the above pricing problem.
While we did not pursue this pricing subproblem, we did consider optimizing over $\U_{n,3}$. We refer to the vectors in $\U_{n,3}$ as ``triples'' for obvious reasons and generally refer to the process of adding atoms drawn from $U_{n,3}$ as optimizing over ``triples''.

% Indeed, this is a polynomially solvable pricing problem, though still an expensive one.
% Our strategy for managing this subproblem on larger instances is described in the next subsection.

%Our implementation is somewhat straightforward and can be obviously improved, yet we are able to demonstrate that optimizing over triples improves over the best bounds obtained by Ahmadi and Majumdar in a similar amount of time (see Section~\ref{subsec:comput_experim_poly}). 

%Our problem instances are again fully dense and generated in exactly the same way as the $n=10$ example of the previous subsection.

Even though one can theoretically solve (\ref{eq:sepsubu}) with $\U=\U_{n,3}$ in polynomial time by simple enumeration of $n^3$ elements, this is very impractical. %, i.e., the separation problem for triples by simple enumeration, this is impractical because of the large number of triples as noted earlier.
Our simple implementation is a partial enumeration and is implemented as follows. We iterate through the triples (in a fixed order), and test to see whether the condition $u^TBu \geq 0$ is violated by a given triple $u$, and collect such violating triples in a list.
We terminate the iteration when we collect a fixed number of violating triples (say $t_1$). We then sort the violating triples by increasing values of $u^TBu$ (remember, these values are all negative for the violating triples) and select the $t_2$ most violated triples (or fewer if less than $t_2$ are violated overall) and add them to our current set of atoms.
In a subsequent iteration we start off enumerating triples right after the last triple enumerated in the current iteration so that we do not repeatedly scan only the same subset of triples. Although our implementation is somewhat straightforward and can be obviously improved, we are able to demonstrate that optimizing over triples improves over the best bounds obtained by Ahmadi and Majumdar in a similar amount of time (see Section~\ref{subsec:comput_experim_poly}). 

We can also have pricing subproblems where the set $\U$ is infinite. Consider e.g. the case $\U=\mathbb{R}^n$ in (\ref{eq:sepsubu}). In this case, if there is a feasible solution with a negative objective value, then the problem is clearly unbounded below. Hence, we look for a solution with the smallest value of ``violation'' of the dual constraint divided by the norm of the violating matrix. In other words, we want the expression $u^TBu / \textup{norm}(uu^T)$ to be as small as possible, where $\textup{norm}$ is the Euclidean
norm of the vector consisting of all entries of $uu^T$. This is the same as minimizing $u^TBu/||u||^2$. The eigenvector corresponding to the smallest eigenvalue yields such a minimizing solution. This is the motivation behind the strategy described in the previous section for our LP column generation scheme. In this case, we can use a similar strategy for our SOCP column generation scheme. We replace $Y \succeq 0$ by $Y\in SDD_n$ in (\ref{eq:SDPform}) and iteratively expand $SDD_n$ by using the ``two most negative eigenvector technique'' described in Section \ref{subsec:socp.based.CG}.

% We note though that for large matrices $B,$ computation of eigenvectors can be expensive, even though polynomially solvable.

%In such a case, it is standard to restrict the norm of a solution matrix $uu^T$ in the pricing subproblem. It is also common in integer programming to not just work with any solution of the pricing subproblem, but with one
%with the smallest value of ``violation'' of the dual constraint $B \cdot X \geq 0$ divided by the norm of the violating matrix.
%In other words, we want the expression $u^TBu / \textup{norm}(uu^T)$ to be as small as possible, where $\textup{norm}$ is the euclidean
%norm of the vector consisting of all entries of $uu^T$. This is the same as minimizing $u^TBu/||u||^2$; of course, if $\U = \R^n$, then the eigenvector corresponding to the smallest eigenvalue yields such a minimizing solution. This was the strategy briefly described in the previous section. We note though that for large matrices $B,$ computation of eigenvectors can be expensive, even though polynomially solvable.

% we have a polynomial-time pricing subproblem, while still generalizing DSOS programming.

\subsection{Experiments with a 10-variable quartic}\label{subsec:10-var.quartic}

We illustrate the behaviour of these different strategies on an example. Let $p(x)$ be a degree-four form defined on 10 variables, where the components of $\coef{p}$ are drawn independently at random from the normal distribution $\mathcal{N}(0,1)$.
Thus $d=2$ and $n=10$, and the form $p(x)$ is `fully dense' in the sense that $\coef{p}$ has essentially all nonzero components.
In Figure~\ref{fig:errorfig}, we show how the lower bound on the optimal value of $p(x)$ over the unit sphere changes per iteration for different methods. The $x$-axis shows the number of iterations of the column generation algorithm, i.e., the number of times columns are added and the LP (or SOCP) is resolved. The $y$-axis shows the lower bound obtained from each LP or SOCP. Each curve represents one way of adding columns. The three horizontal lines (from top to bottom) represent, respectively, the SDP bound, the 1SDSOS bound and the 1DSOS bound. The curve DSOS$_k$ gives the bound obtained by solving LPs, where the first LP has $Y \in DD_n$ and subsequent columns are generated from a single eigenvector corresponding to the most negative eigenvalue of the dual optimal solution as described in Section~\ref{subsec:LP.based.CG}. The LP triples curve also corresponds to an LP sequence, but this time the columns that are added are taken from $U_{n,3}$ and are more than one in each iteration (see the next subsection). This bound saturates when constraints coming from all elements of $U_{n,3}$ are satisfied. Finally, the curve SDSOS$_k$ gives the bound obtained by SOCP-based column generation as explained just above.
\begin{figure}[H]
\centering
\includegraphics[scale=0.7,trim={3cm 8.5cm 3cm 9cm},clip]{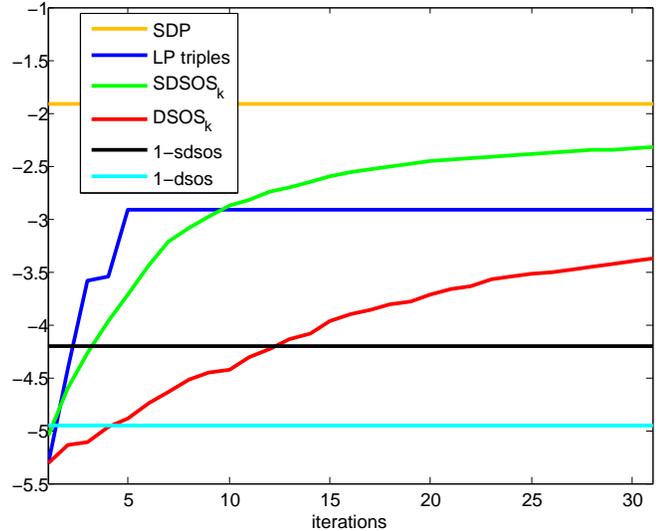}
\caption{Lower bounds for a polynomial of degree 4 in 10 variables obtained via LP and SOCP based column generation}
\label{fig:errorfig}
\end{figure}

\subsection{ Larger computational experiments}\label{subsec:comput_experim_poly}

In this section, we consider larger problem instances ranging from 15 variables to 40 variables: these instances are again fully dense and generated in exactly the same way as the $n=10$ example of the previous subsection. However, contrary to the previous subsection, we only apply our ``triples'' column generation strategy here. This is because the eigenvector-based column generation strategy is too computationally expensive for these problems as we discuss below.
%In this section, we consider larger problem instances ranging from 15 variables to 40 variables and we only apply our ``triples'' column generation strategy. As we discuss below, the eigenvector-based column generation strategy is too computationally expensive for these problems. Our problem instances are again fully dense and generated in exactly the same way as the $n=10$ example of the previous subsection.

To solve the triples pricing subproblem with our partial enumeration strategy, we set $t_1$ to 300,000 and $t_2$ to 5000. Thus in each iteration, we find up to 300,000 violated triples, and add up to 5000 of them. In other words, we augment our LP by up to 5000 columns in each iteration. This is somewhat unusual as in practice at most a few dozen columns are added in each iteration. The logic for this is that primal simplex is very fast in reoptimizing an LP when a small number of additional columns are added to an LP whose optimal basis is known. However, in our context, we observed that the associated LPs are very hard for the simplex routines inside our LP solver (CPLEX 12.4) and take much more time than CPLEX's interior point solver. We therefore use CPLEX's interior point (``barrier'') solver not only for the initial LP but for subsequent LPs after adding columns. Because interior point solvers do not benefit significantly from warm starts, each LP takes a similar amount of time to solve as the initial LP, and therefore it makes sense to add a large number of columns in each iteration to amortize the time for each expensive solve over many columns. 

%In our experiments with stable set (Section~\ref{sec:stable.set}), the LPs are easier to solve ; in that context $t_2$ is smaller.

%In general we observe that even though theoretically the bounds obtained in the first case match the semidefinite programming bounds, in practice one gets worse bounds in the same amount of time.
%This is partially because the matrices $uu^T$ where $u$ is an eigenvector of $B$ are much denser than the elements of $U_{n,3}$.
%For the same reason optimizing over $\U_n$ is very expensive and we do not report computing times with $\U_n$.

Table~\ref{tab:min_hom_opts} is taken from the work of Ahmadi and Majumdar~\cite{isos_journal}, where they report lower bounds on the minimum value of fourth-degree forms on the unit sphere obtained using different methods, and the respective computing times (in seconds).

\begin{table}[H]
\small
\begin{center}
	\scalebox{0.95}{
  \begin{tabular}{| c | c | c | c | c | c | c | c | c| c| c|}
    \hline 
 & \multicolumn{2}{|c|}{n=15}  &  \multicolumn{2}{|c|}{n=20} & \multicolumn{2}{|c|}{n=25} &  \multicolumn{2}{|c|}{n=30} & \multicolumn{2}{|c|}{n=40}\\
\cline{2-11}
& bd & t(s) & bd & t(s) & bd & t(s) & bd & t(s) & bd & t(s) \\
\hline
 DSOS & -10.96 & 0.38  & -18.012 & 0.74 & -26.45 & 15.51 &  -36.85 & 7.88 &  -62.30 & 10.68  \\ \hline
 SDSOS  &  -10.43 & 0.53 &  -17.33 & 1.06 &  -25.79 & 8.72 & -36.04 & 5.65 &  -61.25 & 18.66   \\ \hline
1DSOS &  -9.22 & 6.26 &  -15.72 & 37.98 &  -23.58 & 369.08 & NA & NA  & NA & NA  \\ \hline
1SDSOS  & -8.97 & 14.39 & -15.29 & 82.30 & -23.14 & 538.54 & NA & NA & NA & NA   \\ \hline
 SOS & -3.26 & 5.60 & -3.58 & 82.22 & -3.71 & 1068.66 & NA & NA & NA & NA   \\ \hline
  \end{tabular}}
    \caption{Comparison of optimal values in \cite{isos_journal} for lower bounding a quartic form on the sphere for varying dimension, along with run times (in seconds). These results are obtained on a 3.4 GHz Windows computer with 16 GB of memory. }
  \label{tab:min_hom_opts}
  \end{center}
\end{table}
In Table~\ref{tab:min_triples}, we give our bounds for the { same problem} instances. We report two bounds, obtained at two different times (if applicable).
In the first case ({ rows} labeled R1), the time taken by 1SDSOS in Table~\ref{tab:min_hom_opts} is taken as a limit, and we report the bound from the last column generation iteration occuring { before this time limit}; the 1SDSOS bound is the best non-SDP bound reported in the experiments of Ahmadi and Majumdar. In the {rows} labeled as R2, we take 600 seconds as a limit and report the last bound obtained before this limit.
In a couple of instances ($n=15$ and $n=20$), our column generation algorithm terminates before the 600 second limit, and we report the termination time in this case.
%\begin{table}[H]
%\small
%\begin{center}
%  \begin{tabular}{ | c | c | c | c | c | c | c | c | c | c| c| c| }
%    \hline 
%\  &  \multicolumn{2}{|c|}{n=15}  &  \multicolumn{2}{|c|}{n=20} & \multicolumn{2}{|c|}{n=25} &  \multicolumn{2}{|c|}{n=30} & \multicolumn{2}{|c|}{n=40} \\
%%\cline{2-11}
%\cline{2-11}
%\ & C1 & C2 & C1 & C2 & C1 & C2 & C1 & C2 & C1 & C2 \\
%
%\hline
%time (sec) & 10.96 & 31.19 & 70.70 & 471.39 & 508.63 & 600 &  N/A & 600 & N/A  & 600 \\
%%\hline
%%1-SDSOS bound & -8.97 & \   & -15.29  & \  & -23.14 & \  & N/A & \  & N/A & \   \\
%
%\hline
% 
%triples bound & -6.20 & -5.57  & -12.38  & -9.02 & -20.08 & -20.08 & N/A & -32.38 & N/A & -35.14  \\ \hline
%\end{tabular}
%    \caption{Lower bounds on the optimal value of a form on the sphere for varying degrees of polynomials using Triples on a 2.33 GHz Linux machine with 32 GB of memory.}
%  \label{tab:min_triples}
%  \end{center}
%\end{table}

\begin{table}[H]
	\small
	\begin{center}
		\begin{tabular}{ | c | c | c | c | c | c | c | c | c | c| c| c| }
			\hline 
			\  &  \multicolumn{2}{|c|}{n=15}  &  \multicolumn{2}{|c|}{n=20} & \multicolumn{2}{|c|}{n=25} &  \multicolumn{2}{|c|}{n=30} & \multicolumn{2}{|c|}{n=40} \\
			%\cline{2-11}
			\cline{2-11}
			\ & bd & t(s) & bd & t(s) & bd & t(s) & bd & t(s) & bd & t(s) \\
			
			\hline
			R1 & -6.20 & 10.96 &-12.38 & 70.70 & -20.08 & 508.63 &  N/A & N/A & N/A  & N/A \\
			%\hline
			%1-SDSOS bound & -8.97 & \   & -15.29  & \  & -23.14 & \  & N/A & \  & N/A & \   \\
			
			\hline
			
		R2 & -5.57 & 31.19 &-9.02 & 471.39 & -20.08 & 600 &-32.28 & 600 & -35.14 & 600  \\ \hline
		\end{tabular}
		\caption{Lower bounds on the optimal value of a form on the sphere for varying degrees of polynomials using Triples on a 2.33 GHz Linux machine with 32 GB of memory.}
		\label{tab:min_triples}
	\end{center}
\end{table}

{ We observe that in the same amount of time (and even on a slightly slower machine), we are able to consistently beat the 1SDSOS bound, which is the strongest non-SDP bound produced in~\cite{isos_journal}. We also experimented with the eigenvalue pricing subproblem in the LP case, with a time limit of 600 seconds. For $n=25$, we obtain a bound of $-23.46$ after adding only $33$ columns in 600 seconds. For $n=40$, we  are only able to add 6 columns and the lower bound obtained is $-61.49$. Note that this bound is worse than the triples bound given in Table \ref{tab:min_triples}. The main reason for being able to add so few columns in the time limit is that each column is almost fully dense (the LPs for n=25 have 20,475 rows, and 123,410 rows for $n=40$). Thus, the LPs obtained are very hard to solve after a few iterations and become harder with increasing $n$. As a consequence, we did not experiment with the eigenvalue pricing subproblem in the SOCP case as it is likely to be even more computationally intensive.}

\section{Inner approximations of copositive programs and the maximum stable set problem}\label{sec:stable.set}
%============================================================

Semidefinite programming has been used extensively for approximation of NP-hard combinatorial optimization problems. One such example is finding the \emph{stability number} of a graph. A stable set (or independent set) of a graph $G=(V,E)$ is a set of nodes of $G$, no two of which are adjacent. The size of the largest stable set of a graph $G$ is called the stability number (or independent set number) of $G$ and is denoted by $\alpha(G).$  Throughout, {$G$} is taken to be an undirected, unweighted graph on $n$ nodes. It is known that the problem of testing if $\alpha(G)$ is greater than a given integer $k$ is NP-hard \cite{Karp}. Furthermore, the stability number cannot be approximated { to a factor of} $n^{1-\epsilon}$ for any $\epsilon >0$ unless P$=$NP \cite{HaastadJohan}. The natural integer programming formulation of this problem is given by
\begin{equation} \label{IPformulation}
\begin{aligned}
\alpha(G)=&\underset{x_i}{\max} \sum_{i=1}^n x_i\\
&\text{s.t. } x_i+x_j \leq 1, \forall (i,j) \in E,\\
&x_i \in \{0,1\}, \forall i=1,\ldots,n.
\end{aligned}
\end{equation}
Although this optimization problem is intractable, there are several computationally-tractable relaxations that provide upper bounds on the stability number of a graph. For example, the obvious LP relaxation of (\ref{IPformulation}) can be obtained by relaxing the constraint $x_i \in \{0,1\}$ to $x_i \in [0,1]$:
\begin{equation} \label{LPformulation}
\begin{aligned}
LP(G)=&\underset{x_i}{\max} \sum_i x_i\\
&\text{s.t. } x_i+x_j \leq 1, \forall (i,j) \in E,\\
&x_i \in [0,1], \forall i=1,\ldots,n.
\end{aligned}
\end{equation}
{This bound can be improved upon by adding the so-called \emph{clique inequalities} to the LP, which are of the form $x_{i_1}+x_{i_2}+\ldots+x_{i_k} \leq 1$ when nodes $(i_1,i_2,\ldots,i_k)$ form a clique in $G$. Let $C_k$ be the set of all $k$-clique inequalities in $G$. This leads to a hierarchy of LP relaxations:}
\begin{equation}\label{LPkformulation}
\begin{aligned}
LP_k(G)=&\max \sum_i x_i,\\
&x_i \in [0,1], \forall i=1,\ldots,n,\\
& C_2,\ldots,C_k \text{ are satisfied.}
\end{aligned}
\end{equation}
Notice that for $k=2,$ this simply corresponds to (\ref{LPformulation}), in other words, $LP_2(G)=LP(G)$.

{In addition to LPs, there are also semidefinite programming (SDP)} relaxations that provide upper bounds to the stability number. {The most well-known} is perhaps the Lov\'{a}sz theta number $\vartheta(G)$ \cite{Lovasz}, which is defined as the optimal value of the following SDP:
\begin{equation}\label{Lovasz}
\begin{aligned}
\vartheta(G)\mathrel{\mathop{:}}= &\underset{X}{\max } ~J\cdot X\\
&\text{s.t. } I\cdot X=1,\\
&X_{i,j}=0, \forall (i,j) \in E\\
&X \in P_n.
\end{aligned}
\end{equation}
Here $J$ is the all-ones matrix and $I$ is the identity matrix of size $n$.
The Lov\'{a}sz theta number is known to always give at least as good of an upper bound as the LP in (\ref{LPformulation}), even with the addition of clique inequalities of all sizes (there are exponentially many); see, e.g., \cite[Section 6.5.2]{LaurentVall} for a proof. In other words,
$$\vartheta(G) \leq LP_k(G), \forall k.$$

An alternative SDP relaxation for stable set is due to de Klerk and Pasechnik. In \cite{dp}, they show that the stability number can be obtained through a conic linear program over the set of copositive matrices. Namely,
\begin{equation} \label{Copos}
\begin{aligned}
\alpha(G)=&\min_{\lambda} \lambda\\
&\text{s.t. } \lambda(I+A)-J \in C_n,
\end{aligned}
\end{equation}
where $A$ is the adjacency matrix of $G$.
Replacing $C_n$ by the restriction $P_n + N_n$, {one} obtains the aforementioned relaxation through the following SDP
\begin{equation}\label{CoposSDPRelax}
\begin{aligned}
SDP(G)\mathrel{\mathop{:}}= &\min_{\lambda,X} \lambda\\
&\text{s.t. } \lambda(I+A)-J \geq X,\\
&X \in P_n.  
\end{aligned}
\end{equation}
This latter SDP is more expensive to solve than the Lov\'{a}sz SDP (\ref{Lovasz}), but the bound that it obtains is always at least as good (and sometimes strictly better). 
%In fact, it is known that in some cases, $SDP(G)$ strictly improves on $\vartheta(G)$.
A proof of this statement is given in \cite[Lemma 5.2]{dp}, where it is shown that (\ref{CoposSDPRelax}) is an equivalent formulation of an SDP of Schrijver~\cite{sch}, which produces stronger upper bounds than (\ref{Lovasz}).

Another reason for the interest in the copositive approach is that it allows for well-known SDP and LP hierarchies---developed respectively by Parrilo \cite[Section 5]{PhD:Parrilo} and de Klerk and Pasechnik \cite{dp}---that produce a sequence of improving bounds on the stability number. In fact, by appealing to Positivstellensatz results of P\'{o}lya \cite{Polya}, and Powers and Reznick \cite{PR}, de Klerk and Pasechnik show that their LP hierarchy produces the exact stability number in $\alpha^2(G)$ number of steps \cite[Theorem 4.1]{dp}. This immediately implies the same result for stronger hierarchies, such as the SDP hierarchy of Parrilo \cite{PhD:Parrilo}, or the rDSOS and rSDSOS hierarchies of Ahmadi and Majumdar \cite{isos_journal}.

One notable difficulty with the use of copositivity-based SDP relaxations such as (\ref{CoposSDPRelax}) in applications is scalibility. For example, it takes more than 5 hours to solve (\ref{CoposSDPRelax}) when the input is a randomly generated Erd\'{o}s-Renyi graph with 300 nodes and edge probability $p=0.8$. \footnote{The solver in this case is MOSEK~\cite{mosek} and the machine used has 3.4GHz speed and 16GB RAM; see Table \ref{SOCPbounds} for more results. The solution time with the popular SDP solver SeDuMi~\cite{sedumi} e.g. would be several times larger.} Hence, instead of using (\ref{CoposSDPRelax}), we will solve a sequence of LPs/SOCPs generated in an iterative fashion. These easier optimization problems will provide upper bounds on the stability number in a more reasonable amount of time, though they will be weaker than the ones obtained via (\ref{CoposSDPRelax}). 

%We now describe the process through which we iteratively construct the LPs and the SOCPs used to upper bound the stability number.

We will derive both our LP and SOCP sequences from formulation (\ref{Copos}) of the stability number. To obtain the first LP in the sequence, we replace $C_n$ by $DD_n+N_n$ (instead of replacing $C_n$ by $P_n+N_n$ as was done in (\ref{CoposSDPRelax})) and get
\begin{equation}\label{CoposDDRelax}
\begin{aligned} 
DSOS_1(G)\mathrel{\mathop{:}}= &\min_{\lambda,X} \lambda\\
&\text{s.t. } \lambda(I+A)-J \geq X,\\
&X \in DD_n. 
\end{aligned}
\end{equation}
%It is easy to see that this is an LP as optimizing over the set of diagonally dominant matrices can be done using linear programming (see Theorem \ref{blabla}). Furthermore,
This is an LP whose optimal value is a valid { upper bound} on the stability number as $DD_n \subseteq P_n$.

\begin{theorem}\label{thm:DSOSfinite}
	The LP in (\ref{CoposDDRelax}) is always feasible. %$DSOS_1(G)$ is always finite.
\end{theorem} 
\begin{proof} 
	We need to show that for any $n\times n$ adjacency matrix $A$, there exists a diagonally dominant matrix $D$, a nonnegative matrix $N$, and a scalar $\lambda$ such that 
	\begin{align} \label{lemma:eq}
	\lambda(I+A)-J=D+N.
	\end{align} 
	 Notice first that $\lambda(I+A)-J$ is a matrix with $\lambda -1$ on the diagonal and at entry $(i,j)$, if $(i,j)$ is an edge in the graph, and with $-1$ at entry $(i,j)$ if $(i,j)$ is not an edge in the graph. If we denote by $d_i$ the degree of node $i$, then let us take $\lambda=n-\min_i d_i+1$ and $D$ a matrix with diagonal entries $\lambda-1$ and off-diagonal entries equal to $0$ if there is an edge, and $-1$ if not. This matrix is diagonally dominant as there are at most $n-\min_i d_i$ minus ones on each row. Furthermore, if we take $N$ to be a matrix with $\lambda-1$ at the entries $(i,j)$ where $(i,j)$ is an edge in the graph, then (\ref{lemma:eq}) is satisfied and $N\geq0$.
\end{proof}
%We will see later that this lemma is crucial for the column generation process to be well defined.\\

Feasibility of this LP is important for us as it allows us to initiate column generation. By contrast, if we were to replace the diagonal dominance constraint by a diagonal constraint for example, the LP could fail to be feasible. This fact has been observed by de Klerk and Pasechnik in \cite{dp} and Bomze and de Klerk in \cite{BdK}.

To generate the next LP in the sequence via column generation, we think of the extreme-ray description of the set of diagonally dominant matrices as explained in Section~\ref{sec:cg.overview}. Theorem~\ref{thm:dsos.corners} tells us that these are given by the matrices in $U_{n,2}$ and so we can rewrite (\ref{CoposDDRelax}) as
\begin{equation}\label{CoposDDRelaxCorners}
\begin{aligned} 
DSOS_1(G)\mathrel{\mathop{:}}= &\min_{\lambda,\alpha_i} \lambda\\
&\text{s.t. } \lambda(I+A)-J \geq X,\\
&X=\sum_{u_iu_i^T\in {U}_{n,2}} \alpha_i u_iu_i^T,\\
&\alpha_i \geq 0,\  i=1,\ldots,n^2.
\end{aligned}
\end{equation}
%Notice that we have $n^2$ variables $\alpha_i$. Even though the cardinality of $\mathcal{U}_{n,2}$ is $2n^2$, these vectors will produce only $n^2$ different matrices $u_iu_i^T$; in other words, $|U_{n,2}|=n^2$.

The column generation procedure aims to add new matrix atoms to the existing set $U_{n,2}$ in such a way that the current bound $DSOS_1$ improves. There are numerous ways of choosing these atoms. We focus first on the cutting plane approach based on eigenvectors. The dual of (\ref{CoposDDRelaxCorners}) is the LP
\begin{equation}\label{CoposDDRelaxCornersDual}
\begin{aligned} 
DSOS_1(G)\mathrel{\mathop{:}}= &\max_{X} J\cdot X,\\
&\text{s.t. } (A+I) \cdot X=1,\\
&X\geq 0,\\
& (u_iu_i^T) \cdot X \geq 0, \forall u_iu_i^T \in {U}_{n,2}.
\end{aligned}
\end{equation}
%The strategy is then the following.
If our optimal solution $X^*$ to (\ref{CoposDDRelaxCornersDual}) is positive semidefinite, then we are obtaining the best bound we can possibly produce, which is the SDP bound of (\ref{CoposSDPRelax}). If this is not the case however, we pick our atom matrix to be the outer product $uu^T$ of the eigenvector $u$ corresponding to the most negative eigenvalue of $X^*$. The optimal value of the LP
\begin{equation}\label{iterativeLP}
\begin{aligned} 
DSOS_2(G)\mathrel{\mathop{:}}= &\max_{X} J \cdot X,\\
&\text{s.t. } (A+I) \cdot X=1,\\
&X\geq 0,\\
& (u_iu_i^T) \cdot X \geq 0, \forall u_iu_i^T \in {U}_{n,2},\\
& (uu^T) \cdot X \geq 0
\end{aligned}
\end{equation}
that we derive is guaranteed to be no worse than $DSOS_1$ as the feasible set of (\ref{iterativeLP}) is smaller than the feasible set of (\ref{CoposDDRelaxCornersDual}). Under mild nondegeneracy assumptions (satisfied, e.g., by uniqueness of the optimal solution to (\ref{CoposDDRelaxCornersDual})), the new bound will be strictly better. By reiterating the same process, we create a sequence of LPs whose optimal values $DSOS_1, DSOS_2, \ldots$ are a nonincreasing sequence of upper bounds on the stability number.

 %portance of Lemma \ref{Lemma:DSOSfinite}. Indeed, if (\ref{CoposDDRelaxCorners}) was infeasible (i.e., the upper bound it produces is infinity), (\ref{CoposDDRelaxCornersDual}) would be infeasible, and we would not be able to generate the matrix atom needed for the construction of the next LP. In particular, this could have been the case if we had required that our variable $X$ be diagonal in (\ref{CoposDDRelax}) rather than diagonally dominant, as observed by de Klerk and Pasechnik in \cite{dp} and Bomze and de Klerk in \cite{BdK}.\\

Generating the sequence of SOCPs is done in an analogous way. Instead of replacing the constraint $X \in P_n$ in (\ref{CoposSDPRelax}) by $X \in DD_n$, we replace it by $X \in SDD_n$ and get
\begin{equation}\label{CoposSDDRelax}
\begin{aligned} 
SDSOS_1(G)\mathrel{\mathop{:}}= &\min_{\lambda,X} \lambda\\
&\text{s.t. } \lambda(I+A)-J \geq X,\\
&X \in SDD_n.
\end{aligned}
\end{equation}
Once again, we need to reformulate the problem in such a way that the set of scaled diagonally dominant matrices is described as some combination of psd ``atom" matrices. In this case, we can write any matrix $X \in SDD_n$ as $$X=\sum_{V_i\in\mathcal{V}_{n,2}}V_i\begin{pmatrix} a_i^1 & a_i^2 \\ a_i^2 & a_i^3\end{pmatrix}V_i^T,$$ where $a_i^1,a_i^2, a_i^3$ are variables making the $2\times 2$ matrix psd, and the $V_i$'s are our atoms. Recall from Section~\ref{sec:cg.overview} that the set $\mathcal{V}_{n,2}$ consists of all $n\times 2$ matrices which have zeros everywhere, except for a 1 in the first column in position $j$ and a 1 in the second column in position $k\neq j$. This gives rise to an equivalent formulation of (\ref{CoposSDDRelax}):
\begin{equation}\label{CoposSDDRelaxCorners}
\begin{aligned} 
SDSOS_1(G)\mathrel{\mathop{:}}= &\min_{\lambda,a_i^j} \lambda\\
&\text{s.t. } \lambda(I+A)-J \geq X\\
&X=\sum_{V_i\in\mathcal{V}_{n,2}}V_i\begin{pmatrix} a_i^1 & a_i^2 \\ a_i^2 & a_i^3\end{pmatrix}V_i^T\\
&\begin{pmatrix} a_i^1 & a_i^2 \\ a_i^2 & a_i^3\end{pmatrix} \succeq 0, \  i=1,\ldots,\binom{n}{2}.
\end{aligned}
\end{equation}
Just like the LP case, we now want to generate one (or more) $n \times 2$ matrix $V$ to add to the set $\{V_i\}_i$ so that the bound $SDSOS_1$ improves. We do this again by using a cutting plane approach originating from the dual of (\ref{CoposSDDRelaxCorners}):
\begin{equation}\label{CoposSDDRelaxCornersDual}
\begin{aligned} 
SDSOS_1(G)\mathrel{\mathop{:}}= &\max_{X} J \cdot X\\
&\text{s.t. } (A+I) \cdot X=1,\\
&X\geq 0,\\
& V_i^T\cdot X V_i \succeq 0, \ i=1,\ldots,\binom{n}{2}.
\end{aligned}
\end{equation} 
%As $DSOS_1(G)$ is always finite and $SDSOS_1(G)\leq DSOS_1(G)$, it follows that $SDSOS_1(G)$ is always finite as well. This entails that an optimal solution $X^*$ to (\ref{CoposSDDRelaxCornersDual}) exists.

Note that strong duality holds between this primal-dual pair as it is easy to check that both problems are strictly feasible. We then take our new atom to be $$V=(w_1 ~w_2),$$ where $w_1$ and $w_2$ are two eigenvectors corresponding to the two most negative eigenvalues of $X^*$, the optimal solution of (\ref{CoposSDDRelaxCornersDual}). If $X^*$ only has one negative eigenvalue, we add a linear constraint to our problem; if $X^* \succeq 0$, then the bound obtained is identical to the one obtained through SDP (\ref{CoposSDPRelax}) and we cannot hope to improve. Our next iterate is therefore
\begin{equation}\label{iterativeSOCP}
\begin{aligned} 
SDSOS_2(G)\mathrel{\mathop{:}}= &\max_{X} J \cdot X\\
&\text{s.t. } (A+I) \cdot X=1,\\
&X\geq 0,\\
& V_i^T\cdot X V_i \succeq 0, \ i=1,\ldots,\binom{n}{2},\\
& V^T\cdot X V \succeq 0.
\end{aligned}
\end{equation} 
Note that the optimization problems generated iteratively in this fashion always remain SOCPs and their optimal values form a nonincreasing sequence of upper bounds on the stability number.

To illustrate the column generation method for both LPs and SOCPs, we consider the complement of the Petersen graph as shown in Figure \ref{fig:Petersen} as an example. The stability number of this graph is 2 and one of its maximum stable sets is designated by the two white nodes. In Figure \ref{Petersenbounds}, we compare the upper bound obtained via (\ref{CoposSDPRelax}) and the bounds obtained using the iterative LPs and SOCPs as described in (\ref{iterativeLP}) and (\ref{iterativeSOCP}). 

\begin{figure}[h]
	\begin{center}
		\mbox{
			\subfigure[The complement of Petersen Graph]
			{\label{fig:Petersen}\scalebox{0.25}{\includegraphics[trim={1cm 0cm 2cm 1cm},clip]{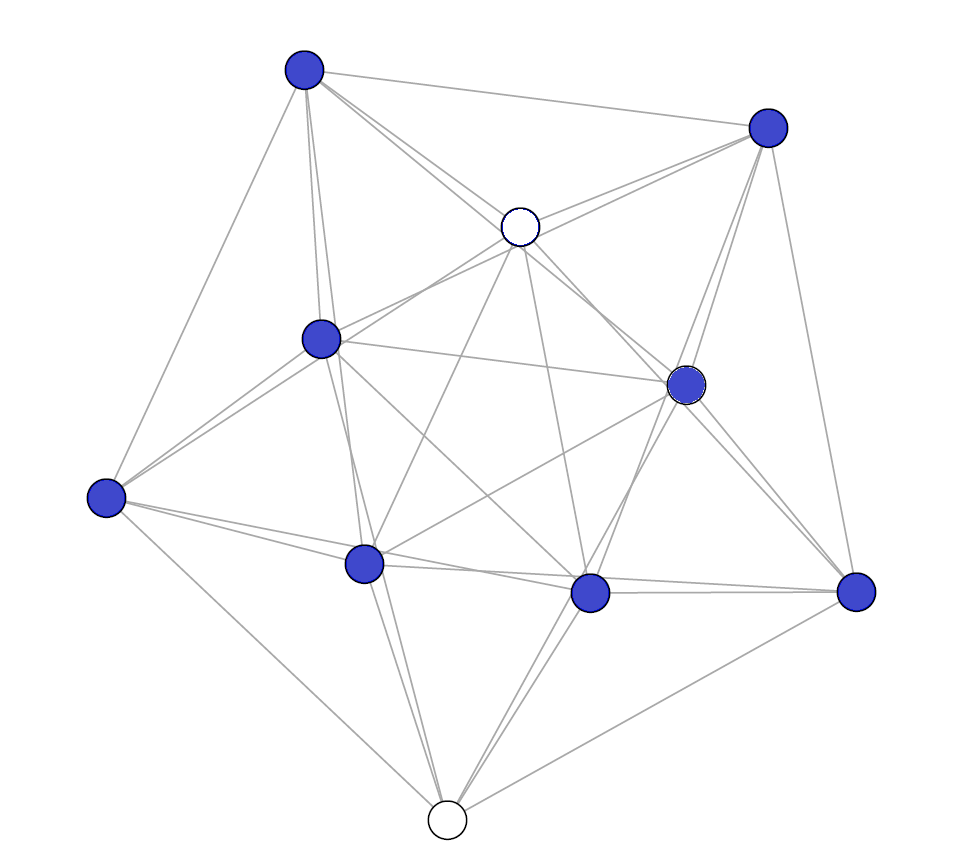}}}}
		\mbox{
			\subfigure[Upper bounds on the stable set number $\alpha(G)$]
			{\label{Petersenbounds}\scalebox{0.5}{\includegraphics[trim={3.5cm 8.5cm 3.5cm 9cm},clip]{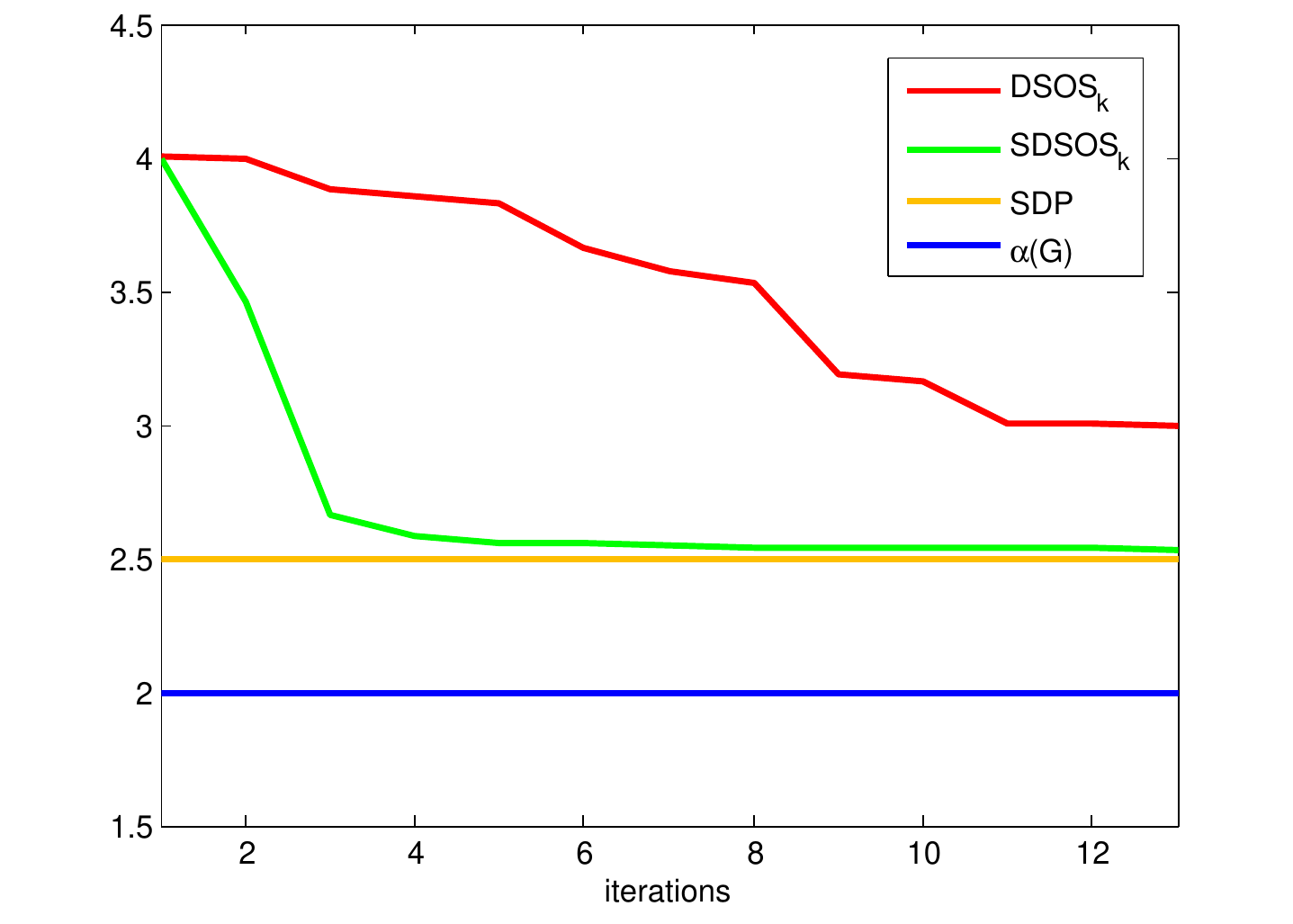}}}
		}
		
		\caption{Bounds obtained through SDP (\ref{CoposSDPRelax}) and iterative SOCPs and LPs for the complement of the Petersen graph.}
		\label{fig:test}
	\end{center}
	\vspace{-20pt}
\end{figure}

%
%\begin{figure}[H]
%	\centering
%	\begin{subfigure}{.5\textwidth}
%		\centering
%		\includegraphics[scale=0.55,trim={3cm 2.5cm 2cm 2.5cm},clip]{prettygraph}
%		\caption{The complement of Petersen Graph}
%		\label{fig:Petersen}
%	\end{subfigure}%
%	\begin{subfigure}{.5\textwidth}
%		\centering
%		\includegraphics[scale=0.65,trim={3.5cm 8.5cm 3.5cm 9cm},clip]{PetersenBounds.pdf}
%		\caption{Upper bounds on the stable set number $\alpha(G)$}
%		\label{Petersenbounds}
%	\end{subfigure}
%	\caption{Bounds obtained through SDP (\ref{CoposSDPRelax}) and iterative SOCPs and LPs for the complement of the Petersen graph}
%	\label{fig:test}
%\end{figure}

Note that it takes 3 iterations for the SOCP sequence to produce an { upper bound} strictly within one unit of the actual stable set number (which would immediately tell us the value of $\alpha$), whereas it takes 13 iterations for the LP sequence to do the same. It is also interesting to compare the sequence of LPs/SOCPs obtained through column generation to the sequence that one could obtain using the concept of $r$-dsos/$r$-sdsos polynomials. Indeed, LP (\ref{CoposDDRelax}) (resp. SOCP (\ref{CoposSDDRelax})) can be written in polynomial form as 
\begin{equation}
\begin{aligned}
DSOS_1(G) \text{ (resp. $SDSOS_1(G)$)}=&\min_{\lambda} \lambda\\
&\text{s.t. } \begin{pmatrix} x_1^2 \\ \vdots \\ x_n^2 \end{pmatrix}^T (\lambda (I+A) -J) \begin{pmatrix} x_1^2 \\ \vdots \\ x_n^2 \end{pmatrix} \text{ is dsos (resp. sdsos).}
%\begin{pmatrix}x_1^2  end{pmatrix} (\lambda(I+A)-J) \begin{pmatrix}x_1^2 \\ \vdots \\ x_n^2 \end{pmatrix} \text{ is dsos (resp. sdsos)}
\end{aligned}
\end{equation}
Iteration $k$ in the sequence of LPs/SOCPs would then correspond to requiring that this polynomial be $k$-dsos or $k$-sdsos. For this particular example, we give the 1-dsos, 2-dsos, 1-sdsos and 2-sdsos bounds in Table \ref{Table1dsos1sdsos}. 

\begin{table}[h]
	\centering
	\begin{tabular}{|c|c|c|}
		\hline
		Iteration & $r$-dsos bounds & $r$-sdsos bounds\\
		\hline
		$r=0$ & 4.00 & 4.00 \\
		\hline
		$r=1$ & 2.71 & 2.52\\
		\hline 
		$r=2$ & 2.50 & 2.50 \\
		\hline
	\end{tabular}
	\caption{Bounds obtained through rDSOS and rSDSOS hierarchies.}
	\label{Table1dsos1sdsos}
	
\end{table}

Though this sequence of LPs/SOCPs gives strong upper bounds, each iteration is more expensive than the iterations done in the column generation approach. Indeed, in each of the column generation iterations, only one constraint is added to our problem, whereas in the rDSOS/rSDSOS hierarchies, the number of constraints is roughly multiplied by $n^2$ at each iteration.

%
%\begin{figure}[H]\label{Petersen}
%\centering
%\includegraphics[scale=0.55,trim={2.7cm 2.7cm 2.7cm 2.7cm},clip]{prettygraph}
%\caption{Complement of Petersen Graph}
%\end{figure}
%%\begin{table}\label{Petersenbounds}
%\centering
%\begin{tabular}{c|c||c|c}
%            \hline \hline
%	\multicolumn{4}{c}{SDP bound}\\
%\hline
%	\multicolumn{4}{c}{2.5}\\
%\hline \hline
%	\multicolumn{4}{c}{LP bounds} \\
%            \cline{1-4}
%             Iteration & Bound & Iteration & Bound  \\
%            \hline
%1 & 4 & 8 & 3.529\\
% 2 & 3.997 &  9 & 3.193\\\
%3 & 3.885 & 10 & 3.166\\
%4 & 3.854 & 11 & 3.006\\
%5 & 3.825 & 12 & 3.003\\
%6  &  3.665 & 13 & 2.99\\
%7 & 3.573 &  &\\
%\hline
%        \end{tabular}
%\end{table}

Finally, we investigate how these techniques perform on graphs with a large number of nodes, where the SDP bound cannot be found in a reasonable amount of time. The graphs we test these techniques on are Erd\"os-R\'{e}nyi graphs $ER(n,p)$; i.e. graphs on $n$ nodes where an edge is added between each pair of nodes independently and with probability $p$. In our case, we take $n$ to be between $150$ and $300$, and $p$ to be either $0.3$ or $0.8$ so as to experiment with both medium and high density graphs.\footnote{All instances used for these tests are available online at \url{http://aaa.princeton.edu/software}.} 

%In very sparse graphs, LP based methods usually perform much better than SDP based methods.

In Table \ref{SOCPbounds}, we present the results of the iterative SOCP procedure and contrast them with the SDP bounds. The third column of the table contains the SOCP upper bound obtained through (\ref{CoposSDDRelaxCornersDual}); the solver time needed to obtain this bound is given in the fourth column. The fifth and sixth columns correspond respectively to the SOCP iterative bounds obtained after 5 mins solving time and 10 mins solving time. Finally, the last two columns chart the SDP bound obtained from (\ref{CoposSDPRelax}) and the time in seconds needed to solve the SDP. All SOCP and SDP experiments were done using Matlab, the solver MOSEK \cite{mosek}, the SPOTLESS toolbox \cite{SPOT_Megretski}, and a computer with 3.4 GHz speed and 16 GB RAM.

\begin{table}[H] 
	\centering
	\scalebox{0.85}{
	\begin{tabular}{c|c|cc|c|c|cc}
		\hline
		n & p & $SDSOS_1$ & time (s) & $SDSOS_k$ (5 mins) & $SDSOS_k$ (10 mins) &$SDP(G)$ & time (s)\\
		\hline
		150 & 0.3 & 105.70 & 1.05 & 39.93 & 37.00 & 20.43 & 221.13\\
		150 & 0.8 & 31.78 & 1.07 & 9.96& 9.43 & 6.02 & 206.28\\
		200 & 0.3 & 140.47 & 1.84& 70.15& 56.37 & 23.73 & 1,086.42\\
		200 & 0.8 & 40.92 & 2.07 & 12.29 & 11.60 & 6.45 & 896.84\\
		250 & 0.3 & 176.25 & 3.51 & 111.63 & 92.93 & 26.78 & 4,284.01 \\
		250 & 0.8 & 51.87 & 3.90 & 17.25 & 15.39 & 7.18 & 3,503.79\\
		300 & 0.3 & 210.32 & 5.69 & 151.71 & 134.14 & 29.13 & 32,300.60\\
		300 & 0.8 & 60.97 & 5.73 & 19.53 & 17.24 & 7.65 & 20,586.02\\
		\hline
	\end{tabular}}
	\caption{SDP bounds and iterative SOCP bounds obtained on ER(n,p) graphs.}
	\label{SOCPbounds}
\end{table} 

From the table, we note that it is better to run the SDP rather than the SOCPs for small $n$, as the bounds obtained are better and the times taken to do so are comparable. However, as $n$ gets bigger, the SOCPs become valuable as they provide good upper bounds in reasonable amounts of time. For example, for $n=300$ and $p=0.8$, the SOCP obtains a bound that is only twice as big as the SDP bound, but it does so 30 times faster. {The sparser graphs don't do as well, a trend that we will also observe in Table \ref{LPbounds}.}
%The sparser graphs don't do as well, but the bounds obtained in $1/30^{th}$ of the time are still only within 4 times the ones obtained with the SDP. 
Finally, notice that the improvement in the first 5 mins is significantly better than the improvement in the last 5 mins. This is partly due to the fact that the SOCPs generated at the beginning are sparser, and hence faster to solve. 

In Table \ref{LPbounds}, we present the results of the iterative LP procedure used on the same instances. All LP results were obtained using a computer with 2.3 GHz speed and 32GB RAM and the solver CPLEX 12.4 \cite{cplex}. The third and fourth columns in the table contain the LP bound obtained with (\ref{CoposDDRelaxCornersDual}) and the solver time taken to do so. Columns 5 and 6 correspond to the LP iterative bounds obtained after 5 mins solving time and 10 mins solving time using the eigenvector-based column generation technique (see discussion around (\ref{iterativeLP})). The seventh and eighth columns are the standard LP bounds obtained using (\ref{LPkformulation}) and the time taken to obtain the bound. Finally, the last column gives bounds obtained by column generation using  ``triples'', as described in Section \ref{subsec:comput_experim_poly}. In this case, we take $t_1=300,000$ and $t_2=500$.

\begin{table}[H] 
	\centering
	\scalebox{0.85}{
	\begin{tabular}{c|c|cc|c|c|cc|c}
		\hline
		n & p & $ DSOS_1 $ & time (s) & $DSOS_k$ (5m) & $DSOS_k$ (10m) & $LP_2 $& time (s) &  $LP_{triples}$ (10m)\\
		\hline
		150 & 0.3 & 117 & $<1$ & 110.64 &  110.26 & 75 & $<1$ & 89.00 \\
		150 & 0.8 & 46 & $<1$ & 24.65 & 19.13 & 75 & $<1$ &23.64  \\
		200 & 0.3 & 157 & $<1 $& 147.12 &  146.71 & 100 & $<1$ & 129.82\\
		200 & 0.8 & 54 & $<1$ & 39.27 & 36.01 & 100 & $<1$ & 30.43 \\
		250 & 0.3 & 194 & $<1$ & 184.89 & 184.31 & 125 & $<1$ & 168.00\\
		250 & 0.8 &  68 & $<1$ & 55.01 & 53.18 & 125 & $<1$ & 40.19\\
		300 & 0.3 & 230 & $<1$ & 222.43 & 221.56 & 150 & $<1$&  205.00 \\
		300 & 0.8 & 78 & $<1$ & 65.77 & 64.84 &  150 & $<1$& 60.00\\
		\hline
	\end{tabular}}
	\caption{LP bounds obtained on the same $ER(n,p)$ graphs.}
	\label{LPbounds}
\end{table} 

We note that in this case the upper bound with triples via column generation does better for this range of $n$ than eigenvector-based column generation in the same amount of time. Furthermore, the iterative LP scheme seems to perform better in the dense regime. In particular, the first iteration does significantly better than the standard LP for $p=0.8$, even though both LPs are of similar size. This would remain true even if the 3-clique inequalities were added as in (\ref{LPkformulation}), since the optimal value of $LP_3$ is always at least $n/3$. This is because the vector $(\frac{1}{3},\ldots,\frac{1}{3})$ is feasible to the LP in (\ref{LPkformulation}) with $k=3$. Note that this LP would have order $n^3$ constraints, which is more expensive than our LP. On the contrary, for sparse regimes, the standard LP, {which hardly takes any time to solve,} gives better bounds than ours.

Overall, the high-level conclusion is that running the SDP is worthwhile for small sizes of the graph. As the number of nodes increases, column generation becomes valuable, providing upper bounds in a reasonable amount of time. Contrasting Tables \ref{SOCPbounds} and \ref{LPbounds}, our initial experiments seem to show that the iterative SOCP bounds are better than the ones obtained using the iterative LPs. It may be valuable to experiment with different approaches to column generation however, as the technique used to generate the new atoms seems to impact the bounds obtained.

\section{Conclusions and future research}\label{sec:future}

For many problems of discrete and polynomial optimization, there are hierarchies of SDP-based sum of squares algorithms that produce provably optimal bounds in the limit~\cite{sdprelax},~\cite{lasserre_moment}. However, these hierarchies can often be expensive computationally. In this chapter, we were interested in problem sizes where even the first level of the hierarchy is too expensive, and hence we resorted to algorithms that replace the underlying SDPs with LPs or SOCPs. We built on the recent work of Ahmadi and Majumdar on DSOS and SDSOS optimization~\cite{isos_journal},~\cite{dsos_ciss14}, which serves exactly this purpose. { We showed that by using ideas from linear programming column generation,} the performance of their algorithms is improvable. We did this by iteratively optimizing over increasingly larger structured subsets of the cone of positive semidefinite matrices, without resorting to the more expensive rDSOS and rSDSOS hierarchies.

%In this chapter, based on the inner representation of the cone of diagonally dominant matrices, we proposed an extension of the DSOS and SDOS-programming ideas developed in 
%By considering specific, structured subsets of the semidefinite cone, and using column generation to optimize over these structured cones (or using cutting plane approaches to optimize over analogous structured relaxations of the semidefinite cone), 

%we demonstrated that one could get better bounds than with DSOS or SDSO programming without resorting to the next levels of the associated heirarchies or solving semidefinite programs, and by spending a reasonable amount of extra time.

There is certainly a lot of room to improve our column generation algorithms. In particular, we only experimented with a few types of pricing subproblems and particular strategies for solving them. {The success of column generation often comes} from good ``engineering'', which fine-tunes the algorithms to the problem at hand. Developing warm-start strategies for our iterative SOCPs for example, would be a very useful problem to work on in the future.

% we feel that our column generation ideas need to be applied differently to different problems, based on problem characteristics. Much more computational experiments are needed to master a good ``engineering'' understanding of the choice of the pricing subproblems and efficient strategies for solving them. This 

%Our separation algorithms are not best possible, and we believe it should be possible to speed up our column generation algorithms with a more careful implementation.

Here is another interesting research direction, which for illustrative purposes we outline for the problem studied in Section~\ref{sec:poly.opt}; i.e., minimizing a form on the sphere. Recall that given a form $p$ of degree $2d$, we are trying to find the largest $\lambda$ such that $p(x)-\lambda (\sum_{i=1}^n x_i^2)^d$ is a sum of squares. Instead of solving this sum of squares program, we looked for the largest $\lambda$ for which we could write $p(x)-\lambda$ as a conic combination of a certain set of nonnegative polynomials. These polynomials for us were always either a single square or a sum of squares of polynomials.
%Each column in our column generation framework represents a very simple nonnegative polynomial, namely one that is a square of a linear combination of the monomials of degree $d$,
%In the context of polynomial optimization, we do not need to restrict ourself to obtaining a sum of squares representation via column generation. Each column in our column generation framework represents a very simple nonnegative polynomial, namely one that is a square of a linear combination of the monomials of degree $d$, i.e., it has the form $(z(x,d) \cdot u)^2$.
%In other words, we are trying to find a representation of $p(x)-\lambda s(x)$ as $\sum_{i} \alpha_i p_i(x)^2$ where $p_i(x)$ is a polynomial of degree $d/2$ as we know that $p_i(x)^2$ is nonnegative.
%
There are polynomials, however, that are nonnegative but not representable as a sum of squares. Two classic examples~\cite{MotzkinSOS},~\cite{Choi_Lam_extremalPSDforms} are the Motzkin polynomial
\[ M(x,y,z) = x^6 + y^4z^2 + y^2z^4 - 3x^2y^2z^2, \]
and the Choi-Lam polynomial
\[ CL(w,x,y,z) = w^4 + x^2y^2 + y^2z^2 + x^2z^2 - 4wxyz. \]

%%There are polynomials, however,  such as the Motzkin polynomial \cite{MotzkinSOS_temp}
%\[ M(x,y,z) = x^6 + y^4z^2 + y^2z^4 - 3x^2y^2z^2 \] that are nonnegative but not representable as a sum of squares of other polynomials, i.e., not {\em SOS-representable}.
%Similarly, the Choi-Lam polynomial \cite{Choi_Lam_extremalPSDforms}
%\[ CL(w,x,y,z) = w^4 + x^2y^2 + y^2z^2 + x^2z^2 - 4wxyz \]
%is nonnegative for all $(w,x,y,z) \in \R^4$ but is not SOS-representable.

Either of these polynomials can be shown to be nonnegative using the arithmetic mean-geometric mean (am-gm) inequality, which states
%inequality of arithmetic means and geometric means (we call it the am-gm inequality) which states
that if $x_1, \ldots, x_k \in \R$, then 
\[ x_1, \ldots, x_k \geq 0 \Rightarrow (\sum_{i=1}^k x_i)/k \geq (\Pi_{i=1}^k x_i)^{\frac{1}{k}}. \]
For example, in the case of the Motzkin polynomial, it is clear that the monomials $x^6, y^4z^2$ and $y^2z^4$ are nonnegative for all $x,y,z \in \R$, and letting $x_1, x_2, x_3$ stand for these monomials respectively, the am-gm inequality  implies that
\[ x^6 + y^4z^2 + y^2z^4 \geq 3x^2y^2z^2 \mbox{ for all } x,y,z \in \R. \]
These polynomials are known to be extreme in the cone of nonnegative polynomials and they cannot be written as a sum of squares (sos)~\cite{Reznick}.

%In other words, these polynomials all lie outside the cone of SOS-representable polynomials.
%Furthemore, they are known to be extreme in the cone of nonnegative polynomials.

It would be interesting to study the separation problems associated with using such non-sos polynomials in column generation. We briefly present one separation algorithm for a family of polynomials whose nonnegativity is provable through the am-gm inequality and includes the Motzkin and Choi-Lam polynomials. This will be a relatively easy-to-solve integer program in itself, whose goal is to find a polynomial $q$ amongst this family which is to be added as our new ``nonnegative atom''.

The family of $n$-variate polynomials under consideration consists of polynomials with only $k+1$ nonzero coefficients, with $k$ of them equal to one, and one equal to $-k$. (Notice that the Motzkin and the Choi-Lam polynomials are of this form with $k$ equal to three and four respectively.) Let $m$ be the number of monomials in $p$. Given a dual vector $\mu$ of (\ref{eq:polycol}) of dimension $m$, one can check if there exists a nonnegative degree $2d$ polynomial $q(x)$ in our family such that $\mu \cdot \coef{q(x)} < 0.$ This can be done by solving the following integer program (we assume that $p(x) = \sum_{i=1}^m x^{\alpha_i}$):
\begin{eqnarray}
	\underset{c,y}{\min} && \sum_{i=1}^m  \mu_ic_i - \sum_{i=1}^m k\mu_iy_i \\\mbox{ s.t.} 
	&&\sum_{i:\alpha_i \mbox{ is even}} \alpha_i c_i = k \sum_{i=1}^m \alpha_iy_i, \nonumber \\
	&&\sum_{i=1}^m c_i = k, \nonumber\\
	&&\sum_{i=1}^m  y_i = 1, \nonumber\\
	&& c_i \in \{0,1\}, y_i \in \{0,1\}, i=1,\ldots, m, c_i=0 \mbox{ if } \alpha_i \mbox{ is not even.}\nonumber
\end{eqnarray}
Here, we have $\alpha_i\in\mathbb{N}^n$  and the variables $c_i, y_i$ form the coefficients of the polynomial $q(x) = \sum_{i=1}^m c_ix^{\alpha_i} - k\sum_{i=1}^m y_i x^{\alpha_i}$. The above integer program has $2m$ variables, but only $n+2$ constraints (not counting the integer constraints).     
If a polynomial $q(x)$ with a negative objective value is found, then one can add it as a new atom for column generation. In our specific randomly generated polynomial optimization examples,
such polynomials did not seem to help in our preliminary experiments. Nevertheless, it would be interesting to consider other instances and problem structures.

Similarly, in the column generation approach to obtaining inner approximations of the copositive cone, one need not stick to positive {semidefinite} matrices. It is known that the $5\times 5$ ``Horn matrix''~\cite{burer2009difference} for example is extreme in the copositive cone but cannot be written as the sum of a nonnegative and a positive semidefinite matrix.
One could define a separation problem for a family of Horn-like matrices and add them in a column generation approach. Exploring such strategies is left for future research.

%\bibliographystyle{abbrv}
%\bibliography{ch-basis_pursuit/pablo_amirali,ch-basis_pursuit/elib}

%\bibliographystyle{plain}
%\begin{thebibliography}{10}
%
%\bibitem{am}
%A. A. Ahmadi and A. Majumdar. 
%DSOS and SDSOS optimization: More tractable alternatives to SOS optimization.
%{\em Manuscript.}
%
%\bibitem{cl}
%  M. D. Choi and T. Y. Lam, Extremal positive semidefinite forms, {\em Math. Ann.} {\bf 231} 1--8 (1977).
%  
%\bibitem{dp}
%E. de Klerk and D. V. Pasechnik.  Approximation of the stability number of a graph via copositive programming. {\em SIAM Journal on Optimization}, {\bf 12}(4), 875-- 892 (2002).
%
%\bibitem{mot}
%T.S. Motzkin, The arithmetic-geometric inequality. Proc. Symposium on Inequalities (ed. O. Shisha), Academic Press, New York, 1967, pp. 205-–224.
%
%\bibitem{mk}
%K. G. Murty and S. N. Kabadi. Some NP-complete problems in quadratic and nonlinear
%programming. {\em Mathematical Programming}, {\bf 39} 117--129 (1987).
%  
%\bibitem{pp}
%P. A. Parrilo. Semidefinite programming relaxations for semialgebraic problems. Mathematical Programming {\bf 96} 293--320 (2003).
%\end{thebibliography}

\chapter{Sum of Squares Basis Pursuit with Linear and Second Order Cone Programming}\label{ch:cholesky}

\section{Introduction}
In recent years, semidefinite programming~\cite{vandenberghe1996semidefinite} and sum of squares optimization~\cite{sdprelax, lasserre_moment, NesterovSquared} have proven to be powerful techniques for tackling a diverse set of problems in applied and computational mathematics. The reason for this, at a high level, is that several fundamental problems arising in discrete and polynomial optimization~\cite{laurent2009sums, Stability_number_SOS, ahmadiOR_letters} or the theory of dynamical systems~\cite{PhD:Parrilo, PositivePolyInControlBook, AAA_PhD} can be cast as linear optimization problems over the cone of nonnegative polynomials. This observation puts forward the need for efficient conditions on the coefficients $c_\alpha\mathrel{\mathop:}=c_{\alpha_1,\ldots,\alpha_n}$ of a multivariate polynomial $$p(x)=\sum_{\alpha } c_{\alpha_1,\ldots,\alpha_n} x_1^{\alpha_1}\ldots x_n^{\alpha_n}$$
that ensure the inequality $p(x)\geq 0,$ for all $x\in\mathbb{R}^n$. If $p$ is a quadratic function, $p(x)=x^TQx+2c^Tx+b,$ then nonnegativity of $p$ is equivalent to the $(n+1)\times (n+1)$ symmetric matrix $$\begin{pmatrix}
Q & c\\ c^T & b
\end{pmatrix}$$
being positive semidefinite and this constraint can be imposed by semidefinite programming. For higher degrees, however, imposing nonnegativity of polynomials is in general an intractable computational task. In fact, even checking if a given quartic polynomial is nonnegative is NP-hard~\cite{copositivity_NPhard}. A particularly popular and seemingly powerful sufficient condition for a polynomial $p$ to be nonnegative is for it to decompose as a sum of squares of other polynomials:
$$p(x)=\sum_i q_i^2(x).$$

This condition is attractive for several reasons. From a computational perspective, for fixed-degree polynomials, a sum of squares decomposition can be checked (or imposed as a constraint) by solving a semidefinite program of size polynomial in the number of variables. From a representational perspective, such a decomposition \emph{certifies} nonnegativity of $p$ in terms of an easily verifiable algebraic identity. From a practical perspective, the so-called ``sum of squares relaxation'' is well-known to produce powerful (often exact) bounds on optimization problems that involve nonnegative polynomials; see, e.g.,~\cite{Minimize_poly_Pablo}. The reason for this is that constructing examples of nonnegative polynomials that are not sums of squares in relatively low dimensions and degrees seems to be a difficult task\footnote{See~\cite{Reznick} for explicit examples of nonnegative polynomials that are not sums of squares.}, especially when additional structure arising from applications is required.

We have recently been interested in leveraging the attractive features of semidefinite programs (SDPs) and sum of squares (SOS) programs, while solving much simpler classes of convex optimization problems, namely \emph{linear programs} (LPs) and \emph{second order cone programs} (SOCPs). Such a research direction can potentially lead to a better understanding of the relative power of different classes of convex relaxations. It also has obvious practical motivations as simpler convex programs come with algorithms that have better scalability and improved numerical conditioning properties. This chapter is a step in this research direction. We present a scheme for solving a sequence of LPs or SOCPs that provide increasingly accurate approximations to the optimal value and the optimal solution of a semidefinite (or a sum of squares) program. With the algorithms that we propose, one can use one of many mature LP/SOCP solvers such as~\cite{cplex, gurobi, mosek}, including simplex-based LP solvers, to obtain reasonable approximations to the optimal values of these more difficult convex optimization problems.

%Aside from practical motivations such as better scalability and improved numerical conditioning, such a research direction can potentially lead to a better understanding of the relative power of different classes of convex relaxations.

The intuition behind our approach is easy to describe with a contrived example.
%
%Consider the matrix $$A=\begin{pmatrix}3 &4&2\\ 4& 22& 6\\ 2& 6& 10
%\end{pmatrix},$$
%which is positive semidefinite, although this may not be immediately evident. We can diagonalize the matrix by writing it in a different basis: $$A=LDL^T=\begin{pmatrix} 1& 0 & 0 \\ \frac{4}{3} &1 & 0\\ \frac{2}{3} &\frac{1}{5} & 1 \end{pmatrix}\begin{pmatrix} 3 & 0 & 0 \\ 0 & 16\frac{2}{3} & 0 \\ 0 & 0 & 8 \end{pmatrix}\begin{pmatrix} 1 &\frac{4}{3} & \frac{2}{3} \\ 0 &1 & \frac{1}{5} \\ 0 & 1 & 1 \end{pmatrix}.$$ 
%
%If we think of the quadratic form $x^TAx$, by a change of basis we mean the transformation that sends the monomial basis $x=(x_1,\ldots,x_n)^T$ to the new basis $L^Tx$. If we had access to $L$, then proving nonnegativity of our quadratic form (i.e., positive semidefiniteness of $A$) would be equivalent to nonnegativity of the diagonal entires of $D$, which is a \emph{linear} condition in $D$.
%
Suppose we would like to show that the degree-4 polynomial 
\begin{equation}\nonumber
\begin{array}{lll}
p(x)&=&x_1^4-6x_1^3x_2+2x_1^3x_3+6x_1^2x_3^2+9x_1^2x_2^2-6x_1^2x_2x_3-14x_1x_2x_3^2+4x_1x_3^3
\\ \ &\ &+5x_3^4-7x_2^2x_3^2+16x_2^4
\end{array}
\end{equation}
has a sum of squares decomposition. One way to do this is to attempt to write $p$ as $${ p(x)=z^T(x)Qz(x),}$$
where
\begin{equation}\label{eq:z(x).quadatic}
z(x)=(x_1^2,x_1x_2,x_2^2,x_1x_3,x_2x_3,x_3^2)^T
\end{equation}  
is the standard (homogeneous) monomial basis of degree 2 and the matrix $Q$, often called the \emph{Gram matrix}, is symmetric and positive semidefinite. The search for such a $Q$ can be done with semidefinite programming; one feasible solution e.g. is as follows:
%One feasible solution to the SDP in (\ref{eq:p=z'Qz}) is given by 

$$Q=\begin{pmatrix}
1 & -3 & 0 & 1 & 0 & 2 \\
-3 & 9 & 0 & -3 & 0 & -6 \\
0 & 0 & 16 & 0 & 0 & -4 \\
1 & -3 & 0 & 2 & -1 & 2 \\
0 & 0 & 0 & -1 & 1 & 0 \\
2 & -6 & 4 & 2 & 0 & 5 
\end{pmatrix}.$$
Suppose now that instead of the basis $z$ in (\ref{eq:z(x).quadatic}), we pick a different basis 
%$$\tilde{z}(x)=(x_1^2-3x_1x_2+x_1x_3+2x_3^2, x_1x_3-x_2x_3, 4x_2^2-x_3^2)^T.$$
\begin{equation}\label{eq:ztilde}
{\tilde{z}(x)=(2x_1^2-6x_1x_2+2x_1x_3+2x_3^2, \quad x_1x_3-x_2x_3, \quad x_2^2-\frac{1}{4}x_3^2)^T.}
\end{equation}
With this new basis, we can get a sum of squares decomposition of $p$ by writing it as 
$$p(x)=\tilde{z}^T(x) \begin{pmatrix}
\frac{1}{2} & 0 & 0 \\ 0& 1 & 0 \\ 0 & 0 & 4.
\end{pmatrix}  \tilde{z}(x).$$
In effect, by using a better basis, we have simplified the Gram matrix and made it diagonal. When the Gram matrix is diagonal, its positive semidefiniteness can be imposed as a \emph{linear} constraint (diagonals should be nonnegative).

Of course, the catch here is that we do not have access to the magic basis $\tilde{z}(x)$ in (\ref{eq:ztilde}) a priori. Our goal will hence be to ``pursue'' this basis (or other good bases) by starting with an arbitrary basis (typically the standard monomial basis), and then iteratively improving it by solving a sequence of LPs or SOCPs and performing some efficient matrix decomposition tasks in the process. Unlike the intentionally simplified example we gave above, we will not ever require our Gram matrices to be diagonal. This requirement is too strong and would frequently lead to our LPs and SOCPs being infeasible. The underlying reason for this is that the cone of diagonal matrices is not full dimensional in the cone of positive semidefinite matrices. Instead, we will be after bases that allow the Gram matrix to be \emph{diagonally dominant} or \emph{scaled diagonally dominant} (see Definition~\ref{def:dd.sdd}). The use of these matrices in polynomial optimization has recently been proposed by Ahmadi and Majumdar~\cite{isos_journal, dsos_ciss14}. We will be building on and improving upon their results in this chapter.

\subsection{Organization of this chapter}
The organization of the rest of the chapter is as follows. In Section~\ref{sec:prelims}, we introduce some notation and briefly review the concepts of ``dsos and sdsos polynomials'' which are used later as the first step of an iterative algorithm that we propose in Section~\ref{sec:basis.pursuit}. In this section, we explain how we inner approximate semidefinite (Subsection~\ref{subsec:InnerApprox}) and sum of squares (Subsection~\ref{subsec:polyopt}) cones with LP and SOCP-based cones by iteratively changing bases. In Subsection~\ref{subsec:Corners}, we give a different interpretation of our LPs in terms of their corner description as opposed to their facet description. Subsection~\ref{subsec:OuterApprox} is about duality, which is useful for iteratively outer approximating semidefinite or sum of squares cones.

In Section~\ref{sec:StableSet}, we apply our algorithms to the Lov\'{a}sz semidefinite relaxation of the maximum stable set problem. It is shown numerically that our LPs and SOCPs converge to the SDP optimal value in very few iterations and outperform some other well-known LP relaxations on a family of randomly generated examples. In Section~\ref{sec:Partition}, we consider the partition problem from discrete optimization. As opposed to the stable set problem, the quality of our relaxations here is rather poor. In fact, even the sum of squares relaxation fails on some completely trivial instances. We show this empirically on random instances, and formally prove it on one representative example (Subsection~\ref{subsec:sos.refutability}). The reason for this failure is existence of a certain family of quartic polynomials that are nonnegative but not sums of squares.

%After briefly reviewing some algebraic concepts and notation in Section \ref{sec:prelims}, we give a high-level overview of our basis pursuit approach applied to a general SDP in Section \ref{sec:basis.pursuit}. Ways of obtaining both lower and upper bounds on the optimal solution of the SDP are presented. Section\ref{sec:StableSet} showcases an application of these techniques to the well-known problem of finding the stability number of a graph, with promising computational results. Finally, Section \ref{sec:Partition} describes another potential application of the basis pursuit approach to discrete optimization: the partition problem. Building on computational results obtained, we study the effectiveness of algebraic methods in refuting partition instances.

%============================================================
\section{Preliminaries}\label{sec:prelims}
%============================================================
%Let us first introduce some notation on matrices.
We denote the set of real symmetric $n\times n$ matrices by $S_n$. Given two matrices $A$ and $B$ in $S_n$, their standard matrix inner product is denoted by $A\cdot B := \sum_{i,j}A_{ij}B_{ij} = \mbox{ Trace}(AB)$.
%The set of symmetric matrices with nonnegative entries is denoted by $N_n$.
A symmetric matrix $A$ is \emph{positive semidefinite} (psd) if $x^TAx\geq 0$ for all $x\in\mathbb{R}^n$; this will be denoted by the standard notation $A\succeq 0$, and our notation for the set of $n\times n$ psd matrices is $P_n$. We say that $A$ is \emph{positive definite} (pd) if $x^TAx> 0$ for all $x\neq 0$.
%A matrix $A$ is \emph{copositive} if $x^TAx\geq 0$ for all $x\geq 0$. The set of copositive matrices is denoted by $C_n$. All three sets $N_n,P_n,C_n$ are convex cones and we have the obvious inclusion $N_n+P_n\subseteq C_n$. This inclusion is strict if $n\geq 5$~\cite{burer2009difference, burer2012copositive}.
Any psd matrix $A$ has an upper triangular Cholesky factor $U=\text{chol}(A)$ satisfying $A=U^TU$. When $A$ is pd, the Cholesky factor is unique and has positive diagonal entries.
%We use the notation $U=\text{chol}(A)$.
For a cone of matrices in $S_n$, we define its dual cone $\mathcal{K}^*$ as $\{Y \in S_n: Y\cdot X \geq 0, \ \forall X \in \mathcal{K}\}$.

For a vector variable $x \in \mathbb{R}^n$ and a vector {$s \in \mathbb{Z}^n_+$, let a monomial in $x$ be denoted as $x^s= \Pi_{i=1}^n x_i^{s_i}$ which by definition has degree $\sum_{i=1}^n s_i$.}
A polynomial is said to be \emph{homogeneous} or a \emph{form} if all of its monomials have the same degree. A form $p(x)$ in $n$ variables is nonnegative if $p(x)\geq 0$ for all $x\in\mathbb{R}^n$, or equivalently for all $x$ on the unit sphere in $\mathbb{R}^n$. The set of nonnegative (or positive semidefinite) forms in $n$ variables and degree $d$ is denoted by $PSD_{n,d}$. A form $p(x)$ is a \emph{sum of squares} (sos) if it can be written as $p(x)=\sum_{i=1}^r q_i^2(x)$ for some forms $q_1,\ldots,q_r$. The set of sos forms in $n$ variables and degree $d$ is denoted by $SOS_{n,d}$. We have the obvious inclusion $SOS_{n,d}\subseteq PSD_{n,d}$, which is strict unless $d=2$, or $n=2$, or $(n,d)=(3,4)$~\cite{Hilbert_1888}. Let $z(x,d)$ be the vector of all monomials of degree exactly $d$; it is well known that a form $p$ of degree $2d$ is sos if and only if it can be written as $p(x)=z^T(x,d)Qz(x,d)$, for some psd matrix $Q$~\cite{sdprelax, PhD:Parrilo}. An SOS optimization problem is the problem of minimizing a linear function over the intersection of the convex cone $SOS_{n,d}$ with an affine subspace. The previous statement implies that SOS optimization problems can be cast as semidefinite programs.

%The size of the matrix $Q$, which is often called the \emph{Gram matrix}, is ${n+d-1 \choose d} \times {n+d-1 \choose d}$. At the price of imposing a semidefinite constraint of this size, one obtains the very useful ability to search and optimize over the convex cone of sos forms via semidefinite programming.

%Let $N_n, S_n, C_n$, stand for the cones of $n\times n$ symmetric nonnegative matrices, diagonally dominant matrices, positive semidefinite matrices (PSD cone) and copositive matrices, respectively.
%
%$D\in D_n$ is a diagonally dominant matrix if $D_{ii} \geq \sum_{j\neq i}|D_{ij}|$ for $i=1, \ldots, n$.
%For a cone of matrices $\mathcal{C} \subseteq \R^{n \times n}$, the dual cone $\mathcal{C}^*$ is defined as $\{Y \in \R^{n\times n}: Y\cdot X \geq 0 \ \forall X \in \mathcal{C}\}$.
%It is well-known that (i) $D_n \subseteq S_n \subseteq C_n$; (ii) $N_n \subseteq C_n$ which also implies that $S_n + N_n \subseteq C_n$; (iii) 
%$S_n^* = S_n$.
%Let $e_1, \ldots e_n$ be the unit vectors in $\R^n$. Let $U_k \subset \R^{n \times n}$ be the set of matrices defined as
%\[ U_k = \{uu^T : u \mbox{ has at most } k \mbox{ nonzero components, each equal to} \pm 1\}. \]
%Clearly $U_k$ is a finite set for each $k=1, \ldots, n$ and it is not hard to see that $D_n = \cone{U_2}$ and therefore $D_n^* = \{X \in \R^{n\times n}: X \cdot V \geq 0 \ \forall V \in U_2\}$.
%Furthermore as $D_n \subseteq S_n$, we have $S_n \subseteq D_n^*$.

\subsection{DSOS and SDSOS optimization}
\label{subsec:dsos.sdsos}

In recent work, Ahmadi and Majumdar introduce more scalable alternatives to SOS optimization that they refer to as \emph{DSOS and SDSOS programs}~\cite{isos_journal,dsos_ciss14}\footnote{The work in~\cite{isos_journal} is currently in preparation for submission; the one in~\cite{dsos_ciss14} is a shorter conference version of~\cite{isos_journal} which has already appeared. The presentation of the current chapter is meant to be self-contained.}. Instead of semidefinite programming, these optimization problems can be cast as linear and second order cone programs respectively. Since we will be building on these concepts, we briefly review their relevant aspects to make our chapter self-contained.

%In order to alleviate the problem of scalability posed by the SDPs arising from sum of squares programs, Ahmadi and Majumdar recently introduced similar-purpose LP and SOCP-based optimization problems that they refer to as \emph{DSOS and SDSOS programs}. 

%In order to address the problem of scalability posed by SDP, we have recently introduced~\cite{isos_journal, dsos_ciss14} alternatives to SOS programming that lead to linear programs (LPs) and second order cone programs (SOCPs). 

The idea in~\cite{isos_journal, dsos_ciss14}  is to replace the condition that the Gram matrix $Q$ be positive semidefinite with stronger but cheaper conditions in the hope of obtaining more efficient inner approximations to the cone $SOS_{n,d}$. Two such conditions come from the concepts of \emph{diagonally dominant} and \emph{scaled diagonally dominant} matrices in linear algebra. We recall these definitions below.

\begin{definition}\label{def:dd.sdd}
	A symmetric matrix $A$ is \emph{diagonally dominant} (dd) if $a_{ii} \geq \sum_{j \neq i} |a_{ij}|$ for all $i$. We say that $A$ is \emph{scaled diagonally dominant} (sdd) if there exists a diagonal matrix $D$, with positive diagonal entries, which makes $DAD$ diagonally dominant.
\end{definition}

We refer to the set of $n \times n$ dd (resp. sdd) matrices as $DD_n$ (resp. $SDD_n$). The following inclusions are a consequence of Gershgorin's circle theorem {\cite{gersh}}:
$$DD_n\subseteq SDD_n\subseteq P_n.$$

Whenever it is clear from the context, we may drop the subscript $n$ from our notation. We now use these matrices to introduce the cones of ``dsos'' and ``sdsos'' forms which constitute special subsets of the cone of sos forms. We remark that in the interest of brevity, we do not give the original definition of dsos and sdsos polynomials as it appears in~\cite{isos_journal} (as sos polynomials of a particular structure), but rather an equivalent characterization of them that is more useful for our purposes. The equivalence is proven in~\cite{isos_journal}. 

%We now introduce some cones that are inner approximations of the cone of nonnegative polynomials and that lend themselves to LP and SOCP. In analogy with the representation of sos polynomials in terms of psd matrices (Theorem \ref{thm:sos.sdp}), we define the \emph{dsos} and \emph{sdsos} polynomials in terms of dd and sdd matrices respectively.

\begin{definition}[\cite{isos_journal,dsos_ciss14}] \label{def:dsos.sdsos.rdsos.rsdsos}
	Recall that $z(x,d)$ denotes the vector of all monomials of degree exactly $d$. A form $p(x)$ of degree $2d$ is said to be
	
	\begin{itemize}
		\item  \emph{diagonally-dominant-sum-of-squares} (dsos) if it admits a representation as $p(x)=z^T(x,d)Qz(x,d)$, where $Q$ is a dd matrix.
		\item  \emph{scaled-diagonally-dominant-sum-of-squares} (sdsos) if it admits a representation as $p(x)=z^T(x,d)Qz(x,d)$, where $Q$ is an sdd matrix.
		%\item \emph{$r$-diagonally-dominant-sum-of-squares} ($r$-dsos) if there exists a positive integer $r$ such that \\$p(x) (\sum_{i=1}^n x_i^2)^r$ is dsos.
		%	\item \emph{$r$-scaled diagonally-dominant-sum-of-squares} ($r$-sdsos) if there exists a positive integer $r$ such that \\$p(x)  (\sum_{i=1}^n x_i^2)^r$ is sdsos.\footnote{Although we gave all the definitions for forms (i.e., homogeneous polynomials), they are no different for non-homogeneous polynomials, once one dehomogenizes the monomial vector $z(x,d)$ by setting one of its variables equal to 1.}
	\end{itemize}
\end{definition}

The definitions for non-homogeneous polynomials are exactly the same, except that we replace the vector of monomials of degree exactly $d$ with the vector of monomials of degree $\leq d$. We observe that a quadratic form $x^TQx$ is dsos/sdsos/sos if and only if the matrix $Q$ is dd/sdd/psd. Let us denote the cone of forms in $n$ variables and degree $d$ that are dsos and sdsos by $DSOS_{n,d}$, $SDSOS_{n,d}$. The following inclusion relations are straightforward: $$DSOS_{n,d}\subseteq SDSOS_{n,d}\subseteq SOS_{n,d}\subseteq PSD_{n,d}.$$

%The multiplier $(\sum_{i=1}^n x_i^2)^r$ should be thought of as a special denominator in the Artin-type representation in (\ref{eq:sos}). 

%The use of the multiplier $(\sum_{i=1}^n x_i^2)^r$ is motivated by a result of Reznick~\cite{Reznick_Unif_denominator} related to Hilbert's 17th problem, which states that if a form is positive definite (i.e., strictly positive on the unit sphere), then there exists and integer $r$ such that $p\cdot (\sum_{i=1}^n x_i^2)^r$ is a sum of squares.

%By appealing to some theorems of real algebraic geometry, it is shown in~\cite{isos_journal} that under some conditions, as the power $r$ increases, the sets $rDSOS_{n,d}$ (and hence $rSDSOS_{n,d}$) fill up the entire cone $PSD_{n,d}.$ We will mostly be concerned with the cones $DSOS_{n,d}$ and $SDSOS_{n,d}$, which correspond to the case where $r=0$. 

From the point of view of optimization, our interest in all of these algebraic notions stems from the following theorem.

\begin{theorem}[\cite{isos_journal,dsos_ciss14}] For any fixed $d$, optimization over the cones $DSOS_{n,d}$ (resp. $SDSOS_{n,d}$) can be done with linear programming (resp. second order cone programming) of size polynomial in $n$.
\end{theorem}

The ``LP part'' of this theorem is not hard to see. The equality $p(x)=z^T(x,d)Qz(x,d)$ gives rise to linear equality constraints between the coefficients of $p$ and the entries of the matrix $Q$ (whose size is $\sim n^{\frac{d}{2}}\times n^{\frac{d}{2}}$ and hence polynomial in $n$ for fixed $d$). The requirement of diagonal dominance on the matrix $Q$ can also be described by linear inequality constraints on $Q$. The ``SOCP part'' of the statement comes from the fact, shown in~\cite{isos_journal}, that a matrix $A$ is sdd if and only if it can be expressed as 	\begin{align}\label{eq:SDSOS.def}A = \sum_{i< j} M_{2 \times 2}^{ij},\end{align}
where each $ M_{2 \times 2}^{ij}$ is an $n\times n$ symmetric matrix with zeros everywhere except for four entries $M_{ii}, M_{ij}, M_{ji}, M_{jj}$, which must make the $2\times 2$ matrix $\begin{bmatrix} M_{ii} & M_{ij}  \\ M_{ji} & M_{jj} \end{bmatrix}$ symmetric and positive semidefinite. These constraints are \emph{rotated quadratic cone} constraints and can be imposed using SOCP~\cite{alizadeh, socp_boyd}:
$$M_{ii}\geq 0, ~\Bigl\lvert\Bigl\lvert\begin{pmatrix}
2M_{ij}\\M_{ii}-M_{jj}
\end{pmatrix}\Bigl\lvert\Bigl\lvert \leq M_{ii}+M_{jj}.$$

We refer to linear optimization problems over the convex cones $DSOS_{n,d}$, $SDSOS_{n,d}$, and $SOS_{n,d}$ as DSOS programs, SDSOS programs, and SOS programs respectively. In general, quality of approximation decreases, while scalability increases, as we go from SOS to SDSOS to DSOS programs. What we present next can be thought of as an iterative procedure for moving from DSOS/SDSOS relaxations towards SOS relaxations without increasing the problem size in each step.

%In this chapter, we will be using SOS optimization (Section~\ref{sec:jamming}) and SDSOS optimization (Sections~\ref{sec:barriers} and~\ref{sec:quadrotor}) in our numerical experiments. The reader is referred to~\cite{dsos_cdc14,dsos_ciss14,Ahmadi14} for many numerical examples involving DSOS optimization. We also remark in passing that SDSOS or even DSOS programming enjoy many of the same theoretical (asymptotic) guarantees of SOS programming---results of this nature are proven in~\cite{Ahmadi14}.

\section{Pursuing improved bases}\label{sec:basis.pursuit}

Throughout this section, we consider the standard SDP
\begin{equation} \label{eq:genericSDP}
\begin{aligned}
SOS^* &\mathrel{\mathop{:}}=\min_{X\in S_n} C \cdot X\\
&\text{s.t. } A_i \cdot X=b_i, i=1,\ldots,m,\\
&\quad \quad \quad \ \ X \succeq 0,
\end{aligned}
\end{equation}
which we assume to have an optimal solution. We denote the optimal value by $SOS^*$ since we think of a semidefinite program as a sum of squares program over quadratic forms (recall that $PSD_{n,2}=SOS_{n,2}$). This is so we do not have to introduce additional notation to distinguish between degree-2 and higher degree SOS programs. The main goal of this section is to construct sequences of LPs and SOCPs that generate bounds on the optimal value of (\ref{eq:genericSDP}). Section  \ref{subsec:InnerApprox} focuses on providing upper bounds on (\ref{eq:genericSDP}) while Section \ref{subsec:OuterApprox} focuses on lower bounds.

%At optimality $X^*$ is psd, and it has a Cholesky decomposition $U^*$ that verifies $$X^*=U^{*T}U^*.$$ The main goal of this section is to construct sequences of LPs and SOCPs that generate bounds on the optimal value of (\ref{eq:genericSDP}). This is done by creating candidate matrices $U_k$ verifying $X_{k-1}=U_{k}^TU_{k}$ where $X_{k-1}$ is the optimal solution at iterate $k-1$ of the sequence. At each iteration, the bounds obtained through the LPs/SOCPs improve. The hope is then that the sequence of matrices $\{U_k\}_k$ will eventually converge to $U^*$, providing us thereby with the exact SDP bound.

%%%%%%%%%%%%%%%%%%%%%%%%%%%%%%%%%%%%%%%%%%%%%%%%
\subsection{Inner approximations of the psd cone} \label{subsec:InnerApprox}
%%%%%%%%%%%%%%%%%%%%%%%%%%%%%%%%%%%%%%%%%%%%%%%

To obtain upper bounds on (\ref{eq:genericSDP}), we need to replace the constraint $X \succeq 0$ by a stronger condition. In other words, we need to provide \emph{inner approximations} to the set of psd matrices.

First, let us define a family of cones $$DD(U)\mathrel{\mathop{:}}=\{M \in S_n~|~ M=U^TQU \text{ for some dd matrix } Q \},$$parametrized by an $n \times n$ matrix $U$. Optimizing over the set $DD(U)$ is an LP since $U$ is fixed, and the defining constraints are linear in the coefficients of the two unknowns $M$ and $Q$. Furthermore, the matrices in $DD(U)$ are all psd; i.e., $\forall U,$ $DD(U) \subseteq P_n$.

The iteration number $k$ in the sequence of our LPs consists of replacing the condition $X \succeq 0$ by $X \in DD(U_k)$:
\begin{equation}\label{eq:LPChol}
\begin{aligned}
DSOS_k &\mathrel{\mathop{:}}=\min C \cdot X\\
&\text{s.t. } A_i \cdot X=b_i, ~i=1,\ldots,m,\\
&X \in DD(U_k).
\end{aligned}
\end{equation}
To define the sequence $\{U_k\}$, we assume that an optimal solution $X_k$ to (\ref{eq:LPChol}) exists for every iteration. As it will become clear shortly, this assumption will be implied simply by assuming that only the first LP in the sequence is feasible. The sequence $\{U_k\}$ is then given recursively by
\begin{equation}\label{eq:defUk}
\begin{aligned}
U_0&=I\\
U_{k+1}&=\text{chol}(X_k).
\end{aligned}
\end{equation}

Note that the first LP in the sequence optimizes over the set of diagonally dominant matrices as in the work of Ahmadi and Majumdar~\cite{isos_journal,dsos_ciss14}. By defining $U_{k+1}$ as a Cholesky factor of $X_k$, improvement of the optimal value is guaranteed in each iteration. Indeed, as $X_k=U_{k+1}^T I U_{k+1}$, and the identity matrix $I$ is diagonally dominant, we see that $X_{k} \in DD(U_{k+1})$ and hence is feasible for iteration $k+1$. This entails that the optimal value at iteration $k+1$ is at least as good as the optimal value at the previous iteration; i.e., $DSOS_{k+1}\leq DSOS_k$. Since the sequence $\{DSOS_k\}$ is lower bounded by $SOS^*$ and monotonic, it must converge to a limit $DSOS^*\geq SOS^*$. We have been unable to formally rule out the possibility that $DSOS^*>SOS^*$. In all of our numerical experiments, convergence to $SOS^*$ happens (i.e., $DSOS^*=SOS^*$), though the speed of convergence seems to be problem dependent (contrast e.g. the results of Section~\ref{sec:StableSet} with Section~\ref{sec:Partition}). What is easy to show, however, is that if $X_k$ is positive definite\footnote{This would be the case whenever our inner approximation is not touching the boundary of the psd cone in the direction of the objective. As far as numerical computation is concerned, this is of course always the case.}, then the improvement from step $k$ to $k+1$ is actually \emph{strict}.

%\begin{theorem}\label{thm:strict.imp}
%Let $X_k$ (resp. $X_{k+1}$) be an optimal solution of iterate $k$ (resp. $k+1$) of (\ref{eq:LPChol}), and assume that $DSOS_{k+1}>DSOS^*$, where $DSOS^*$ is the limit of the sequence $\{DSOS_k\}_k$ as defined above. Then $$DSOS_{k+1}<DSOS_{k}.$$
%%In other words, improvement is strict at each iteration.
%\end{theorem}
\begin{theorem}\label{thm:strict.imp}
	Let $X_k$ (resp. $X_{k+1}$) be an optimal solution of iterate $k$ (resp. $k+1$) of (\ref{eq:LPChol}) and assume that $X_k$ is pd and $SOS^*<DSOS_{k}$. Then, $$DSOS_{k+1}<DSOS_{k}.$$
	%In other words, improvement is strict at each iteration.
\end{theorem}

\begin{proof}
	We show that for some $\lambda\in(0,1)$, the matrix $\hat{X}\mathrel{\mathop{:}}= (1-\lambda)X_k+\lambda X^*$ is feasible to the LP in iteration number $k+1$. We would then have that $$DSOS_{k+1}\leq C \cdot \hat{X}=(1-\lambda)C \cdot X_k+\lambda C \cdot X^*<{ DSOS_k,}$$
	as we have assumed that $C \cdot X^*=SOS^*<DSOS_k=C \cdot X_k.$ To show feasibility of $\hat{X}$ to LP number $k+1$, note first that as both $X_k$ and $X^*$ satisfy the affine constraints $A_i \cdot X=b_i$, then $\hat{X}$ must also. Since $X_k=U_{k+1}^TU_{k+1}$ and $X_k$ is pd,  $U_{k+1}$ must have positive diagonal entries and is invertible. Let $$X_{k+1}^*\mathrel{\mathbb:}=U_{k+1}^{-T}X^*U_{k+1}^{-1}.$$ For $\lambda$ small enough the matrix $(1-\lambda)I+\lambda X_{k+1}^*$ will be dd since we know the identity matrix is strictly diagonally dominant. Hence, the matrix $$\hat{X}=U_{k+1}^T( (1-\lambda)I+\lambda X_{k+1}^*) U_{k+1}$$ is feasible to LP number $k+1$.
	%Since the identity matrix is strictly diagonally dominant, for $\lambda$ small enough the matrix $(1-\lambda)I+\lambda X_{k+1}^*$ will be dd. Hence, the matrix $$\hat{X}=U_{k+1}^T( (1-\lambda)I+\lambda X_{k+1}^*) U_{k+1}$$ is feasible to LP number $k+1$.
\end{proof}

%Theorem \ref{thm:strict.imp} given below proves an even stronger result: the sequence $\{DSOS_k\}_k$ is in fact \emph{strictly} decreasing. Because the LP optimal value $DSOS_k$ is lower bounded by $SOS^*$ (the SDP optimal value, assumed to be finite), and is monotonically decreasing with $k$, $DSOS_k$ converges to some limit $DSOS^*$. Theoretically, it could be that $DSOS^*$ be strictly greater than the optimal value obtained in SDP (\ref{eq:genericSDP}). In practice however, we seem to always obtain the optimal value of the SDP, though convergence can sometimes be slow. 

%where the matrix $U_k$ is defined recursively. We start with $U_0=I$, so our initial LP is optimizing over the set of diagonally dominant matrices. This LP is lower bounded by $SOS^*$, and we assume that it is feasible. This assumption is satisfied 

%constructed from an optimal solution to the previous problem in such a way that the optimal value at each iteration improves. The sequence can be written
%\begin{equation}\label{eq:LPChol}
%\begin{aligned}
%DSOS_k &\mathrel{\mathop{:}}=\min C \cdot X\\
%&\text{s.t. } A_i \cdot X=b_i, ~\forall i\\
%&X \in DD(U_k).
%\end{aligned}
%\end{equation}
%To define the sequence $\{U_k\}$, we first assume that an optimal solution $X_k$ to (\ref{eq:LPChol}) exists for every iteration $k$. We then take
%\begin{equation}\label{eq:defUk}
%\begin{aligned}
%U_0&=I\\
%U_{k+1}&=\text{chol}(X_k).
%\end{aligned}
%\end{equation}

A few remarks are in order. First, instead of the Cholesky decomposition, we could have worked with some other decompositions such as the LDL decomposition $X_k=LDL^T$ or the spectral decomposition $X_k=H^T\Lambda H$ (where $H$ has the eigenvectors of $X_k$ as columns). Aside from the efficiency of the Cholesky decomposition, the reason we made this choice is that the decomposition allows us to write $X_k$ as $U^TIU$ and the identity matrix $I$ is at the analytic center of the set of diagonally dominant matrices {\cite[Section 8.5.3]{BoydBook}}. Second, the reader should see that feasibility of the first LP implies that all future LPs are feasible and lower bounded. While in most applications that we know of the first LP is automatically feasible (see, e.g., the stable set problem in Section~\ref{sec:StableSet}), sometimes the problem needs to be modified to make this the case. An example where this happens appears in Section~\ref{sec:Partition} (see Theorem~\ref{thm:feas.dsos}), where we apply an SOS relaxation to the partition problem.

Alternatively, one can first apply our iterative procedure to a Phase-I problem
\begin{equation} \label{eq:PhaseI}
\begin{aligned}
\alpha_k &\mathrel{\mathop:}= \min \alpha\\
&\text{s.t. } A_i \cdot X=b_i, i=1,\ldots,m\\
&X+\alpha I \in DD(U_k),
\end{aligned}
\end{equation}
with $U_k$ defined as in (\ref{eq:defUk}). Indeed, for $\alpha$ large enough, the initial problem in (\ref{eq:PhaseI}) (i.e., with $U_0=I$) is feasible. Thus all subsequent iterations are feasible and continually decrease $\alpha$. If for some iteration $k$ we get $\alpha_k \leq 0,$ then we can start the original LP sequence (\ref{eq:LPChol}) with the matrix $U_k$ obtained from the last iteration of the Phase-I algorithm.

In an analogous fashion, we can construct  a sequence of SOCPs that provide upper bounds on $SOS^*$. This time, we define a family of cones $${SDD(U)\mathrel{\mathop{:}}=\{M \in S_n ~|~ M=U^TQU, \text{ for some sdd matrix } Q\},}$$ parameterized again by an $n\times n$ matrix $U$. For any $U$, optimizing over the set $SDD(U)$ is an SOCP and we have $SDD(U)\subseteq P_n$. This leads us to the following iterative SOCP sequence:
\begin{equation} \label{eq:SOCPchol}
\begin{aligned} 
SDSOS_k  &\mathrel{\mathop{:}}= \min C \cdot X\\
&\text{s.t. } A_i \cdot X=b_i, i=1,\ldots,m,\\
&X \in SDD(U_k).
\end{aligned}
\end{equation}
Assuming existence of an optimal solution $X_k$ at each iteration, we can once again define the sequence $\{U_k\}$ iteratively as
\begin{align*}
U_0&=I\\
U_{k+1}&=\text{chol}(X_k).
\end{align*}

{ The previous statements concerning strict improvement of the LP sequence as described in Theorem \ref{thm:strict.imp}, as well as its convergence} carry through for the SOCP sequence. In our experience, our SOCP bounds converge to the SDP optimal value often faster than our LP bounds do. While it is always true that $SDSOS_0\leq DSOS_0$ (as $DD\subseteq SDD$), the inequality can occasionally reverse in future iterations.  

%As in the LP case, the existence of an optimal solution is contingent on the feasibility of the first iteration of (\ref{eq:SOCPchol}). In the case where the first iteration is not feasible, we can also define a Phase-I type algorithm for the SOCP sequence, replacing $X+ \alpha I \in DD(U_k)$ by $X +\alpha I \in SDD(U_k)$ in (\ref{eq:PhaseI}). Finally, the convergence results remain true: in particular, each iteration strictly improves on the previous one. This can be seen through an analoguous proof to the one given in Theorem \ref{thm:strict.imp}.

\begin{figure}[h!]
	\begin{center}
		\mbox{
			\subfigure[LP inner approximations]
			{\label{subfig:DDCones1D}\scalebox{0.46}{\includegraphics{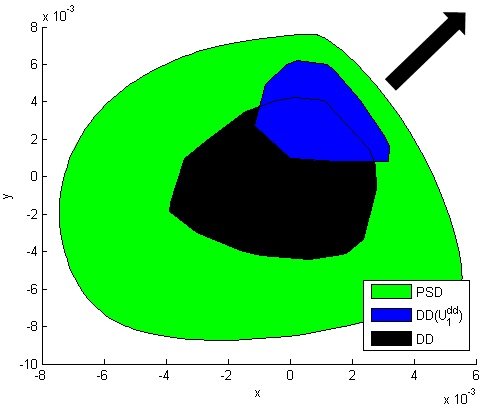}}}}
		\mbox{
			\subfigure[SOCP inner approximations]
			{\label{subfig:SDDCones1D}\scalebox{0.45}{\includegraphics{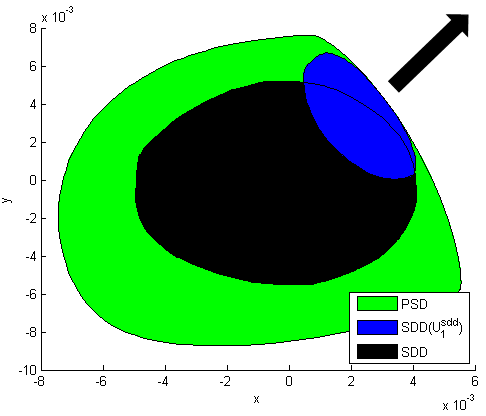}}}
		}
		
		\caption{Improvement after one Cholesky decomposition when maximizing the objective function $x+y$}
		\label{fig:DDSDDCones1D}
	\end{center}
	%\vspace{20pt}
\end{figure}

An illustration of both procedures is given in Figure \ref{fig:DDSDDCones1D}. We generated two random symmetric matrices $A$ and $B$ of size $10 \times 10$. The outermost set is the feasible set of an SDP with the constraint $I+xA+yB \succeq 0$. The goal is to maximize the function $x+y$ over this set. The set labeled $DD$ in Figure \ref{subfig:DDCones1D} (resp. $SDD$ in Figure \ref{subfig:SDDCones1D}) consists of the points $(x,y)$ for which $I+xA+yB$ is dd (resp. sdd). Let ($x_{dd}^*,y_{dd}^*$) (resp. ($x_{sdd}^*,y_{sdd}^*$)) be optimal solutions to the problem of maximizing $x+y$ over these sets. The set labeled $DD(U_1^{dd})$ in Figure \ref{subfig:DDCones1D} (resp. $SDD(U_1^{sdd})$ in Figure \ref{subfig:SDDCones1D}) consists of the points $(x,y)$ for which $I+xA+yB \in DD(U_1^dd)$ (resp. $\in SDD(U_1^{sdd}$)) where $U_1^{dd}$ (resp. $U_1^{sdd}$) corresponds to the Cholesky decomposition of $I+x_{dd}^*A+y_{dd}^*B$ (resp. $I+x_{sdd}^*A+y_{sdd}^*B$). Notice the interesting phenomenon that while the new sets happen to shrink in volume, they expand in the direction that we care about. Already in one iteration, the SOCP gives the perfect bound here.

%In other words, we are improving the set in the direction given by $x+y$. Notice that our inner approximations after one iteration are worse in terms of volume but better in terms of the objective we care about.

%\begin{figure}[H]
%	\begin{center}
%		\mbox{
%			\subfigure{LP inner approximations]
%			{\label{subfig:DDCones}\scalebox{0.46}{\includegraphics{CholDD1D}}}}
%		\mbox{
%			\subfigure[SOCP inner approximations]
%			{\label{subfig:SDDCones}\scalebox{0.45}{\includegraphics{CholeskyconesMax1DSDD}}}
%		}
%		
%		\caption{Improvement in one step in one direction given by $\begin{pmatrix}1 1 \end{pmatrix}^T$}
%		\label{fig:DDSDDCones1D}
%	\end{center}
%	%\vspace{20pt}
%\end{figure}

\begin{figure}[h!]
	\begin{center}
		\mbox{
			\subfigure[LP inner approximations]
			{\label{subfig:DDCones}\scalebox{0.46}{\includegraphics{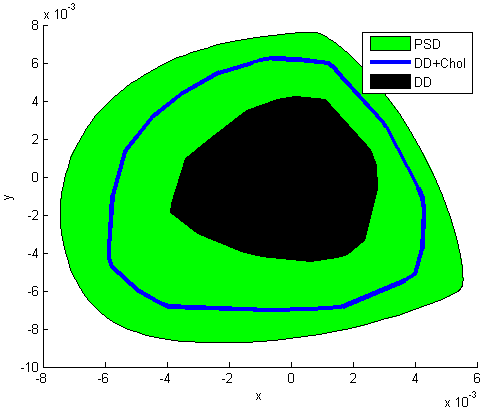}}}}
		\mbox{
			\subfigure[SOCP inner approximations]
			{\label{subfig:SDDCones}\scalebox{0.45}{\includegraphics{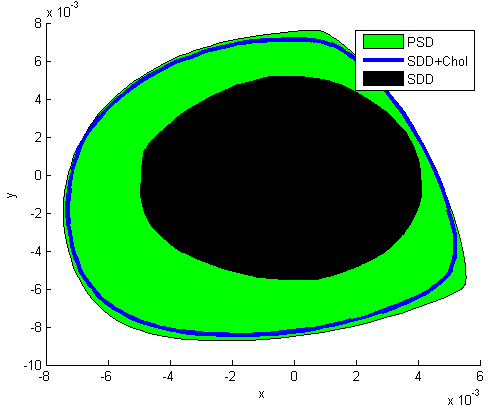}}}
		}
		
		\caption{Improvement in all directions after one Cholesky decomposition}
		\label{fig:DDSDDCones}
	\end{center}
	%\vspace{20pt}
\end{figure}

In Figure \ref{subfig:DDCones}, instead of showing the improvement in just the North-East direction, we show it in all directions. This is done by discretizing a large set of directions $d_i=(d_{i,x},d_{i,y})$ on the unit circle and optimizing along them. More concretely, for each $i$, we maximize $d_{i,x}x+d_{i,x}y$ over the set $I+xA+yB \in DD_n$. We extract an optimal solution every time and construct a matrix $U_{1,d_i}$ from its Cholesky decomposition. We then maximize in the same direction once again but this time over the set $I+xA+yB \in DD(U_{1,d_i})$. The set of all new optimal solutions is what is plotted with the thick blue line in the figure. We proceed in exactly the same way with our SOCPs to produce Figure~\ref{subfig:SDDCones}. Notice that both inner approximations after one iteration improve substantially. The SOCP in particular fills up almost the entire spectrahedron.

%In both Figure \ref{subfig:DDCones} and \ref{subfig:SDDCones}, the innermost and outermost sets are identical to the ones drawn in Figure \ref{fig:DDSDDCones}. Notice that both inner approximations after one iteration improve substantially on the initial inner approximations. This is particularly the case for the SOCP-based algorithm.

\subsection{Inner approximations to the cone of nonnegative polynomials}\label{subsec:polyopt}

%As mentioned in the Introduction, many polynomial optimization problems consist of optimizing over the set of nonnegative polynomials, which is generally inner approximated by the computationally tractable set of sum-of-squares polynomials. Indeed, as seen in Section \ref{sec:prelims}, sum-of-square constraints can be encoded as SDPs. Generating a sequence of LP or SOCP-based bounds on the SDP bound must then be done by inner approximating the psd cone if one wants to obtain valid bounds on the original problem. This is where the algorithms defined in Section \ref{subsec:InnerApprox} come in useful.

A problem domain where inner approximations to semidefinite programs can be useful is in sum of squares programming. This is because the goal of SOS optimization is already to inner approximate the cone of nonnegative polynomials. So by further inner approximating the SOS cone, we will get bounds in the same direction as the SOS bounds.

Let $z(x)$ be the vector of monomials of degree up to $d$. Define a family of cones of degree-$2d$ polynomials $${DSOS(U)\mathrel{\mathop:}=\{p ~|~ p(x)=z^T(x)U^TQUz(x), \text{ for some dd matrix } Q \},}$$ parameterized by an $n\times n$ matrix $U$. We can think of this set as the cone of polynomials that are dsos in the basis $Uz(x)$. If an SOS program has a constraint ``$p$ sos'', we will replace it iteratively by the constraint $p \in DSOS(U_k)$. The sequence of matrices $\{U_k\}$ is again defined recursively with
\begin{align*}
U_0&=I\\
U_{k+1}&=\mbox{chol}(U_k^TQ_kU_k),
%p_k&=\mathrel{\mathop{:}}z(x)^TU_{k+1}^TU_{k+1}z(x).
\end{align*}
where $Q_k$ is an optimal Gram matrix of iteration $k$. 

Likewise, let $${ SDSOS(U)\mathrel{\mathop:}=\{p~|~ p(x)=z(x)^TU^TQUz(x), \text{ for some sdd matrix } Q \}.}$$  This set can also be viewed as the set of polynomials that are sdsos in the basis $Uz(x)$. To construct a sequence of SOCPs that generate improving bounds on the sos optimal value, we replace the constraint $p$ sos by $p \in SDSOS(U_k)$, where $U_k$ is defined as above.

\begin{figure}[h!]
	\begin{center}
		\mbox{
			\subfigure[LP inner approximations]
			{\label{subfig:DSOSCones}\scalebox{0.46}{\includegraphics{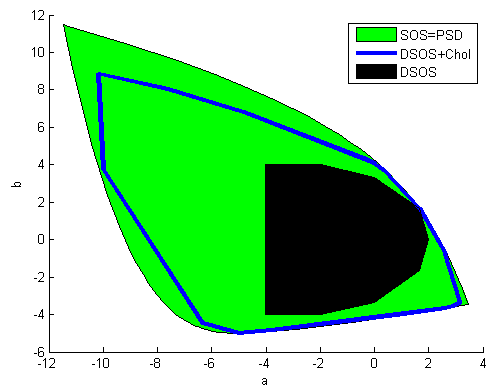}}}}
		\mbox{
			\subfigure[SOCP inner approximations]
			{\label{subfig:SDSOSCones}\scalebox{0.46}{\includegraphics{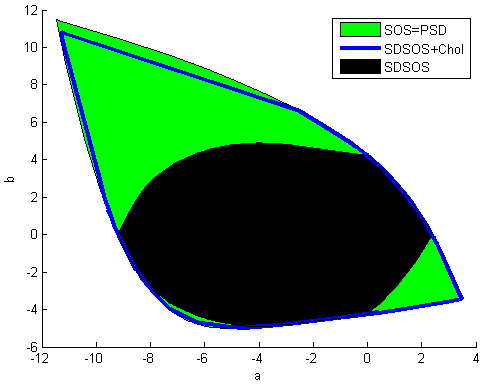}}}
		}
		
		\caption{Improvement in all directions after one Cholesky decomposition}
		\label{fig:DSOSSDSOSCones}
	\end{center}
	%\vspace{20pt}
\end{figure}

In Figure \ref{fig:DSOSSDSOSCones}, we consider a parametric family of polynomials
$$p_{a,b}(x_1,x_2)=2x_1^4+2x_2^4+ax_1^3x_2+(1-a)x_2^2x_2^2+bx_1x_2^3.$$ The outermost set in both figures corresponds to the set of { pairs} $(a,b)$ for which $p_{a,b}$ is sos. As $p_{a,b}$ is a bivariate quartic, this set coincides with the set of $(a,b)$ for which $p_{a,b}$ is nonnegative. The innermost sets in the two subfigures correspond to $(a,b)$ for which $p_{a,b}$ is dsos (resp. sdsos). The { thick blue} lines illustrate the optimal points achieved when maximizing in all directions over the sets obtained from a single Cholesky decomposition. (The details of the procedure are exactly the same as Figure~\ref{fig:DDSDDCones}.) Once again, the inner approximations after one iteration improve substantially over the DSOS and SDSOS approximations.

% In Figure \ref{subfig:DSOSCones}, the intermediate  boundary is obtained by discretizing the set of all possible directions to a finite set $\{d_i\}$. We define $p_{a,b}^{d_i}$ to be an optimal solution of the problem of optimizing in direction $d_i$ over the set $p_{a,b} \in DSOS$. From this polynomial, we construct $U_1^{d_i}$ such that $p_{a,b}^{d_i}=z(x)^TU_1^{d_i~ T}U_1^{d_i}z(x)$. The intermediate boundary corresponds to solutions when optimizing in all directions on the set $DSOS(U_1^{d_i}).$ Once again, the inner approximations after one iteration improve substantially on the initial inner approximations.

\subsection{Extreme-ray interpretation of the change of basis} \label{subsec:Corners}
In this section, we present an alternative but equivalent way of expressing the LP and SOCP-based sequences. This characterization is based on the extreme-ray description of the cone of diagonally dominant/scaled diagonally dominant matrices. It will be particularly useful when we consider outer approximations of the psd cone in Section \ref{subsec:OuterApprox}.

\begin{lemma}[Barker and Carlson \cite{dd_extreme_rays}]\label{lem:dd.corners}
	A symmetric matrix $M$ is diagonally dominant if and only if it can be written as $$M=\sum_{i=1}^{n^2} \alpha_i v_iv_i^T, \alpha_i\geq 0,$$ where $\{v_i\}$ is the set of all {nonzero} vectors in $\mathbb{R}^n$ with at most $2$ nonzero components, each equal to $\pm 1$.
\end{lemma}
The vectors $v_i$ are the extreme rays of the $DD_n$ cone. This characterization of the set of diagonally dominant matrices leads to a convenient description of the dual cone:
\begin{equation}\label{eq:DD*}
DD^*_n=\{X \in { S_n}~|~ v_i^TXv_i \geq 0, i=1,\ldots,n^2\},
\end{equation}
which we will find to be useful in the next subsection. Using Lemma~\ref{lem:dd.corners}, we can rewrite {the sequence of LPs given in (\ref{eq:LPChol}) as}
%constructed in Section \ref{subsec:InnerApprox}} as 
\begin{equation}\label{eq:LPCorners}
\begin{aligned}
DSOS_k &\mathrel{\mathop{:}}=\min_{X,\alpha_i} C \cdot X\\
&\text{s.t. } A_i \cdot X=b_i,\  i=1,\ldots,m,\\
&X=\sum_{i=1}^{n^2} \alpha_i (U_k^Tv_i)(U_k^Tv_i)^T,\\
&{\alpha_i \geq 0, i=1,\ldots,n^2.}
\end{aligned}
\end{equation}
Let $X_k$ be an optimal solution to the LP in iteration $k$. The sequence of matrices $\{U_k\}$ is defined just as before:
\begin{align*}
U_0&=I\\
U_{k+1}&=\text{chol}(X_k).
\end{align*}

In the first iteration, a linear map is sending (or intuitively ``rotating'') the extreme rays $\{v_iv_i^T\}$ of the dd cone to a new set of extreme rays $\{(U_1^Tv_i)(U_1^Tv_i)^T\}$. This procedure keeps repeating itself without ever changing the number of extreme rays. 

%We see that in each iteration a linear map is sending the previous set of extreme rays $\{\}$ to new ex

%In each iteration, a linear map is created that maps the old extreme rays to new ones, without ever changing their number which stays equal to $n^2$. This is different from recent work done by Ahmadi et al.\cite{AAASDGH} where new elements were added to the existing extreme rays at each iteration.

As the sequence of LPs defined in (\ref{eq:LPCorners}) is equivalent to the sequence defined in (\ref{eq:LPChol}), the optimal value of (\ref{eq:LPCorners}) improves in each iteration. This can be seen directly: Indeed, $X_{k}$ is feasible for iteration $k+1$ of (\ref{eq:LPCorners}) by taking $\alpha_i=1$ when {$v_i$ has} exactly one nonzero entry equal to $1$ and $\alpha_i=0$ otherwise. This automatically implies that $DSOS_{k+1}\leq DSOS_k.$ Moreover, the improvement is strict under the assumptions of Theorem~\ref{thm:strict.imp}.

The set of scaled diagonally dominant matrices can be described in a similar fashion. In fact, from (\ref{eq:SDSOS.def}), we know that any scaled diagonally dominant matrix $M$ can be written as $${M=\sum_{i=1}^{ {n \choose 2} }V_i \Lambda_i V_i^T,}$$
where $V_i$ is an $ n \times 2$ matrix whose columns each contain exactly one nonzero element which is equal to $1,$ and $\Lambda_i$ is a $2 \times 2$ symmetric psd matrix.

This characterization of $SDD_n$ gives an immediate description of the dual cone $$SDD^*_n={\left\{X \in S^n ~|~ V_i^TXV_i \succeq 0, i=1,\ldots, {n \choose 2}\right\}},$$ which will become useful later. Our SOCP sequence in explicit form is then
\begin{equation}\label{eq:socp.corner.repres}
\begin{aligned}
SDSOS_k  &= \min_{X,\Lambda_i} C \cdot X\\
&\text{s.t. } A_i \cdot X=b_i, i=1,\ldots,m, \\
&X =\sum_{i=1}^{n \choose 2} (U_k^TV_i) \Lambda_i (U_k^TV_i)^T,\\
&\Lambda_i \succeq 0.
\end{aligned}
\end{equation}

If $X_k$ is an optimal solution at step $k$, the matrix sequence $\{U_k\}$ is defined as before:
\begin{align*}
U_0&=I\\
U_{k+1}&=\text{chol}(X_k).
\end{align*}

The interpretation of (\ref{eq:socp.corner.repres}) is similar to that of (\ref{eq:LPCorners}).
%, except that we no longer have a finite set of extreme rays.

\subsection{Outer approximations of the psd cone}\label{subsec:OuterApprox}

In Section \ref{subsec:InnerApprox}, we considered inner approximations of the psd cone to obtain upper bounds on (\ref{eq:genericSDP}). In many applications, semidefinite programming is used as a ``relaxation'' to provide outer approximations to some nonconvex sets. This approach is commonly used for relaxing quadratic programs; see, e.g., Section~\ref{sec:StableSet}, where we consider the problem of finding the largest stable set of a graph. In such scenarios, it does not make sense for us to inner approximate the psd cone: to have a valid relaxation, we need to outer approximate it. {This can be easily achieved by working with the dual problems, which we will derive explicitly in this section.}
% Luckily, duality will get this for us for free. In this section, we write down these dual problems explicitly.

%is of no interest to inner approximate the psd cone; in fact, we need to construct outer approximations of it. The bounds obtained with the outer approximations will then constitute valid bounds on the optimal value of the nonconvex problem we are interested in.

Since $P_n \subseteq DD^*_n$, the first iteration in our LP sequence for outer approximation will be
\begin{align*}
DSOSout_0 &\mathrel{\mathop{:}}= \min_{X} C \cdot X\\
&\text{s.t. } A_i \cdot X =b_i, i=1,\ldots,m,\\
&X \in DD_n^*.
\end{align*}
By the description of the dual cone in (\ref{eq:DD*}), we know this can be equivalently written as
\begin{equation} \label{eq:LP0Outer}
\begin{aligned}
DSOSout_0&= \min_{X} C \cdot X\\
&\text{s.t. } A_i \cdot X =b_i, \forall i\\
&v_i^TXv_i \geq 0, i=1,\ldots,n^2,
\end{aligned}
\end{equation}
where the $v_i$'s are the extreme rays of the set of diagonally dominant matrices as described in Section \ref{subsec:Corners}; namely, all vectors with at most two nonzero elements which are either $+1$ or $-1$. Recall that when we were after inner approximations (Subsection~\ref{subsec:InnerApprox}), the next LP in our sequence was generated by replacing the vectors $v_i$ by $U^Tv_i$, where the choice of $U$ was dictated by a Cholesky decomposition of an optimal solution of the previous iterate. %As seen in Theorem \ref{th:strictimp}, this method guarantees strict improvement of the objective value at each iteration. 
In the outer approximation setting, we seemingly do not have access to a psd matrix that would provide us with a Cholesky decomposition.  However, we can simply get this from the dual of (\ref{eq:LP0Outer})
\begin{align*}
DSOSout^d_0 &\mathrel{\mathop{:}}= \max_{y,\alpha} b^Ty\\
&\text{s.t. } C- \sum_{i=1}^m y_iA_i=\sum_{i=1}^{n^2} \alpha_iv_iv_i^T,\\
&\alpha_i\geq 0, i=1,\ldots,n^2,
\end{align*}
by taking $U_1=\text{chol}(C-\sum_i y_i^*A_i)$. We then replace $v_i$ by $U_1^Tv_i$ in (\ref{eq:LP0Outer}) to get the next iterate and proceed. In general, the sequence of LPs can be written as
\begin{align*}
DSOSout_k&= \min_{X} C \cdot X\\
&\text{s.t. } A_i \cdot X =b_i, i=1,\ldots,m,\\
&v_i^TU_kXU_k^Tv_i \geq 0,
\end{align*}
where $\{U_k\}$ is a sequence of matrices defined recursively as
\begin{align*}
U_0 &=I\\
U_k &=\text{chol}\left(C-\sum_i {y_i^{(k-1)}} A_i \right).
\end{align*}
The vector $y^{k-1}$ here is an optimal solution to the dual problem at step $k-1$:
\begin{align*}
DSOSout^d_{k-1} &\mathrel{\mathop{:}}= \max_{y,\alpha} b^Ty\\
&\text{s.t. } C- \sum_{i=1}^m y_iA_i=\sum_{i=1}^{n^2} \alpha_i(U_{k-1}^Tv_i)(U_{k-1}^Tv_i)^T,\\
&\alpha_i\geq 0, i=1,\ldots,n^2.
\end{align*}
This algorithm again strictly improves the objective value at each iteration. Indeed, from LP strong duality, we have $$DSOSout_k=DSOSout_k^d$$ and Theorem \ref{thm:strict.imp} applied to the dual problem states that $$DSOSout_{k-1}^d<DSOSout_{k}^d.$$

The sequence of SOCPs for outer approximation can be constructed in an analogous manner:
\begin{align*}
SDSOSout_k&= \min_{X} C \cdot X\\
&\text{s.t. } A_i \cdot X =b_i, i=1,\ldots,m,\\
&V_i^TU_kXU_k^TV_i \succeq 0, i=1,\ldots,{n \choose 2},
\end{align*}
where $V_i$'s are $n \times 2$ matrices containing exactly one 1 in each column, and $\{U_k\}$ is a sequence of matrices defined as
\begin{align*}
U_0 &=I\\
U_k &=\text{chol}\left(C-\sum_i {y_i^{(k-1)}} A_i \right)
\end{align*}
Here again, the vector {$y^{(k-1)}$} is an optimal solution to the dual SOCP at step $k-1$:
\begin{align*}
SDSOSout^d_{k-1} &\mathrel{\mathop{:}}= \max_{y,\Lambda_i} b^Ty\\
&\text{s.t. } C- \sum_{i=1}^m y_iA_i=\sum_{i=1}^{n \choose 2} (U_{k-1}^Tv_i)\Lambda_i(U_{k-1}^Tv_i)^T,\\
&\Lambda_i\succeq 0, i=1,\ldots,{n \choose 2},
\end{align*}
where each $\Lambda_i$ is a $2 \times 2$ unknown symmetric matrix.

% Similarly to the LP case, Theorem \ref{thm:strict.imp} guarantees that the dual bounds will strictly improve from one iteration to another.

\begin{remark}
	Let us end with some concluding remarks about our algorithm. There are other ways of improving the DSOS and SDSOS bounds. For example, Ahmadi and Majumdar~\cite{isos_journal, majumdar2014control} propose the requirement that $(\sum_{i=1}^n x_i^2)^rp(x)$ be dsos or sdsos as a sufficient condition for nonnegativity of $p$. As $r$ increases, the quality of approximation improves, although the problem size also increases very quickly. Such hierarchies are actually commonly used in the sum of squares optimization literature. But unlike our approach, they do not take into account a particular objective function and may improve the inner approximation to the PSD cone in directions that we do not care about.
	%
	%Unlike our approach, such a hierarchy, which is actually commonly used in the sum of squares optimization literature, does not take into account a particular objective function; it may improve the inner approximation to the PSD cone in directions that we do not care about.
	%
	Nevertheless, these hierarchies have interesting theoretical implications. Under some assumptions, one can prove that as $r\rightarrow\infty$, the underlying convex programs succeed in optimizing over the entire set of nonnegative polynomials; see, e.g.,~\cite{Reznick_Unif_denominator, jesus_polya, PhD:Parrilo, isos_journal}.

	Another approach to improve on the DSOS and SDSOS bounds appears in Chapter~\ref{chap:bp}. {We show there how ideas from column generation in large-scale integer and linear programming can be used to iteratively improve inner approximations to semidefinite cones. The LPs and SOCPs proposed in that work take the objective function into account and increase the problem size after each iteration by a moderate amount.} By contrast, the LPs and SOCPs coming from our Cholesky decompositions in this chapter have exactly the same size in each iteration. We should remark however that the LPs from iteration two and onwards are typically more dense than the initial LP (for DSOS) and slower to solve. A worthwhile future research direction {would be to systematically compare the performance of the two approaches and to explore customized solvers for the LPs and the SOCPs that arise in our algorithms.}

\end{remark}

\section{The maximum stable set problem} \label{sec:StableSet}
A classic problem in discrete optimization is that of finding the stability number of a graph. The graphs under our consideration in this section are all undirected and unweighted. A \emph{stable set} (or \emph{independent set}) of a graph $G=(V,E)$ is a set of nodes of $G$ no two of which are adjacent. The stability number of {$G$}, often denoted by $\alpha(G)$, is the size of its maximum stable set(s). The problem of determining $\alpha$ has many applications in {scheduling (see, e.g., \cite{golumbic2005algorithmic}) and coding theory \cite{Lovasz}}. As an example, the maximum number of final exams that can be scheduled on the same day at a university without requiring any student to take two exams is given by the stability number of a graph. This graph has courses IDs as nodes and an edge between two nodes if and only if there is at least one student registered in both courses. Unfortunately, the problem of testing whether $\alpha(G)$ is greater than a given integer $k$ is well known to be NP-complete~\cite{Karp}. Furthermore, the stability number cannot be approximated within a factor $|V|^{1-\epsilon}$ for any $\epsilon >0$ unless P$=$NP \cite{HaastadJohan}.

A straightforward integer programming formulation of $\alpha(G)$ is given by
\begin{align*}
\alpha(G) &=\max \sum_i x_i\\
&\text{s.t. } x_i+x_j \leq 1, \text{ if } {\{i,j\}} \in E\\
&x_i \in \{0,1\}.
\end{align*} 
The standard LP relaxation for this problem is obtained by changing the binary constraint $x_i \in \{0,1\}$ to the linear constraint $x_i \in [0,1]$:
\begin{equation}\label{eq:standard.LP}
\begin{aligned}
LP &\mathrel{\mathop{:}}=\max \sum_i x_i\\
&\text{s.t. } x_i+x_j \leq 1, \text{ if } {\{i,j\}} \in E\\
&x_i \in [0,1].
\end{aligned} 
\end{equation}
Solving this LP results in an upper bound on the stability number. The quality of this upper bound can be improved by adding the so-called \emph{clique inequalities}. The set of $k$-clique inequalities, denoted by $C_k$, is the set of constraints of the type $x_{i_1}+x_{i_2}+\ldots+x_{i_k}\leq 1$, if $(i_1,\ldots,i_k)$ form a clique (i.e., a complete subgraph) of $G$. Observe that these inequalities must be satisfied for binary solutions to the above LP, but possibly not for fractional ones. Let us define a family of LPs indexed by $k$:
\begin{equation}\label{eq:standard.LP.with.clique.inequ}
\begin{aligned}
LP^k &\mathrel{\mathop{:}}=\max \sum_i x_i\\
&x_i \in [0,1]\\
&C_1,\ldots,C_k \text{ are satisfied.}
\end{aligned} 
\end{equation}
Note that $LP=LP^2$ by construction and $\alpha(G)\leq LP^{k+1}\leq LP^k$ for all $k$. We will be comparing the bound obtained by some of these well-known LPs with those achieved via the new LPs that we propose further below.

A famous semidefinite programming based upper bound on the stability number is due to Lov\'{a}sz~\cite{Lovasz}:
\begin{align*}
\vartheta(G) &\mathrel{\mathop{:}}= \max_X J \cdot X\\
&\text{s.t. } I \cdot X=1\\
&X_{ij}=0, ~\forall {\{i,j\}} \in E\\
&X \succeq 0,
\end{align*}
where $J$ here is the all ones matrix and $I$ is the {identity matrix.} The optimal value $\vartheta(G)$ is called the Lov\'{a}sz theta number of the graph. We have the following inequalities $$\alpha(G)\leq \vartheta(G) \leq LP^k, \ \forall k.$$

The fact that $\alpha(G)\leq \vartheta(G)$ is easily seen by noting that if $S$ is a stable set of maximum size and $1_S$ is its indicator vector, then the rank-one matrix $\frac{1}{|S|} 1_S 1_S^T$ is feasible to the SDP and gives the objective value $|S|$. The other inequality states that this SDP-based bound is stronger than the aforementioned LP bound even with all the clique inequalities added (there are exponentially many). A proof can be found e.g. in \cite[Section 6.5.2]{LaurentVall}.

Our goal here is to obtain LP and SOCP based sequences of upper bounds on the Lov\'{a}sz theta number. To do this, we construct a series of outer approximations of the set of psd matrices as described in Section \ref{subsec:OuterApprox}. The first bound in the sequence of LPs is given by:
\begin{align*}
DSOS_0(G) &\mathrel{\mathop{:}}= \max_X J \cdot X\\
&\text{s.t. } I \cdot X=1\\
&X_{ij}=0, ~\forall { \{i,j\}} \in E\\
&X \in DD_n^*.
\end{align*}

%Notice that this problem is always feasible on the first iteration as one can take $\frac{1}{n}I$ as a solution, hence the algorithm is guaranteed to be feasible for all iterations.

%Through the extreme-ray description of the set of diagonally dominant matrices (see Section~\ref{subsec:OuterApprox}), we see that 

In view of (\ref{eq:DD*}), this LP can be equivalently written as

\begin{equation} \label{eq:DSOSLovasz}
\begin{aligned}
DSOS_0(G) &= \max_X J \cdot X\\
&\text{s.t. } I \cdot X=1\\
&X_{ij}=0, ~\forall { \{i,j\}} \in E\\
&v_i^TXv_i \geq 0, i=1,\ldots,n^2,
\end{aligned}
\end{equation}
where $v_i$ is a vector with at most two nonzero entries, each nonzero entry being either $+1$ or $-1$. {This LP is always feasible (e.g., with $X=\frac{1}{n}I$). Furthermore, it is bounded above.} Indeed, the last constraints in (\ref{eq:DSOSLovasz}) imply in particular that for all $i,j$, we must have $$X_{i,j}\leq \frac{1}{2}(X_{ii}+X_{jj}).$$
This, together with the constraint $I\cdot X=1$, implies that the objective $J\cdot X$ must remain bounded. As a result, the first LP in our iterative sequence will give a finite upper bound on $\alpha$. 

To progress to the next iteration, we will proceed as described in Section \ref{subsec:OuterApprox}. The new basis for solving the problem is obtained through the dual\footnote{The reader should not be confused to see both the primal and the dual as maximization problems. We can make the dual a minimization problem by changing the sign of $y$.} of (\ref{eq:DSOSLovasz}):
%In the inner approximation case, the sequence of LPs is generated as described in Section \ref{Section:bla}: at each iteration, we pick a new linear combination of the vectors $v_i$. This is done through a change of basis dictated by the Cholesky decomposition of the previous solution. As seen in Theorem \ref{Theorem:}, this method guarantees strict improvement of the objective value at each iteration. 
%In the outer approximation case, we seemingly do not have access to a dd matrix that would provide us with the needed Cholesky decomposition.  This can be obtained however through the dual of (\ref{eq:DSOSLovasz}):
\begin{equation}\label{eq:DSOSdualLovasz}
\begin{aligned}
DSOS_0^d(G) &\mathrel{\mathop{:}}=\max y\\
&\text{s.t. } yI+Y-J=\sum_{i=1}^{n^2} \alpha_i v_i v_i^T\\
& Y_{ij}=0 \text{ if } i=j \text{ or } { \{i,j\}} \notin E\\
&\alpha_i \geq 0, i=1,\ldots,n^2.
\end{aligned}
\end{equation}
The second constraint in this problem is equivalent to requiring that $yI+Y-J$ be dd. We  can define $$U_1=\text{chol}(y_0^*I+Y_0^*-J)$$ where $(y_1^*,Y_1^*)$ are optimal solutions to (\ref{eq:DSOSdualLovasz}). We then solve 
\begin{align*}
DSOS_1(G) &\mathrel{\mathop{:}}= \max_X J \cdot X\\
&\text{s.t. } I \cdot X=1\\
&X_{ij}=0, ~\forall { \{i,j\}} \in E\\
&v_i^TU_1XU_1^Tv_i \geq 0, i=1,\ldots,n^2,
\end{align*}
to obtain our next iterate. The idea remains exactly the same for a general iterate $k$: We construct the dual 
\begin{align*}
DSOS_{k}^d(G) &\mathrel{\mathop{:}}=\max y\\
&\text{s.t. } yI+Y-J=\sum_{i=1}^{n^2} \alpha_i U_{k}^T v_i (U_{k}^T v_i)^T\\
& Y_{ij}=0 \text{ if } i=j \text{ or } { \{i,j\}} \notin E\\
&\alpha_i \geq 0, \forall i,
\end{align*}
and define 
\begin{align*}
U_{k+1}\mathrel{\mathop{:}}=\text{chol}(y_k^*+Y_k^*-J), %\cdot U_{k-1},
\end{align*}
where $(y_k^*, Y_k^*)$ is an optimal solution to the dual. The updated primal is then
\begin{equation}\label{eq:DSOSLovaszk}
\begin{aligned}
DSOS_{k+1}(G) &\mathrel{\mathop{:}}= \max_X J \cdot X\\
&\text{s.t. } I \cdot X=1\\
&X_{ij}=0, ~\forall { \{i,j\}} \in E\\
&v_i^TU_{k+1}XU_{k+1}^Tv_i \geq 0, i=1,\ldots,n^2.
\end{aligned}
\end{equation} 
As stated in Section \ref{subsec:OuterApprox}, the optimal values of (\ref{eq:DSOSLovaszk}) are guaranteed to strictly improve as a function of $k$. Note that to get the bounds, we can just work with the dual problems throughout. 

An analoguous technique can be used to obtain a sequence of SOCPs. For the initial iterate, instead of requiring that $X \in DD^*$ in (\ref{eq:DSOSLovasz}), we require that $X \in SDD^*$. This problem must also be bounded and feasible as $$P_n\subseteq SDD^*\subseteq DD^*.$$ Then, for a given iterate $k$, the algorithm consists of solving
\begin{align*}
SDSOS_k(G) &\mathrel{\mathop{:}}=\max_X J \cdot X\\
&\text{s.t. } I \cdot X=1\\
&X_{ij}=0, \forall { \{i,j\}} \in E\\
&V_i^TU_kXU_k^TV_i \succeq 0, i=1,\ldots,{n \choose 2},
\end{align*}
where as explained in Section~\ref{subsec:Corners} each $V_i$ is an $ n \times 2$ matrix whose columns contain exactly one nonzero element which is equal to $1$. The matrix $U_k$ here is fixed and obtained by first constructing the dual SOCP
\begin{align*}
SDSOS_{k}^d(G) &\mathrel{\mathop{:}}=\max y\\
&\text{s.t. } yI+Y-J=\sum_{i=1}^{n\choose 2} U_{k}^T V_i \Lambda_i (U_{k}^T V_i)^T\\
& Y_{ij}=0 \text{ if } i=j \text{ or } { \{i,j\}} \notin E\\
&\Lambda_i \succeq 0, \forall i,
\end{align*}
(each $\Lambda_i$ is a symmetric $2\times 2$ matrix decision variable) and then taking $$U_{k}=\text{chol}(y_k^*I+Y_k^*-J).$$

Once again, one can just work with the dual problems to obtain the bounds.

%As stated in Section \ref{subsec:OuterApprox}, this algorithm is guaranteed to strictly improve as long as strong duality holds. One can check that this is the case as both the primal and the dual are strictly feasible. 

% One can check that this is the case as $X=\frac{1}{n}I$ is strictly feasible for the primal and the primal optimal value is 

%always lower bounded by the stability number **???****cite**.

%and $(Y,y)=(0_n,n+1)$ are strictly feasible for the dual. In other words, 
%$$SDSOS_k^d(G)=SDSOS_k(G).$$
%and we have shown strict improvement of the sequence $\{SDSOS_k\}_k$ in Section \ref{subsec:InnerApprox}.

As our first example, we apply both techniques to the problem of finding the stability number of the complement of the Petersen graph (see Figure \ref{subfig:PetersenGraph}). The exact stability number here is 2 and an example of a maximum stable set is illustrated by the two white nodes in Figure \ref{subfig:PetersenGraph}. The Lov\'asz theta number is 2.5 and has been represented by the continuous line in Figure \ref{subfig:IterDSOSSDSOS}. The dashed lines represent the optimal values of the LP and SOCP-based sequences of approximations for 7 iterations. Notice that already within one iteration, the optimal values are within one unit of the true stability number, which is good enough for knowing the exact bound (the stability number is an integer). From the fifth iteration onwards, they differ from the Lov\'asz theta number by only $10^{-2}$.
%By strong duality, this is equivalent to solving 
%\begin{align*}
%DSOS_2^d(G) &\mathrel{\mathop{:}}=\max y\\
%&\text{s.t. } yI+Y-J=\sum_{i} \alpha_i U_1^Tv_i (U_1^Tv_i)^T\\
%& Y_{ij}=0 \text{ if } i=j \text{ or } { \{i,j\}} \notin E.
%\end{align*}
%as $DSOS_2^d(G)=DSOS_2(G).$ 

\begin{figure}[h!]
	\begin{center}
		\mbox{
			\subfigure[Complement of Petersen graph]
			{\label{subfig:PetersenGraph}\scalebox{0.18}{\includegraphics{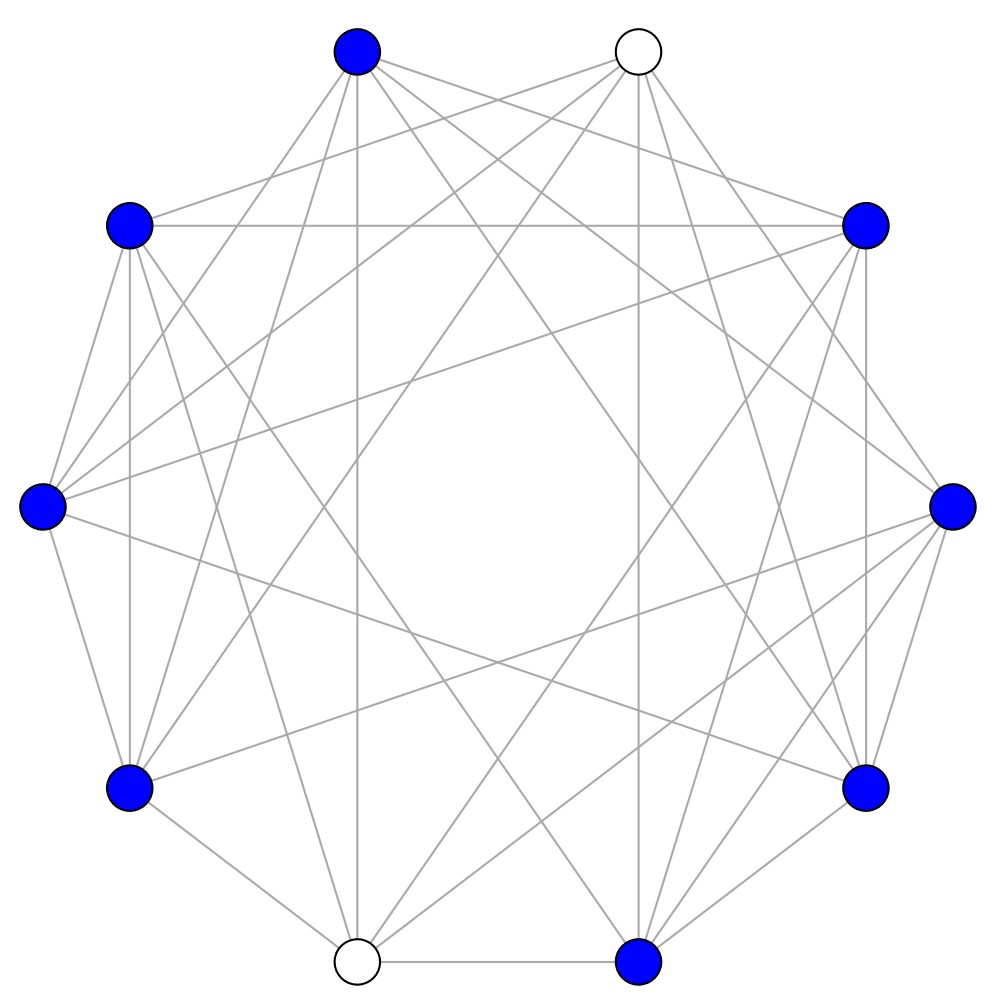}}}}
		\mbox{
			\subfigure[The Lov\'{a}sz theta number and iterative bounds bounds obtained by LP and SOCP]
			{\label{subfig:IterDSOSSDSOS}\scalebox{0.22}{\includegraphics{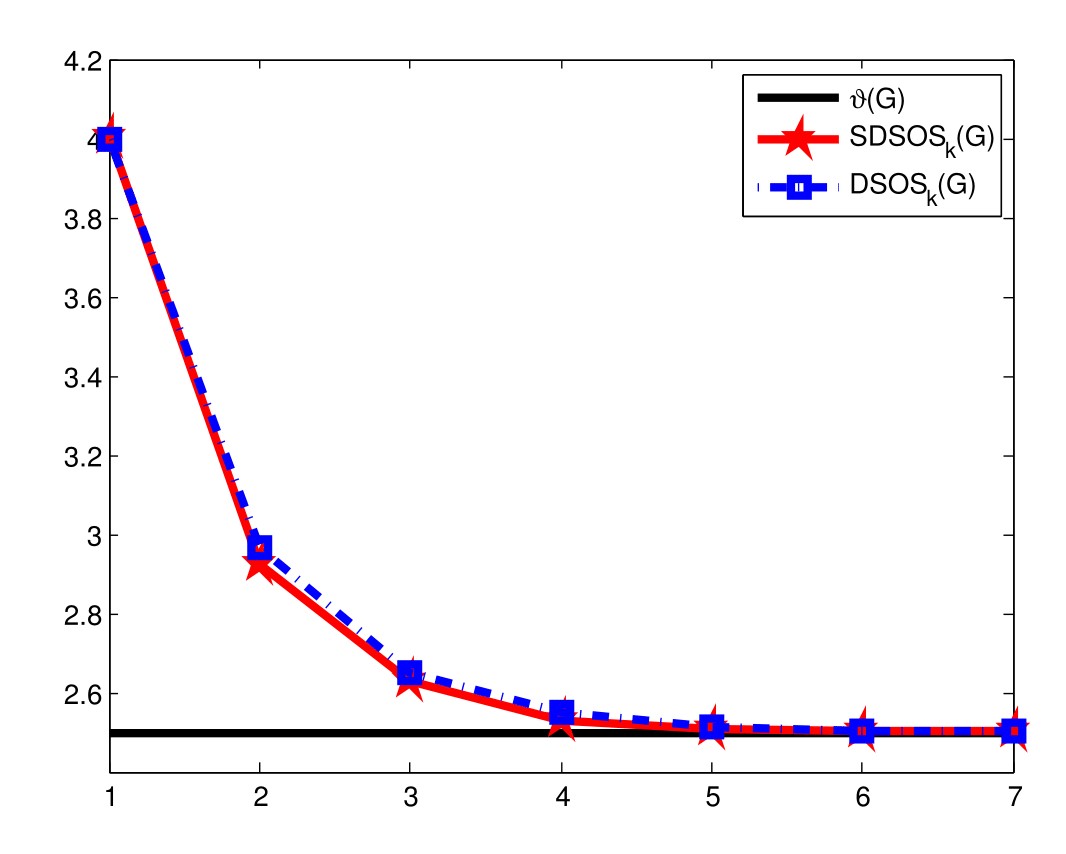}}}
		}
		
		\caption{Upper bounding the stability number of the complement of the Petersen graph}
		\label{fig:dsos.sdsos.Lovasz.Petersen}
	\end{center}
	%\vspace{-20pt}
\end{figure}

\vspace{5mm}

Finally, in Table \ref{table:stable_set_success}, we have generated 100 instances of 20-node Erd\"{o}s-R\'{e}nyi graphs with edge probability $0.5$. For each instance, we compute the bounds from the Lov\'{a}sz SDP, the standard LP in (\ref{eq:standard.LP}), the standard LP with all 3-clique inequalities added ($LP^3$ in (\ref{eq:standard.LP.with.clique.inequ})), and our LP/SOCP iterative sequences. We focus here on iterations 3,4 and 5 because there is no need to go further. We compare our bounds with the standard LP and the standard LP with 3-clique inequalities because they are LPs of roughly the same size. If any of these bounds are within one unit of the true stable set number, we count this as a success and increment the counter. As can be seen in Table \ref{table:stable_set_success}, the Lov\'{a}sz theta number is always within a unit of the stable set number, and so are our LP and SOCP sequences ($DSOS_k,SDSOS_k$) after four or at most five iterations. If we {look just} at the bound after 3 iterations, the success rate of SDSOS is noticeably higher than the success rate of DSOS. Also note that the standard LP with or without the three clique inequalities never succeeds in giving a bound within one unit of $\alpha(G)$.\footnote{All numerical experiments in this chapter have been parsed using either SPOT~\cite{SPOT_Megretski} or YAMIP~\cite{yalmip} and solved using the LP/SOCP/SDP solver of MOSEK~\cite{mosek}.}

%. Furthermore, the LP and SOCP iterative sequences do very well, obtaining bounds within the requisite one unit of the Lov\'{a}sz number ***???*** in 4 iterations on average. One can also note that the SOCP formulation seems to perform slightly better, and that the standard LP with or without the three clique inequalities never succeed in giving a bound within one unit of $\alpha$.

\begin{table}[h!]
	\small
	\begin{tabular}{|c|c|c|c|c|c|c|c|c|}
		\hline
		$\vartheta(G)$  &$LP$ & $LP^3$ & $DSOS_3$ &$DSOS_4$ & $DSOS_5$ & $SDSOS_3$ & $SDSOS_4$ & $SDSOS_5$\\
		\hline
		100\% & 0\% & 0\% & 14\% & 83\% & 100\% & 69\% & 100\% & 100\%\\
		\hline
	\end{tabular}
	\vspace{2mm}
	\caption{Percentage of instances out of 100 where the bound obtained is less than a unit away from the {stability number}}
	\label{table:stable_set_success}
\end{table}
\vspace{0.5mm}

%%%%%%%%%%%%%%%%%%%%%%%%%%%%%%%%%%%%%%%%%%%%%%%
\section{Partition}\label{sec:Partition}
%%%%%%%%%%%%%%%%%%%%%%%%%%%%%%%%%%%%%%%%%%%%%%%

The partition problem is arguably the simplest NP-complete problem to state: Given a list of positive integers $a_1,\ldots,a_n$, is it possible to split them into two sets with equal sums? We say that a partition instance is feasible if the answer is yes (e.g., \{5,2,1,6,3,8,5,4,1,1,10\}) and infeasible if the answer is no (e.g., \{47,20,13,15,36,7,46\}). The partition problem is NP-complete but only weakly. In fact, the problem admits a pseudopolynomial time algorithm based on dynamic programming that can deal with rather large problem sizes efficiently. This algorithm has polynomial running time on instances where the bit size of the integers $a_i$ are bounded by a polynomial in $\log n$~\cite{GareyJohnson_Book}. {In this section, we investigate the performance and mostly limitations of algebraic techniques for refuting feasibility of partition instances.}

Feasibility of a partition instance can always be certified by a short proof (the partition itself). However, unless P=co-NP, we do not expect to always have short certificates of infeasibility.
%
%If an instance is infeasible, however, unless P=co-NP, we do not expect to always have short certificates. 
Nevertheless, we can try to look for such a certificate through a sum of squares decomposition. Indeed, given an instance $a\mathrel{\mathop{:}}=\{a_1,\ldots,a_n\}$, {it is not hard to see\footnote{ This equivalence is apparent in view of the zeros of the polynomial on the right hand side of (\ref{eq:partition.no.eps}) corresponding to a feasible partition.} that} the following equivalence holds:

\begin{align} \label{eq:partition.no.eps}
{ \begin{bmatrix} a \text{ is an infeasible }\\ \text{ partition instance} \end{bmatrix}} \Leftrightarrow p_a(x)\mathrel{\mathop{:}}=\sum_i (x_i^2-1)^2+ (\sum_i a_ix_i)^2>0,~ \forall x\in\mathbb{R}^n.
\end{align}

So if for some $\epsilon >0$ we could prove that $p_a(x)-\epsilon$ is nonnegative, we would have refuted the feasibility of our partition instance.

%We can then use sum-of-squares polynomials to determine whether an instance of partition is feasible.

\begin{definition}\label{def:sos.refut}
	An instance of partition $a_1,\ldots,a_n$ is said to be \emph{sos-refutable} if there exists $\epsilon>0$ such that $p_a(x)-\epsilon$ is sos. 
\end{definition}
Obviously, any instance of partition that is sos-refutable is infeasible. This suggests that we can consider solving the following semidefinite program

%The converse however is false. 

%We will discuss this in more detail at the end of this section. Consider now the optimization problem 
\begin{equation} \label{eq:sospartition}
\begin{aligned}
SOS &\mathrel{\mathop{:}}=\max \ \epsilon \\
&\text{s.t. } q_a(x) \mathrel{\mathop{:}}=p_a(x)-\epsilon \text{ is sos}
\end{aligned}
\end{equation}
and examining its optimal value. Note that the optimal value of this problem is always greater than or equal to zero as $p_a$ is sos by construction. If the optimal value is positive, we have succeeded in proving infeasibility of the partition instance $a$.

%Let $\epsilon^*$ be an optimal solution to (\ref{eq:sospartition}). If $\epsilon^*$ is strictly positive then the instance $a$ is sos-refutable and conversely. 

We would like to define the notions of \emph{dsos-refutable} and \emph{sdsos-refutable} instances analogously by replacing the condition $q_a(x)$ sos by the condition $q_a(x)$ dsos or sdsos. Though (\ref{eq:sospartition}) is guaranteed to always be feasible by taking $\epsilon =0$, this is not necessarily the case for dsos/sdsos versions of (\ref{eq:sospartition}). For example, the optimization problem 
\begin{align}\label{eq:dsos.partition.nonhom}
\max_\epsilon\{\epsilon ~|~ p_a(x) -\epsilon  \text{ dsos}\}
\end{align}
on the instance $\{1,2,2,1,1\}$ is infeasible.\footnote{Under other structures on a polynomial, the same type of problem can arise for sos. For example, consider the Motzkin polynomial~\cite{MotzkinSOS} $M(x_1,x_2)=x_1^2x_2^4+x_2^2x_1^4-3x_1^2x_2^2+1$  which is nonnegative everywhere. The problem $\max_\epsilon\{\epsilon ~|~ M(x) -\epsilon  \text{     sos}\}$ is infeasible.} This is a problem for us as we need the first LP to be feasible to start our iterations. We show, however, that we can get around this issue by modeling the partition problem with homogeneous polynomials.

%  Therefore, defining dsos/sdsos-refutable similarly to sos-refutable would not work, particularly as our end goal is to generate a sequence of LPs (resp. SOCPs) using (\ref{eq:dsos.partition.nonhom}) (resp. (\ref{eq:dsos.partition.nonhom}) with sdsos) as the first iteration. Because of this, we consider a homogeneous version of the definition given above.

\begin{definition}
	Let $p_a$ be as in (\ref{eq:partition.no.eps}). An instance of partition $a_1,\ldots,a_n$ is said to be \emph{dsos-refutable} (resp. \emph{sdsos-refutable}) if there exists $\epsilon>0$ such that the quartic form
	\begin{align}\label{eq:homog}
	q_{a,\epsilon}^h(x)\mathrel{\mathop{:}}=p_a \left( \frac{x}{\left(\frac{1}{n}\sum_i x_i^2\right)^{1/2} }\right)\left( \frac{1}{n}\sum_i  x_i^2\right)^2-\epsilon \left( \frac{1}{n} \sum_i x_i^2 \right)^2
	\end{align}
	is dsos (resp. sdsos).
\end{definition}
Notice that $q_{a,\epsilon}^h$ is indeed a polynomial as it can be equivalently written as
$$ \sum_i x_i^4 +\left(\left (\sum_i a_ix_i\right)^2-2\sum_i x_i^2 \right) \cdot \left( \frac{1}{n}\sum_i  x_i^2\right) +(n-\epsilon) \cdot \left( \frac{1}{n}\sum_i  x_i^2\right)^2.$$ 

What we are doing here is homogenizing a polynomial that does not have odd monomials by multiplying its lower degree monomials with appropriate powers of $\sum_i  x_i^2$. The next theorem tells us how we can relate nonnegativity of this polynomial to feasibility of partition.

%For simplicity, we will take 
%\begin{align} \label{eq:def.homog.no.eps}
%p_a^h(x) \mathrel{\mathop{:}}=p_a \left( \frac{x}{\left(\frac{1}{n}\sum_i x_i^2\right)^{1/2} }\right)\left( \frac{1}{n}\sum_i  x_i^2\right)^2.
%\end{align}

%One can see that this homogenization procedure only works on polynomials that contain no monomials of odd degree, which is the case of $p_a$ and $q_a$.
\begin{theorem} \label{thm:valid.homog}
	A partition instance $a=\{a_1,\ldots,a_n\}$ is infeasible if and only if there exists $\epsilon>0$ for which the quartic form $q_{a,\epsilon}^h(x)$ defined in (\ref{eq:homog}) is nonnegative.
\end{theorem}
\begin{proof}
	For ease of reference, let us define 
	\begin{align} \label{eq:def.homog.no.eps}
	p_a^h(x) \mathrel{\mathop{:}}=p_a \left( \frac{x}{\left(\frac{1}{n}\sum_i x_i^2\right)^{1/2} }\right)\left( \frac{1}{n}\sum_i  x_i^2\right)^2.
	\end{align}
	
	Suppose partition is feasible, i.e, the integers $a_1,\ldots,a_n$ can be placed in two sets $\mathcal{S}_1$ and $\mathcal{S}_2$ with equal sums. Let $\bar{x}_i$=1 if $a_i$ is placed in set $\mathcal{S}_1$ and $\bar{x}_i=-1$ if $a_i$ is placed in set $\mathcal{S}_2$. Then $||\bar{x}||_2^2=n$ and $p_a(\bar{x})=0$. This implies that $$p_a^h(\bar{x})=p_a(\bar{x})=0,$$  and hence having $\epsilon>0$ would make $$q_{a,\epsilon}^h(\bar{x})=-\epsilon<0.$$

	%Let $\epsilon>0$  and suppose partition is feasible, i.e., there exists $I \in [1,\ldots,n]$ such that (\ref{eq:def.partition}) holds. Taking $x_i=1, \forall i \in I$ and $x_i=-1, \forall i \in I^C$, we have $||x||_2^2=n$ and $p_a(x)=0$. This entails that $$p_a^h(x)=p_a(x)=0,$$  and $$q_a^h(x)=-\epsilon<0,$$ which concludes the proof of the implication.
	
	If partition is infeasible, then $p_a(x)>0,~ \forall x\in\mathbb{R}^n$. In view of (\ref{eq:def.homog.no.eps}) we see that $p_a^h(x)>0$ on the sphere $\mathbb{S}$ of radius $n$. Since $p_a^h$ is continuous, its minimum $\hat{\epsilon}$ on the compact set $\mathbb{S}$ is achieved and must be positive. So we must have

	% Given (\ref{eq:def.homog.no.eps}), this implies that $p_a^h(x)>0, \forall x \neq 0$ and $p_a^h(x)=0$ for $x=0$. In particular, $p_a^h$ is positive on the sphere $S$ of radius $n$. As $p_a^h$ is continuous and $S$ is bounded and closed, $p_a^h$ achieves its minimum $\epsilon>0$ on $S$. Then
	
	$$q_{a,\hat{\epsilon}}^h(x)=p_a^h(x)-\hat{\epsilon} \left( \frac{1}{n} \sum_i x_i^2 \right)^2 \geq 0, \forall x \in \mathbb{S}.$$
	By homogeneity, this implies that $q_{a,\hat{\epsilon}}^h$ is nonnegative everywhere.
	%
	%As $q_a^h$ is homogeneous, this implies that $q_a^h(x) \geq 0, \forall x \neq 0$. Bearing in mind that $q_a^h(x)=0$ for $x=0$, this concludes the proof.
\end{proof}

Consider now the LP
\begin{equation} \label{eq:dsos.nh}
\begin{aligned}
&\max_{\epsilon}\  \epsilon\\
&\text{s.t. } q_{a,\epsilon}^h(x) \ \text{dsos.}
\end{aligned}
\end{equation}
\begin{theorem}\label{thm:feas.dsos}
	The LP in (\ref{eq:dsos.nh}) is always feasible.
\end{theorem}
\begin{proof}
	Let $h(x)\mathrel{\mathop:}=\left(\frac{1}{n}\sum_ix_i^2\right)^2$ and recall that $z(x,2)$ denotes the vector of all monomials of degree exactly $2$. We can write $$h(x)=z^T(x,2)Q_hz(x,2)$$ where $Q_h$ is in the strict interior of the $DD_n$ cone (i.e., its entries $q_{ij}$ satisfy $q_{ii}>\sum_j |q_{ij}|, \forall i$). Furthermore, let $Q$ be a symmetric matrix such that $p_a^h(x)=z(x,2)^TQz(x,2).$ Then
	$$q_{a,\epsilon}^h(x)=p_a^h(x)-\epsilon h(x)=z(x,2)^T(Q-\epsilon Q_h)z(x,2).$$
	As $Q_h$ is in the strict interior of $DD_n$, $\exists \lambda>0$ such that $$\lambda Q+(1-\lambda) Q_h \text{ is dd.}$$ Taking $\epsilon=-\frac{1-\lambda}{\lambda}$, $Q-\epsilon Q_h$ will be diagonally dominant and $q_{a,\epsilon}^h$ will be dsos. 
\end{proof}
As an immediate consequence, the SOCP
\begin{equation} \label{eq:sdsos.nh}
\begin{aligned}
&\max_{\epsilon} \ \epsilon\\
&\text{s.t. } q_{a,\epsilon}^h(x) \text{ sdsos}
\end{aligned}
\end{equation}
is also always feasible. We can now define our sequence of LPs and SOCPs as we have guaranteed feasibility of the first iteration. This is done following the strategy and notation of Section \ref{subsec:polyopt}:
\begin{equation}\label{eq:LP.sequence}
\begin{aligned}
DSOS_k \text{ (resp. $SDSOS_k$)} &\mathrel{\mathop{:}}= \max_{\epsilon} \ \epsilon\\
&\text{s.t. } q_{a,\epsilon}^h(x) \in DSOS(U_k) \text{ (resp. $SDSOS(U_k)$)},
\end{aligned}
\end{equation}
where $\{U_k\}$ is a sequence of matrices recursively defined with $U_0=I$ and $U_{k+1}$ defined as the Cholesky factor of an optimal dd (resp. sdd) Gram matrix of the optimization problem in iteration $k$.

%
%  using an optimal solution $q_k$ of (\ref{eq:LP.sequence}):
%\begin{align*}
%U_0&=I\\
%q_k &=\mathrel{\mathop{:}}z(x)^TU_{k+1}^TU_{k+1}z(x).
%\end{align*}

We illustrate the performance of these LP and SOCP-based bounds on the infeasible partition instance $\{1,2,2,1,1\}$. The results are in Figure \ref{fig:instance.12211}. We can use the sum of squares relaxation to refute the feasibility of this instance by either solving (\ref{eq:sospartition}) (the ``non-homogenized version'') or solving (\ref{eq:dsos.nh}) with dsos replaced with sos (the ``homogenized version''). Both approaches succeed in refuting this partition instance, though the homogenized version gives a slightly better (more positive) optimal value. As a consequence, we only plot the homogeneous bound, denoted by $SOS_h$, in Figure~\ref{fig:instance.12211}. Notice that the LP and SOCP-based sequences refute the instance from the $6^{th}$ iteration onwards.

\begin{figure}[h!]
	\begin{center}
		\mbox{
			\subfigure[Bounds $SOS_h$, $DSOS_k$ and $SDSOS_k$]
			{\label{subfig:Partition}\scalebox{0.45}{\includegraphics{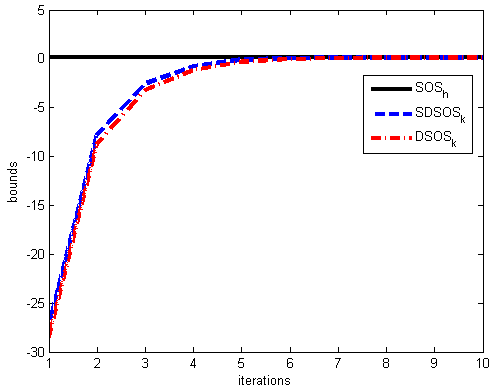}}}}
		\mbox{
			\subfigure[Zoomed-in version of Figure \ref{subfig:Partition}]
			{\label{subfig:PartitionZoom}\scalebox{0.45}{\includegraphics{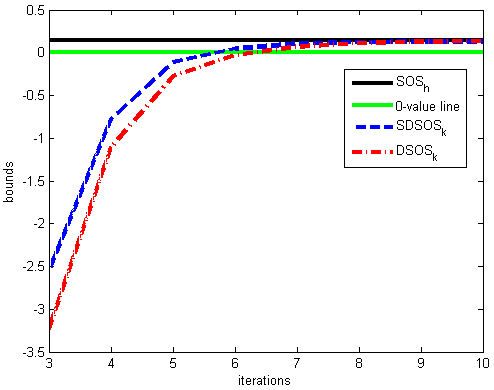}}}
		}
		
		\caption{Bounds obtained on the \{1,2,2,1,1\} instance of the partition problem using SDP, as well as the LP/SOCP-based sequences}
		\label{fig:instance.12211}
	\end{center}
	%\vspace{-20pt}
\end{figure}

As our final experiment, we generate 50 infeasible instances of partition with 6 elements randomly generated between 1 and 15. These instances are \emph{trivially infeasible} because we made sure that $a_1+\cdots+a_6$ is an odd number. %The results are presented in Table \ref{tab:partition.success}.
%For Table \ref{tab:partition.success}, we generated 50 infeasible instances of partition of size $6$. 
In the first column of Table~\ref{tab:partition.success}, we count the number of successes for sos-refutability (non homogeneous version as defined in Definition \ref{def:sos.refut}), where a failure is defined as the optimal value of (\ref{eq:sospartition}) being 0 up to numerical precision. The second column corresponds to the number of successes for sos-refutability (homogeneous version). The last 4 columns show the success rate of the LP and SOCP-based sequences as defined in (\ref{eq:LP.sequence}), after 20 iterations and 40 iterations.
\begin{table}[H]
	\centering
	\begin{tabular}{|c|c|c|c|c|c|}
		\hline
		$SOS$ & $SOS_{h}$ & $DSOS_{20}$ & $DSOS_{40}$ & $SDSOS_{20}$ & $SDSOS_{40}$\\
		\hline
		56\% & 56\% & 12\% & 16 \% & 14\% & 14\%\\
		\hline
	\end{tabular}
	\caption{Rate of success for refutability of infeasible instances of partition}
	\label{tab:partition.success}
\end{table}
From the experiments, the homogeneous and non-homogeneous versions of (\ref{eq:sospartition}) have the same performance in terms of their ability to refute feasibility. However, we observe that they both fail to refute a large number of completely trivial instances! We prove why this is the case for one representative instance in the next section. The LP and SOCP-based sequences also perform poorly and their convergence is much slower than what we observed for the maximum stable set problem in Section~\ref{sec:StableSet}.

\subsection{Failure of the sum of squares relaxation on trivial partition instances.} \label{subsec:sos.refutability}

For complexity reasons, one would expect there to be infeasible instances of partition that are not sos-refutable. What is surprising however is that the sos relaxation is failing on many instances that are totally trivial to refute as the sum of their input integers is odd. We present a proof of this phenomenon on an instance which is arguably the simplest one.\footnote{If we were to instead consider the instance [1,1,1], sos would succeed in refuting it.}

%As mentioned before, it is not true that all infeasible instances of partition are sos-refutable. This could be expected for complexity reasons. Indeed, deciding whether partition is infeasible is an NP-complete problem, whereas we can solve any SDP to arbitrary precision. If all infeasible instances were sos-refutable, i.e. $\exists \epsilon>0$ such that $p(x)-\epsilon$ sos, we could obtain a solution to the SDP that is as close as needed to $\epsilon$ to become positive. This would then be a certificate of infeasibility of the partition instance obtained in polynomial time, thereby contradicting the NP-completeness of partition. What was surprising to us however was the triviality of some of the instances on which the algebraic techniques failed. An example of such an instance is given below.

\begin{proposition} \label{th:infeas.partition.inst}
	The infeasible partition instance $\{1,1,1,1,1\}$ is not sos-refutable.
\end{proposition}
\begin{proof}
	Let $p_a$ be the polynomial defined in (\ref{eq:partition.no.eps}). To simplify notation, we let $p(x)$ represent $p_a(x)$ for $a=\{1,1,1,1,1\}$. We will show that $p$ is on the boundary of the SOS cone even though we know it is strictly inside the PSD cone. This is done by presenting a dual functional $\mu$ that vanishes on $p,$ takes a nonnegative value on all quartic sos polynomials, and a negative value on $p(x)-\epsilon$ for any $\epsilon>0.$ (See Figure~\ref{fig:proof.partition} for an intuitive illustration of this.)
	
	\begin{figure}[H]
		\centering
		\includegraphics[scale=0.5]{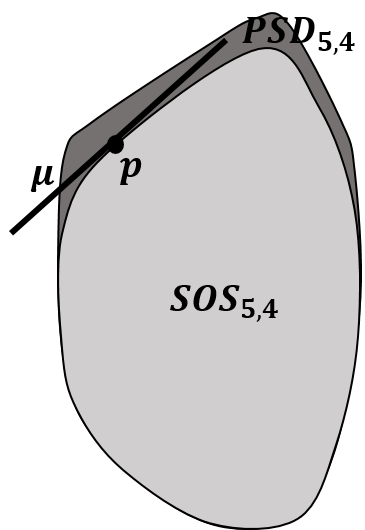}
		\caption{The geometric idea behind the proof of Proposition \ref{th:infeas.partition.inst}}
		\label{fig:proof.partition}
	\end{figure}

	%We will prove that $p(x)-\epsilon$ is not sos by presenting a dual functional $\mu$ that separates it from the set of sos polynomials.  

	The polynomial $p$ when expanded out reads
	\begin{align} \label{eq:p.expanded}
	p(x)=n-\sum_i x_i^2+2\sum_{i<j} x_i x_j +\sum_i x_i^4.
	\end{align}
	Consider the vector of coefficients of $p$ with the ordering as written in (\ref{eq:p.expanded}):
	\setcounter{MaxMatrixCols}{21}
	\scalefont{0.6}
	\begin{align} \label{eq:pvec}
	\overrightarrow{p}=\begin{pmatrix} 5&-1&-1&-1&-1&-1&2&2&2&2&2&2&2&2&2&2&-1&-1&-1&-1&-1
	\end{pmatrix}.
	\end{align}
	\normalsize
	This is a reduced representation of the vector of coefficients of $p$, in that there are many zeros associated with all other monomials of degree less than or equal to 4, which we are not writing out.
	
	%all monomials of degree less than or equal to 4, but only those for which the coefficient of $p$ is nonzero.
	
	Our goal is to find a vector $\mu$ that satisfies
	\begin{align}
	\langle \mu, \overrightarrow{p} \rangle &=0 \nonumber \\
	\langle \mu,\overrightarrow{q} \rangle &\geq 0, \text{ for all $q$ sos of degree 4}. \label{eq:inner.prod.sos.zero}
	\end{align}
	If such a $\mu$ exists and its first element is nonzero (which by rescaling can then be taken to be 1), then $\langle \mu, \overrightarrow{p-\epsilon} \rangle=\langle \mu,\overrightarrow{p} \rangle -\langle \mu, \overrightarrow{\epsilon} \rangle=-\epsilon<0$. This provides us with the required functional that separates $p(x)-\epsilon$ from the set of sos polynomials.
	
	Selecting the same reduced basis as the one used in (\ref{eq:pvec}), we take
	\begin{align*}
	\overrightarrow{\mu_{reduced}}=\begin{pmatrix} 1 & \textbf{1}_5^T & -\frac{1}{4} \cdot \textbf{1}_1^T & \textbf{1}_5^T 
	\end{pmatrix}
	\end{align*}
	where $\textbf{1}_n$ is the all ones vector of size $n$. The subscript ``reduced'' denotes the fact that in $\overrightarrow{\mu_{reduced}}$, only the elements of $\mu$ needed to verify $\langle \mu, \overrightarrow{p}\rangle=0$ are presented. Unlike $\overrightarrow{p}$, the entries of $\mu$ corresponding to the other monomials are not all zero. This can be seen from the entries of the matrix $M$ that appears further down.
	
	We now show how (\ref{eq:inner.prod.sos.zero}) holds. Consider any sos polynomial $q$ of degree less than or equal to 4. We know that it can be written as $$q(x)=z^TQz=\mbox{Tr}~ Q \cdot zz^T,$$ for some $Q \succeq 0$, and a vector of monomials 
	$$z^T=[1,x_1,x_2,\ldots,x_n,x_1^2,\ldots,x_n^2,x_1x_2, \ldots, x_{n-1}x_n].$$
	It is not difficult to see that $$ \langle \mu, \overrightarrow{q} \rangle = \mbox{Tr}~ Q \cdot (zz^T)|_\mu$$
	where by $(zz)^T|_\mu$, we mean a matrix where each monomial in $zz^T$ is replaced with the corresponding element of the vector $\mu$. This yields the matrix
	\scalefont{1}
	\[ M=\begin{pmatrix}
	1&    0&            0 &       0        &     0      &       0         &    1          &   1    &         1   &           1    &         1    &       b &          b &          b &          b &          b &          b &          b &          b &          b &          b  \\    
	0         &    1  &         b &          b &          b &          b &            0     &        0        &     0      &       0     &        0   &          0 &             0       &      0   &          0    &         0        &     0       &      0     &        0       &      0     &        0      \\
	0    &       b &            1 &          b &          b &          b &            0 &            0    &         0      &       0         &    0   &          0            & 0 &            0        &     0          &   0       &      0     &        0   &           0   &           0        &     0      \\
	0   &        b &          b &            1      &     b &          b &            0       &      0          &   0      &       0       &      0       &      0&             0            & 0            & 0           &  0          &   0           &  0          &   0          &   0         &    0      \\
	0  &         b &          b &          b &            1   &        b &            0   &          0   &          0   &          0     &        0  &           0 &            0         &    0          &   0      &       0      &       0        &     0         &    0            & 0         &    0      \\
	0    &       b &          b &          b &          b &            1  &           0     &        0   &          0       &      0  &           0       &      0          &   0        &     0           &  0       &      0          &   0  &           0     &        0        &     0           &  0      \\
	1   &          0   &          0     &        0        &     0      &       0    &         1         &    1       &      1  &           1       &      1  &         b &          b &          b &          b &          b &          b &          b &          b &          b &          b \\     
	1      &       0          &   0        &     0     &        0        &     0     &        1     &        1         &    1       &      1       &      1        &   b &          b &          b &          b &          b &          b &          b &          b &          b &          b    \\  
	1      &       0        &     0    &         0     &        0    &         0       &      1         &    1      &       1           &  1            & 1           &b &          b &          b &          b &          b &          b &          b &          b &          b &          b  \\
	1       &      0       &      0     &        0      &       0      &       0      &       1        &     1        &     1  &           1        &     1    &       b &          b &          b &          b &          b &          b &          b &          b &          b &          b    \\
	1    &         0       &      0      &       0     &        0      &       0    &         1       &      1     &        1     &        1        &     1       &    b &          b &          b &          b &          b &          b &          b &          b &          b &          b    \\
	b &            0    &         0   &          0     &        0         &    0      &     b &          b &          b &          b &          b &            1        &   b &          b &          b &          b &          b &          b &           a &           a &           a\\     
	b &            0    &         0  &           0      &       0      &       0      &     b &          b &          b &          b &          b &          b &            1     &      b &          b &          b &           a &           a &          b &          b &           a   \\  
	b &            0  &           0        &     0        &     0          &   0       &    b &          b &          b &          b &          b &          b &          b &            1    &       b &           a &          b &           a &          b &           a &          b      \\
	b &            0          &   0            & 0      &       0     &        0        &   b &          b &          b &          b &          b &          b &          b &          b &            1          &  a &           a &          b &           a &          b &          b     \\
	b &            0   &          0      &       0       &      0       &      0       &    b &          b &          b &          b &          b &          b &          b &           a &           a &            1        &   b &          b &          b &          b &           a      \\
	b &            0     &        0       &      0      &       0            & 0 &          b &          b &          b &          b &          b &          b &           a &          b &           a &          b &            1    &       b &          b &           a &          b     \\
	b &            0       &      0    &         0   &          0   &          0        &   b &          b &          b &          b &          b &          b &           a &           a &          b &          b &          b &            1     &       a &          b &          b   \\
	b &            0       &      0        &     0      &       0          &   0   &        b &          b &          b &          b &          b &           a &          b &          b &           a &          b &          b &           a &            1       &    b &          b    \\
	b &            0         &    0      &       0      &       0      &       0    &       b &          b &          b &          b &          b &           a &          b &           a &          b &          b &           a &          b &          b &            1        &   b      \\
	b &            0     &        0      &       0     &        0      &       0  &         b &          b &          b &          b &          b &           a &           a &          b &          b &           a &          b &          b &          b &          b &            1      \\
	\end{pmatrix},
	\]
	\normalsize
	where $a=\frac{3}{8}$ and $b=-\frac{1}{4}$. We can check that $M \succeq 0$. This, together with the fact that $Q\succeq 0,$ implies that (\ref{eq:inner.prod.sos.zero}) holds.\footnote{It can be shown in a similar fashion that $\{1,1,1,1,1\}$ is not sos-refutable in the homogeneous formulation of (\ref{eq:homog}) either.}
\end{proof}

{
	\subsection{Open problems}
	
	We showed in the previous subsection that the infeasible partition instance $\{1,1,1,1,1\}$ was not sos-refutable. Many more randomly-generated partition instances that we knew to be infeasible (their sum being odd) also failed to be sos-refutable. This observation motivates the following open problem:
	
	%\paragraph{Open Problem 1} It can easily be seen that $[1,1,1,1,1]$ is infeasible as the sum of its elements is odd. Though this criterion for infeasibility is an easy one to check, such trivial instances are not necessarily sos-refutable as seen above. Is it possible to characterize the odd instances that are sos-refutable? 
	%\paragraph{Open Problem 2} An instance of partition $a_1, \ldots, a_n$ is said to be $r$-sos-refutable for $r \in \mathbb{N}$ if $\exists \epsilon>0$ such that $(p(x)-\epsilon)(\sum_i x_i^2+1)^r$ is sos. Though the instance $\{1,1,1,1,1\}$ is not sos-refutable, it is in fact $1$-sos-refutable. Furthermore, the instance $[1,1,1,1,1,1,1]$ (vector of all ones of length 7) is not sos-refutable, nor $1$-sos-refutable but it is $2$-sos-refutable. If we consider the instance of $n$ ones with $n$ odd, and we define $\tilde{r}$ to be the minimum $r$ such that $[1,1,\ldots,1]$ is $r$-sos-refutable. Is it true that $\tilde{r}$ grows with $n$?

	\paragraph{Open Problem 1} Characterize the set of partition instances $\{a_1,\ldots,a_n\}$ that have an odd sum but are not sos-refutable (see Definition~\ref{def:sos.refut}).
	
	Our second open problem has to do with the power of higher order sos relaxations for refuting feasibility of partition instances.
	\paragraph{Open Problem 2} For a positive integer $r$, let us call a partition instance $\{a_1, \ldots, a_n\}$ \emph{$r$-sos-refutable} if $\exists \epsilon>0$ such that $(p(x)-\epsilon)(\sum_i x_i^2+1)^r$ is sos. Note that this is also a certificate of infeasibility of the instance. Even though the $\{1,1,1,1,1\}$ instance is not sos-refutable, it is $1$-sos-refutable. Furthermore, we have numerically observed that the instance $\{1,1,1,1,1,1,1\}$ (vector of all ones of length 7) is not sos-refutable or $1$-sos-refutable, but it is $2$-sos-refutable. If we consider the instance consisting of $n$ ones with $n$ odd, and define $\tilde{r}$ to be the minimum $r$ such that $\{1,1,\ldots,1\}$ becomes $r$-sos-refutable, is it true that $\tilde{r}$ must grow with $n$?
	
}
%\section{Acknowledgments} We are grateful to Anirudha Majumdar and Sanjeeb Dash for insightful discussions. %Ani and for resolving some of the issues with our code. 

%\bibliographystyle{amsplain} 
%\bibliography{pablo_amirali}

\chapter{On the Construction of Converging Hierarchies for Polynomial Optimization Based on Certificates of Global Positivity}\label{ch:positiv}

\section{Introduction}

A polynomial optimization problem (POP) is an optimization problem of the form
\begin{equation}\label{eq:POP}
\begin{aligned}
&\inf_{x \in \mathbb{R}^n} &&p(x) \\
&\text{s.t. } &&g_i(x)\geq 0, ~i=1,\ldots,m,
\end{aligned}
\end{equation}
where $p, g_i,~i=1,\ldots,m,$ are polynomial functions in $n$ variables $x\mathrel{\mathop{:}}=(x_1,\ldots,x_n)$ and with real coefficients. It is well-known that polynomial optimization is a hard problem to solve in general. For example, simply testing whether the optimal value of problem (\ref{eq:POP}) is smaller than or equal to some rational number $k$ is NP-hard already when the objective is quadratic and the constraints are linear \cite{pardalos1991quadratic}. Nevertheless, these problems remain topical due to their numerous applications throughout engineering, operations research, and applied mathematics (see, e.g., \cite{lasserre2009moments,blekherman2012semidefinite,ahmadiOR_letters}). In this chapter, we are interested in obtaining lower bounds on the optimal value of problem (\ref{eq:POP}). We focus on a class of methods which construct hierarchies of tractable convex optimization problems whose optimal values are lowerbounds on the optimal value of (\ref{eq:POP}), with convergence to it as the sequence progresses.  This implies that even though the original POP is nonconvex, one can obtain increasingly accurate lower bounds on its optimal value by solving convex optimization problems. One method for constructing these hierarchies of optimization problems that has gained attention in recent years relies on the use of \emph{Positivstellens\"atze} (see, e.g., \cite{laurent2009sums} for a survey). Positivstellens\"atze are algebraic identities that certify infeasibility of a set of polynomial inequalities, or equivalently\footnote{Note that the set $\{x \in \mathbb{R}^n ~|~ g_1(x) \geq 0, \ldots, g_m(x)\geq 0\}$ is empty if and only if $-g_1(x)>0$ on the set $\{x \in \mathbb{R}^n ~|~ g_2(x) \geq 0, \ldots, g_m(x) \geq 0\}$.}, positivity of a polynomial on a basic semialgebraic set. (Recall that a basic semialgebraic set is a set defined by finitely many polynomial inequalities.) These Positivstellens\"atze can be used to prove lowerbounds on POPs. Indeed, if we denote the feasible set of (\ref{eq:POP}) by $S$, the optimal value of problem (\ref{eq:POP}) is equivalent to 
\begin{equation} \label{eq:gamma.opt}
\begin{aligned}
&\sup_{\gamma} &&\gamma\\
&\text{s.t. } &&p(x)-\gamma \geq 0,~\forall x \in S.
\end{aligned}
\end{equation}
Hence if $\gamma$ is a strict lower bound on (\ref{eq:POP}), we have that $p(x)-\gamma>0$ on $S$, a fact that can be certified using Positivstellens\"atze.  At a conceptual level, hierarchies that provide lower bounds on (\ref{eq:POP}) are constructed thus: we fix the ``size of the certificate'' at each level of the hierarchy and search for the largest $\gamma$ such that the Positivstellens\"atze at hand can certify positivity of $p(x)-\gamma$ over $S$ with a certificate of this size. As the sequence progresses, we increase the size of the certificates allowed, hence obtaining increasingly accurate lower bounds on (\ref{eq:POP}). 

Below, we present three of the better-known Positivstellens\"atze, given respectively by Stengle~\cite{stengle1974nullstellensatz}, Schm\"udgen~\cite{schmudgen1991thek}, and Putinar~\cite{putinar1993positive}. These all rely on {sum of squares} certificates. We recall that a polynomial is a \emph{sum of squares} (sos) if it can be written as a sum of squares of other polynomials. We start with Stengle's Positivstellensatz, which certifies infeasibility of a set of polynomial inequalities. It is sometimes referred to as ``the Positivstellensatz'' in related literature as it requires no assumptions, contrarily to Schm\"udgen and Putinar's theorems which can be viewed as refinements of Stengle's result under additional assumptions.

\begin{theorem}[Stengle's Positivstellensatz~\cite{stengle1974nullstellensatz}]\label{th:stengle}
	The basic semialgebraic set $$S=\{x\in \mathbb{R}^n ~|~ g_1(x)\geq 0,\ldots, g_m(x)\geq 0\}$$ is empty if and only if there exist sum of squares polynomials $s_0(x)$,$s_1(x)$,$\ldots$, $s_m(x)$, $s_{12}(x)$, $s_{13}(x)$,$\ldots$, $s_{123\ldots m}(x)$ such that
	$$-1=s_0(x)+\sum_{i} s_{i}(x)g_{i}(x) +\sum_{\{i,j\}} s_{ij}(x)g_{i}(x)g_{j}(x)+\ldots+s_{123\ldots m}(x)g_{1}(x)\ldots g_{m}(x).$$
\end{theorem}
The next two theorems, due to Schm\"udgen and Putinar, certify positivity of a polynomial $p$ over a basic semialgebraic set $S$. They impose additional compactness assumptions comparatively to Stengle's Positivstellensatz.
\begin{theorem}[Schm\"udgen's Positivstellensatz \cite{schmudgen1991thek}]\label{th:schmudgen}
	Assume that the set $$S=\{x \in \mathbb{R}^n ~|~ g_1(x)\geq 0, \ldots, g_m(x)\geq 0\}$$ is compact. If a polynomial $p$ is positive on $S$, then $$p(x)=s_0(x)+\sum_{i} s_{i}(x)g_{i}(x) +\sum_{\{i,j\}} s_{ij}(x)g_{i}(x)g_{j}(x)+\ldots+s_{123\ldots m}(x)g_{1}(x)\ldots g_{m}(x),$$ where $s_0(x)$,$s_1(x)$,$\ldots$, $s_m(x)$, $s_{12}(x)$, $s_{13}(x)$,$\ldots$, $s_{123\ldots m}(x)$ are sums of squares.
\end{theorem}

\begin{theorem}[Putinar's Positivstellensatz~\cite{putinar1993positive}]\label{th:putinar}
	Let $$S=\{x \in \mathbb{R}^n ~|~ g_1(x)\geq 0, \ldots, g_m(x)\geq 0\}$$ and assume that $\{g_1,\ldots,g_m\}$ satisfy the Archimedean property, i.e., there exists $N \in~\mathbb{N}$ such that $$N-\sum_i x_i^2=\sigma_0(x)+\sigma_1(x) g_1(x)+\ldots+\sigma_m(x) g_m(x),$$ where $\sigma_1(x),\ldots,\sigma_m(x)$ are sums of squares. If a polynomial $p$ is positive on $S$, then $$p(x)=s_0(x)+s_1(x)g_1(x)+\ldots+s_m(x)g_m(x),$$ where $s_1(x),\ldots,s_m(x)$ are sums of squares.
\end{theorem}
Note that these three Positivstellens\"atze involve in their expressions sum of squares polynomials of unspecified degree. To construct hierarchies of tractable optimization problems for (\ref{eq:gamma.opt}), we fix this degree: at level $r$, we search for the largest $\gamma$ such that positivity of $p(x)-\gamma$ over $S$ can be certified using the Positivstellens\"atze where the degrees of all sos polynomials are taken to be less than or equal to $2r$. Solving each level of these hierarchies is then a semidefinite program (SDP). This is a consequence of the fact that one can optimize over (or test membership to) the set of sum of squares polynomials of fixed degree using semidefinite programming \cite{sdprelax,PhD:Parrilo,lasserre_moment}. Indeed, a polynomial $p$ of degree $2d$ and in $n$ variables is a sum of squares if and only if there exists a symmetric matrix $Q\succeq 0$ such that $p(x)=z(x)^TQz(x)$, where $z(x)=(1,x_1,\ldots,x_n,\ldots,x_n^d)^T$ is the standard vector of monomials in $n$ variables and of degree less than or equal to $d$. We remark that the hierarchy obtained from Stengle's Positivstellensatz was proposed and analyzed by Parrilo in \cite{sdprelax}; the hierarchy obtained from Putinar's Positivstellensatz was proposed and analyzed by Lasserre in~\cite{lasserre_moment}. There have been more recent works that provide constructive proofs of Schm\"udgen and Putinar's Positivstellens\"atze; see \cite{averkov, schweighofer2002algorithmic, schweighofer2005optimization}. These proofs rely on other Positivstellens\"atze, e.g., a result by Poly\'a (see Theorem \ref{th:polya} below) in \cite{schweighofer2002algorithmic, schweighofer2005optimization}, and the same result by Poly\'a, Farkas' lemma, and Stengle's Positivstellensatz in \cite{averkov}. There has further been an effort to derive complexity bounds for Schm\"udgen and Putinar's Positivstellens\"atze in recent years; see \cite{nie2007complexity,schweighofer2004complexity}.

%We briefly discuss the assumptions required for both Positivstellens\"atze to hold. Theorem \ref{th:schmudgen} requires only compactness whereas Theorem \ref{th:putinar} requires a stronger assumption, namely an algebraic proof of compactness (the so-called Archimedean property). The assumption that we require in the hierarchies that we derive also relates to compactness: we need to know the value of a constant $R$ such that $S \subseteq B(0,R)$, where $B(0,R)$ is the ball of radius $R$ and centered at $0.$
On a historical note, Stengle, Schm\"udgen, and Putinar's Positivstellens\"atze were derived in the latter half of the 20th century. As mentioned previously, they all certify positivity of a polynomial over an arbitrary basic semialgebraic set (modulo compactness assumptions). By contrast, there are Positivstellens\"atze from the early 20th century that certify positivity of a polynomial \emph{globally.} Perhaps the most well-known Positivstellensatz of this type is due to Artin in 1927, in response to Hilbert's 17th problem. Artin shows that any nonnegative polynomial is a sum of squares of rational functions. Here is an equivalent formulation of this statement:

\begin{theorem}[Artin \cite{Artin_Hilbert17}]\label{th:artin}
	For any nonnegative polynomial $p$, there exists an sos polynomial $q$ such that $p\cdot q$ is a sum of squares.
\end{theorem}
To the best of our knowledge, in this area, all converging hierarchies of lower bounds for POPs are based off of Positivstellens\"atze that certify nonnegativity of a polynomial over an arbitrary basic semialgebraic set. In this chapter, we show that in fact, under compactness assumptions, it suffices to have only global certificates of nonnegativity (such as the one given by Artin) to produce a converging hierarchy for general POPs. As a matter of fact, even weaker statements that apply only to globally positive (as opposed to globally nonnegative) forms are enough to derive converging hierarchies for POPs. Examples of such statements are due to Habicht \cite{habicht1939zerlegung} and Reznick \cite{Reznick_Unif_denominator}. With such an additional positivity assumption, more can usually be said about the structure of the polynomial $q$ in Artin's result. Below, we present the result by Reznick.

\begin{theorem}[Reznick \cite{Reznick_Unif_denominator}]\label{th:reznick.uniform}
	For any positive definite form $p$, there exists $r \in \mathbb{N}$ such that $p(x) \cdot (\sum_i x_i^2)^r$ is a sum of squares.	
\end{theorem}
We show in this chapter that this Positivstellensatz also gives rise to a converging hierarchy for POPs with a compact feasible set similarly to the one generated by Artin's Positivstellensatz.

Through their connections to sums of squares, the two hierarchies obtained using the theorems of Reznick and Artin are semidefinite programming-based. In this chapter, we also derive an ``optimization-free'' converging hierarchy for POPs with compact feasible sets where each level of the hierarchy only requires that we be able to test nonnegativity of the coefficients of a given fixed polynomial. To the best of our knowledge, this is the first converging hierarchy of lower bounds for POPs which does not require that convex optimization problems be solved at each of its levels. To construct this hierarchy, we use a result of Poly\'a \cite{Polya}, which just like Artin's and Reznick's Positivstellens\"atze, certifies global positivity of forms. However this result is restricted to even forms. Recall that a form $p$ is \emph{even} if each of the variables featuring in its individual monomials has an even power. This is equivalent (see \cite[Lemma 2]{de2005equivalence}) to $p$ being invariant under change of sign of each of its coordinates, i.e., \begin{equation*}
p(x_1,\ldots,x_n)=p(-x_1,\ldots,x_n)=\cdots=p(x_1,\ldots,-x_n).
\end{equation*}
\begin{theorem}[Poly\'a \cite{Polya}]\label{th:polya}
	For any positive definite even form $p$, there exists $r \in \mathbb{N}$ such that $p(x) \cdot (\sum_i x_i^2)^r$ has nonnegative coefficients.\footnote{A perhaps better-known but equivalent formulation of this theorem is the following: for any form $h$ that is positive on the standard simplex, there exists $r \in \mathbb{N}$ such that $h(x)\cdot (\sum_i x_i)^r$ has nonnegative coefficients. The two formulations are equivalent by simply letting $p(x)=h(x^2)$.} 
\end{theorem}
Our aforementioned hierarchy enables us to obtain faster-converging linear programming (LP) and second-order cone programming (SOCP)-based hierarchies for general POPs with compact feasible sets that rely on the concepts of \emph{dsos} and \emph{sdsos} polynomials. These are recently introduced inner approximations to the set of sos polynomials that have shown much better scalability properties in practice \cite{isos_journal}.

As a final remark, we wish to stress the point that the goal of this chapter is first and foremost theoretical, i.e., to provide methods for constructing converging hierarchies of lower bounds for POPs using as sole building blocks certificates of global positivity. We do not make any claims that these hierarchies can outperform the popular existing hierarchies due, e.g., to Lasserre \cite{lasserre_moment} and Parrilo \cite{sdprelax}. We do believe however that the optimization-free hierarchy presented in Section \ref{subsec:opt.free} could potentially be of interest in large-scale applications where the convex optimization problems appearing in traditional hierarchies are too cumbersome to solve.

\subsection{Outline of the chapter} The chapter is structured as follows. In Section \ref{sec:nonneg.approx}, we show that if one can inner approximate the cone of positive definite forms arbitrarily well (with certain basic properties), then one can produce a converging hierarchy of lower bounds for POPs with compact feasible sets (Theorem~\ref{th:hierarchy}). This relies on a reduction (Theorem~\ref{th:slb}) that reduces the problem of certifying a strict lower bound on a POP to that of proving positivity of a certain form. In Section \ref{sec:sdp.hierarchy}, we see how this result can be used to derive semidefinite programming-based converging hierarchies (Theorems \ref{th:reznick.hierarchy} and \ref{th:artin.hierarchy}) from the Positivstellens\"atze by Artin (Theorem \ref{th:artin}) and Reznick (Theorem \ref{th:reznick.uniform}). In Section~\ref{sec:opt.free.LP.SOCP}, we derive an optimization-free hierarchy (Theorem \ref{th:sms.hierarchy}) from the Positivstellensatz of Poly\'a (Theorem \ref{th:polya}) as well as LP and SOCP-based hierarchies which rely on dsos/sdsos polynomials (Corollary \ref{cor:dsos.sdsos.hierarchy}). We conclude with a few open problems in Section \ref{sec:conclusion}.

\subsection{Notation and basic definitions}

We use the standard notation $A\succeq 0$ to denote that a symmetric matrix $A$ is positive semidefinite. Recall that a \emph{form} is a homogeneous polynomial, i.e., a polynomial whose monomials all have the same degree. We denote the degree of a form $f$ by $deg(f)$. We say that a form $f$ is \emph{nonnegative} (or positive semidefinite) if $f(x) \geq 0$, for all $x \in \mathbb{R}^n$ (we write $f\geq 0$). A form $f$ is \emph{positive definite} (pd) if $f(x) >0,$ for all nonzero $x$ in $\mathbb{R}^n$ (we write $f>0$). Throughout the chapter, we denote the set of forms (resp. the set of nonnegative forms) in $n$ variables and of degree $d$ by $H_{n,d}$ (resp $P_{n,d}$). We denote the ball of radius $R$ and centered at the origin by $B(0,R)$ and the unit sphere in $x$-space, i.e., $\{x\in \mathbb{R}^n~|~||x||_2=1\}$, by $S_x$. We use the shorthand $f(y^2-z^2)$ for $y,z \in \mathbb{R}^n$ to denote $f(y_1^2-z_1^2,\ldots,y_n^2-z_n^2).$ We say that a scalar $\gamma$ is a strict lower bound on (\ref{eq:POP}) if $p(x)>\gamma,~\forall x\in S$. Finally, we ask the reader to carefully read Remark~\ref{rem:notation} which contains the details of a notational overwriting occurring before Theorem \ref{th:hierarchy} and valid from then on throughout the chapter. This overwriting makes the chapter much simpler to parse.

\section{Constructing converging hierarchies for POP using global certificates of positivity}\label{sec:nonneg.approx}

Consider the polynomial optimization problem in (\ref{eq:POP}) and denote its optimal value by $p^*$. Let $d$ be such that $2d$ is the smallest even integer larger than or equal to the maximum degree of $p, g_i, i=1,\ldots,m$. We denote the feasible set of our optimization problem by $$S=\{x \in \mathbb{R}^n~|~g_i(x)\geq 0, i=1,\ldots,m\}$$ and assume that $S$ is contained within a ball of radius $R$. From this, it is easy to provide (possibly very loose) upper bounds on $g_i(x)$ over the set $S$: as $S$ is contained in a ball of radius $R$, we have $|x_i|\leq R$, for all $i=1,\ldots,n$. We then use this to upper bound each monomial in $g_i$ and consequently $g_i$ itself. We use the notation $\eta_i$ to denote these upper bounds, i.e., $g_i(x) \leq \eta_i$, for all $i=1,\ldots,m$ and for all $x\in S$. Similarly, we can provide an upperbound on $-p(x)$. We denote such a bound by $\beta$, i.e., $-p(x) \leq \beta,$ $\forall x \in S.$

The goal of this section is to produce a method for constructing converging hierarchies of lower bounds for POPs if we have access to arbitrarily accurate inner approximations of the set of positive definite forms. The first theorem (Theorem \ref{th:slb}) connects lower bounds on (\ref{eq:POP}) to positive definiteness of a related form. The second theorem (Theorem~\ref{th:hierarchy}) shows how this can be used to derive a hierarchy for POPs.

\begin{theorem}\label{th:slb} Consider the general polynomial optimization problem in (\ref{eq:POP}) and recall that $d$ is such that $2d$ is the smallest even integer larger than or equal to the maximum degree of $p, g_i, i=1,\ldots,m$. Suppose $S \subseteq B(0,R)$ for some positive scalar $R$. Let $\eta_i, i=1,\ldots,m$ (resp.  $\beta$) be any finite upper bounds on $g_i(x), i=1,\ldots,m$ (resp. $-p(x)$).
	
	Then, a scalar $\gamma$ is a strict lower bound on (\ref{eq:POP}) if and only if the homogeneous sum of squares polynomial
	\begin{equation}\label{eq:f.gamma}
	\begin{aligned}
	f_\gamma(x,s,y)\mathrel{\mathop{:}}=&\left( \gamma y^{2d}-y^{2d}p(x/y)-s_0^2y^{2d-2} \right)^2+\sum_{i=1}^m\left(y^{2d}g_i(x/y)-s_i^2 y^{2d-2}\right)^2 \\
	&+\left((R+\sum_{i=1}^m \eta_i +\beta+\gamma)^dy^{2d}-(\sum_{i=1}^n x_i^2+\sum_{i=0}^m s_i^2)^d-s_{m+1}^{2d}\right)^2
	\end{aligned}
	\end{equation} of degree $4d$ and in $n+m+3$ variables $(x_1,\ldots,x_n,s_0,\ldots,s_m,s_{m+1},y)$ is positive definite.
\end{theorem}

\begin{proof}
	It is easy to see that $\gamma$ is a strict lower bound on (\ref{eq:POP}) if and only if the set $$T\mathrel{\mathop{:}}=\{x \in \mathbb{R}^n ~|~ \gamma- p(x) \geq 0; \quad g_i(x)\geq 0, i=1,\ldots,m;\quad \sum_i x_i^2 \leq R\}$$ is empty. Indeed, if $T$ is nonempty, then there exists a point $\check{x} \in S$ such that $p(\check{x})\leq \gamma$. This implies that $\gamma$ cannot be a strict lower bound on (\ref{eq:POP}). Conversely, if $T$ is empty, the intersection of $S$ with $\{x~|~\gamma-p(x)\geq 0\}$ is empty, which implies that $\forall x \in S$, $p(x)>\gamma$.
	
	We now define the set:
	\begin{equation}\label{eq:def.Ts}
	\begin{aligned}
	T_s=\{(x,s) \in \mathbb{R}^{n+m+2} ~|~ \gamma-p(x)=s_0^2; \quad g_i(x)=s_i^2, i=1,\ldots,m; \phantom\}\\
	\quad \phantom\{(R+\sum_{i=1}^m \eta_i+\beta+\gamma)^d -(\sum_{i=1}^n x_i^2+\sum_{i=0}^m s_i^2)^d-s_{m+1}^{2d}=0\}. 
	\end{aligned}
	\end{equation}
	Note that $T_s$ is empty if and only if $T$ is empty. Indeed, if $T_s$ is nonempty, then there exists $\hat{x} \in \mathbb{R}^n$ and $\hat{s} \in \mathbb{R}^{m+2}$ such that the three sets of equations are satisfied. This obviously implies that $\gamma -p(\hat{x})\geq 0$ and that $g_i(\hat{x})\geq 0$, for all $i=1,\ldots,m.$ It further implies that $\sum_i \hat{x}_i^2 \leq R$ as by assumption, if $\hat{x} \in S$, then $\hat{x}$ is in a ball of radius $R$. 
	Conversely, suppose now that $T$ is nonempty. There exists $\tilde{x}$ such that $\gamma-p(\tilde{x})\geq 0$, $g_i(\tilde{x})\geq 0$ for $i=1,\ldots,m$, and $\sum_i \tilde{x_i}^2 \leq R.$ Hence, there exist $\tilde{s_0},\ldots,\tilde{s_m}$ such that $$\gamma-p(\tilde{x})=\tilde{s_0}^2 \text{ and } g_i(\tilde{x})=\tilde{s_i}^2,~ i=1,\ldots,m.$$ Combining the fact that $\sum_i \tilde{x_i}^2 \leq R$ and the fact that $\eta_i$, $i=1,\ldots,m$ (resp. $\gamma+\beta$) are upperbounds on $g_i$ (resp. $\gamma-p(\tilde{x})$), we obtain: 
	$$R+\sum_{i=1}^m \eta_i+\beta+\gamma \geq \sum_{i=1}^n \tilde{x_i}^2+\sum_{i=0}^m \tilde{s_i}^2.$$  By raising both sides of the inequality to the power $d$, we show the existence of $\tilde{s}_{m+1}$.
	
	We now show that $T_s$ is empty if and only if $f_{\gamma}(x,s,y)$ is positive definite. Suppose that $T_s$ is nonempty, i.e., there exists $(\breve{x},\breve{s}) \in \mathbb{R}^{n+m+2}$ such that the equalities given in (\ref{eq:def.Ts}) hold. Note then that $f_{\gamma}(\breve{x},\breve{s},1)=0$. As $(\breve{x},\breve{s},1)$ is nonzero, this implies that $f_{\gamma}(\breve{x},\breve{s},\breve{y})$ is not positive definite. 
	
	For the converse, assume that $f_{\gamma}(x,s,y)$ is not positive definite. As $f_{\gamma}(x,s,y)$ is a sum of squares and hence nonnegative, this means that there exists nonzero $(\bar{x},\bar{s},\bar{y})$ such that $f(\bar{x},\bar{s},\bar{y})=0$. We proceed in two cases. If $\bar{y} \neq 0$, it is easy to see that $(\bar{x}/\bar{y},\bar{s}/\bar{y}) \in T_s$ and $T_s$ is nonempty. Consider now the case where $\bar{y}=0$. The third square in $f_{\gamma}$ being equal to zero gives us: $$-(\sum_i \bar{x}_i^2+\sum_{i=0}^m \bar{s}_i^2)^d=\bar{s}_{m+1}^{2d}.$$
	This implies that $\bar{s}_{m+1}=0$ and that $\bar{x}_1=\ldots=\bar{x}_m=\bar{s_0}=\ldots=\bar{s}_m=0$ which contradicts the fact that $(\bar{x},\bar{s},\bar{y})$ is nonzero.
\end{proof}

\begin{remark}
	Note that Theorem \ref{th:slb} implies that testing feasibility of a set of polynomial inequalities is no harder than checking whether a homogeneous polynomial that is sos has a zero. Indeed, as mentioned before, the basic semialgebraic set $$\{x~|~g_1(x)\geq 0,\ldots, g_m(x)\geq 0\}$$ is empty if and only if $\gamma=0$ is a strict lower bound on the POP
	\begin{equation*}
	\begin{aligned}
	&\inf_x &&-g_1(x)\\
	&\text{s.t. } &&g_2(x) \geq 0,\ldots,g_{m}(x)\geq 0.
	\end{aligned}
	\end{equation*}
	In principle, this reduction can open up new possibilities for algorithms for testing feasibility of a basic semialgebraic set. For example, the work in \cite{AAA_Cubic_vec_field} shows that positive definiteness of a form $f$ is equivalent to global asymptotic stability of the polynomial vector field $\dot{x}=-\nabla f(x).$ One could as a consequence search for Lyapunov functions, as is done in \cite[Example 2.1.]{AAA_Cubic_vec_field}, to certify positivity of forms. Conversely, simulating trajectories of the above vector field can be used to minimize $f$ and potentially find its nontrivial zeros, which, by our reduction, can be turned into a point that belongs to the basic semialgebraic set at hand. 
	
	We further remark that one can always take the degree of the sos form $f_{\gamma}$ in (\ref{eq:f.gamma}) whose positivity is under consideration to be equal to four. This can be done by changing the general POP in (\ref{eq:POP}) to only have quadratic constraints and a quadratic objective via an iterative introduction of new variables and new constraints in the following fashion: $x_{ij}=x_ix_j$.
	
\end{remark}

\vspace{10pt}

\begin{remark}[Notational remark]\label{rem:notation}
	As a consequence of Theorem~\ref{th:slb}, we now know that certifying lower bounds on (\ref{eq:POP}) is equivalent to proving positivity of the form $f_{\gamma}$ that appears in (\ref{eq:f.gamma}). To simplify notation, we take this form to have $n$ variables and be of degree $2d$ from now on (except for our Positivstellens\"atze in Corollaries \ref{cor:SDP.psatz} and \ref{cor:Positiv.LP} which stand on their own). To connect back to problem (\ref{eq:POP}) and the original notation, the reader should replace every occurrence of $n$ and $d$ in the future as follows: $$n \leftarrow n+m+3, \quad d \leftarrow 2d.$$ Recall that $n$ was previously the dimension of the decision variable of problem (\ref{eq:POP}), $d$ was such that $2d$ is the smallest even integer larger than or equal to the maximum degree of $g_i$ and $p$ in (\ref{eq:POP}), and $m$ was the number of constraints of problem (\ref{eq:POP}).
\end{remark}
\vspace{10pt}

Our next theorem shows that, modulo some technical assumptions, if one can inner approximate the set of positive definite forms arbitrarily well (conditions (a) and (b)), then one can construct a converging hierarchy for POPs.

\begin{theorem} \label{th:hierarchy}
	Let $K_{n,2d}^r$ be a sequence of sets (indexed by $r$) of homogeneous polynomials in $n$ variables and of degree $2d$ with the following properties:
	\begin{enumerate}[(a)]
		\item $K_{n,2d}^r \subseteq P_{n,2d}, \forall r,$ and there exists a pd form $s_{n,2d} \in K_{n,2d}^0.$
		\item If $p>0$, then $\exists r \in \mathbb{N}$ such that $p \in K_{n,2d}^r.$
		\item $K_{n,2d}^r \subseteq K_{n,2d}^{r+1}$, $\forall r$.
		\item If $p \in K_{n,2d}^r$, then $\forall \epsilon \in [0,1]$, $p+\epsilon s_{n,d} \in K_{n,2d}^{r}.$ 
		
	\end{enumerate}
	
	Recall the definition of $f_{\gamma}(z)$ given in (\ref{eq:f.gamma}). Consider the hierarchy of optimization problems indexed by $r$:
	\begin{equation}\label{eq:hierarchy}
	\begin{aligned}
	l_r\mathrel{\mathop{:}}=&\sup_{\gamma} &&\gamma\\
	&\text{s.t. } &&f_{\gamma}(z)-\frac{1}{r}s_{n,2d}(z) \in K_{n,2d}^r.
	\end{aligned}
	\end{equation}
	Then, $l_r \leq p^*$ for all $r$, $\{l_r\}$ is nondecreasing, and $\lim_{r \rightarrow \infty} l_r=p^*.$
	
\end{theorem}

\begin{proof} We first show that the sequence $\{l_r\}$ is upperbounded by $p^*$. Suppose that a scalar $\gamma$ satisfies $$f_{\gamma}(z) -\frac{1}{r}s_{n,2d}(z) \in K_{n,2d}^r.$$ We then have $f_{\gamma}(z)-\frac{1}{r}s_{n,2d}(z) \in P_{n,2d}$ using (a). This implies that $f_{\gamma}(z) \geq \frac{1}{r}s_{n,2d}(z) $, and hence $f_{\gamma}$ is pd as $s_{n,2d}$ is pd. From Theorem \ref{th:slb}, it follows that $\gamma$ has to be a strict lower bound on (\ref{eq:POP}). As any such $\gamma$ satisfies $\gamma < p^*$, we have that $l_r \leq p^*$ for all $r$.\\
	
	We now show monotonicity of the sequence $\{l_r\}$. Let $\gamma$ be such that $$f_{\gamma}(z)-\frac{1}{r} s_{n,2d}(z) \in K_{n,2d}^r.$$ We have the following identity:$$f_{\gamma}(z)-\frac{1}{r+1} s_{n,2d}(z)=f_{\gamma}(z)-\frac{1}{r}s_{n,2d}(z)+\frac{1}{r(r+1)}s_{n,2d}(z).$$
	Now, using the assumption and properties (c) and (d), we conclude that 
	$$f_{\gamma}(z)-\frac{1}{r+1}s_{n,2d}(z) \in K_{n,2d}^{r+1}.$$ This implies that $\{\gamma ~|~ f_{\gamma}(z)-\frac{1}{r}s_{n,2d}(z) \in K_{n,2d}^r\} \subseteq \{\gamma ~|~ f_{\gamma}(z)-\frac{1}{r+1}s_{n,2d}(z) \in K_{n,2d}^{r+1}\}$ and that $l_r \leq l_{r+1}.$\\

	Note that as the sequence $\{l_r\}$ is upperbounded and nondecreasing, it converges. Let us show that the limit of this sequence is $p^*$. To do this, we show that for any strict lower bound $\gamma$ on (\ref{eq:POP}), there exists a positive integer $r$ such that $f_{\gamma}(z)-\frac{1}{r} s_{n,2d}(z) \in K_{n,2d}^r$. By Theorem~\ref{th:slb}, as $\gamma$ is a strict lower bound, $f_{\gamma}(z)$ is positive definite. Hence, by continuity, there exists a positive integer $r'$ such that $f_{\gamma}(z)-\frac{1}{r'}s_{n,2d}(z)$ is positive definite. Using (b), this implies that there exists a positive integer $r''$ such that 
	\begin{equation}\label{eq:f.gam.in.cone}
	f_{\gamma}(z)-\frac{1}{r'}s_{n,2d}(z) \in K_{n,2d}^{r''}.
	\end{equation}
	
	We now proceed in two cases. If $r'' \leq r'$, we take $r=r'$ and use property (c) to conclude. If $r' \leq r''$, we have
	$$f_{\gamma}(z)-\frac{1}{r''}s_{n,2d}(z)=f_{\gamma}(z)-\frac{1}{r'}s_{n,2d}(z)+\frac{r''-r'}{r'\cdot r''}s_{n,2d}(z).$$
	We take $r=r''$ and use (\ref{eq:f.gam.in.cone}) and properties (c) and (d) to conclude.
\end{proof}

\begin{remark} Note that condition (d) is subsumed by the more natural condition that $K_{n,d}^r$ be a convex cone for any $n,d,$ and $r$. However, there are interesting and relevant cones which we cannot prove to be convex though they trivially satisfy condition (d) (see Theorem \ref{th:reznick.hierarchy} for an example).
\end{remark}

\section{Semidefinite programming-based hierarchies obtained from Artin's and Reznick's Positivstellens\"atze}\label{sec:sdp.hierarchy}

In this section, we construct two different semidefinite programming-based hierarchies for POPs using Positivstellens\"atze derived by Artin (Theorem \ref{th:artin}) and Reznick (Theorem \ref{th:reznick.uniform}). To do this, we introduce two sets of cones that we call the Artin and Reznick cones.

\begin{definition}\label{def:reznick.artin.cones}
	We define the \emph{Reznick cone} of level $r$ to be $$R_{n,2d}^{r}\mathrel{\mathop{:}}=\{p \in H_{n,2d}~|~ p(x)\cdot \left(\sum_{i=1}^n x_i^2\right)^r \text{ is sos}\}.$$
	Similarly, we define the \emph{Artin cone} of level $r$ to be $$A_{n,2d}^{r} \mathrel{\mathop{:}}=\{p \in H_{n,2d}~|~ p(x) \cdot q(x) \text{ is sos for some sos form $q$ of degree $2r$} \}.$$
	
\end{definition}

We show that both of these cones produce hierarchies of the type discussed in Theorem \ref{th:hierarchy}. Recall that $p^*$ is the optimal value of problem (\ref{eq:POP}) and that $f_{\gamma}$ is defined as in (\ref{eq:f.gamma}) with the change of notation discussed in Remark \ref{rem:notation}.

\begin{theorem}\label{th:reznick.hierarchy}
	Consider the hierarchy of optimization problems indexed by $r$:
	\begin{equation}\label{eq:reznick.hierarchy}
	\begin{aligned}
	l_r\mathrel{\mathop{:}}=&\sup_{\gamma} &&\gamma\\
	&\text{s.t. } &&f_{\gamma}(z)-\frac{1}{r}(\sum_{i=1}^n z_i^2)^{d} \in R_{n,2d}^r.
	\end{aligned}
	\end{equation}
	Then, $l_r \leq p^*$ for all $r$, $\{l_r\}$ is nondecreasing, and $\lim_{r \rightarrow \infty} l_r=p^*.$
\end{theorem}

\begin{proof}
	It suffices to show that the Reznick cones $R_{n,2d}^r$ satisfy properties (a)-(d) in Theorem \ref{th:hierarchy}. The result will then follow from that theorem. For property (a), it is clear that, as $(\sum_i x_i^2)^r>0$ and $p(x) \cdot (\sum_i x_i^2)^r$ is a sum of squares and hence nonnegative, $p(x)$ must be nonnegative, so $R_{n,2d}^r \subseteq P_{n,2d}.$ Furthermore, the form $s_{n,2d}\mathrel{\mathop{:}}=(\sum_i x_i^2)^{d}$ belongs to $R_{n,2d}^0$ and is positive definite. Property (b) is verified as a consequence of Theorem \ref{th:reznick.uniform}. For (c), note that if $p(x) \cdot (\sum_i x_i^2)^r$ is sos, then $p(x) \cdot (\sum_i x_i^2)^{r+1}$ is sos since the product of two sos polynomials is sos. Finally, for property (d), note that $R_{n,2d}^r$ is a convex cone. Indeed, for any $\lambda \in [0,1]$, $$(\lambda p(x) +(1-\lambda) q(x)) \cdot (\sum_i x_i^2)^r=\lambda p(x) (\sum_i x_i^2)^r+(1-\lambda) q(x) (\sum_i x_i^2)^r $$
	is sos if $p$ and $q$ are in $R_{n,2d}^r$. Combining the fact that $R_{n,2d}^r$ is a convex cone and the fact that $(\sum_i x_i^2)^d \in R_{n,d}^r$, we obtain (d).
\end{proof}

\begin{remark}\label{rem:bisection}
	To solve a fixed level $r$ of the hierarchy given in Theorem \ref{th:reznick.hierarchy}, one must proceed by \emph{bisection} on $\gamma.$ Bisection here would produce a sequence of upper bounds $\{U_k\}$ and lower bounds $\{L_k\}$ on $l_r$ as follows. At iteration $k$, we test whether $\gamma=\frac{U_k+L_k}{2}$ is feasible for (\ref{eq:reznick.hierarchy}). If it is, then we take $L_{k+1}=\frac{U_k+L_k}{2}$ and $U_{k+1}=U_k$. If it is not, we take $U_{k+1}=\frac{U_k+L_k}{2}$ and $L_{k+1}=L_k$. We stop when $|U_{k_\epsilon}-L_{k_\epsilon}|<\epsilon$, where $\epsilon$ is a prescribed accuracy, and the algorithm returns $l_{r,\epsilon}=L_{k_\epsilon}.$ Note that $l_{r}-\epsilon \leq l_{r,\epsilon} \leq l_r$ and that to obtain $l_{r,\epsilon}$, one needs to take a logarithmic (in $\frac{1}{\epsilon}$) number of steps using this method. 
	
	Hence, solving the $r^{th}$ level of this hierarchy using bisection can be done by semidefinite programming. Indeed, for a fixed $r$ and $\gamma$ given by the bisection algorithm, one simply needs to test membership of $$\left( f_{\gamma}(z)-\frac{1}{r}(\sum_i z_i^2)^d\right) \cdot (\sum_i z_i^2)^r$$ to the set of sum of squares polynomials. This amounts to solving a semidefinite program. We remark that all semidefinite programming-based hierarchies available only produce an approximate solution to the optimal value of the SDP solved at level $r$ in polynomial time. This is independent of whether they use bisection (e.g., such as the hierarchy given in Theorem \ref{th:reznick.hierarchy} or the one based on Stengle's Positivstellensatz) or not (e.g., the Lasserre hierarchy).
\end{remark}

\vspace{10pt}

Our next theorem improves on our previous hierarchy by freeing the multiplier $(\sum_{i=1}^n z_i^2)^r$ and taking advantage of our ability to search for an optimal multiplier using semidefinite programming.

\begin{theorem}\label{th:artin.hierarchy} Recall the definition of Artin cones from Definition \ref{def:reznick.artin.cones}. Consider the hierarchy of optimization problems indexed by $r$:
	\begin{equation}\label{eq:artin.hierarchy}
	\begin{aligned}
	l_r\mathrel{\mathop{:}}=&\sup_{\gamma,q} &&\gamma\\
	&\text{s.t. } &&f_{\gamma}(z)-\frac{1}{r}(\sum_{i=1}^n z_i^2)^{d} \in A_{n,2d}^r.
	\end{aligned}
	\end{equation}
	Then, $l_r \leq p^*$ for all $r$, $\{l_r\}$ is nondecreasing, and $\lim_{r \rightarrow \infty} l_r=p^*.$
\end{theorem}

\begin{proof}
	Just as the previous theorem, it suffices to show that the Artin cones $A_{n,2d}^r$ satisfy properties (a)-(d) of Theorem \ref{th:hierarchy}. The proof of property (a) follows the proof given for Theorem \ref{th:reznick.hierarchy}. Property (b) is satisfied as a (weaker) consequence of Artin's result (see Theorem \ref{th:artin}). For (c), we have that if $p(x) \cdot q(x)$ is sos for some sos polynomial of degree $2r$, then $p(x) \cdot q(x) \cdot (\sum_i x_i^2)$ is sos, and $q(x)\cdot (\sum_i x_i^2)$ has degree $2(r+1)$. Finally, for (d), suppose that $p \in A_{n,2d}^r$. Then there exists an sos form $q$ such that $p(x) \cdot q(x)$ is sos. We have $$\left(p(x)+\epsilon (\sum_i x_i^2)^{d}\right) \cdot q(x)=p(x)\cdot q(x)+\epsilon (\sum_i x_i^2)^{d} \cdot q(x),$$ which is sos as the product (resp. sum) of two sos polynomials is sos.
\end{proof}

Note that again, for any fixed $r$, the level $r$ of the hierarchy can be solved using bisection which leads to a sequence of semidefinite programs.

Our developments in the past two sections can be phrased in terms of a Positivstellensatz.

\begin{corollary}[A new Positivstellensatz]\label{cor:SDP.psatz}
	Consider the basic semialgebraic set $$S\mathrel{\mathop{:}}=\{x \in \mathbb{R}^n~|~ g_i(x)\geq 0, i=1,\ldots,m\}$$ and a polynomial $p\mathrel{\mathop{:}}=p(x)$. Suppose that $S$ is contained within a ball of radius $R$. Let $\eta_i$ and $\beta$ be any finite upperbounds on $g_i(x)$ and, respectively, $-p(x)$ over the set $S$.\footnote{As discussed at the beginning of Section \ref{sec:nonneg.approx}, such bounds are very easily computable.} Let $d$ be such that $2d$ is the smallest integer larger than or equal to the maximum degree of $p, g_i, i=1,\ldots,m$. Then, $p(x)>0$ for all $x \in S$ if and only if there exists a positive integer $r$ such that
	$$\left(h(x,s,y)-\frac{1}{r} \left(\sum_{i=1}^n x_i^2+\sum_{j =0}^{m+1} s_j^2+y^2\right)^{2d}\right) \cdot \left(\sum_{i=1}^n x_i^2+\sum_{j =0}^{m+1} s_j^2+y^2\right)^{r}$$ is a sum of squares, where the form $h$ in variables $(x_1,\ldots,x_n, s_0,\ldots,s_{m+1},y)$ is as follows:
	\begin{align*}
	h(x,s,y)\mathrel{\mathop{:}}=&\left(y^{2d}p(x/y)+s_0^2y^{2d-2} \right)^2+\sum_{i=1}^m\left(y^{2d}g_i(x/y)-s_i^2 y^{2d-2}\right)^2\\
	&+\left((R+\sum_{i=1}^m \eta_i +\beta)^dy^{2d}-(\sum_{i=1}^n x_i^2+\sum_{i=0}^m s_i^2)^d-s_{m+1}^{2d}\right)^2.
	\end{align*}

\end{corollary} 

\begin{proof}
	This is an immediate corollary of arguments given in the proof of Theorem \ref{th:slb} and in the proof of Theorem \ref{th:reznick.hierarchy} for the case where $\gamma=0.$
\end{proof}

\section{Poly\'a's theorem and hierarchies for POPs that are optimization-free, LP-based, and SOCP-based}\label{sec:opt.free.LP.SOCP}

In this section, we use a result by Poly\'a on global positivity of even forms to obtain new hierarchies for polynomial optimization problems. In Section~\ref{subsec:opt.free}, we present a hierarchy that is \emph{optimization-free}, in the sense that each level of the hierarchy only requires multiplication of two polynomials and checking if the coefficients of the resulting polynomial are nonnegative. In Section~\ref{subsec:lp.socp}, we use the previous hierarchy to derive linear programming and second-order cone programming-based hierarchies with faster convergence. These rely on the recently developed concepts of dsos and sdsos polynomials (see Definition \ref{def:dsos.sdsos} and \cite{isos_journal}), which are alternatives to sos polynomials that have been used in diverse applications to improve scalability; see \cite[Section 4]{isos_journal}.

\subsection{An optimization-free hierarchy of lower bounds for POPs}\label{subsec:opt.free}

The main theorem in this section presents an optimization-free hierarchy of lower bounds for general POPs with compact feasible sets:

\begin{theorem}\label{th:sms.hierarchy}
	Recall the definition of $f_{\gamma}(z)$ as given in (\ref{eq:f.gamma}), with $z \in \mathbb{R}^n$ and $deg(f_{\gamma})=2d.$ Let $(v,w) \in \mathbb{R}^{2n}$ and define	
	\begin{equation}\label{eq:def.Knd}
	\begin{aligned}
	&Pol_{n,2d}^r\mathrel{\mathop{:}}=&&\{p \in H_{n,2d}~|~\left(p(v^2-w^2)+ \frac{1}{2r}(\sum_{i=1}^n (v_i^4+w_i^4))^d\right) \cdot (\sum_i v_i^2+\sum_i w_i^2)^{r^2} \phantom\}\\
	& &&\phantom\{ \text{ has nonnegative coefficients } \}.
	\end{aligned}
	\end{equation}

	Consider the hierarchy of optimization problems indexed by $r$:
	\begin{equation}\label{eq:sms.hierarchy}
	\begin{aligned}
	l_r\mathrel{\mathop{:}}=&\sup_{\gamma} &&\gamma\\
	&\text{s.t. } && f_{\gamma}(z)-\frac{1}{r}(\sum_{i=1}^n z_i^2)^d \in Pol_{n,2d}^r.
	\end{aligned}
	\end{equation}
	Let $m_r=\max_{i=1,\ldots,r} l_i$. Then $m_r \leq p^*$ for all $r$, $\{m_r\}$ is nondecreasing, and $\lim_{r \rightarrow \infty} m_r=p^*$.
\end{theorem}

As before, we use bisection to obtain the optimal value $l_r$ of the $r^{th}$ level of the hierarchy up to a fixed precision $\epsilon$ (see Remark \ref{rem:bisection}). At each step of the bisection algorithm, one simply needs to multiply two polynomials together and check nonnegativity of the coefficients of the resulting polynomial to proceed to the next step. As a consequence, this hierarchy is optimization-free as we do not need to solve (convex) optimization problems at each step of the bisection algorithm. To the best of our knowledge, no other converging hierarchy of lower bounds for general POPs dispenses altogether with the need to solve convex subprograms. We also provide a Positivstellensatz counterpart to the hierarchy given above (see Corollary \ref{cor:Positiv.LP}). This corollary implies in particular that one can always certify infeasibility of a basic semialgebraic set by recursively multiplying polynomials together and simply checking nonnegativity of the coefficients of the resulting polynomial.

We now make a few remarks regarding the techniques used in the proof of Theorem~\ref{th:sms.hierarchy}. Unlike Theorems \ref{th:reznick.hierarchy} and \ref{th:artin.hierarchy}, we do not show that $Pol_{n,d}^r$ satisfies properties (a)-(d) as given in Theorem~\ref{th:hierarchy} due to some technical difficulties. It turns out however that we can avoid showing properties (c) and (d) by using a result by Reznick and Powers \cite{PR} that we present below. Regarding properties (a) and (b), we show that a slightly modified version of (a) holds and that (b), which is the key property in Theorem \ref{th:hierarchy}, goes through as is. We note though that obtaining (b) from Poly\'a's result (Theorem \ref{th:polya}) is not as immediate as obtaining (b) from Artin's and Reznick's results. Indeed, unlike the theorems by Artin and Reznick (see Theorems \ref{th:artin} and \ref{th:reznick.uniform}) which certify global positivity of \emph{any} form, Poly\'a's result only certifies global positivity of \emph{even} forms. To make this latter result a statement about general forms, we work in an appropriate lifted space. This is done by replacing any form $p(z)$ in variables $z \in \mathbb{R}^n$ by the even form $p(v^2-w^2)$ in variables $(v,w) \in \mathbb{R}^{2n}$. This lifting operation preserves nonnegativity, but unfortunately it does not preserve positivity: even if $p(z)$ is pd, $p(v^2-w^2)$ always has zeros (e.g., when $v=w$). Hence, though we now have access to an even form, we still cannot use Poly\'a's property as $p(v^2-w^2)$ is not positive. This is what leads us to consider the slightly more complicated form $p(v^2-w^2)+\frac{1}{2r}(\sum_i v_i^4+w_i^4)^d$ in (\ref{eq:def.Knd}). 

\begin{theorem}[Powers and Reznick \cite{PR}]\label{th:reznick}
	Let $\alpha=(\alpha_1,\ldots,\alpha_n) \in \mathbb{N}^n$, $x^\alpha=x_1^{\alpha_1}\ldots x_n^{\alpha_n}$, and write $|\alpha|=\alpha_1+\ldots+\alpha_n.$ Denote the standard simplex by $\Delta_n$. Assume that $f$ is a form of degree $2d$ that is positive on $\Delta_n$ and let $$\lambda=\lambda(f)\mathrel{\mathop{:}}=\min_{x \in \Delta_n} f(x).$$
	Define $c(\alpha)=\frac{(2d)!}{\alpha_1! \ldots \alpha_n!}.$ We have:
	$$f(x)=\sum_{|\alpha|=2d} a_{\alpha} x^{\alpha}=\sum_{|\alpha|=2d} b_{\alpha}c(\alpha)x^{\alpha}.$$
	Let $||f(x)||\mathrel{\mathop{:}}=\max_{|\alpha|=2d} |b_{\alpha}|$.\footnote{As defined, $||f||$ is a submultiplicative norm; see \cite{schweighofer2004complexity}.}
	
	Then, the coefficients of $$f(x_1,\ldots,x_n) \cdot (x_1+\ldots+x_n)^N$$ are nonnegative for $N> d(2d-1)\frac{||f(x)||}{\lambda}-2d$.
\end{theorem}

Note that here the bound is given in the case where one considers the alternative (but equivalent) formulation of Poly\'a's Positivstellensatz to the one given in Theorem \ref{th:polya}, i.e., when one is concerned with positivity of a form over the simplex. The result can easily be adapted to the formulation where one considers global positivity of an even form as shown below.

\begin{lemma}\label{th:Reznick.even}
	Let $p\mathrel{\mathop{:}}=p(x)$ be an even form of degree $2d$ that is positive definite. Let $\beta>0$ be its minimum on $S_x$. Then, $$p(x_1,\ldots,x_n) \cdot (\sum_i x_i^2)^N$$ has nonnegative coefficients for $N>d(2d-1)\frac{||p(\sqrt{x})||}{\beta}-2d$.
\end{lemma}

\begin{proof}
	Let $f(x_1,\ldots,x_n)=p(\sqrt{x_1},\ldots,\sqrt{x_n})$. Since $p(x)\geq \beta$ on $S_x$, then $f(x) \geq \beta$ on $\Delta_n.$ Indeed, by contradiction, suppose that there exists $\hat{x} \in \Delta_n$ such that $f(\hat{x}) =\beta -\epsilon$ (where $\epsilon>0$) and let $y=\sqrt{\hat{x}}$. Note that as $\sum_i \hat{x}_i=1$, we have $\sum_i y_i^2=1$. Furthermore, $p(y)=f(\hat{x})=\beta-\epsilon$ which contradicts the assumption. Hence, using Theorem~\ref{th:reznick}, we have that when $N>d(2d-1)\frac{||p(\sqrt{x})||}{\beta}-2d$, $$f(x)(\sum_i x_i)^N $$
	has nonnegative coefficients. Hence,
	$$f(y^2)(\sum_i y_i^2)^N =p(y)(\sum_i y_i^2)^N$$
	also has nonnegative coefficients.
\end{proof}

Before we proceed with the proof of Theorem \ref{th:sms.hierarchy}, we need the following lemma.

\begin{lemma} \label{lem:outstrip}
	Let 	
	\begin{equation}\label{eq:def.p.gamma}
	p_{\gamma,r}(v,w)\mathrel{\mathop{:}}=f_\gamma(v^2-w^2)-\frac{1}{r}(\sum_{i=1}^n (v_i^2-w_i^2)^2)^d+\frac{1}{2r} (\sum_{i=1}^n (v_i^4+w_i^4))^d,
	\end{equation}
	where $f_{\gamma}$ is defined as in (\ref{eq:f.gamma}) and let $$N(r)=d(2d-1) \cdot \frac{||p_{\gamma,r}(\sqrt{v},\sqrt{w})||}{\min_{S_{v,w}}p_{\gamma,r}(v,w)}-2d.$$ If $f_{\gamma}(z)$ is positive definite, there exists $\hat{r}$ such that $r^2 \geq N(r)$,  for all $r \geq \hat{r}$.
\end{lemma}

\begin{proof}
	As $f_{\gamma}(z)$ is positive definite, there exists a positive integer $r_0$ such that $f_{\gamma}(z)-\frac{1}{r}(\sum_i z_i^2)^d$ is positive definite for all $r \geq r_0$ and hence 
	\begin{equation}\label{eq:nonneg.proof}
	f_{\gamma}(v^2-w^2)-\frac{1}{r} (\sum_i (v_i^2-w_i^2)^2)^d
	\end{equation}
	is nonnegative for all $r \geq r_0$. Recall now that $||x||_p=(\sum_i x_i^p)^{1/p}$ is a norm for $p \geq 1$ and that $$||x||_2 \leq \sqrt{n}||x||_4.$$ This implies that $$(\sum_i v_i^4+\sum_i w_i^4)^d \geq \frac{1}{(2n)^{2d}} (\sum_i v_i^2+\sum_i w_i^2)^{2d}$$ and hence in view of (\ref{eq:nonneg.proof}) and the definition of $p_{\gamma,r}$, we have $$p_{\gamma,r}(v,w) \geq \frac{1}{2^{2d+1}n^{2d}r} (\sum_i v_i^2+\sum_{i} w_i^2)^{2d}, \forall r\geq r_0.$$
	This enables us to conclude that 
	\begin{align}\label{eq:min.p.gamma}
	\min_{S_{v,w}} p_{\gamma,r}(v,w) \geq \frac{1}{2^{2d+1}n^{2d}r}, \text{ for any } r \geq r_0.
	\end{align}
	Further, notice that using properties of the norm, we have the following chain of inequalities for any positive integer $r$:
	\begin{align*}
	||p_{\gamma,r}(\sqrt{v},\sqrt{w})|| &\leq ||f_\gamma(v-w)||+\frac{1}{r}||(\sum_i (v_i-w_i)^2)^d||+\frac{1}{2r} ||(\sum_i (v_i^2+w_i^2))^d||\\
	&\leq ||f_\gamma(v-w)||+||(\sum_i (v_i-w_i)^2)^d||+ ||(\sum_i v_i^2+w_i^2)^d||=\mathrel{\mathop{:}}c_{\gamma}.
	\end{align*}
	As a consequence, combining this with the definition of $N(r)$ and (\ref{eq:min.p.gamma}), we have $$N(r) \leq d(2d-1)2^{{2d+1}}rn^{2d}c_{\gamma}, ~\forall r \geq r_0.$$
	Now taking $\hat{r}=\max(r_0,\lceil d(2d-1)2^{{2d+1}}n^{2d}c_{\gamma}\rceil)$, we have $r^2 \geq N(r), \forall r\geq \hat{r}.$
\end{proof}
We now proceed with the proof of Theorem \ref{th:sms.hierarchy}.

\begin{proof} (Proof of Theorem \ref{th:sms.hierarchy})
	By definition, the sequence $\{m_r\}$ is nondecreasing. We show that it is upperbounded by $p^*$ by showing that if $\gamma$ is such that $$f_{\gamma}(z)-\frac{1}{r}(\sum_i z_i^2)^d \in Pol_{n,2d}^r,$$ for some $r$, then $f_{\gamma}$ must be positive definite. Then Theorem \ref{th:slb} gives us that $\gamma$ is a strict lower bound on (\ref{eq:POP}). As $p^* > \gamma$ for any such $\gamma$, we have that $l_r\leq p^*,\forall r$ and hence $m_r \leq p^*, \forall r.$ 
	
	Assume that $\gamma$ is such that
	$$f_{\gamma}(z)-\frac{1}{r}(\sum_i z_i^2)^d \in Pol_{n,2d}^r$$ for some $r$.
	By definition of $Pol_{n,2d}^r$ and as $(\sum_i v_i^2+\sum_i w_i^2)^{r^2}$ is nonnegative, we get that the form
	$$f_\gamma(v^2-w^2)-\frac{1}{r}(\sum_i (v_i^2-w_i^2)^2)^d+\frac{1}{2r} (\sum_i v_i^4+w_i^4)^d$$ is nonnegative.
	This implies that 
	\begin{align}\label{eq:ineq.tv.screen}
	f_{\gamma}(v^2-w^2)-\frac{1}{r}(\sum_i (v_i^2-w_i^2)^2)^d \geq -\frac{1}{2r} \text{ for } (v,w) \in \{(v,w)~|~ \sum_i v_i^4+\sum_i w_i^4 =1\},
	\end{align} which gives 
	\begin{align}\label{eq:ineq.sphere}
	f_{\gamma}(z)-\frac{1}{r}(\sum_i z_i^2)^d \geq -\frac{1}{2r},~ \forall z\in S_z.
	\end{align}
	Indeed, suppose that there exists $\hat{z} \in S_z$ such that (\ref{eq:ineq.sphere}) does not hold. Then, let $\hat{z}^+=\max(\hat{z},0)$ and $\hat{z}^-=\max(-\hat{z},0)$. Note that both $\hat{z}^+$ and $\hat{z}^-$ are nonnegative so we can take $\hat{v}=\sqrt{\hat{z}^+}$ and $\hat{w}=\sqrt{\hat{z}^-}.$ We further have that as $\hat{z} \in S_z$ and $\hat{z}=\hat{v}^2-\hat{w}^2$, $\sum_i \hat{v}_i^4+\sum_i \hat{w}_i^4=1$. Substituting $\hat{z}$ by $\hat{v}^2-\hat{w}^2$ in (\ref{eq:ineq.sphere}) then violates (\ref{eq:ineq.tv.screen}). Using (\ref{eq:ineq.sphere}), we conclude that $$f_{\gamma}(z)\geq \frac{1}{2r},~\forall z\in S_z$$ and that $f_{\gamma}$ is positive definite. 
	
	We now show that the hierarchy converges, i.e., that $\lim_{r \rightarrow \infty} m_r=p^*$. To do this, we show that if $\gamma$ is a strict lower bound on (\ref{eq:POP}), or equivalently from Theorem \ref{th:slb}, if $f_{\gamma}(z)$ is positive definite, then there exists $r'$ such that $$f_{\gamma}(z)-\frac{1}{r'} (\sum_{i}z_i^2)^d \in Pol_{n,2d}^{r'}.$$ 
	Since $f_{\gamma}$ is pd, there exists a positive integer $r_0$ such that $f_{\gamma}(z)-\frac{1}{r}(\sum_{i=1}^n z_i^2)^d$ is pd for any $r \geq r_0$. This implies that $f_{\gamma}(v^2-w^2)-\frac{1}{r}(\sum_i (v_i^2-w_i^2)^2)^d$ is nonnegative and $$f_\gamma(v^2-w^2)-\frac{1}{r}(\sum_i (v_i^2-w_i^2)^2)^d+\frac{1}{2r} (\sum_i(v_i^4+w_i^4))^d$$ is positive definite for $r \geq r_0$. Using Lemma \ref{th:Reznick.even} and the definition of $N(r)$ in Lemma \ref{lem:outstrip}, for any $r\geq r_0$, we have that $$\left(f_\gamma(v^2-w^2)-\frac{1}{r}(\sum_i (v_i^2-w_i^2)^2)^d+\frac{1}{2r} (\sum_i (v_i^4+w_i^4))^d \right) \cdot (\sum_i v_i^2+\sum_i w_i^2)^{\lceil N(r)\rceil}$$ has nonnegative coefficients. From Lemma \ref{lem:outstrip}, there exists $\hat{r}$ such that $r \geq \hat{r}$ implies $r^2 \geq N(r).$ Taking $r'=\max\{r_0,\hat{r}\}$ and considering $p_{\gamma,r'}$ as defined in (\ref{eq:def.p.gamma}), we get that
	\begin{align*}
	&p_{\gamma,r'}(v,w) (\sum_i v_i^2+\sum_i w_i^2)^{r'^2} \\ &=p_{\gamma,r'}(v,w)(\sum_i v_i^2+\sum_i w_i^2)^{\lceil N(r') \rceil} \cdot (\sum_i v_i^2+\sum_i w_i^2)^{r'^2-\lceil N(r') \rceil} 
	\end{align*}
	has nonnegative coefficients, which is the desired result. This is because $$p_{\gamma,r'}(v,w)(\sum_i v_i^2+\sum_i w_i^2)^{\lceil N(r') \rceil}$$ has nonnegative coefficients as $r' \geq r_0$, and $$(\sum_i v_i^2+\sum_i w_i^2)^{r'^2-\lceil N(r') \rceil}$$ has nonnegative coefficients as $r' \geq \hat{r}$, and that the product of two polynomials with nonnegative coefficients has nonnegative coefficients.
\end{proof}

\begin{corollary}[An optimization-free Positivstellensatz]\label{cor:Positiv.LP} Consider the basic semialgebraic set $$S\mathrel{\mathop{:}}=\{x \in \mathbb{R}^n~|~ g_i(x)\geq 0, i=1,\ldots,m\}$$ and a polynomial $p\mathrel{\mathop{:}}=p(x)$. Suppose that $S$ is contained within a ball of radius $R$. Let $\eta_i$ and $\beta$ be any finite upperbounds on $g_i(x)$ and, respectively, $-p(x)$ over the set $S$.\footnote{Once again, as discussed at the beginning of Section \ref{sec:nonneg.approx}, such bounds are very easily computable.} Let $d$ be such that $2d$ is the smallest even integer larger than or equal to the maximum degree of $p, g_i, i=1,\ldots,m$.  Then, $p(x)>0$ for all $x \in S$ if and only if there exists a positive integer $r$ such that	
	\begin{align*}
	\left(h(v^2-w^2)-\frac{1}{r}(\sum_{i=1}^{n+m+3} (v_i^2-w_i^2)^2)^d+\frac{1}{2r} (\sum_{i=1}^{n+m+3} (v_i^4+w_i^4))^d\right) \\
	\cdot \left(\sum_{i=1}^{n+m+3}v_i^2+\sum_{i=1}^{{n+m+3}} w_i^2\right)^{r^2}
	\end{align*}
	has nonnegative coefficients, where the form $h \mathrel{\mathop{:}}=h(z)$ in variables $$(z_1,\ldots,z_{n+m+3})\mathrel{\mathop{:}}=(x_1,\ldots,x_n,s_0,\ldots,s_{m+1},y)$$ is as follows:
	\begin{align*}
	h(x,s,y)\mathrel{\mathop{:}}=&\left(y^{2d}p(x/y)+s_0^2y^{2d-2} \right)^2+\sum_{i=1}^m\left(y^{2d}g_i(x/y)-s_i^2 y^{2d-2}\right)^2\\
	&+\left((R+\sum_{i=1}^m \eta_i +\beta)^dy^{2d}-(\sum_{i=1}^n x_i^2+\sum_{i=0}^m s_i^2)^d-s_{m+1}^{2d}\right)^2.
	\end{align*}
\end{corollary}

\begin{proof} 
	This is an immediate corollary of arguments given in the proof of Theorem~\ref{th:slb} and in the proof of Theorem \ref{th:sms.hierarchy} for the case where $\gamma=0.$
\end{proof}

\subsection{Linear programming and second-order cone programming-based hierarchies for POPs}\label{subsec:lp.socp}

In this section, we present a linear programming and a second-order cone programming-based hierarchy for general POPs which by construction converge faster than the hierarchy presented in Section~\ref{subsec:opt.free}. These hierarchies are based on the recently-introduced concepts of dsos and sdsos polynomials \cite{isos_journal} which we briefly revisit below to keep the presentation self-contained.

\begin{definition}\label{def:dd.sdd}
	A symmetric matrix $M$ is said to be
	\begin{itemize}
		\item \emph{diagonally dominant (dd)} if $M_{ii} \geq \sum_{j \neq i}|M_{ij}|$ for all $i$. 
		\item  \emph{scaled diagonally dominant (sdd)} if there exists a diagonal matrix $D,$ with positive diagonal entries, such that $DAD$ is dd.
	\end{itemize} 
\end{definition}
We have the following implications as a consequence of Gershgorin's circle theorem:
\begin{align}\label{eq:implic.matrices}
\text{M } dd \Rightarrow \text{M } sdd \Rightarrow \text{M } \succeq 0.
\end{align}
Requiring $M$ to be dd (resp. sdd) can be encoded via a linear program (resp. a second-order cone program) (see \cite{isos_journal} for more details). These notions give rise to the concepts of dsos and sdsos polynomials.

\begin{definition}[\cite{isos_journal}]\label{def:dsos.sdsos}
	Let $z(x)=(x_1^d,x_1^{d-1}x_2,\ldots,x_n^d)^T$ be the vector of monomials in $(x_1,\ldots,x_n)$ of degree $d$. A form $p \in H_{n,2d}$ is said to be 
	\begin{itemize}
		\item \emph{diagonally-dominant-sum-of-squares (dsos)} if it admits a representation$$p(x)=z^T(x)Qz(x), \text{ where $Q$ is a dd matrix.}$$
		\item \emph{scaled-diagonally-dominant-sum-of-squares (sdsos)} if it admits a representation$$p(x)=z^T(x)Qz(x), \text{ where $Q$ is a sdd matrix.}$$
	\end{itemize}
\end{definition}

The following implications are a consequence of (\ref{eq:implic.matrices}):
\begin{align}\label{eq:implicsdsos}
p(x) \text{ dsos} \Rightarrow p(x) \text{ sdsos} \Rightarrow p(x) \text{ sos} \Rightarrow p(x) \text{ nonnegative}.
\end{align}
Given the fact that our Gram matrices and polynomials are related to each other via linear equalities, it should be clear that optimizing over the set of dsos (resp. sdsos) polynomials is an LP (resp. SOCP).

We now present our LP and SOCP-based hierarchies for POPs.

\begin{corollary}\label{cor:dsos.sdsos.hierarchy}
	
	Recall the definition of $f_{\gamma}(z)$ as given in (\ref{eq:f.gamma}), with $z \in \mathbb{R}^n$ and $deg(f)=2d$, and let $p_{\gamma,r}$ be as in (\ref{eq:def.p.gamma}). Consider the hierarchy of optimization problems indexed by $r$:
	\begin{equation}\label{eq:dsos.sdsos.hierarchy}
	\begin{aligned}
	l_r\mathrel{\mathop{:}}=&\sup_{\gamma,q} &&\gamma\\
	&\text{s.t. } && p_{\gamma,r}(v,w) \cdot q(v,w) \text{ is s/dsos}\\
	& &&q(v,w) \text{ is s/dsos and of degree $2r^2$}.
	\end{aligned}
	\end{equation}
	Let $m_r=\max_{i=1,\ldots,r} l_i$. Then, $m_r \leq p^*$ for all $r$, $\{m_r\}$ is nondecreasing, and we have $\lim_{r \rightarrow \infty} m_r=p^*$.
\end{corollary}

\begin{proof}
	This is an immediate consequence of the fact that any even form $p \in H_{n,2d}$ with nonnegative coefficients can be written as $p(x)=z(x)^TQz(x)$ where $Q$ is diagonal and has nonnegative (diagonal) entries. As such a $Q$ is  dd (and also sdd), we conclude that $p$ is dsos (and also sdsos). The corollary then follows from Theorem~\ref{th:sms.hierarchy}.
\end{proof}

Note that similarly to our previous hierarchies, one must proceed by bisection on $\gamma$ to solve the level $r$ of the hierarchy. At each step of the hierarchy, we solve a linear program (resp. second-order cone program) that searches for the coefficients of $q$ that make $q$ dsos (resp. sdsos) and $p_{\gamma,r} \cdot q$ dsos (resp. sdsos).

There is a trade-off between the hierarchies developed in this subsection and the one developed in the previous subsection: the hierarchy of Section \ref{subsec:opt.free} is optimization-free whereas those of Section \ref{subsec:lp.socp} use linear or second-order cone programming. Hence the former hierarchy is faster to run at each step. However, the latter hierarchies could potentially take fewer levels to converge. This is similar to the trade-off observed between the hierarchies presented in Theorems \ref{th:reznick.hierarchy} and \ref{th:artin.hierarchy}.

\section{Open problems}\label{sec:conclusion}

To conclude, we present two open problems spawned by the writing of this chapter. The first one concerns the assumptions needed to construct our hierarchies.

\paragraph{Open problem 1} Theorems \ref{th:slb} and \ref{th:hierarchy} require that the feasible set $S$ of the POP given in (\ref{eq:POP}) be contained in a ball of radius $R$. Can these theorems be extended to the case where there is no compactness assumption on $S$?

\indent The second open problem is linked to the Artin and Reznick cones presented in Definition \ref{def:reznick.artin.cones}. 
\paragraph{Open problem 2} As mentioned before, Reznick cones $R_{n,2d}^r$ are convex for all $r$. We are unable to prove however that Artin cones $A_{n,2d}^r$ are convex (even though they satisfy properties (a)-(d) of Theorem \ref{th:hierarchy} like Reznick cones do). Are Artin cones convex for all $r$? We know that they are convex for $r=0$ and for $r$ large enough as they give respectively the sos and psd cones (see \cite{lombardi2014elementary} for the latter claim). However, we do not know the answer already for $r=1$.

%\bibliographystyle{abbrv}
%\bibliography{ch-basis_pursuit/pablo_amirali,ch-basis_pursuit/elib}

%\bibliographystyle{plain}
%\begin{thebibliography}{10}
%
%\bibitem{am}
%A. A. Ahmadi and A. Majumdar. 
%DSOS and SDSOS optimization: More tractable alternatives to SOS optimization.
%{\em Manuscript.}
%
%\bibitem{cl}
%  M. D. Choi and T. Y. Lam, Extremal positive semidefinite forms, {\em Math. Ann.} {\bf 231} 1--8 (1977).
%  
%\bibitem{dp}
%E. de Klerk and D. sechnik.  Approximation of the stability number of a graph via copositive programming. {\em SIAM Journal on Optimization}, {\bf 12}(4), 875-- 892 (2002).
%
%\bibitem{mot}
%T.S. Motzkin, The arithmetic-geometric inequality. Proc. Symposium on Inequalities (ed. O. Shisha), Academic Press, New York, 1967, pp. 205-–224.
%
%\bibitem{mk}
%K. G. Murty and S. N. Kabadi. Some NP-complete problems in quadratic and nonlinear
%programming. {\em Mathematical Programming}, {\bf 39} 117--129 (1987).
%  
%\bibitem{pp}
%P. A. Parrilo. Semidefinite programming relaxations for semialgebraic problems. Mathematical Programming {\bf 96} 293--320 (2003).
%\end{thebibliography}

\part{Optimizing over Convex Polynomials}

\chapter{DC Decomposition of Nonconvex Polynomials with Algebraic Techniques}\label{ch:dc}

\section{Introduction}\label{sec:intro}

%In this paper, we focus on a class of optimization problems called difference of convex (dc) programs. This corresponds to problems of the type

A difference of convex (dc) program is an optimization problem of the form

\begin{equation}\label{eq:dcP}
\begin{aligned}
&\min f_0(x)\\
&\text{s.t. } f_i(x)\leq 0, i=1,\ldots,m,
\end{aligned}
\end{equation}
where $f_0,\ldots,f_m$ are difference of convex functions; i.e., 
\begin{equation} \label{eq:dceq}
\begin{aligned} 
&f_i(x)=g_i(x)-h_i(x),i=0,\ldots,m, \\
%&g_i:\mathbb{R}^n \rightarrow \mathbb{R},~h_i:\mathbb{R}^n \rightarrow \mathbb{R} \text{ are convex functions.}
\end{aligned}
\end{equation} 
and $g_i:\mathbb{R}^n \rightarrow \mathbb{R},~h_i:\mathbb{R}^n \rightarrow \mathbb{R}$ are convex functions. The class of functions that can be written as a difference of convex functions is very broad containing for instance all functions that are twice continuously differentiable~\cite{Hartman}, \cite{hiriart}. Furthermore, any continuous function over a compact set is the uniform limit of a sequence of dc functions; see, e.g., reference \cite{Horst} where several properties of dc functions are discussed.

%As shown in \cite{Hartman}, the class of functions that can be written as a difference of convex functions is very broad, containing in particular all functions that are twice continuously differentiable. Furthermore, any continuous function over a compact set is the uniform limit of a sequence of dc functions. A more detailed overview of properties of dc functions can be found in \cite{Horst}.

Optimization problems that appear in dc form arise in a wide range of applications. Representative examples from the literature include machine learning and statistics (e.g., kernel selection \cite{Argyriou}, feature selection in support vector machines \cite{le2008dc}, sparse principal component analysis \cite{lipp2014variations}, and reinforcement learning \cite{piot2014difference}), operations research (e.g., packing problems and production-transportation problems \cite{Tuy}), communications and networks \cite{pang},\cite{lou2015}, circuit design \cite{lipp2014variations}, finance and game theory \cite{gulpinar2010}, and computational chemistry \cite{floudas2013}. We also observe that dc programs can encode constraints of the type $x \in \{0,1\}$ by replacing them with the dc constraints $0 \leq x \leq 1, x-x^2 \leq 0$. This entails that any binary optimization problem can in theory be written as a dc program, but it also implies that dc problems are hard to solve in general.
%compressed sensing \cite{lou2015}, mode switching \cite{feng2014}

As described in \cite{tao1997}, there are essentially two schools of thought when it comes to solving dc programs. The first approach is global and generally consists of rewriting the original problem as a concave minimization problem (i.e., minimizing a concave function over a convex set; see \cite{tuy1988}, \cite{tuy1987}) or as a reverse convex problem (i.e., a convex problem with a linear objective and one constraint of the type $h(x)\geq 0$ where $h$ is convex). We refer the reader to \cite{Tuy86} for an explanation on how one can convert a dc program to a reverse convex problem, and to \cite{HJ80b} for more general results on reverse convex programming. These problems are then solved using branch-and-bound or cutting plane techniques (see, e.g., \cite{Tuy} or \cite{Horst}). The goal of these approaches is to return global solutions but their main drawback is scalibility. The second approach by contrast aims for local solutions while still exploiting the dc structure of the problem by applying the tools of convex analysis to the two convex components of a dc decomposition. One such algorithm is the Difference of Convex Algorithm (DCA) introduced by Pham Dinh Tao in \cite{PDT76} and expanded on by Le Thi Hoai An and Pham Dinh Tao. This algorithm exploits the duality theory of dc programming~\cite{toland79} and is popular because of its ease of implementation, scalability, and ability to handle nonsmooth problems.

%contrast is local and applies tools from convex analysis to the two convex components of a dc decomposition. One such algorithm is the Difference of Convex Algorithm (DCA), introduced by Pham Dinh Tao in \cite{PDT76} and expanded on by Le Thi Hoai An and Phan Dinh Tao.

%, which is based on the concept of dc duality \cite{toland79}. 
%Its simplified form \cite{tao1997} constructs a primal program $$(P) \qquad \alpha=\inf \{f(x)\mathrel{\mathop{:}}=g(x)-h(x) : x \in X\}$$ and a dual program $$(D) \qquad \alpha=\inf\{h^*(y)-g^*(y): y \in Y\},$$ where $g^*(y) = \sup\{\langle x, y\rangle - g(x) : x \in X\}$ and similarly for $h^*(y)$. From (P) and (D) two sequences of convex programs are constructed denoted by ($P_k$) and ($D_k$), generating sequences of vectors $\{x_k\}_k$ and $\{y_k\}_k$ in the following way
%\begin{align*}
%&y_k=\text{argmin} \{h^*(y)-[g^*(y^{k-1})+\langle x^k,y-y^{k-1} \rangle] : y \in Y \} \qquad (D_k)\\
%&x_{k+1}=\text{argmin} \{g(x)-[h(x^k)+ \langle x-x^k, y^k \rangle] : x \in X\} \qquad (P_k).
%\end{align*}
%In other words, ($P_k$) (resp. $D_k$) is obtained from ($P$) (resp. ($D$)) by replacing $h$ (resp. $g^*$) with an affine minorization $y_k \in \partial h(x^k)$ (resp. $x^k \in \partial g^*(y^{k-1})$).
%The solutions obtained using DCA are generally local solutions, and the functions $g_i$ and $h_i$ in (\ref{eq:dceq}) are only required to be lower semi-continuous convex.

In the case where the functions $g_i$ and $h_i$ in (\ref{eq:dceq}) are differentiable, DCA reduces to another popular algorithm called the Convex-Concave Procedure (CCP) \cite{Lanckriet}. The idea of this technique is to simply replace the concave part of $f_i$ (i.e., $-h_i$) by a linear overestimator as described in Algorithm \ref{alg:CCP}. By doing this, problem (\ref{eq:dcP}) becomes a convex optimization problem that can be solved using tools from convex analysis. The simplicity of CCP has made it an attractive algorithm in various areas of application. These include statistical physics (for minimizing Bethe and Kikuchi free energy functions \cite{YR}), machine learning \cite{lipp2014variations},\cite{Fung},\cite{Chapelle}, and image processing \cite{Wang}, just to name a few. In addition, CCP enjoys two valuable features: (i) if one starts with a feasible solution, the solution produced after each iteration remains feasible, and (ii) the objective value improves in every iteration, i.e., the method is a descent algorithm.
%
%There are several immediate advantages to CCP. If one starts with a feasible point, then the solution at every iteration is feasible. CCP is also a descent algorithm as the objective value improves at each step.
The proof of both claims readily comes out of the description of the algorithm and can be found, e.g., in \cite[Section 1.3.]{lipp2014variations}, where several other properties of the method are also laid out. Like many iterative algorithms, CCP relies on a stopping criterion to end. This criterion can be chosen amongst a few alternatives. For example, one could stop if the value of the objective does not improve enough, or if the iterates are too close to one another, or if the norm of the gradient of $f_0$ gets small. 

\begin{algorithm}[H]
	\caption{ CCP}
	\label{alg:CCP}
	\begin{algorithmic}[1]
		\Require $x_0,~ f_i=g_i-h_i, i=0,\ldots,m$
		\State $k\leftarrow 0$
		\While{stopping criterion not satisfied}
		\State Convexify: $f_i^{k}(x)\mathrel{\mathop{:}}=g_i(x)-(h_i(x_k)+\nabla h_i(x_k)^T(x-x_k)),~ i=0,\ldots,m$
		\State Solve convex subroutine: $\min f_0^k(x)$, s.t. $f_i^k(x) \leq 0, i=1,\ldots,m$
		\State $x_{k+1}\mathrel{\mathop{:}}= \underset{f_i^{k}(x) \leq 0}{\text{argmin}} f_0^k(x)$
		\State $k \leftarrow k+1$
		\EndWhile
		\Ensure $x_k$
	\end{algorithmic}
\end{algorithm}

%The formulation we use of CCP is similar to the one used by Yuille and Rangarajan in \cite{YR}.
Convergence results for CCP can be derived from existing results found for DCA, since CCP is a subcase of DCA as mentioned earlier. But CCP can also be seen as a special case of the family of majorization-minimization (MM) algorithms. Indeed, the general concept of MM algorithms is to iteratively upperbound the objective by a convex function and then minimize this function, which is precisely what is done in CCP. This fact is exploited by Lanckriet and Sriperumbudur in \cite{Lanckriet} and Salakhutdinov et al. in \cite{Salak} to obtain convergence results for the algorithm, showing, e.g., that under mild assumptions, CCP converges to a stationary point of the optimization problem (\ref{eq:dcP}).

\subsection{Motivation and organization of the chapter}

Although a wide range of problems already appear in dc form (\ref{eq:dceq}), such a decomposition is not always available. In this situation, algorithms of dc programming, such as CCP, generally fail to be applicable. Hence, the question arises as to whether one can (efficiently) compute a difference of convex decomposition (dcd) of a given function. This challenge has been raised several times in the literature. For instance, Hiriart-Urruty~\cite{hiriart} states ``All the proofs [of existence of dc decompositions] we know are ``constructive'' in the sense that they indeed yield [$g_i$] and [$h_i$] satisfying (\ref{eq:dceq}) but could hardly be carried over [to] computational aspects''. As another example, Tuy~\cite{Tuy} writes: ``The dc structure of a given problem is not always apparent or easy to disclose, and even when it is known explicitly, there remains for the problem solver the hard task of bringing this structure to a form amenable to computational analysis.''

%As one can see from the algorithm, CCP requires input functions to be in form (\ref{eq:dceq}), i.e., the input functions are given as a difference of convex functions. Though this is naturally the case for many applications where CCP is used, can we obtain such a decomposition efficiently if this isn't the case?

Ideally, we would like to have not just the ability to find one dc decomposition, but also to optimize over the set of valid dc decompositions. Indeed, dc decompositions are not unique: Given a decomposition $f=g-h$, one can produce infinitely many others by writing $f=g+p-(h+p),$ for any convex function $p$. This naturally raises the question whether some dc decompositions are better than others, for example for the purposes of CCP. 

%Once again, we will consider this problem for polynomials.

In this chapter we consider these decomposition questions for multivariate polynomials. Since polynomial functions are finitely parameterized by their coefficients, they provide a convenient setting for a computational study of the dc decomposition questions. Moreover, in most practical applications, the class of polynomial functions is large enough for modeling purposes as polynomials can approximate any continuous function on compact sets with arbitrary accuracy. It could also be interesting for future research to explore the potential of dc programming techniques for solving the polynomial optimization problem. This is the problem of minimizing a multivariate polynomial subject to polynomial inequalities and is currently an active area of research with applications throughout engineering and applied mathematics. In the case of quadratic polynomial optimization problems, the dc decomposition approach has already been studied~\cite{Bomze},\cite{quadratic_dc}.

With these motivations in mind, we organize the chapter as follows. In Section~\ref{sec:undominated}, we start by showing that unlike the quadratic case, the problem of testing if two given polynomials $g,h$ form a valid dc decomposition of a third polynomial $f$ is NP-hard (Proposition~\ref{prop:dcd.NPhard}). We then investigate a few candidate optimization problems for finding dc decompositions that speed up the convex-concave procedure. In particular, we extend the notion of an undominated dc decomposition from the quadratic case~\cite{Bomze} to higher order polynomials. We show that an undominated dcd always exists (Theorem~\ref{thm:undom.dcd}) and can be found by minimizing a certain linear function of one of the two convex functions in the decomposition. However, this optimization problem is proved to be NP-hard for polynomials of degree four or larger (Proposition~\ref{prop:undominated.NPhard}). To cope with intractability of finding optimal dc decompositions, we propose in Section~\ref{sec:ConvRelax} a class of algebraic relaxations that allow us to optimize over subsets of dcds. These relaxations will be based on the notions of \emph{dsos-convex, sdsos-convex,} and \emph{sos-convex} polynomials (see Definition~\ref{def:alternatives.convexity}), which respectively lend themselves to {\emph{linear, second order cone,}} and \emph{semidefinite programming}. In particular, we show that a dc decomposition can always be found by linear programming (Theorem~\ref{thm:diff.dsos}). Finally, in Section~\ref{sec:numerical.results}, we perform some numerical experiments to compare the scalability and performance of our different algebraic relaxations.

\section{Polynomial dc decompositions and their complexity} \label{sec:undominated}

To study questions around dc decompositions of polynomials more formally, let us start by introducing some notation. A multivariate \emph{polynomial} $p(x)$ in variables
$x\mathrel{\mathop:}=(x_1,\ldots,x_n)^T$ is a function from
$\mathbb{R}^n$ to $\mathbb{R}$ that is a finite linear combination
of monomials:
\begin{equation}
p(x)=\sum_{\alpha}c_\alpha x^\alpha=\sum_{\alpha_1, \ldots,
	\alpha_n} c_{\alpha_1,\ldots,\alpha_n} x_1^{\alpha_1} \cdots
x_n^{\alpha_n} ,
\end{equation}
where the sum is over $n$-tuples of nonnegative
integers $\alpha_i$. 
The \emph{degree} of a monomial $x^\alpha$ is equal to $\alpha_1 +
\cdots + \alpha_n$. The degree of a polynomial $p(x)$ is defined to
be the highest degree of its component monomials. A simple counting
argument shows that a polynomial of degree $d$ in $n$ variables has
$\binom{n+d}{d}$ coefficients. A \emph{homogeneous polynomial} (or a
\emph{form}) is a polynomial where all the monomials have the same
degree. An $n$-variate form $p$ of degree $d$ has $\binom{n+d-1}{d}$ coefficients. {We denote the set of polynomials (resp. forms) of degree $2d$ in $n$ variables by $\tilde{\mathcal{H}}_{n,2d}$ (resp. $\mathcal{H}_{n,2d}$). }

Recall that a symmetric matrix $A$ is positive semidefinite (psd) if $x^TAx \geq 0$ for all $x \in \mathbb{R}^n$; this will be denoted by the standard notation $A \succeq 0.$ Similarly, a polynomial $p(x)$ is said to be \emph{nonnegative} or
positive semidefinite if $p(x)\geq0$ for all $x\in\mathbb{R}^n$.
For a polynomial $p$, we denote its Hessian by $H_p$. The second order characterization of convexity states that $p$ is convex if and only if $H_p(x) \succeq 0$, $\forall x \in\mathbb{R}^n.$

\begin{definition}
	We say a polynomial $g$ is a \emph{dcd} of a polynomial $f$ if $g$ is convex and $g-f$ is convex. 
\end{definition}
Note that if we let $h\mathrel{\mathop{:}}=g-f$, then indeed we are writing $f$ as a difference of two convex functions $f=g-h.$
It is known that any polynomial $f$ has a (polynomial) dcd $g$. A proof of this is given, e.g., in \cite{Wang}, or in Section \ref{subsec:diffconv}, where it is obtained as corollary of a stronger theorem (see Corollary \ref{cor:SDSOSetc}). By default, all dcds considered in the sequel will be of even degree. Indeed, if $f$ is of even degree $2d$, then it admits a dcd $g$ of degree $2d$. If $f$ is of odd degree $2d-1$, it can be viewed as a polynomial $\tilde{f}$ of even degree $2d$ with highest-degree coefficients which are 0. The previous result then remains true, and $\tilde{f}$ admits a dcd of degree $2d$. 

Our results show that such a decomposition can be found efficiently (e.g., by linear programming); see Theorem \ref{th:mainth}. Interestingly enough though, it is not easy to check if a candidate $g$ is a valid dcd of $f$.

\begin{proposition}\label{prop:dcd.NPhard}
	Given two $n$-variate polynomials $f$ and $g$ of degree 4, with $f\neq g$, it is strongly NP-hard \footnote{For a
		strongly NP-hard problem, even a pseudo-polynomial time algorithm cannot exist unless P=NP \cite{GareyJohnson_Book}.} to determine whether $g$ is a dcd of $f$.\footnote{If we do not add the condition on the input that $f\neq g$, the problem would again be NP-hard (in fact, this is even easier to prove). However, we believe that in any interesting instance of this question, one would have $f\neq g$.}
\end{proposition}
\begin{proof}
	We will show this via a reduction from the problem of testing nonnegativity of biquadratic forms, which is already known to be strongly NP-hard \cite{Ling_et_al_Biquadratic}, \cite{NPhard_Convexity_MathProg}. A biquadratic form $b(x,y)$ in the variables $x=(x_1,\ldots,x_n)^T$ and $y=(y_1,\ldots,y_m)^T$ is a quartic form that can be written as $$b(x;y)=\sum_{i\leq j, k\leq l} a_{ijkl} x_ix_jy_ky_l.$$
	Given a biquadratic form $b(x;y)$, define the $n \times n$ polynomial matrix $C(x,y)$ by setting $[C(x,y)]_{ij} \mathrel{\mathop{:}}=\frac{\partial b(x;y)}{\partial x_i \partial y_j},$  and let $\gamma$ be the largest coefficient in absolute value of any monomial present in some entry of $C(x,y)$. Moreover, we define 
	$$r(x;y)\mathrel{\mathop{:}}=\frac{n^2 \gamma}{2}\sum_{i=1}^n x_i^4+\sum_{i=1}^{n} y_i^4+\sum_{1\leq i<j\leq n} x_i^2x_j^2+\sum_{1\leq i<j \leq n} y_i^2y_j^2.$$ 
	It is proven in \cite[Theorem 3.2.]{NPhard_Convexity_MathProg} that $b(x;y)$ is nonnegative if and only if $$q(x,y) \mathrel{\mathop{:}}=b(x;y)+r(x,y)$$ is convex. We now give our reduction. Given a biquadratic form $b(x;y)$, we take 
	$g=q(x,y)+r(x,y)$ and $f=r(x,y)$. If $b(x;y)$ is nonnegative, from the theorem quoted above, $g-f=q$ is convex. Furthermore, it is straightforward to establish that $r(x,y)$ is convex, which implies that $g$ is also convex. This means that $g$ is a dcd of $f$. If $b(x;y)$ is not nonnegative, then we know that $q(x,y)$ is not convex. This implies that $g-f$ is not convex, and so $g$ cannot be a dcd of $f$.  
\end{proof}

Unlike the quartic case, it is worth noting that in the quadratic case, it is easy to test whether a polynomial $g(x)=x^TGx$ is a dcd of $f(x)=x^TFx$. Indeed, this amounts to testing whether $F \succeq 0$ and $G -F \succeq 0$ which can be done in $O(n^3)$ time. 

As mentioned earlier, there is not only one dcd for a given polynomial $f$, but an infinite number. Indeed, if $f=g-h$ with $g$ and $h$ convex then any convex polynomial $p$ generates a new dcd $f=(g+p)-(h+p)$. It is natural then to investigate if some dcds are better than others, e.g., for use in the convex-concave procedure.

Recall that the main idea of CCP is to upperbound the non-convex function $f=g -h$ by a convex function $f^k$. These convex functions are obtained by linearizing $h$ around the optimal solution of the previous iteration. Hence, a reasonable way of choosing a good dcd would be to look for dcds of $f$ that minimize the curvature of $h$ around a point. Two natural formulations of this problem are given below. The first one attempts to minimize the \emph{average}\footnote{ Note that $\mbox{Tr } H_h(\bar{x})$ (resp. $\lambda_{\max} H_h(\bar{x})$) gives the average (resp. maximum) of $y^TH_h(\bar{x})y$ over $\{y~|~||y||=1\}$.} curvature of $h$ at a point $\bar{x}$ over all directions:
\begin{equation} \label{eq:trace.point}
\begin{aligned}
&\min_{g} \mbox{Tr } H_h(\bar{x})\\
&\text{s.t. } f=g-h, g,h \text{ convex}.
\end{aligned}
\end{equation}
The second one attempts to minimize the \emph{worst-case}\footnotemark[\value{footnote}] curvature of $h$ at a point $\bar{x}$ over all directions:
\begin{equation} \label{eq:lambda.max.point}
\begin{aligned}
&\min_{g} \lambda_{\max} H_h(\bar{x})\\
&\text{s.t. } f-g-h, g, h \text{ convex}.
\end{aligned}
\end{equation}
A few numerical experiments using these objective functions will be presented in Section \ref{subsec:scalibility}. 

Another popular notion that appears in the literature and that also relates to finding dcds with minimal curvature is that of \emph{undominated dcds.} These were studied in depth by Bomze and Locatelli in the quadratic case \cite{Bomze}. We extend their definition to general polynomials here. 

\begin{definition}
	Let $g$ be a dcd of $f$. A dcd $g'$ of $f$ is said to dominate $g$ if $g-g'$ is convex and nonaffine.
	A dcd $g$ of $f$ is \emph{undominated} if no dcd of $f$ dominates $g$.
\end{definition}
Arguments for chosing undominated dcds can be found in \cite{Bomze}, \cite[Section 3]{Dur}. One motivation that is relevant to CCP appears in Proposition \ref{th:CCPundom}\footnote{A variant of this proposition in the quadratic case appears in \cite[Proposition 12]{Bomze}.}. Essentially, the proposition shows that if we were to start at some initial point and apply one iteration of CCP, the iterate obtained using a dc decomposition $g$ would always beat an iterate obtained using a dcd dominated by $g$.
%
%The next theorem provides intuition as to why it is relevant to use undominated dcds for CCP. Indeed, when starting with the same iterate, the next iterate obtained using an undominated dcd $g$ always gives tighter bounds than the next iterate obtained using a dcd dominated by $g$.
%
%%In the next theorem, we consider for simplicity the unconstrained version of CCP: given $x_0$ and a dcd $g$ of $f$, the goal is to minimize $f$ over $\mathbb{R}^n$. This is done as in Algorithm \ref{alg:CCP}. At each iteration, we construct the convexified version of $f$, $$f_g^{k}(x)=g(x)-(h(x_k)+\nabla h(x_k)^T(x-x_k)$$ and take $x_{k+1}\mathrel{\mathop{:}}=\text{argmin} f_g^{k}(x).$
\begin{proposition} \label{th:CCPundom}
	Let $g$ and $g'$ be two dcds of $f$. Define the convex functions $h\mathrel{\mathop{:}}=g-f$ and $h'\mathrel{\mathop{:}}=g'-f$, and assume that $g'$ dominates $g$. For a point $x_0$ in $\mathbb{R}^n$, define the convexified versions of $f$
	\begin{align*}
	f_g(x)\mathrel{\mathop{:}}=g(x)-(h(x_0)+\nabla h(x_0)^T(x-x_0)),\\
	f_{g'}(x)\mathrel{\mathop{:}}=g'(x)-(h'(x_0)+\nabla h'(x_0)^T(x-x_0)).
	\end{align*}  
	Then, we have $$f_g'(x) \leq f_{g}(x), \forall x.$$
\end{proposition}

\begin{proof}
	As $g'$ dominates $g$, there exists a nonaffine convex polynomial $c$ such that $c=g-g'$.
	We then have $g'=g-c$ and $h'=h-c$, and
	\begin{equation*}
	\begin{aligned}
	f_g'(x)&=g(x)-c(x)-h(x_0)+c(x_0)-\nabla h(x_0)^T(x-x_0)+\nabla c(x_0)^T(x-x_0)\\
	&= f_{g}(x)-(c(x)-c(x_0)-\nabla c(x_0)^T (x-x_0)).\\
	\end{aligned}
	\end{equation*}
	The first order characterization of convexity of $c$ then gives us  
	\begin{align*}
	f_{g'}(x) \leq f_{g}(x), \forall x. 
	\end{align*}
\end{proof}

In the quadratic case, it turns out that an optimal solution to (\ref{eq:trace.point}) is an undominated dcd \cite{Bomze}. A solution given by (\ref{eq:lambda.max.point}) on the other hand is not necessarily undominated. Consider the quadratic function $$f(x)=8x_1^2-2x_2^2-8x_3^2$$ and assume that we want to decompose it using (\ref{eq:lambda.max.point}). An optimal solution is given by $g^*(x)=8x_1^2+6x_2^2$ and $h^*(x)=8x_2^2+8x_3^2$ with $\lambda_{\max}H_h=8.$ This is clearly dominated by $g'(x)=8x_1^2$ as $g^*(x)-g'(x)=6x_2^2$ which is convex.

When the degree is higher than 2, it is no longer true however that solving (\ref{eq:trace.point}) returns an undominated dcd. Consider for example the degree-4 polynomial $$f(x)=x^{12}-x^{10}+x^{6}-x^{4}.$$
A solution to (\ref{eq:trace.point}) with $\bar{x}=0$ is given by $g(x)=x^{12}+x^{6}$ and $h(x)=x^{10}+x^{4}$ (as $\mbox{Tr}H_h(0)=0$). This is dominated by the dcd $g(x)=x^{12}-x^8+x^6$ and $h(x)=x^{10}-x^8+x^4$ as $g-g'=x^8$ is clearly convex.

It is unclear at this point how one can obtain an undominated dcd for higher degree polynomials, or even if one exists. In the next theorem, we show that such a dcd always exists and provide an optimization problem whose optimal solution(s) will always be undominated dcds. This optimization problem involves the integral of a polynomial over a sphere which conveniently turns out to be an explicit linear expression in its coefficients.

\begin{proposition}[\cite{intSphere}]
	Let $S^{n-1}$ denote the unit sphere in $\mathbb{R}^n$. For a monomial $x_1^{\alpha_1}\ldots x_n^{\alpha_n}$, define $\beta_j\mathrel{\mathop{:}}=\frac{1}{2} (\alpha_j+1)$.
	Then $$\int_{S^{n-1}} x_1^{\alpha_1}\ldots x_n^{\alpha_n} d\sigma = \begin{cases} 0 &\text{ if some $\alpha_j$ is odd}, \\ \frac{2\Gamma(\beta_1) \ldots \Gamma(\beta_n)}{\Gamma(\beta_1+\ldots+\beta_n)} &\text{ if all $\alpha_j$ are even,} \end{cases}$$
\end{proposition}
where $\Gamma$ denotes the gamma function, and $\sigma$ is the rotation invariant probability measure on $S^{n-1}.$
%Notice that Theorem \ref{th:CCPundom} only holds for the next iterate when the previous iterate is the same. This guarantees better performance for the first iterate of CCP when using undominated dcds and a fixed initial point. No such guarantee exists however for subsequent iterates as the starting point is not the same.

%\subsection{Existence of undominated dcds}

%It is known that any polynomial $f$ admits a dcd. We will now show existence of an undominated dcd for any polynomial $f$. Recall that a polynomial of degree $2d$ and in $n$ variables is a linear combination of monomials in $x=(x_1,\ldots,x_n)$ of degree up to $2d$. The space of polynomials of degree $2d$ and in $n$ variables is denoted by $\mathbb{R}[x].$  A polynomial of degree $2d$ and in $n$ variables is said to be \emph{homogeneous} if it is a linear combination of monomials in $x$ of degree \emph{exactly} $2d$. Furthermore, if $f$ is a polynomial, its Hessian will be denoted by $H_f(x)$. 

\begin{theorem}\label{thm:undom.dcd}
	Let $f \in \tilde{\mathcal{H}}_{n,2d}.$  Consider the optimization problem
	\begin{equation} \label{eq:undom.dcd}
	\begin{aligned}
	\min_{g \in \tilde{\mathcal{H}}_{n,2d}} &\frac{1}{\mathcal{A}_n}\int_{S^{n-1}} \mbox{Tr } H_g d \sigma\\
	\text{s.t. } &g \text{ convex}, \\
	&g-f \text{ convex},
	\end{aligned}
	\end{equation}
	where $\mathcal{A}_n=\frac{2\pi^{n/2}}{\Gamma(n/2)}$ is a normalization constant which equals the area of $S^{n-1}$. Then, an optimal solution to (\ref{eq:undom.dcd}) exists and any optimal solution is an undominated dcd of $f$.
\end{theorem}
Note that problem (\ref{eq:undom.dcd}) is exactly equivalent to (\ref{eq:trace.point}) in the case where $n=2$ and so can be seen as a generalization of the quadratic case.
\begin{proof}
	We first show that an optimal solution to (\ref{eq:undom.dcd}) exists. As any polynomial $f$ admits a dcd, (\ref{eq:undom.dcd}) is feasible. Let $\tilde{g}$ be a dcd of $f$ and define $\gamma \mathrel{\mathop{:}}= \int_{S^{n-1}} \mbox{Tr }H_{\tilde{g}} d\sigma.$
	Consider the optimization problem given by (\ref{eq:undom.dcd}) with the additional constraints:
	\begin{equation} \label{eq:undom.dcd.new.constraint}
	\begin{aligned}
	\min_{g \in \tilde{\mathcal{H}}_{n,2d}} &\frac{1}{\mathcal{A}_n}\int_{S^{n-1}} \mbox{Tr } H_g d \sigma\\
	\text{s.t. } &g \text{ convex and with no affine terms} \\
	&g-f \text{ convex},\\
	&\int_{S^{n-1}} \mbox{Tr } H_{g} d\sigma \leq \gamma.
	\end{aligned}
	\end{equation}
	Notice that any optimal solution to (\ref{eq:undom.dcd.new.constraint}) is an optimal solution to (\ref{eq:undom.dcd}). Hence, it suffices to show that (\ref{eq:undom.dcd.new.constraint}) has an optimal solution. Let $\mathcal{U}$ denote the feasible set of (\ref{eq:undom.dcd.new.constraint}). Evidently, the set $\mathcal{U}$ is closed and $g \rightarrow \int_{S^{n-1}} \mbox{Tr }H_g d \sigma$ is continuous. If we also show that $\mathcal{U}$ is bounded, we will know that the optimal solution to (\ref{eq:undom.dcd.new.constraint}) is achieved.
	To see this, assume that $\mathcal{U}$ is unbounded. Then for any $\beta$,  there exists a coefficient $c_g$ of some $g \in \mathcal{U}$ that is larger than $\beta$. By absence of affine terms in $g$, $c_g$ features in an entry of $H_g$ as the coefficient of a nonzero monomial. Take $\bar{x} \in S^{n-1}$ such that this monomial evaluated at $\bar{x}$ is nonzero: this entails that at least one entry of $H_g(\bar{x})$ can get arbitrarily large. However, since $g \rightarrow \mbox{Tr }H_g$ is continuous and $\int_{S^{n-1}} \mbox{Tr }H_g d\sigma \leq \gamma$, $\exists \bar{\gamma}$ such that $\mbox{Tr }H_g(x) \leq \bar{\gamma}$, $\forall x \in S^{n-1}$. This, combined with the fact that $H_g(x) \succeq 0~ \forall x$, implies that $||H_g(x)|| \leq \bar{\gamma}, ~\forall x \in S^{n-1}$, which contradicts the fact that an entry of $H_g(\bar{x})$ can get arbitrarily large.

	%The constraint space is then compact. Indeed, it is closed as the set of convex functions is closed. The boundednesss is slightly harder to see. Each undetermined coefficient of $g$ features in its Hessian at some location $(i,j)$. If $i=j$, then the coefficient is obviously bounded as $\int_{S^{n-1}} \mbox{Tr }H_g(x)$ is linear in the coefficients of $g$ that are on the diagonal of the Hessian \cite{intSphere}. If $i \neq j$ consider the principal minor obtained from $(i,j)$. Pick $x$ in such a way that the off-diagonal is nonzero. The positive semidefiniteness of the Hessian for any $x$ implies that the coefficient is bounded by a linear combination of the coefficients on the diagonal, which are themselves bounded as seen previously. As $g \rightarrow \int_{S^{n-1}} \mbox{Tr }H_g(x)$ is continuous, the minimum is achieved, which entails that the minimum of (\ref{eq:undom.dcd}) is achieved.
	
	We now show that if $g^*$ is any optimal solution to (\ref{eq:undom.dcd}), then $g^*$ is an undominated dcd of $f$. Suppose that this is not the case. Then, there exists a dcd $g'$ of $f$ such that $g^*-g'$ is nonaffine and convex. As $g'$ is a dcd of $f$, $g'$ is feasible for (\ref{eq:undom.dcd}). The fact that $g^*-g'$ is nonaffine and convex implies that $$\int_{S^{n-1}} \mbox{Tr } H_{g^*-g'} d\sigma >0 \Leftrightarrow \int_{S^{n-1}} \mbox{Tr} H_{g^*} d \sigma >\int_{S^{n-1}} \mbox{Tr } H_{g'} d\sigma,$$  which contradicts the assumption that $g^*$ is optimal to (\ref{eq:undom.dcd}).  
\end{proof}

Although optimization problem (\ref{eq:undom.dcd}) is guaranteed to produce an undominated dcd, we show that unfortunately it is intractable to solve.

\begin{proposition}\label{prop:undominated.NPhard}
	Given an n-variate polynomial $f$ of degree 4 with rational coefficients, and a rational number $k$, it is strongly NP-hard to decide whether there exists a feasible solution to (\ref{eq:undom.dcd}) with objective value $\leq k$.
	%\begin{align} \label{eq:condition.g}
	%k\mathrel{\mathop{:}}=\frac{1}{\mathcal{A}_S^n} \int_{S^{n-1}} \mbox{Tr } H_f(x),
	%\frac{1}{\mathcal{A}_S^n} \int_{S^{n-1}} \mbox{Tr } H_g(x)\leq k.
	%\end{align}
\end{proposition}
\begin{proof}
	We give a reduction from the problem of deciding convexity of quartic polynomials. Let $q$ be a quartic polynomial. We take $f=q$ and $k=\frac{1}{\mathcal{A}_n} \int_{S^{n-1}} \mbox{Tr } H_q(x)$. If $q$ is convex, then $g=q$ is trivially a dcd of $f$ and
	\begin{align} \label{eq:condition.g}
	\frac{1}{\mathcal{A}_n} \int_{S^{n-1}} \mbox{Tr } H_g d\sigma \leq k.
	\end{align} 
	If $q$ is not convex, assume that there exists a feasible solution $g$ for (\ref{eq:undom.dcd}) that satisfies (\ref{eq:condition.g}). From (\ref{eq:condition.g}) we have
	\begin{align} \label{eq:ineq1}
	\int_{S^{n-1}} \mbox{Tr } H_g(x) \leq \int_{S^{n-1}} \mbox{Tr } H_f d\sigma \Leftrightarrow \int_{S^{n-1}} \mbox{Tr } H_{f-g} d \sigma \geq 0.
	\end{align}
	But from (\ref{eq:undom.dcd}), as $g-f$ is convex, $\int_{S^{n-1}} \mbox{Tr }H_{g-f} d\sigma \geq 0.$
	Together with (\ref{eq:ineq1}), this implies that $$\int_{S^{n-1}} \mbox{Tr }H_{g-f}d \sigma=0$$ which in turn {implies that $H_{g-f}(x)=H_g(x)-H_f(x)=0$. To see this, note that $\mbox{Tr}(H_{g-f})$ is a nonnegative polynomial which must be identically equal to $0$ since its integral over the sphere is $0$. As $H_{g-f}(x)\succeq 0,\forall x$, we get that $H_{g-f}=0.$ Thus, $H_g(x)=H_f(x), ~\forall x$, which is not possible as $g$ is convex and $f$ is not.} 
\end{proof}
We remark that solving (\ref{eq:undom.dcd}) in the quadratic case (i.e., $2d=2$) is simply a semidefinite program.

\section{Alegbraic relaxations and more tractable subsets of the set of convex polynomials} \label{sec:ConvRelax}

We have just seen in the previous section that for polynomials with degree as low as four, some basic tasks related to dc decomposition are computationally intractable. In this section, we identify three subsets of the set of convex polynomials that lend themselves to polynomial-time algorithms. These are the sets of \emph{sos-convex, sdsos-convex}, and \emph{dsos-convex} polynomials, which will respectively lead to semidefinite, second order cone, and linear programs. The latter two concepts are to our knowledge new and are meant to serve as more scalable alternatives to sos-convexity. All three concepts certify convexity of polynomials via explicit algebraic identities, which is the reason why we refer to them as algebraic relaxations.

%The optimization problems we introduced in the previous section are not computationally tractable as we do not know how to optimize over the set of convex functions. In this section, we appeal to the concept of \emph{sos-convex} polynomials, a well-known subset of the set of convex polynomials, to get around this difficulty. Indeed, optimizing over the set of sos-convex polynomials amounts to solving a semidefinite program, which can be done efficiently. We also introduce the concepts of \emph{dsos-convex} and \emph{sdsos-convex} polynomials, which are other tractable alternatives to convexity based on linear programmming and second order cone programming. Their main advantage when compared to sos-convexity is their scalibility when the degree or the number of variables of the polynomial gets high.

\subsection{DSOS-convexity, SDSOS-convexity, SOS-convexity}

To present these three notions we need to introduce some notation and briefly review the concepts of sos, dsos, and sdsos polynomials.
%In this subsection, we will recall the definition of an sos-convex polynomial and define dsos-convex and sdsos-convex polynomials. To do this, we introduce some notation and briefly review the concepts of sos, dsos, sdsos polynomials.

%Recall that a symmetric matrix $A$ is \emph{positive semidefinite} (psd) if $x^TAx\geq 0$ for all $x\in\mathbb{R}^n$; this will be denoted by the standard notation $A\succeq 0$.
%, and our notation for the set of $n\times n$ psd matrices is $P_n$.
% Similarly, a polynomial $p$ in $n$ variables is \emph{nonnegative} (or positive semidefinite) if $p(x)\geq 0$ for all $x\in\mathbb{R}^n$.
We denote the set of nonnegative polynomials (resp. forms) in $n$ variables and of degree $d$ by $\tilde{PSD}_{n,d}$ (resp. $PSD_{n,d}$).
% a form $p$ in $n$ variables is \emph{nonnegative} (or positive semidefinite) if $p(x)\geq 0$ for all $x\in\mathbb{R}^n$, or equivalently for all $x$ on the unit sphere in $\mathbb{R}^n$. The set of nonnegative forms in $n$ variables and degree $d$ is denoted by $PSD_{n,d}$. 
A polynomial $p$ is a \emph{sum of squares} (sos) if it can be written as $p(x)=\sum_{i=1}^r q_i^2(x)$ for some polynomials $q_1,\ldots,q_r$. The set of sos polynomials (resp. forms) in $n$ variables and of degree $d$ is denoted by $\tilde{SOS}_{n,d}$ (resp. $SOS_{n,d}$).
%forms in $n$ variables and degree $d$ is denoted by $SOS_{n,d}$.
We have the obvious inclusion $\tilde{SOS}_{n,d}\subseteq \tilde{PSD}_{n,d}$ (resp. $SOS_{n,d}\subseteq PSD_{n,d}$), which is strict unless  $d=2$, or $n=1$, or $(n,d)=(2,4)$ (resp. $d=2$, or $n=2$, or $(n,d)=(3,4)$)~\cite{Hilbert_1888},~\cite{Reznick}.

Let $\tilde{z}_{n,d}(x)$ (resp. $z_{n,d}(x)$) denote the vector of all monomials in $x=(x_1,\ldots,x_n)$ of degree up to (resp. exactly) $d$; the length of this vector is $\binom{n+d}{d}$ (resp. $\binom{n+d-1}{d}$). It is well known that a polynomial (resp. form) $p$ of degree $2d$ is sos if and only if it can be written as $p(x)=\tilde{z}^T_{n,d}(x)Q\tilde{z}_{n,d}(x)$ (resp. $p(x)=z_{n,d}^T(x)Qz_{n,d}(x)$), for some psd matrix $Q$~\cite{sdprelax},~\cite{PhD:Parrilo}. The matrix $Q$ is generally called the Gram matrix of $p$. An SOS optimization problem is the problem of minimizing a linear function over the intersection of the convex cone $SOS_{n,d}$ with an affine subspace. The previous statement implies that SOS optimization problems can be cast as semidefinite programs.

We now define dsos and sdsos polynomials, which were recently proposed by Ahmadi and Majumdar~\cite{isos_journal},~\cite{dsos_ciss14} as more tractable subsets of sos polynomials. When working with dc decompositions of $n$-variate polynomials, we will end up needing to impose sum of squares conditions on polynomials that have $2n$ variables (see Definition~\ref{def:alternatives.convexity}). While in theory the SDPs arising from sos conditions are of polynomial size, in practice we rather quickly face a scalability challenge. For this reason, we also consider the class of dsos and sdsos polynomials, which while more restrictive than sos polynomials, are considerably more tractable. For example, Table \ref{tab:dsos} in Section \ref{subsec:scalibility} shows that when $n=14$, dc decompositions using these concepts are about 250 times faster than an sos-based approach. At $n=18$ variables, we are unable to run the sos-based approach on our machine. With this motivation in mind, let us start by recalling some concepts from linear algebra.

%Unfortunately, this quickly limits the size of problems that we can solve by semidefinite programming 

%As we will see, optimizing over these two sets amounts to solving a linear program or a second order cone program, which means that DSOS and SDSOS programming scale better with $n$ and $d$ than SOS programming does. To make the link between dsos/sdsos polynomials and linear/second order cone programming more apparent, we will define dsos and sdsos polynomials via their Gram matrices in the following. This is different to what is done in \cite{isos_journal} where they are defined as subcases of sos polynomials with a very specific structure.

\begin{definition}\label{def:dd.sdd}
	A symmetric matrix $M$ is said to be \emph{diagonally dominant (dd)} if $m_{ii} \geq \sum_{j \neq i}|m_{ij}|$ for all $i$, and \emph{strictly diagonally dominant} if  $m_{ii} > \sum_{j \neq i}|m_{ij}|$ for all $i$. 
	We say that $M$ is \emph{scaled diagonally dominant (sdd)} if there exists a diagonal matrix $D,$ with positive diagonal entries, such that $DAD$ is dd.
\end{definition}
We have the following implications from Gershgorin's circle theorem
\begin{align}\label{eq:implic.matrices}
\text{M } dd \Rightarrow \text{M } sdd \Rightarrow \text{M } psd.
\end{align}
Furthermore, notice that requiring $M$ to be dd can be encoded via a linear program (LP) as the constraints are linear inequalities in the coefficients of $M$. Requiring that $M$ be sdd can be encoded via a second order cone program (SOCP). This follows from the fact that $M$ is sdd if and only if $$M=\sum_{i<j} M^{ij}_{2\times 2},$$ where each $M^{ij}_{2 \times 2}$ is an $n \times n$ symmetric matrix with zeros everywhere except four entries $M_{ii},M_{ij},M_{ji}, M_{jj}$ which must make the $2 \times 2$ matrix $\begin{pmatrix} M_{ii} & M_{ij} \\ M_{ji} & M_{jj} \end{pmatrix}$ symmetric positive semidefinite \cite{isos_journal}. These constraints are \emph{rotated quadratic cone constraints} and can be imposed via SOCP \cite{alizadeh}.

\begin{definition}[\cite{isos_journal}]
	A polynomial $p \in \tilde{\mathcal{H}}_{n,2d}$ is said to be 
	\begin{itemize}
		\item \emph{diagonally-dominant-sum-of-squares (dsos)} if it admits a representation $p(x)=\tilde{z}^T_{n,d}(x)Q\tilde{z}_{n,d}(x)$, where $Q$ is a dd matrix.
		\item \emph{scaled-diagonally-dominant-sum-of-squares (sdsos)} it it admits a representation $p(x)=\tilde{z}^T_{n,d}(x)Q\tilde{z}_{n,d}(x),$ where $Q$ is an sdd matrix.
	\end{itemize}
	Identical conditions involving $z_{n,d}$ instead of $\tilde{z}_{n,d}$ define the sets of dsos and sdsos forms. 
\end{definition}

The following implications are again straightforward:
\begin{align}\label{eq:implicsdsos}
p(x) \text{ dsos} \Rightarrow p(x) \text{ sdsos} \Rightarrow p(x) \text{ sos} \Rightarrow p(x) \text{ nonnegative}.
\end{align}
Given the fact that our Gram matrices and  polynomials are related to each other via linear equalities, it should be clear that optimizing over the set of dsos (resp. sdsos, sos) polynomials is an LP (resp. SOCP, SDP). 

Let us now get back to convexity.

\begin{definition} \label{def:alternatives.convexity}
	Let $y=(y_1,\ldots,y_n)^T$ be a vector of variables. A polynomial $p\mathrel{\mathop:}=p(x)$ is said to be
	\begin{itemize}
		\item \emph{dsos-convex} if $y^TH_p(x)y$ is dsos {(as a polynomial in $x$ and $y$)}.
		\item \emph{sdsos-convex} if $y^TH_p(x)y$ is sdsos {(as a polynomial in $x$ and $y$)}.
		\item \emph{sos-convex} if $y^TH_p(x)y$ is sos {(as a polynomial in $x$ and $y$)}.\footnote{The notion of sos-convexity has already appeared in the study of semidefinite representability of convex sets~\cite{helton2010} and in  applications such as shaped-constrained regression in statistics~\cite{Magnani}.}
	\end{itemize}

	We denote the set of dsos-convex (resp. sdsos-convex, sos-convex, convex) forms in $\mathcal{H}_{n,2d}$ by $\Sigma_DC_{n,2d}$ (resp. $\Sigma_SC_{n,2d}$, $\Sigma C_{n,2d}$, $C_{n,2d}$). Similarly, $\tilde{\Sigma}_DC_{n,2d}$ (resp. $\tilde{\Sigma}_SC_{n,2d}$, $\tilde{\Sigma} C_{n,2d}$, $\tilde{C}_{n,2d}$) denote the set of dsos-convex (resp. sdsos-convex, sos-convex, convex) polynomials in $\tilde{\mathcal{H}}_{n,2d}$.
\end{definition}
The following inclusions 
\begin{align}\label{eq:inclusion.cones}
\Sigma_DC_{n,2d} \subseteq \Sigma_SC_{n,2d} \subseteq \Sigma C_{n,2d} \subseteq C_{n,2d}
\end{align}
are a direct consequence of (\ref{eq:implicsdsos}) and the second-order necessary and sufficient condition for convexity which reads $$ p(x) \text{ is convex } \Leftrightarrow H_p(x) \succeq 0, \forall x \in \mathbb{R}^n \Leftrightarrow y^TH_p(x)y \geq 0, \forall x,y \in \mathbb{R}^n.$$ 
Optimizing over $\Sigma_DC_{n,2d}$ (resp. $\Sigma_S C_{n,2d}$, $\Sigma C_{n,2d}$) is an LP (resp. SOCP, SDP). The same statements are true for $\tilde{\Sigma}_DC_{n,2d}$, $\tilde{\Sigma}_SC_{n,2d}$ and $\tilde{\Sigma} C_{n,2d}$.

Let us draw these sets for a parametric family of polynomials
\begin{align} \label{eq:parampoly}
p(x_1,x_2)=2x_1^4+2x_2^4+ax_1^3x_2+bx_1^2x_2^2+cx_1x_2^3.
\end{align} Here, $a,b$ and $c$ are parameters. It is known that for bivariate quartics, all convex polynomials are sos-convex; i.e., $\Sigma C_{2,4}=C_{2,4}.$\footnote{In general, constructing polynomials that are convex but not sos-convex seems to be a nontrivial task~\cite{AAA_PP_not_sos_convex_journal}. A complete characterization of the dimensions and degrees for which convexity and sos-convexity are equivalent is given in~\cite{ahmadi2013complete}.}  To obtain Figure \ref{fig:cones}, we fix $c$ to some value and then plot the values of $a$ and $b$ for which $p(x_1,x_2)$ is s/d/sos-convex. As we can see, the quality of the inner approximation of the set of convex polynomials by the sets of dsos/sdsos-convex polynomials can be very good (e.g., $c=0$) or less so (e.g., $c=1$).

\begin{figure}[h!]
	\centering
	\includegraphics[scale=0.38]{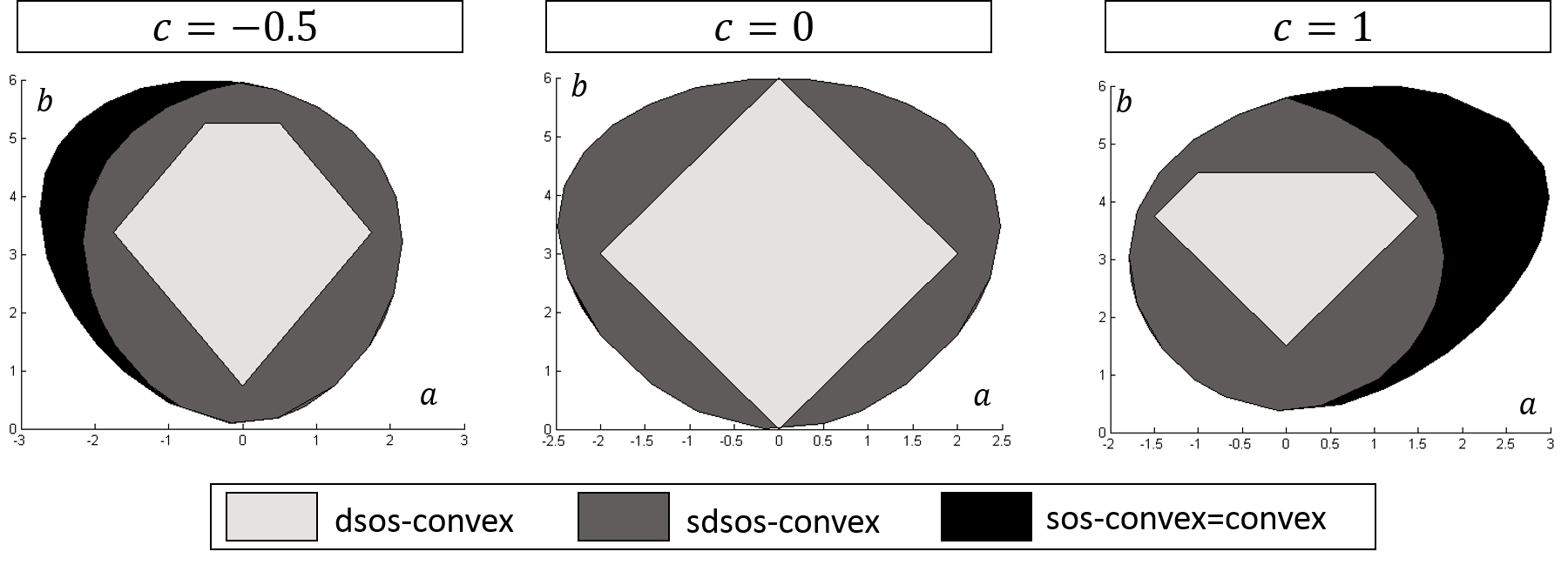}
	\caption{The sets $\Sigma_DC_{n,2d}, \Sigma_SC_{n,2d}$ and $\Sigma C_{n,2d}$ for the parametric family of polynomials in (\ref{eq:parampoly})}
	\label{fig:cones}
	%\vspace{-10mm}
\end{figure}

\subsection{Existence of difference of s/d/sos-convex decompositions of polynomials} \label{subsec:diffconv}
%
%
%[**** Place this notation somewhere appropriate
%
%If $y=(y_1,\ldots,y_n)$ is a vector of variables of length $n$, we define $$w_{n,d}(x,y) \mathrel{\mathop{:}}= y \cdot z_{n,d}(x),$$
%where $y \cdot z_{n,d}(x)=(y_1 z_{n,d}(x), \ldots, y_n  z_{n,d}(x))^T.$ 
%%$y \cdot z_{n,d}(x)=(y_1 \cdot z_{n,d}(x), \ldots, y_n \cdot z_{n,d}(x))^T$ 
%%
%%and $y_k \cdot z_{n,d}(x)$ is simply the vector $$y_k\cdot z_{n,d}(x) \mathrel{\mathop{:}}=(y_k x_1^d, y_k x_1^{d-1}x_2, \ldots, y_kx_n^d)^T.$$
%%
%The length of $w_{n,d}(x,y)$ is then $n \times \binom{n+d-1}{d}$. Similarly, we define $$\tilde{w}_{n,d}(x,y)\mathrel{\mathop{:}}=y \cdot \tilde{z}_{n,d}(x).$$
%
%
%
%****]

The reason we introduced the notions of s/d/sos-convexity is that in our optimization problems for finding dcds, we would like to replace the condition $$f=g-h,~ g,h \text{ convex}$$ with the computationally tractable condition $$f=g-h, ~g,h \text{ s/d/sos-convex.}$$ 
The first question that needs to be addressed is whether for any polynomial such a decomposition exists. In this section, we prove that this is indeed the case. This in particular implies that a dcd can be found efficiently. 

%We have seen previously that for any polynomial $f$, there exist $g$ and $h$ convex such that $f=g-h$. We will see now that this result still holds in the case where we require $g$ and $h$ to be s/d/sos-convex. This ensures that a dcd can be efficiently computed for any polynomial. Furthermore, this enables us to optimize over a large subset of the set of dcds.
% This will ensure that (\ref{eq:undom.dcd}) is always feasible. 

We start by proving a lemma about cones. 

\begin{lemma}\label{lemma:cones}
	Consider a vector space $E$ and a full-dimensional cone $K \subseteq E$. Then, any $v \in E$ can be written as $v=k_1-k_2,$ where $k_1,k_2 \in K.$
\end{lemma}

\begin{proof}
	Let $v \in E$. If $v \in K$, then we take $k_1=v$ and $k_2=0$. Assume now that $v \notin K$ and let $k$ be any element in the interior of the cone $K$. As $k \in int(K)$, there exists $0<\alpha<1$ such that $k'\mathrel{\mathop{:}}=(1-\alpha)v+\alpha k \in K.$ Rewriting the previous equation, we obtain $$v=\frac{1}{1-\alpha}k'-\frac{\alpha}{1-\alpha}k.$$ By taking $k_1\mathrel{\mathop{:}}=\frac{1}{1-\alpha}k'$ and $k_2\mathrel{\mathop{:}}=\frac{\alpha}{1-\alpha}k$, we observe that $v=k_1-k_2$ and $k_1,k_2 \in K$.  
\end{proof}

The following theorem is the main result of the section. 
\begin{theorem}\label{thm:diff.dsos}
	Any polynomal $p \in \tilde{\mathcal{H}}_{n,2d}$ can be written as the difference of two dsos-convex polynomials in $\tilde{\mathcal{H}}_{n,2d}.$
\end{theorem}

\begin{corollary}\label{cor:SDSOSetc}
	Any polynomial $p \in \tilde{\mathcal{H}}_{n,2d}$ can be written as the difference of two sdsos-convex, sos-convex, or convex polynomials in $\tilde{\mathcal{H}}_{n,2d}$.
\end{corollary}
\begin{proof}
	This is straightforward from the inclusions $$\tilde{\Sigma}_DC_{n,2d} \subseteq \tilde{\Sigma}_SC_{n,2d} \subseteq \tilde{\Sigma}C_{n,2d} \subseteq \tilde{C}_{n,2d}.\ \  $$
\end{proof}
In view of Lemma \ref{lemma:cones}, it suffices to show that $\tilde{\Sigma}_DC_{n,2d}$ is full dimensional in the vector space $\tilde{\mathcal{H}}_{n,2d}$ to prove Theorem \ref{thm:diff.dsos}. We do this by constructing a polynomial in $int(\tilde{\Sigma}_DC_{n,2d})$ for any $n,d$.

%We wish to prove that any polynomial can be written as a difference of two dsos-convex polynomials. Indeed, the results on sdsos-convexity and sos-convexity would follow by the inclusions $\tilde{\Sigma}_DC_{n,2d} \subseteq \tilde{\Sigma}_SC_{n,2d} \subseteq \tilde{\Sigma}C_{n,2d}$. Given Lemma \ref{lemma:cones}, it only remains to show that the cone $\tilde{\Sigma}_D C_{n,2d}$ is full dimensional in the vector space $\tilde{\mathcal{H}}_{n,2d}$.  This can be obtained by proving existence of a polynomial $p \in \tilde{\mathcal{H}}_{n,2d}$ that is in the interior of $\tilde{\Sigma}_D C_{n,2d}$. We briefly review some notation on monomials to do this. 

Recall that $z_{n,d}$ (resp. $\tilde{z}_{n,d}$) denotes the vector of all monomials in $x=(x_1,\ldots,x_n)$ of degree exactly (resp. up to) $d$. If $y=(y_1,\ldots,y_n)$ is a vector of variables of length $n$, we define $$w_{n,d}(x,y) \mathrel{\mathop{:}}= y \cdot z_{n,d}(x),$$
where $y \cdot z_{n,d}(x)=(y_1 z_{n,d}(x), \ldots, y_n  z_{n,d}(x))^T.$ 
%$y \cdot z_{n,d}(x)=(y_1 \cdot z_{n,d}(x), \ldots, y_n \cdot z_{n,d}(x))^T$ 
%
%and $y_k \cdot z_{n,d}(x)$ is simply the vector $$y_k\cdot z_{n,d}(x) \mathrel{\mathop{:}}=(y_k x_1^d, y_k x_1^{d-1}x_2, \ldots, y_kx_n^d)^T,$$
%
%The length of $w_{n,d}(x,y)$ is then $n \times \binom{n+d-1}{d}$.
Analogously, we define $$\tilde{w}_{n,d}(x,y)\mathrel{\mathop{:}}=y \cdot \tilde{z}_{n,d}(x).$$

\begin{theorem} \label{th:mainth}
	For all $n,d$, there exists a polynomial $p \in \tilde{\mathcal{H}}_{n,2d}$ such that 
	\begin{align}\label{eq:mainth}
	y^TH_p(x)y=\tilde{w}^T_{n,d-1}(x,y)Q\tilde{w}_{n,d-1}(x,y),
	\end{align} where $Q$ is strictly dd.
\end{theorem}
Any such polynomial will be in $int(\tilde{\Sigma}_D C_{n,2d})$. Indeed, if we were to pertub the coefficients of $p$ slightly, then each coefficient of $Q$ would undergo a slight perturbation. As $Q$ is strictly dd, $Q$ would remain dd, and hence $p$ would remain dsos-convex. 

We will prove Theorem \ref{th:mainth} through a series of lemmas. First, we show that this is true in the homogeneous case and when $n=2$ (Lemma \ref{lemma:homogn2}). By induction, we prove that this result still holds in the homogeneous case for any $n$ (Lemma \ref{lemma:induction}). We then extend this result to the nonhomogeneous case.

\begin{lemma} \label{lemma:homogn2}
	For all $d$, there exists a polynomial $p \in \tilde{\mathcal{H}}_{2,2d}$ such that 
	\begin{align}\label{eq:DSOS.Conv.int}
	y^TH_p(x)y=w^T_{2,d-1}(x,y)Qw_{2,d-1}(x,y), 
	\end{align}
	for some strictly dd matrix $Q$.
\end{lemma}

{We remind the reader that Lemma \ref{lemma:homogn2} corresponds to the base case of a proof by induction on $n$ for Theorem \ref{th:mainth}.}

\begin{proof}
	{In this proof, we show that there exists a polynomial $p$ that satisfies (\ref{eq:DSOS.Conv.int}) for some strictly dd matrix $Q$ in the case where $n=2$, and for any $d\geq 1.$}
	
	First, if $2d=2$, we simply take $p(x_1,x_2)=x_1^2+x_2^2$ as $y^TH_p(x)y=2y^T I y$ {and the identity matrix is strictly dd}. Now, assume $2d>2$.  We consider two cases depending on whether $d$ is divisible by $2$.
	
	In the case that it is, we construct $p$ as
	\scalefont{0.7}
	\begin{align*}
	p(x_1,x_2) \mathrel{\mathop{:}}=a_0x_1^{2d}+a_1x_1^{2d-2}x_2^2+a_2x_1^{2d-4}x_2^4+\ldots+a_{d/2}x_1^{d}x_2^{d}+\ldots+a_1x_1^2x_2^{2d-2}+a_0x_2^{2d},
	\end{align*}
	\normalsize
	with the sequence $\{a_k\}_{k=0,\ldots,\frac{d}{2}}$ defined as follows
	\begin{equation}\label{eq:ak.even}
	\begin{aligned}
	a_1&=1\\
	a_{k+1}&=\left( \frac{2d-2k}{2k+2}\right)a_k,~ k=1,\ldots, \frac{d}{2}-1 \text{ (for } 2d>4)\\
	a_0&=\frac{1}{d}+\frac{d}{2(2d-1)}a_{\frac{d}{2}}.
	\end{aligned}
	\end{equation}
	
	%We want to show that $p$ as defined above verifies (\ref{eq:DSOS.Conv.int}) where we recall that $w_{2,d-1}(x,y)$ is given by $$\begin{pmatrix}y_1x_1^{d-1} & \cdots & y_1x_1^{d-1-k}x_2^{k} & \cdots & y_1x_2^{d-1} & y_2x_1^{d-1} & \cdots & y_2x_1^{d-k-1}x_2^{k} & \cdots & y_2x_2^{d-1} \end{pmatrix}^T.$$
	
	Let 
	\begin{equation}\label{eq:def.beta.gamma.delta}
	\begin{aligned}
	\beta_k&=a_k(2d-2k)(2d-2k-1), k=0,\ldots,\frac{d}{2}-1,\\
	\gamma_k&=a_k\cdot 2k(2k-1), k=1,\ldots,\frac{d}{2},\\
	\delta_k&=a_k(2d-2k)\cdot 2k, k=1,\ldots,\frac{d}{2}.
	\end{aligned}
	\end{equation}
	We claim that the matrix $Q$ defined as  
	\setcounter{MaxMatrixCols}{29}
	\begin{align*}
	\scalemath{0.7}{
	\begin{pmatrix}
	\beta_0 & & & & & &  &  & &  & & \vline & 0 & \delta_1 & & & & & & & & &\delta_{\frac{d}{2}}\\
	& \ddots & & & & &  & & & & & \vline &  & \ddots & \ddots & &  & & & & & &\\
	&  & \beta_{k} & & & & & & & & &\vline & & & 0 & \delta_{k+1}   & & & & & &\\
	& & & \ddots& & & & & & & & \vline & & & & \ddots & \ddots& & & & &\\
	& & & & \beta_{\frac{d}{2}-2} & & & & & & & \vline & & & & & \ddots& \delta_{\frac{d}{2}-1} & & & &\\
	& & & & & \beta_{\frac{d}{2}-1} & & & & & &\vline & & &  & & &0 & 0& & & &\\
	\hline
	&  & & & & &\gamma_{\frac{d}{2}} & & & & &\vline & &   & & & & & 0 & \delta_{\frac{d}{2}-1} & & & \\
	&  & & & & & & \ddots & & & & \vline & & &  & & & & &\ddots  &\ddots& & \\
	&  & & & & & & & \gamma_k&  & &\vline & & &  & & & & &  & 0 &\delta_{k-1} & \\
	&  & & &  & &  & &  &  \ddots & &\vline & & &  & & & & & & &\ddots & \ddots \\
	&  & &   & & &   & &  & & \gamma_1 &\vline & & &  & & & & & &  & & 0 \\
	\hline \hline
	0& & & & & & & & &  & & \vline & \gamma_1 &  & & & & & & & & & &\\
	\ddots & \ddots & & & & & & & & & & \vline &  &\ddots  & & & & & & & & & &\\
	& \delta_{k-1}& 0  & & & & & & & &  & \vline & &  &\gamma_k  & & & & & & & & & &\\
	& & \ddots & \ddots & & & & & &  & &\vline & &  & & \ddots & & & & & & & &\\
	& &  & \delta_{\frac{d}{2}-1}& 0 & & & & & & & \vline & & &  & &   \gamma_{\frac{d}{2}} & & & & & \\
	\hline & &  &  &0 &0 & & & &  & & \vline & &  & & &  &   \beta_{\frac{d}{2}-1} & & & & & &\\
	& &  &  & &\delta_{\frac{d}{2}-1} & 0 & & & & & \vline & &  & & &  &   & \beta_{\frac{d}{2}-2} & & & & & &\\
	& &  & & & &\ddots & \ddots & & &  & \vline & &  & & &  & & &  \ddots & & & & &\\
	& &  & & & & & \delta_{k+1} &0 & & & \vline & &  & & &  & &  & &\beta_k & & & &\\
	& &  &  & & & &  &\ddots &\ddots &  &\vline & &  & & &  &  & & & &\ddots & & &\\
	\delta_{\frac{d}{2}}& &  &  & & &  & & &  \delta_1&  0& \vline & &  & & & &  &  & & & &\beta_0 & &\\
	\end{pmatrix}}
	\normalsize
	\end{align*}
	is strictly dd and satisfies (\ref{eq:DSOS.Conv.int}) with $w_{2,d-1}(x,y)$ ordered as $$\begin{pmatrix}y_1x_1^{d-1}, y_1 x_1^{d-2}x_2, \ldots, y_1x_1x_2^{d-2}, y_1x_2^{d-1},y_2x_1^{d-1}, y_2 x_1^{d-2}x_2, \ldots, y_2 x_1x_2^{d-2}, y_2x_2^{d-1} \end{pmatrix}^T.$$
	To show (\ref{eq:DSOS.Conv.int}), one can derive the Hessian of $p$, expand both sides of the equation, and verify equality. To ensure that the matrix is strictly dd, we want all diagonal coefficients to be strictly greater than the sum of the elements on the row. This translates to the following inequalities
	\begin{align*}
	\beta_0 &>\delta_1+\delta_{\frac{d}{2}}\\
	\beta_k &>\delta_{k+1}, \forall k=1,\ldots,\frac{d}{2}-2\\
	\beta_{\frac{d}{2}-1}&>0, \gamma_1>0 \\
	\gamma_{k+1} &>\delta_{k}, \forall k=1,\ldots, \frac{d}{2}-1.
	\end{align*}
	Replacing the expressions of $\beta_k,\gamma_k$ and $\delta_k$ in the previous inequalities using (\ref{eq:def.beta.gamma.delta}) and the values of $a_k$ given in (\ref{eq:ak.even}), one can easily check that these inequalities are satisfied.
	
	We now consider the case where $d$ is not divisable by 2 and take
	\scalefont{0.60}
	\begin{align*}
	p(x_1,x_2)\mathrel{\mathop{:}}=a_0x_1^{2d}+a_1x_1^{2d-2}x_2^2+\ldots+a_{(d-1)/2}x_1^{d+1}x_2^{d-1}+a_{(d-1)/2}x_1^{d-1}x_2^{d+1}
	+\ldots+a_1x_1^2x_2^{2d-2}+a_0x_2^{2d},
	\end{align*}
	\normalsize
	with the sequence $\{a_k\}_{k=0,\ldots,\frac{d-1}{2}}$ defined as follows
	\begin{equation}\label{eq:ak.odd}
	\begin{aligned}
	a_1&=1\\
	a_{k+1}&=\left( \frac{2d-2k}{2k+2}\right)a_k,~ k=1,\ldots, \frac{d-3}{2}\\
	a_0&=1+\frac{2(2d-2)}{2d(2d-1)}.
	\end{aligned}
	\end{equation}
	Again, we want to show existence of a strictly dd matrix $Q$ that satisfies (\ref{eq:DSOS.Conv.int}). Without changing the definitions of the sequences $\{\beta_k\}_{k=1,\ldots,\frac{d-3}{2}}$,$\{\gamma_k\}_{k=1,\ldots,\frac{d-1}{2}}$ and $\{\delta_k\}_{k=1,\ldots,\frac{d-1}{2}}$, we claim this time that the matrix $Q$ defined as
	\begin{align*}
	\scalemath{0.7}{
	\begin{pmatrix}
	\beta_0 & & & & & &  &  & &  & \vline & 0 & \delta_1 & & & & & & & &\textbf{0}\\
	& \ddots & & & & &  & & & & \vline &  & \ddots & \ddots & &  & & & & & &\\
	&  & \beta_{k} & & & & & & & &\vline & & & 0 & \delta_{k+1}   & & & & & &\\
	& & & \ddots& & & & & & & \vline & & & &\ddots & \ddots& & & & &\\
	& & & &\beta_{\frac{{d-3}}{2}} & & & & & &\vline & &  & & &0 & \delta_{\frac{{d-1}}{{2}}}& & & &\\
	\hline
	&  & & & &\gamma_{\frac{{d-1}}{2}} & & & & &\vline & &   & & & & 0 & \delta_{\frac{{d-1}}{2}-1} & & & \\
	&  & & & & & \ddots & & & & \vline & & &  & & &  &\ddots  &\ddots& & \\
	&  & & & & & & \gamma_k&  & &\vline & & &  & & & &  & 0 &\delta_{k-1} & \\
	&  & & &  & & &  &  \ddots & &\vline & & &  & & & & & &\ddots & \ddots \\
	&  & &   & & &   &  & & \gamma_1 &\vline & & &  & & & & &  & & 0 \\
	\hline \hline
	0& & & & & & & & &  & \vline & \gamma_1 &  & & & & & & & & & &\\
	\ddots & \ddots & & & & & & & &  & \vline &  &\ddots  & & & & & & & & & &\\
	& \delta_{k-1}& 0  & & & & & & &  & \vline & &  &\gamma_k  & & & & & & & & & &\\
	& & \ddots & \ddots & & & & & &  & \vline & &  & & \ddots & & & & & & & &\\
	& &  & \delta_{\frac{d-1}{2}-1}& 0 & & & & & & \vline & & &  & &   \gamma_{\frac{d-1}{2}} & & & & & \\
	\hline & &  &  &  \delta_{\frac{{d-1}}{\textbf{2}}}&0 & & & &  & \vline & &  & & &  &   \beta_{\frac{{d-3}}{2}} & & & & & &\\
	& &  &  &&\ddots & \ddots & & &  & \vline & &  & & &  &  &\ddots & & & & &\\
	& &  &  & & & \delta_{k+1} &0 & &  & \vline & &  & & &  &  & &\beta_k & & & &\\
	& &  &  & & &  &\ddots &\ddots &  & \vline & &  & & &  &  & & &\ddots & & &\\
	\textbf{0}& &  &  & & &  & &  \delta_1&  0& \vline & &  & & &  &  & & & &\beta_0 & &\\
	\end{pmatrix}}
	\normalsize
	\end{align*}
	satisfies (\ref{eq:DSOS.Conv.int}) and is strictly dd. Showing (\ref{eq:DSOS.Conv.int}) amounts to deriving the Hessian of $p$ and checking that the equality is verified. To ensure that $Q$ is strictly dd, the inequalities that now must be verified are 
	\begin{align*}
	\beta_k &>\delta_{k+1}, \forall k=0,\ldots,\frac{d-1}{2}-1\\
	\gamma_{k} &>\delta_{k-1}, \forall k=2,\ldots, \frac{d-1}{2} \\
	\gamma_{1}&>0.
	\end{align*}
	These inequalities can all be shown to hold using (\ref{eq:ak.odd}).  
	
\end{proof}

\begin{lemma}\label{lemma:induction}
	For all $n,d,$ there exists a form $p_{n,2d} \in \mathcal{H}_{n,2d}$ such that  
	\begin{align*}
	y^TH_{p_{n,2d}}(x)y=w^T_{n,d-1}(x,y)Q_{p_{n,2d}}w_{n,d-1}(x,y) 
	\end{align*}
	and $Q_{p_{n,2d}}$ is a strictly dd matrix.
\end{lemma}
\begin{proof}
	We proceed by induction on $n$ with fixed and arbitrary $d$. The property is verified for $n=2$ by Lemma \ref{lemma:homogn2}. 
	Suppose that there exists a form $p_{n,2d} \in \mathcal{H}_{n,2d}$ such that 
	\begin{align}\label{eq:induction.hyp}
	y^TH_{p_{n,2d}}y=w^T_{n,d-1}(x,y)Q_{p_{n,2d}}w_{n,d-1}(x,y),
	\end{align}
	for some strictly dd matrix $Q_{p_{n,2d}}.$  We now show that 
	$$p_{n+1,2d}\mathrel{\mathop{:}}=q+\alpha v $$
	with
	\begin{equation} \label{eq:defqv}
	\begin{aligned}
	q&\mathrel{\mathop{:}}=\sum_{\{i_1,\ldots,i_n\} \in \{1,\ldots,n+1\}^n} p_{n,2d}(x_{i_1},\ldots,x_{i_n})\\
	v &\mathrel{\mathop{:}}=\sum_{2i_1+\ldots 2i_{n+1}=2d, i_1,\ldots,i_{n+1}>0} x_1^{2i_1}x_2^{2i_2}\ldots x_{n+1}^{2i_{n+1}},
	\end{aligned}
	\end{equation}
	and $\alpha>0$ small enough, verifies 
	\begin{align} \label{eq:induction.result}
	y^TH_{p_{n+1,2d}}y=w^T_{n+1,d-1}(x,y)Q_{p_{n+1,2d}}w_{n+1,d-1}(x,y),
	\end{align}
	for some strictly dd matrix $Q_{p_{n+1,2d}}$. Equation (\ref{eq:induction.result}) will actually be proved using an equivalent formulation that we describe now. {Recall that $$w_{n+1,d-1}(x,y)=y\cdot z_{n+1,d-1},$$ where $z_{n+1,d-1}$ is the standard vector of monomials in $x=(x_1,\ldots,x_{n+1})$ of degree exactly $d-1$. Let $\hat{w}_n$ be a vector containing all monomials from $w_{n+1,d-1}$  that include up to $n$ variables in $x$ and $\hat{w}_{n+1}$ be a vector containing all monomials from $w_{n+1,d-1}$ with exactly $n+1$ variables in $x$.} Obviously, $w_{n+1,d-1}$ is equal to $$\hat{w}\mathrel{\mathop{:}}=\begin{pmatrix} \hat{w}_n \\ \hat{w}_{n+1} \end{pmatrix}$$ up to a permutation of its entries. If we show that there exists a strictly dd matrix $\hat{Q}$ such that 
	\begin{align} \label{eq:induction.permutation}
	y^TH_{p_{n+1,2d}}(x)y=\hat{w}^T(x,y)\hat{Q}\hat{w}(x,y)
	\end{align}
	{then one can easily construct a strictly dd matrix $Q_{p_{n+1,2d}}$ such that (\ref{eq:induction.result}) will hold by simply permuting the rows of $\hat{Q}$ appropriately.}
	%then (\ref{eq:induction.result}) will hold, as the rows of $Q_{p_{n+1,2d}}$ will be a permutation of the rows of $\hat{Q}$. 
	
	%Denote by $\hat{Q}_q$ (resp. $\hat{Q}_{v}$) a matrix such that 
	%\begin{align}
	%y^TH_q(x)y&=\hat{w}^T(x,y)\hat{Q}_q\hat{w}(x,y) \label{eq:induction.q}\\
	%\text{ (resp. } y^TH_v (x) y&=\hat{w}^T(x,y)\hat{Q}_{v}\hat{w}(x,y)).\label{eq:induction.gam}
	%\end{align}
	{We now show the existence of such a $\hat{Q}$. To do this, we claim and prove the following:}
	\begin{itemize}
		\item \emph{Claim 1:} there exists a strictly dd matrix $\hat{Q}_q$ such that
		\begin{align}\label{eq:permut.q}
		y^TH_{q}(x)y=\begin{pmatrix} \hat{w}_n \\ \hat{w}_{n+1} \end{pmatrix}^T \begin{pmatrix} \hat{Q}_q & 0 \\ 0 & 0 \end{pmatrix} \begin{pmatrix} \hat{w}_n \\ \hat{w}_{n+1} \end{pmatrix}.
		\end{align} 
		\item \emph{Claim 2:} there exist a symmetric matrix $\hat{Q}_{v}$,  and $q_1,\ldots,q_m>0$ (where $m$ is the length of $\hat{w}_{n+1}$) such that
		\begin{align}\label{eq:permut.gam}
		y^TH_{v}(x)y=\begin{pmatrix} \hat{w}_n \\ \hat{w}_{n+1} \end{pmatrix}^T \begin{pmatrix} \hat{Q}_{v} & 0 \\ 0 & \text{diag}(q_1,\ldots,q_m)\end{pmatrix} \begin{pmatrix} \hat{w}_n \\ \hat{w}_{n+1} \end{pmatrix}.
		\end{align}
	\end{itemize}
	
	{Using these two claims and the fact that $p_{n+1,2d}=q+\alpha v$, we get that $$y^TH_{p_{n+1,2d}}(x)y=y^TH_q(x)y+\alpha y^T H_v(x)y=\hat{w}^T(x,y)\hat{Q}\hat{w}(x,y)$$ where $$\hat{Q}=\begin{pmatrix} \hat{Q}_q+\alpha \hat{Q}_v & 0 \\ 0 & \alpha \text{ diag}(q_1,\ldots,q_m) \end{pmatrix}.$$} As $\hat{Q}_q$ is strictly dd, we can pick $\alpha>0$ small enough such that $\hat{Q}_q+\alpha \hat{Q}_v$ is strictly dd. This entails that $\hat{Q}$ is strictly dd, and  (\ref{eq:induction.permutation}) holds.
	
	{It remains to prove the two claims to be done.}\\
	
	{\emph{Proof of Claim 1:} Claim 1 concerns the polynomial $q$, defined as the sum of polynomials $p_{n,2d}(x_{i_1},\ldots,x_{i_n})$. Note from (\ref{eq:induction.hyp}) that the Hessian of each of these polynomials has a strictly dd Gram matrix in the monomial vector $w_{n,d-1}.$ However, the statement of Claim 1 involves the monomial vector $\hat{w}_n$. So, we start by linking the two monomial vectors.}
	If we denote by $$M=\mathop{\cup}_{(i_1,\ldots,i_n) \in \{1,\ldots,n+1\}^n} \{\text{monomials in } w_{n,d-1}(x_{i_1},\ldots, x_{i_n},y)\},$$ then $M$ is exactly equal to $\hat{M}=\{\text{monomials in }\hat{w}_n(x,y)\}$ { as the entries of both are monomials of degree 1 in $y$ and of degree $d-1$ and in $n$ variables of $x=(x_1,\ldots,x_{n+1}).$
		
		By definition of $q$, we have that 
		\scalefont{0.69}
		\begin{align*}
		y^TH_qy=\sum_{(i_1,\ldots,i_n) \in \{1,\ldots,n+1\}^n} w_{n,d-1}(x_{i_1}, \ldots,x_{i_n},y)^T Q_{p_{n,2d}(x_{i_1},\ldots,x_{i_n})} w_{n,d-1}(x_{i_1}, \ldots,x_{i_n},y)
		\end{align*}
		\normalsize
		
		We now claim that there exists a strictly dd matrix $\hat{Q}_q$ such that $$y^TH_qy=\hat{w}_n^T \hat{Q}_q \hat{w}_n.$$ This matrix is constructed by padding the strictly dd matrices $Q_{p_{n,2d}(x_{i_1},\ldots,x_{i_n})}$ with rows of zeros and then adding them up. The sum of two rows that verify the strict diagonal dominance condition still verifies this condition. So we only need to make sure that there is no row in $\hat{Q}_q$ that is all zero. This is indeed the case because $\hat{M} \subseteq M.$}

	{ \emph{Proof of Claim 2:}} Let $I \mathrel{\mathop{:}}=\{i_1,\ldots,i_n~|~ i_1+\ldots+i_{n+1}=d, i_1,\ldots,i_{n+1}>0\}$ and $\hat{w}_{n+1}^i$ be the $i^{th}$ element of $\hat{w}_{n+1}$. To prove (\ref{eq:permut.gam}), we need to show that
	\begin{equation}\label{eq:hessian.gam.ondiag}
	\begin{aligned}
	y^TH_{v}(x)y=\sum_{i_1,\ldots,i_n \in I}\sum_{k=1}^{n+1}2i_k(2i_{k}-1)x_1^{2i_1}\ldots x_k^{2i_k-2}\ldots x_{n+1}^{2i_{n+1}}y_k^2 \\
	+4\sum_{i_1,\ldots,i_m \in I} \sum_{j \neq k} i_k i_j x_1^{2i_1}\ldots x_j^{2i_j-1}\ldots x_k^{2i_k-1}\ldots x_{n+1}^{2i_{n+1}}y_jy_k.
	\end{aligned}
	\end{equation}
	can equal
	\begin{align} \label{eq:expansion}
	\hat{w}_n^T(x,y)\hat{Q}_v \hat{w}_n^T(x,y)+\sum_{i=1}^m q_{i}(\hat{w}_{n+1}^i)^2
	\end{align}
	for some symmetric matrix $\hat{Q}_{v}$ and positive scalars $q_1,\ldots, q_m$.
	{ We first argue that all monomials contained in $y^TH_v(x)y$ appear in the expansion (\ref{eq:expansion}). This means that we do not need to use any other entry of the Gram matrix in (\ref{eq:permut.gam}). Since every monomial appearing in the first double sum of (\ref{eq:hessian.gam.ondiag}) involves only even powers of variables, it can be obtained via the diagonal entries of $Q_v$ together with the entries $q_1,\ldots,q_m.$ Moreover, since the coefficient of each monomial in this double sum is positive and since the sum runs over all possible monomials consisting of even powers in $n+1$ variables, we conclude that $q_i>0$, for $i=1,\ldots,m.$

		%It is clear that any monomial in the first term of (\ref{eq:hessian.gam.ondiag}) will feature in (\ref{eq:expansion}) via the diagonal terms of the matrix $\hat{Q}_{v}$ and the entries $q_1,\ldots,q_m$. Furthermore, it is not hard to see from the structure of $v$ in (\ref{eq:defqv}) that each $q_i>0$.
		
		Consider now any monomial contained in the second double sum  of (\ref{eq:hessian.gam.ondiag}). We claim that any such monomial can be obtained from off-diagonal entries in $\hat{Q}_v.$ To prove this claim, we show that it can be written as the product of two monomials $m'$ and $m''$ with $n$ or fewer variables in $x=(x_1,\ldots,x_{n+1})$. Indeed, at least two variables in the monomial must have degree less than or equal to $d-1$.} Placing one variable in $m'$ and the other variable in $m''$ and then filling up $m'$ and $m''$ with the remaining variables (in any fashion as long as the degrees at $m'$ and $m''$ equal $d-1$) yields the desired result.  

\end{proof}
\begin{proof}[of Theorem \ref{th:mainth}] Let $p_{n,2k} \in \mathcal{H}_{n,2k}$ be the form constructed in the proof of Lemma \ref{lemma:induction} which is in the interior of $\Sigma_D C_{n,2k}.$ Let $Q_{k}$ denote the strictly diagonally dominant matrix which was constructed to satisfy $$y^TH_{p_{n,2k}}y=w_{n,2k}^T(x,y)Q_{k}w_{n,2k}.$$ To prove Theorem \ref{th:mainth}, we take $$p \mathop{\mathrel{:}}=\sum_{k=1}^{d} p_{n,2k} \in \tilde{\mathcal{H}}_{n,2d}.$$ We have
	\begin{align*}
	y^TH_p(x)y&=\begin{pmatrix} w_{n,1}(x,y) \\ \vdots \\ w_{n,d-1}(x,y) \end{pmatrix}^T \begin{pmatrix} Q_1 & & \\ & \ddots &  \\ & & Q_d \end{pmatrix} \begin{pmatrix} w_{n,1}(x,y) \\ \vdots \\ w_{n,d-1}(x,y) \end{pmatrix}\\
	&=\tilde{w}_{n,d-1}(x,y)^TQ\tilde{w}_{n,d-1}(x,y).
	\end{align*} We observe that $Q$ is strictly dd, which shows that $p \in int(\tilde{\Sigma}_DC_{n,2d}).$  
\end{proof}

\begin{remark}
	If we had only been interested in showing that any polynomial in $\tilde{\mathcal{H}}_{n,2d}$ could be written as a difference of two sos-convex polynomials, this could have been easily done by showing that $p(x)=\left( \sum_i x_i^2 \right)^d \in int(\Sigma C_{n,2d})$. However, this form is not dsos-convex or sdsos-convex for all $n,d$ (e.g., for $n=3$ and $2d=8$). We have been unable to find a simpler proof for existence of sdsos-convex dcds that does not go through the proof of existence of dsos-convex dcds.
\end{remark}
\begin{remark}
	If we solve problem (\ref{eq:undom.dcd}) with the convexity constraint replaced by a dsos-convexity (resp. sdsos-convexity, sos-convexity) requirement, the same arguments used in the proof of Theorem \ref{thm:undom.dcd} now imply that the optimal solution $g^*$ is not dominated by any dsos-convex (resp. sdsos-convex, sos-convex) decomposition.
\end{remark} 

\section{Numerical results}\label{sec:numerical.results}
In this section, we present a few numerical results to show how our algebraic decomposition techniques affect the convex-concave procedure. The objective function $p \in \tilde{\mathcal{H}}_{n,2d}$ in all of our experiments is generated randomly following the ensemble of \cite[Section 5.1.]{Minimize_poly_Pablo}. This means that $$p(x_1,\ldots,x_n)=\sum_{i=1}^n x_i^{2d}+g(x_1,\ldots,x_n),$$ where $g$ is a random polynomial of total degree $\leq 2d-1$ whose coefficients are random integers uniformly sampled from $[-30,30].$ An advantage of polynomials generated in this fashion is that they are bounded below and that their minimum $p^*$ is achieved over $\mathbb{R}^n.$ {We have intentionally restricted ourselves to polynomials of degree equal to $4$ in our experiments as this corresponds to the smallest degree for which the problem of finding a dc decomposition of $f$ is hard, without being too computationally expensive. Experimenting with higher degrees however would be a worthwhile pursuit in future work.} The starting point of CCP was generated randomly from a zero-mean Gaussian distribution. 

One nice feature of our decomposition techniques is that all the polynomials $f_i^k, i=0,\ldots, m$ in line 4 of Algorithm \ref{alg:CCP} in the introduction are sos-convex. This allows us to solve the convex subroutine of CCP exactly via a single SDP \cite[Remark 3.4.]{Monique_Etienne_Convex}, \cite[Corollary 2.3.]{lasserre2009convexity}:
\begin{equation} \label{eq:Lasserre.hierarchy}
\begin{aligned}
&\min \gamma \\
&\text{s.t. } f_0^k -\gamma=\sigma_0+\sum_{j=1}^m \lambda_j f_j^k\\
&\sigma_0 \text{ sos}, ~\lambda_j \geq 0, j=1,\ldots,m.
\end{aligned}
\end{equation}
The degree of $\sigma_0$ here is taken to be the maximum degree of $f_0^k,\ldots, f_m^k$. We could have also solved these subproblems using standard descent algorithms for convex optimization. However, we are not so concerned with the method used to solve this convex problem as it is the same for all experiments. All of our numerical examples were done using MATLAB, the polynomial optimization library SPOT \cite{SPOT_Megretski}, and the solver MOSEK \cite{mosek}.

\subsection{Picking a good dc decomposition for CCP} \label{subsec:numexpOneDecomp}
In this subsection, we consider the problem of minimizing a random polynomial $f_0 \in \tilde{\mathcal{H}}_{8,4}$ over a ball of radius $R$, where $R$ is a random integer in $[20,50].$ The goal is to compare the impact of the dc decomposition of the objective on the performance of CCP. To monitor this, we decompose the objective in 4 different ways and then run CCP using the resulting decompositions. These decompositions are obtained through different SDPs that are listed in Table~\ref{tab:diff.objs}.

\begin{table}[h!]
	\centering
	\begin{tabular}{|c|c|c|c|}
		\hline
		Feasibility & $\lambda_{\max}H_h(x_0)$ & $\lambda_{\max,B} H_h$ & Undominated\\
		\hline
		& $\min t $ & $\min_{g,h} t$ & \\
		$\min 0$ &  $\text{s.t. } f_0=g-h,$ &  $\text{s.t. } f_0=g-h,$ & $\min\frac{1}{\mathcal{A}_n}\int \mbox{Tr} H_g d\sigma$ \\
		$\text{s.t. } f_0=g-h,$ & $g,h$ sos-convex & $g,h$ sos-convex &  $\text{s.t. } f_0=g-h,$ \\
		$g,h$ sos-convex & $tI-H_{h}(x_0) \succeq 0 $ & $y^T(tI-H_h(x)+f_1 \tau(x))y$ sos & $g,h$ sos-convex\\
		& & $y^T \tau(x) y$ \tablefootnote{Here, $\tau(x)$ is an $n \times n$ matrix where each entry is in $\tilde{\mathcal{H}}_{n,2d-4}$}  sos& \\
		\hline
	\end{tabular}
	\caption{Different decomposition techniques using sos optimization}
	\label{tab:diff.objs}
	\vspace{-10pt}
\end{table}

The first SDP in Table \ref{tab:diff.objs} is simply a feasibility problem. The second SDP minimizes the largest eigenvalue of $H_h$ at the initial point $x_0$ inputed to CCP. The third minimizes the largest eigenvalue of $H_h$ over the ball $B$ of radius $R$. Indeed, let $f_1 \mathrel{\mathop{:}}=\sum_i x_i^2-R^2.$ Notice that $\tau(x) \succeq 0, \forall x$ and if $x \in B$, then $f_1(x) \leq 0$. This implies that $tI \succeq H_h(x), \forall x \in B.$ The fourth SDP searches for an undominated dcd.

Once $f_0$ has been decomposed, we start CCP. After $4$ mins of total runtime, the program is stopped and we recover the objective value of the last iteration. This procedure is repeated on 30 random instances of $f_0$ and $R$, and the average of the results is presented in Figure \ref{fig:Onedecomp}.
\begin{figure}[h!]
	\centering
	\includegraphics[scale=0.45]{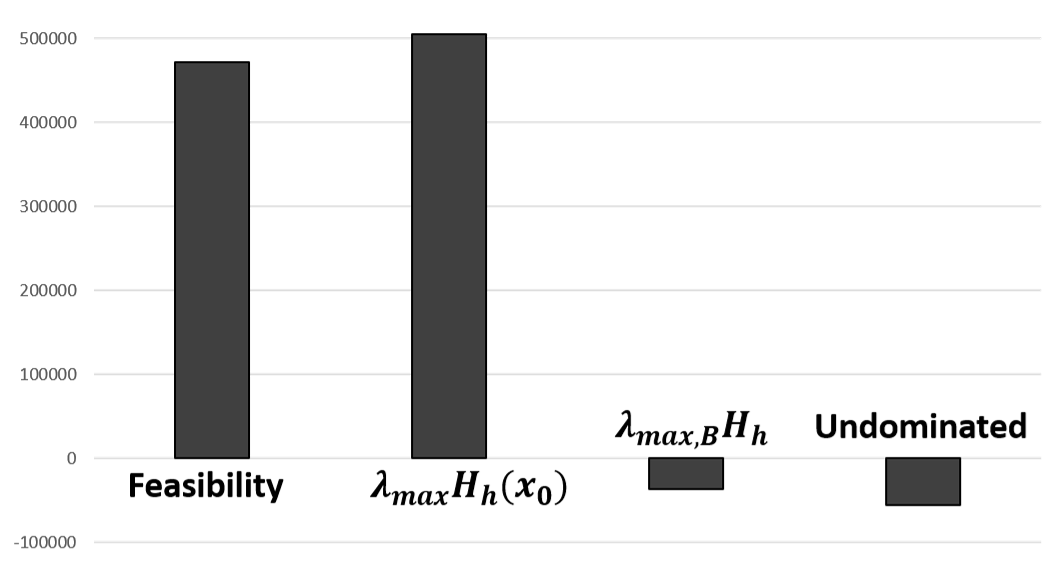}
	\caption{Impact of choosing a good dcd on CCP ($n=8,2d=4$)}
	\label{fig:Onedecomp}
\end{figure}
From the figure, we can see that the choice of the initial decomposition impacts the performance of CCP considerably, with the region formulation of $\lambda_{\max}$ and the undominated decomposition giving much better results than the other two. It is worth noting that all formulations have gone through roughly the same number of iterations of CCP (approx. 400). Furthermore, these results seem to confirm that it is best to pick an undominated decomposition when applying CCP.

\subsection{Scalibility of s/d/sos-convex dcds and the multiple decomposition CCP}\label{subsec:scalibility}

While solving the last optimization problem in Table \ref{tab:diff.objs} usually gives very good results, it relies on an sos-convex dc decomposition. However, this choice is only reasonable in cases where the number of variables and the degree of the polynomial that we want to decompose are low. When these become too high, obtaining an sos-convex dcd can be too time-consuming. The concepts of dsos-convexity and sdsos-convexity then become interesting alternatives to sos-convexity.
This is illustrated in Table \ref{tab:dsos}, where we have reported the time taken to solve the following decomposition problem:
\begin{equation} \label{eq:decomp.relax}
\begin{aligned}
&\min \frac{1}{\mathcal{A}_n} \int_{S^{n-1}} \mbox{Tr } H_g d\sigma\\
&\text{s.t. } f=g-h, g,h \text{ s/d/sos-convex}
\end{aligned}
\end{equation}
In this case, $f$ is a random polynomial of degree $4$ in $n$ variables. We also report the optimal value of (\ref{eq:decomp.relax}) (we know that (\ref{eq:decomp.relax}) is always guaranteed to be feasible from Theorem \ref{th:mainth}).
\begin{table}[h!] 
	\begin{tabular}{|c|c|c|c|c|c|c|c|c|}
		\hline
		&  \multicolumn{2}{|c|}{n=6} & \multicolumn{2}{|c|}{n=10} & \multicolumn{2}{|c|}{n=14} & \multicolumn{2}{|c|}{n=18}\\
		& Time & Value & Time & Value & Time & Value & Time & Value\\
		\hline dsos-convex & $<1s$ & 62090 & $<$1s & 168481 & 2.33s &136427 & 6.91s & 48457 \\
		\hline sdsos-convex & $<1s$ & 53557 & 1.11 s& 132376& 3.89s &99667 & 12.16s & 32875 \\
		\hline sos-convex & $<1s$ & 11602 & 44.42s &18346  &800.16s &9828 & 30hrs+ & ------\\
		\hline
	\end{tabular}
	\caption{Time and optimal value obtained when solving (\ref{eq:decomp.relax})}
	\label{tab:dsos}
	\vspace{-15pt}
\end{table}
Notice that for $n=18$, it takes over 30 hours to obtain an sos-convex decomposition, whereas the run times for s/dsos-convex decompositions are still in the range of 10 seconds. This increased speed comes at a price, namely the quality of the decomposition. For example, when $n=10$, the optimal value obtained using sos-convexity is nearly 10 times lower than that of sdsos-convexity. 

Now that we have a better quantitative understanding of this tradeoff, we propose a modification to CCP that leverages the speed of s/dsos-convex dcds for large $n$.
The idea is to modify CCP in such a way that one would compute a new s/dsos-convex decomposition of the functions $f_i$ after each iteration. Instead of looking for dcds that would provide good global decompositions (such as undominated sos-convex dcds), we look for decompositions that perform well locally. From Section \ref{sec:undominated}, candidate decomposition techniques for this task can come from formulations (\ref{eq:trace.point}) and (\ref{eq:lambda.max.point}) that minimize the maximum eigenvalue of the Hessian of $h$ at a point or the trace of the Hessian of $h$ at a point. This modified version of CCP is described in detail in Algorithm \ref{alg:itCCP}. We will refer to it as \emph{multiple decomposition CCP}.

We compare the performance of CCP and multiple decomposition CCP on the problem of minimizing a polynomial $f$ of degree 4 in $n$ variables, for varying values of $n$. In Figure \ref{fig:sdsos.vs.sos}, we present the optimal value (averaged over 30 instances) obtained after 4 mins of total runtime. The ``SDSOS" columns correspond to multiple decomposition CCP (Algorithm \ref{alg:itCCP}) with sdsos-convex decompositions at each iteration. The ``SOS" columns correspond to classical CCP where the first and only decomposition is an undominated sos-convex dcd. From Figure \ref{fig:Onedecomp}, we know that this formulation performs well for small values of $n$. This is still the case here for $n=8$ and $n=10$. However, this approach performs poorly for $n=12$ as the time taken to compute the initial decomposition is too long. In contrast, multiple decomposition CCP combined with sdsos-convex decompositions does slightly worse for $n=8$ and $n=10$, but significantly better for $n=12$.

\begin{algorithm}[h]
	\caption{ Multiple decomposition CCP ($\lambda_{\max}$ version)}
	\label{alg:itCCP}
	\begin{algorithmic}[1]
		\Require $x_0,~ f_i, i=0,\ldots,m$
		\State $k\leftarrow 0$
		\While{stopping criterion not satisfied}
		\State Decompose: $\forall i$ find $g_i^k,h_i^k$ s/d/sos-convex that min. $t$, s.t. $tI-H_{h_i^k}(x_k)$ s/dd\footnotemark and $f_i=g_i^k-h_i^k$
		\State Convexify: $f_i^{k}(x)\mathrel{\mathop{:}}=g_i^k(x)-(h_i^k(x_k)+\nabla h_i^k(x_k)^T(x-x_k)),~ i=0,\ldots,m$
		\State Solve convex subroutine: $\min f_0^k(x)$, s.t. $f_i^k(x) \leq 0, i=1,\ldots,m$
		\State $x_{k+1}\mathrel{\mathop{:}}= \underset{f_i^{k}(x) \leq 0}{\text{argmin}} f_0^k(x)$
		\State $k \leftarrow k+1$
		\EndWhile
		\Ensure $x_k$
	\end{algorithmic}
\end{algorithm}
\footnotetext{Here dd and sdd matrices refer to notions introduced in Definition \ref{def:dd.sdd}. Note that any $t$ which makes $tI-A$ dd or sdd gives an upperbound on $\lambda_{\max}(A).$ By formulating the problem this way (instead of requiring $tI\succeq A$) we obtain an LP or SOCP instead of an SDP.}

\begin{figure}[h!]
	\centering
	\includegraphics[scale=0.5]{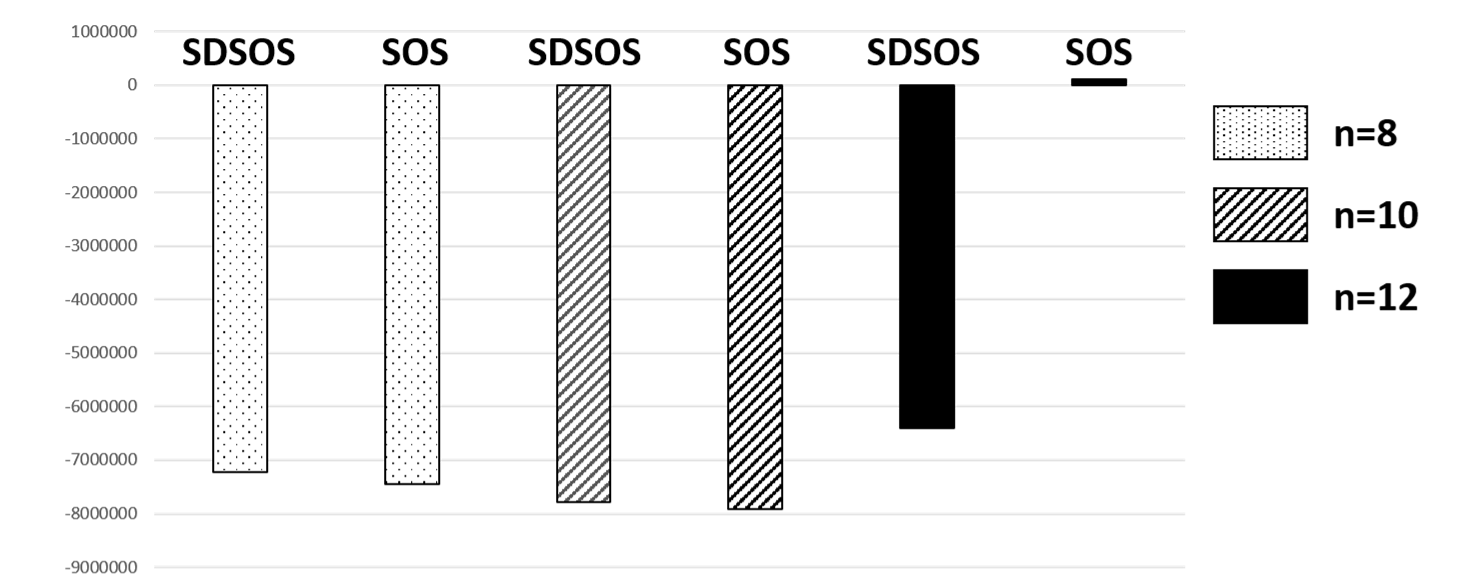}
	\caption{Comparing multiple decomposition CCP using sdsos-convex decompositions against CCP with a single undominated sos-convex decomposition}
	\label{fig:sdsos.vs.sos}
	\vspace{-15pt}
\end{figure}

In conclusion, our overall observation is that picking a good dc decomposition noticeably affects the perfomance of CCP. While optimizing over all dc decompositions is intractable for polynomials of degree greater or equal to $4$, the algebraic notions of sos-convexity, sdsos-convexity and dsos-convexity can provide valuable relaxations. The choice among these options depends on the number of variables and the degree of the polynomial at hand. Though these classes of polynomials only constitute subsets of the set of convex polynomials, we have shown that even the smallest subset of the three contains dcds for any polynomial.

%\bibliographystyle{abbrv}
%\bibliography{ch-basis_pursuit/pablo_amirali,ch-basis_pursuit/elib}

%\bibliographystyle{plain}
%\begin{thebibliography}{10}
%
%\bibitem{am}
%A. A. Ahmadi and A. Majumdar. 
%DSOS and SDSOS optimization: More tractable alternatives to SOS optimization.
%{\em Manuscript.}
%
%\bibitem{cl}
%  M. D. Choi and T. Y. Lam, Extremal positive semidefinite forms, {\em Math. Ann.} {\bf 231} 1--8 (1977).
%  
%\bibitem{dp}
%E. de Klerk and D. V. Pasechnik.  Approximation of the stability number of a graph via copositive programming. {\em SIAM Journal on Optimization}, {\bf 12}(4), 875-- 892 (2002).
%
%\bibitem{mot}
%T.S. Motzkin, The arithmetic-geometric inequality. Proc. Symposium on Inequalities (ed. O. Shisha), Academic Press, New York, 1967, pp. 205-–224.
%
%\bibitem{mk}
%K. G. Murty and S. N. Kabadi. Some NP-complete problems in quadratic and nonlinear
%programming. {\em Mathematical Programming}, {\bf 39} 117--129 (1987).
%  
%\bibitem{pp}
%P. A. Parrilo. Semidefinite programming relaxations for semialgebraic problems. Mathematical Programming {\bf 96} 293--320 (2003).
%\end{thebibliography}

\chapter{Polynomials Norms}\label{ch:polynorms}

%\keywords{Column generation, conic optimization, sum of squares programming }

%============================================================
\section{Introduction}
%============================================================

A function $f:\mathbb{R}^n \rightarrow \mathbb{R}$ is a \emph{norm} if it satisfies the following three properties:
\begin{enumerate}[(i)]
	\item positive definiteness: $f(x)>0, ~\forall x \neq 0,$ and $f(0)=0$.
	\item $1$-homogeneity: $f(\lambda x)=|\lambda| f(x),~ \forall x\in \mathbb{R}^n, ~\forall \lambda \in \mathbb{R}$.
	\item triangle inequality: $f(x+y)\leq f(x)+f(y), ~\forall x,y \in \mathbb{R}^n.$
\end{enumerate}
Some well-known examples of norms include the $1$-norm, $f(x)=\sum_{i=1}^n |x_i|$, the $2$-norm, $f(x)=\sqrt{\sum_{i=1}^n x_i^2}$, and the $\infty$-norm, $f(x)=\max_{i} |x_i|.$ Our focus throughout this chapter is on norms that can be derived from multivariate polynomials. More specifically, we are interested in establishing conditions under which the $d^{th}$ root of a homogeneous polynomial of degree $d$ is a norm, where $d$ is an even number. We refer to the norm obtained when these conditions are met as \emph{a polynomial norm}. It is easy to see why we restrict ourselves to $d^{th}$ roots of degree-$d$ homogeneous polynomials. Indeed, nonhomogeneous polynomials cannot hope to satisfy the homogeneity condition and homogeneous polynomials of degree $d>1$ are not 1-homogeneous unless we take their $d^{th}$ root. The question of when the square root of a homogeneous quadratic polynomial is a norm (i.e., when $d=2$) has a well-known answer (see, e.g., \cite[Appendix A]{BoydBook}): a function $f(x)=\sqrt{x^TQx}$ is a norm if and only if the symmetric $n \times n$ matrix $Q$ is positive definite. In the particular case where $Q$ is the identity matrix, one recovers the $2$-norm. Positive definiteness of $Q$ can be checked in polynomial time using for example Sylvester's criterion (positivity of the $n$ leading principal minors of $Q$). This means that testing whether the square root of a quadratic form is a norm can be done in polynomial time. A similar characterization in terms of conditions on the coefficients are not known for polynomial norms generated by forms of degree greater than 2. In particular, it is not known whether one can efficiently test membership or optimize over the set of polynomial norms. 

\paragraph{Outline and contributions.} In this chapter, we study polynomial norms from a computational perspective. In Section \ref{sec:eq.charac.comp}, we give two different necessary and sufficient conditions under which the $d^{th}$ root of a degree-$d$ form $f$ will be a polynomial norm: namely, that $f$ be strictly convex (Theorem \ref{th:norm.str.conv}), or (equivalently) that $f$ be convex and postive definite (Theorem \ref{th:norm.conv.pd}). Section \ref{sec:approx.norms} investigates the relationship between general norms and polynomial norms: while many norms are polynomial norms (including all $p$-norms with $p$ even), some norms are not (consider, e.g., the $1$-norm). We show, however, that any norm can be approximated to arbitrary precision by a polynomial norm (Theorem \ref{th:approx.poly.norm.sphere}). In Section \ref{sec:sos.approx}, we move on to complexity results and show that simply testing whether the $4^{th}$ root of a quartic form is a norm is strongly NP-hard (Theorem \ref{th:NP.hard}). We then provide a semidefinite programming-based test for checking whether the $d^{th}$ root of a degree $d$ form is a norm (Theorem \ref{th:test.poly.norm}) and a semidefinite programming-based hierarchy to optimize over a subset of the set of polynomial norms (Theorem \ref{th:opt.poly.norms}). The latter is done by introducing the concept of $r$-sum of squares-convexity (see Definition \ref{def:r.sos.convex}). We show that any form with a positive definite Hessian is $r$-sos-convex for some value of $r$, and present a lower bound on that value (Theorem \ref{th:r.sos.convex}). We also show that the level $r$ of the semidefinite programming hierarchy cannot be bounded as a function of the number of variables and the degree only (Theorem \ref{th:unif.bnd}). Finally, we cover a few applications of polynomial norms in statistics and dynamical systems in Section \ref{sec:apps}. In Section \ref{sec:norm.reg}, we compute approximations of two different types of norms, polytopic gauge norms and $p$-norms with $p$ noneven, using polynomial norms. The techniques described in this section can be applied to norm regression. In Section \ref{sec:JSR.comp}, we use polynomial norms to prove stability of a switched linear system, a task which is equivalent to computing an upperbound on the joint spectral radius of a family of matrices.

%==============================================
\section{Two equivalent characterizations of polynomial norms} \label{sec:eq.charac.comp}
%==================================================
We start this section with two theorems that provide conditions under which the $d^{th}$ root of a degree-$d$ form is a norm. These will be useful in Section \ref{sec:sos.approx} to establish semidefinite programming-based approximations of polynomial norms. Note that throughout this chapter, $d$ is taken to be an even positive integer. 

\begin{theorem}\label{th:norm.conv.pd}
	The $d^{th}$ root of a degree-$d$ form $f$ is a norm if and only if $f$ is convex and positive definite.
\end{theorem}

\begin{proof}
	If $f^{1/d}$ is a norm, then $f^{1/d}$ is positive definite, and so is $f$. Furthermore, any norm is convex and the $d^{th}$ power of a nonnegative convex function remains convex.
	
	Assume now that $f$ is convex and positive definite. We show that $f^{1/d}$ is a norm. Positivity and homogeneity are immediate. It remains to prove the triangle inequality. Let $g\mathrel{\mathop{:}}=f^{1/d}$. Denote by $S_f$ and $S_g$ the 1-sublevel sets of $f$ and $g$ respectively. It is clear that $$S_g=\{x~|~f^{1/d}(x)\leq 1\}=\{x~|~f(x)\leq 1\}=S_f,$$
	and as $f$ is convex, $S_f$ is convex and so is $S_g$. Let $x,y \in \mathbb{R}^n.$ We have that $\frac{x}{g(x)} \in S_g$ and $\frac{y}{g(y)} \in S_g$. From convexity of $S_g$,
	$$g \left( \frac{g(x)}{g(x)+g(y)} \cdot \frac{x}{g(x)}+\frac{g(y)}{g(x)+g(y)} \cdot \frac{y}{g(y)}  \right) \leq 1.$$
	Homogeneity of $g$ then gives us $$\frac{1}{g(x)+g(y)}g(x+y)\leq 1$$ which shows that triangle inequality holds.
\end{proof}

\begin{theorem}\label{th:norm.str.conv}
	The $d^{th}$ root of a degree-$d$ form $f$ is a norm if and only if $f$ is strictly convex, i.e., $$f(\lambda x+ (1-\lambda)y)<\lambda f(x)+(1-\lambda)f(y),~\forall x\neq y,~\forall \lambda \in (0,1).$$
\end{theorem}

\begin{proof}
	We will show that a degree-d form $f$ is strictly convex if and only $f$ is convex and positive definite. The result will then follow from Theorem~\ref{th:norm.conv.pd}.
	
	Suppose $f$ is strictly convex, then the first-order characterization of strict convexity gives us that $$f(y)>f(x)+\nabla f(x)^T(y-x), ~\forall y\neq x.$$ For $x=0$, the inequality becomes $f(y)>0, ~\forall y\neq 0 $, as $f(0)=0$ and $\nabla f(0)=0$. Hence, $f$ is positive definite. Of course, a strictly convex function is also convex. 
	
	Suppose now that $f$ is convex, positive definite, but not strictly convex, i.e., there exists $\bar{x},\bar{y} \in \mathbb{R}^n$ with $\bar{x}\neq \bar{y}$, and $\gamma \in (0,1)$ such that $$f\left(\gamma \bar{x}+(1-\gamma)\bar{y} \right)=\gamma f(\bar{x})+ (1-\gamma) f(\bar{y}).$$ Let $g(\alpha)\mathrel{\mathop{:}}=f(\bar{x}+\alpha(\bar{y}-\bar{x})).$ Note that $g$ is a restriction of $f$ to a line and, consequently, $g$ is a convex, positive definite, univariate polynomial in $\alpha$. We now define \begin{align}\label{eq:expression.g}
	h(\alpha)\mathrel{\mathop{:}}=g(\alpha)-(g(1)-g(0))\alpha-g(0).
	\end{align} Similarly to $g$, $h$ is a convex univariate polynomial as it is the sum of two convex univariate polynomials. We also know that $h(\alpha)\geq 0, \forall \alpha \in (0,1)$. Indeed, by convexity of $g$, we have that $g(\alpha x+(1-\alpha)y)\geq \alpha g(x)+(1-\alpha)g(y), \forall x,y \in \mathbb{R}$ and $\alpha \in (0,1)$. This inequality holds in particular for $x=1$ and $y=0$, which proves the claim. Observe now that $h(0)=h(1)=0$. By convexity of $h$ and its nonnegativity over $(0,1)$, we have that $h(\alpha)=0$ on $(0,1)$ which further implies that $h=0$. Hence, from (\ref{eq:expression.g}), $g$ is an affine function. As $g$ is positive definite, it cannot be that $g$ has a nonzero slope, so $g$ has to be a constant. But this contradicts that $\lim_{\alpha \rightarrow \infty} g(\alpha) =\infty.$ To see why this limit must be infinite,  we show that $\lim_{||x|| \rightarrow \infty} f(x)=\infty.$ As $\lim_{\alpha \rightarrow \infty} ||\bar{x}+\alpha(\bar{y}-\bar{x})||=\infty$ and $g(\alpha)=f(\bar{x}+\alpha(\bar{y}-\bar{x}))$, this implies that $\lim_{\alpha \rightarrow \infty} g(\alpha)=\infty.$ To show that $\lim_{||x|| \rightarrow \infty} f(x)=\infty$, let $$x^*=\underset{||x||=1}{\text{argmin }} f(x).$$ By positive definiteness of $f$, $f(x^*)>0.$ Let $M$ be any positive scalar and define $R\mathrel{\mathop{:}}=(M/f(x^*))^{1/d}$. Then for any $x$ such that $||x||=R$, we have $$f(x)\geq \min_{||x||=R} f(x) \geq R^d f(x^*)=M,$$ where the second inequality holds by homogeneity of $f.$ Thus $\lim_{||x|| \rightarrow \infty} f(x)=\infty$.
\end{proof}

%===================================================
\section{Approximating norms by polynomial norms}  \label{sec:approx.norms}
%===================================================

It is easy to see that not all norms are polynomial norms. For example, the 1-norm $||x||_1=\sum_{i=1}^n |x_i|$ is not a polynomial norm. Indeed, all polynomial norms are differentiable at all but one point (the origin) whereas the 1-norm is nondifferentiable whenever one of the components of $x$ is equal to zero. In this section, we show that, though not every norm is a polynomial norm, any norm can be approximated to arbitrary precision by a polynomial norm (Theorem \ref{th:approx.poly.norm.sphere}). The proof of this theorem is inspired from a proof by Ahmadi and Jungers in \cite{sosconvex_Lyap_cdc,Ahmadi_Jungers}. A related result is given by Barvinok in \cite{barvinok2003}. In that chapter, he shows that any norm can be approximated by the $d$-th root of a nonnegative  degree-$d$ form, and quantifies the quality of the approximation as a function of $n$ and $d$. The form he obtains however is not shown to be convex. In fact, in a later work \cite[Section 2.4]{barvinok2006computational}, Barvinok points out that it would be an interesting question to know whether any norm can be approximated by the $d^{th}$ root of a convex form with the same quality of approximation as for $d$-th roots of nonnegative forms. The result below is a step in that direction though no quantitative result on the quality of approximation is given. Throughout, $S^{n-1}$ denotes the unit sphere in $\mathbb{R}^n.$

\begin{theorem} \label{th:approx.poly.norm.sphere}
	Let $||\cdot ||$ be any norm on $\mathbb{R}^n.$ For any $\epsilon>0$, there exist an even integer $d$ and a convex positive definite form $f$ of degree $d$ such that $$\max_{x \in S^{n-1}} |~f^{1/d}(x)-||x||~| \leq \epsilon.$$
\end{theorem}

Note that, from Theorem~\ref{th:norm.conv.pd}, $f^{1/d}$ is a polynomial norm as $f$ is a convex positive definite form. To show this result, we start with the following lemma.

\begin{lemma}\label{lem:approx.level.set}
	Let $|| \cdot||$ be any norm on $\mathbb{R}^n$. For any $\epsilon>0$, there exist an even integer $d$ and an $n$-variate convex positive definite form $f$ of degree $d$ such that 
	\begin{align}\label{eq:norm.approx.level.set}
	||x|| \leq f^{1/d}(x) \leq (1+\epsilon) ||x||,~ \forall x \in \mathbb{R}^n.
	\end{align}
\end{lemma}

\begin{proof} Throughout, we let $B_{\alpha}\mathrel{\mathop{:}}=\{x~|~ ||x|| \leq \alpha\}.$ When $\alpha=1$, we drop the subscript and simply denote by $B$ the unit ball of $||\cdot||$. We will also use the notation $\partial S$ to denote the boundary of a set $S$ and $int(S)$ to denote its interior. Let $\bar{\epsilon}\mathrel{\mathop{:}}=\frac{\epsilon}{1+\epsilon}$. The crux of the proof lies in proving that there exists an integer $d$ and a positive definite convex form $f$ of degree $d$ such that \begin{align} \label{eq:level.set.incl}
	B_{1-\bar{\epsilon}} \subseteq \{x~|~ f(x)\leq 1\} \subseteq B.
	\end{align}
	If we prove this, then Lemma \ref{lem:approx.level.set} can be obtained as follows. Let $x \in \mathbb{R}^n.$ To show the first inequality in (\ref{eq:norm.approx.level.set}), we proceed by contradiction. Suppose that $||x|| > f^{1/d}(x)$. If $f^{1/d}(x)\neq 0$, then $||x/f^{1/d}(x)||>1$ while $f(x/f^{1/d}(x))=1$. (If $f^{1/d}(x)=0$ then $x=0$ and the inequality holds.) Hence, $$x/f^{1/d}(x) \in \{x~|~ f(x)\leq 1\}$$ but $x/f^{1/d}(x) \notin B$ which contradicts (\ref{eq:level.set.incl}). To prove the second inequality in (\ref{eq:norm.approx.level.set}), note that the first inclusion of $(\ref{eq:level.set.incl})$ gives us $f^{1/d}((1-\bar{\epsilon})x/||x||) \leq 1$, which is equivalent to $f^{1/d}(x/||x||) \leq \frac{1}{1-\bar{\epsilon}}=1+\epsilon.$ Multiplying by $||x||$ on both sides gives us the result.\\
	
	We now focus on showing the existence of a positive definite convex form $f$ that satisfies  (\ref{eq:level.set.incl}). The proof is a simplification of the proof of Theorem 3.2.$  $ in \cite{sosconvex_Lyap_cdc,Ahmadi_Jungers} with some modifications.
	
	Let $x \in \partial B_{1-\bar{\epsilon}/2}$. To any such $x$, we associate a dual vector $v(x)$ orthogonal to a supporting hyperplane of $B_{1-\bar{\epsilon}/2}$ at $x$. By definition of a supporting hyperplane, we have that $v(x)^Ty \leq v(x)^Tx$,  $\forall y \in B_{1-\bar{\epsilon}/2}$, and, as $B_{1-\bar{\epsilon}} \subset B_{1-\bar{\epsilon}/2}$, we have \begin{align} \label{eq:prop.C.eps}
	v(x)^Ty < v(x)^Tx, \forall y \in B_{1-\bar{\epsilon}}.
	\end{align} 
	Let $$S(x)\mathrel{\mathop{:}}=\{y~|~ v(x)^Ty>v(x)^Tx\text{ and } ||y||=1\}.$$ It is easy to see that $S(x)$ is an open subset of the boundary $\partial B$ of $B$. Futhermore, since $x \in int(B)$, $x/||x|| \in S(x)$ which implies that $S(x)$ is nonempty and that the family of sets $S(x)$ (as $x$ ranges over $\partial B_{1-\bar{\epsilon}/2})$ is a covering of $\partial B$. As $\{S(x)\}_{x \in \partial B_{1-\bar{\epsilon}/2}}$ is an open covering of the compact set $\partial B$, there exists a finite covering of $\partial B$, i.e., one can choose $x_1, \ldots, x_N \in \partial B_{1-\bar{\epsilon}/2}$ in such a way that $\cup_{i=1}^N S(x_i)=\partial B.$
	
	\begin{figure}[H]
		\centering
		\includegraphics[scale=0.5]{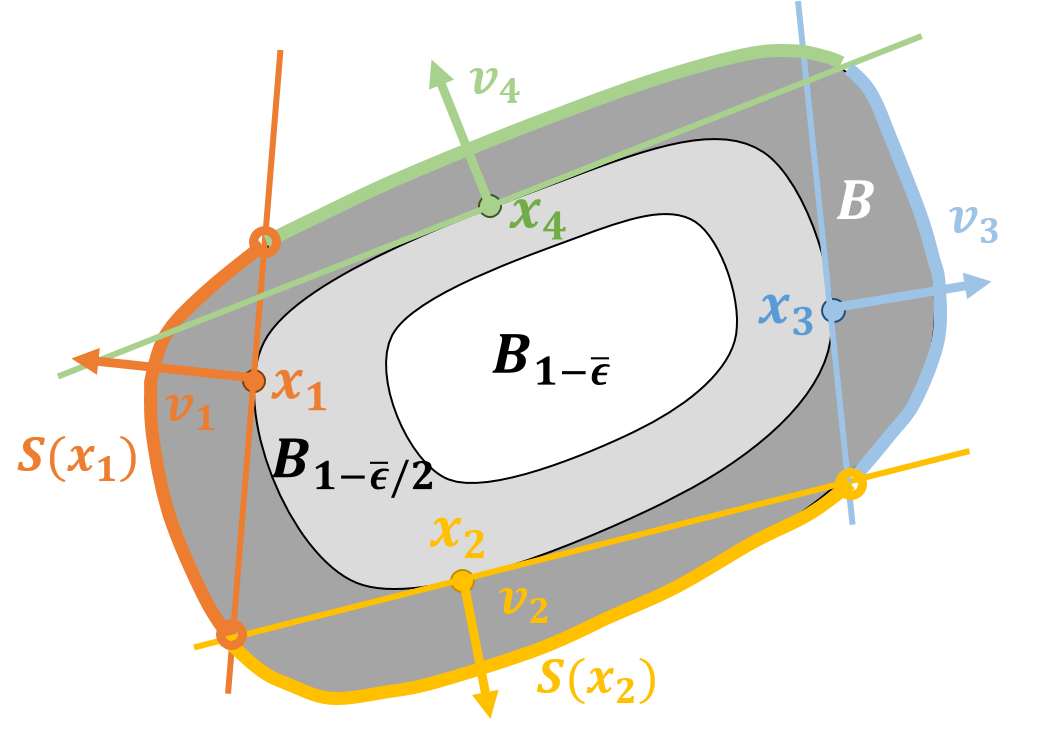}
		\caption{An illustration of the construction of the open covering of $\partial B$.}
		\label{fig:proof.part.1}
	\end{figure}
	
	For ease of notation, we now let $v_i\mathrel{\mathop{:}}=v(x_i)$ for all $i=1,\ldots,N$. From (\ref{eq:prop.C.eps}), we have that $\frac{v_i^Ty}{v_i^Tx_i}<1$ for any $i$ and for any $y \in B_{1-\bar{\epsilon}}$\footnote{Note that $v_i^Tx_i \neq 0$, $\forall i$. In fact, we have $v_i^Tx_i>0, \forall i.$ To see this, recall that by definition of a supporting hyperplane, $v_i \neq 0$ and $v_i^Tx_i \geq v_i^Ty$, for all $y \in B_{1-\bar{\epsilon}/2}$. In particular, there exists $\alpha_i>0$ such that $\alpha_i v_i \in B_{1-\bar{\epsilon}/2}$. Hence, $v_i^Tx_i \geq \alpha ||v_i||_2^2>0.$}. Since $B_{1-\bar{\epsilon}}$ is compact, we get $ \max_{i} \max_{y \in B_{1-\bar{\epsilon}}} \left( \frac{v_i^Ty}{v_i^Tx_i}\right)<1$. Hence, there exists an integer $d$ such that 
	\begin{align}\label{eq:def.degree}
	\left( \max_{i} \max_{y \in B_{1-\bar{\epsilon}}} \left( \frac{v_i^Ty}{v_i^Tx_i}\right)\right)^{d}<\frac{1}{N}.
	\end{align}
	We now define
	\begin{align} \label{eq:poly.def}
	f(y)\mathrel{\mathop{:}}=\sum_{i=1}^N \left( \frac{v_i^Ty}{v_i^Tx_i}\right)^{d},
	\end{align}
	where $d$ is any integer satisfying (\ref{eq:def.degree}). The form $f$ is convex as a sum of even powers of linear forms. Let $$\mathcal{L}\mathrel{\mathop{:}}=\{y~|~ f(y)\leq 1\}.$$ 
	By (\ref{eq:def.degree}), it is straightforward to see that $B_{1-\bar{\epsilon}} \subseteq \mathcal{L}.$ 
	
	We now show that $\mathcal{L} \subseteq int(B).$ Let $y \in \mathcal{L}$, then $f(y) \leq 1.$ If a sum of nonnegative terms is less than or equal to 1, then each term has to be less than or equal to $1$, which implies that $\frac{v_i^Ty}{v_i^Tx_i} \leq 1$, for all $i=1,\ldots,N.$ From this, we deduce that $y \notin \partial B$. Indeed, if $y \in \partial B$, there exists $i \in \{1,\ldots,N\}$ such that $y \in S(x_i)$ as $\{S(x_i)\}_i$ is a cover of $\partial B$. But, by definition of $S(x_i)$, we would have $v_i^T y>v_i^Tx_i$ which contradicts the previous statement. We have that $\partial B \cap \mathcal{L}= \emptyset$ as a consequence. However, as $\mathcal{L}$ and $B$ both contain the zero vector, this implies that $\mathcal{L} \subseteq int(B).$ Note that the previous inclusion guarantees positive definiteness of $f$. Indeed, if $f$ were not positive definite, $\mathcal{L}$ would be unbounded and could not be a subset of $B$ (which is bounded). 
\end{proof}

\begin{proof}[Proof of Theorem \ref{th:approx.poly.norm.sphere}]
	Let $\epsilon >0$ and denote by $\alpha\mathrel{\mathop{:}}=\max_{x\in S^{n-1}} ||x||$.
	By Lemma \ref{lem:approx.level.set}, there exists an integer $d$ and a convex form $f$ such that
	$$||x|| \leq f^{1/d}(x) \leq \left(1+\frac{\epsilon}{\alpha}\right)||x||, ~\forall x.$$
	This is equivalent to
	$$0 \leq f^{1/d}(x)-||x|| \leq \frac{\epsilon}{\alpha} ||x||, \forall x.$$
	For $x \in S^{n-1}$, as $||x||/\alpha\leq1$, this inequality becomes  
	$$0 \leq f^{1/d}(x)-||x|| \leq \epsilon.$$
\end{proof}

\begin{remark}
	We remark that the polynomial norm constructed in Theorem \ref{th:approx.poly.norm.sphere} is the $d^{th}$-root of an \emph{sos-convex} polynomial. Hence, one can approximate any norm on $\mathbb{R}^n$ by searching for a polynomial norm using semidefinite programming. To see why the polynomial $f$ in (\ref{eq:poly.def}) is sos-convex, observe that linear forms are sos-convex and that an even power of an sos-convex form is sos-convex.
\end{remark}

%=================================================
\section{Semidefinite programming-based approximations of polynomial norms} \label{sec:sos.approx}
%====================================================

\subsection{Complexity}
It is natural to ask whether testing if the $d^{th}$ root of a given degree-$d$ form is a norm can be done in polynomial time. In the next theorem, we show that, unless $P=NP$, this is not the case even when $d=4$.
\begin{theorem}\label{th:NP.hard}
	{Deciding} whether the $4^{th}$ root of a quartic form is a norm is strongly NP-hard.
\end{theorem}
\begin{proof}
	The proof of this result is adapted from a proof in \cite{NPhard_Convexity_MathProg}. Recall that the CLIQUE problem can be described thus: given a graph $G=(V,E)$ and a positive integer $k$, decide whether $G$ contains a clique of size at least $k$. The CLIQUE problem is known to be NP-hard \cite{GareyJohnson_Book}. We will give a reduction from CLIQUE to the problem of testing convexity and positive definiteness of a quartic form. The result then follows from Theorem \ref{th:norm.conv.pd}. Let $\omega(G)$ be the clique number of the graph at hand, i.e., the number of vertices in a maximum clique of $G$. Consider the following quartic form $$b(x;y)\mathrel{\mathop{:}}=-2k\sum_{i,j \in E} x_ix_jy_iy_j-(1-k)\left(\sum_i x_i^2\right)\left(\sum_i y_i^2\right).$$ In \cite{NPhard_Convexity_MathProg}, using in part a result in \cite{Ling_et_al_Biquadratic}, it is shown that 
	\begin{align} \label{eq:clique.str.conv}
	&\omega(G)\leq k \Leftrightarrow b(x;y)+\frac{n^2 \gamma}{2} \left( \sum_{i=1}^n x_i^4+\sum_{i=1}^n y_i^4 +\sum_{1\leq i<j \leq n} (x_i^2x_j^2+ y_i^2y_j^2)\right)
	\end{align}
	is convex and $b(x;y)$ is positive semidefinite. Here, $\gamma$ is a positive constant defined as the largest coefficient in absolute value of any monomial present in some entry of the matrix $\left[\frac{\partial^2 b(x;y)}{\partial x_i \partial y_j}\right]_{i,j}$. As $\sum_i x_i^4+\sum_i y_i^4$ is positive definite and as we are adding this term to a positive semidefinite expression, the resulting polynomial is positive definite. Hence, the equivalence holds if and only if the quartic on the righthandside of the equivalence in (\ref{eq:clique.str.conv}) is convex and positive definite. 
\end{proof}

Note that this also shows that strict convexity is hard to test for quartic forms (this is a consequence of Theorem \ref{th:norm.str.conv}). A related result is Proposition 3.5. in \cite{NPhard_Convexity_MathProg}, which shows that testing strict convexity of a polynomial of even degree $d\geq 4$ is hard. However, this result is not shown there for \emph{forms}, hence the relevance of the previous theorem.

Theorem \ref{th:NP.hard} motivates the study of tractable sufficient conditions to be a polynomial norm. The sufficient conditions we consider next are based on semidefinite programming.

%We now present a series of sufficient conditions under which $d^{th}$ roots of degree-$d$ forms become norms. Contrary to the conditions described in Section \ref{sec:eq.charac.comp}, these conditions are computationally tractable; we show in particular that they can be reformulated as semidefinite programs. However, as one could expect due to differences in computational difficulty, the conditions we present are not necessary. We formally quantify how much is lost from inner approximating the set of polynomial norms by semidefinite programming-based approximations.

\subsection{Sum of squares polynomials and semidefinite programming review} \label{sec:sos.review}

We start this section by reviewing the notion of \emph{sum of squares polynomials} and related concepts such as \emph{sum of squares-convexity}. We say that a polynomial $f$ is a \emph{sum of squares} (sos) if $f(x)=\sum_i q_i^2(x)$, for some polynomials $q_i$. Being a sum of squares is a sufficient condition for being nonnegative. The converse however is not true, as is exemplified by the Motzkin polynomial
\begin{align}\label{eq:motzkin}
M(x,y)=x^4y^2 + x^2y^4 − 3x^2y^2 + 1
\end{align}
which is nonnegative but not a sum of squares \cite{MotzkinSOS}.
The sum of squares condition is a popular surrogate for nonnegativity due to its tractability. Indeed, while testing nonnegativity of a polynomial of degree greater or equal to 4 is a hard problem, testing whether a polynomial is a sum of squares can be done using \emph{semidefinite programming.} This comes from the fact that a polynomial $p$ of degree $d$ is a sum of squares if and only if there exists a positive semidefinite matrix $Q$ such that $f(x)=z(x)^TQz(x)$, where $z(x)$ is the standard vector of monomials of degree up to $d$ (see, e.g., \cite{PhD:Parrilo}). As a consequence, any optimization problem over the coefficients of a set of polynomials which includes a combination of affine constraints and sos constraints on these polynomials, together with a linear objective can be recast as a semidefinite program. These type of optimization problems are known as \emph{sos programs}.

Though not all nonnegative polynomials can be written as sums of squares, the following theorem by Artin \cite{Artin_Hilbert17} circumvents this problem using sos multipliers.
\begin{theorem}[Artin \cite{Artin_Hilbert17}]\label{th:artin}
	For any nonnegative polynomial $f$, there exists an sos polynomial $q$ such that $q \cdot f$ is sos. 
\end{theorem}
This theorem in particular implies that if we are given a polynomial $f$, then we can always check its nonnegativity using an sos program that searches for $q$ (of a fixed degree). However, the condition does not allow us to optimize over the set of nonnegative polynomials using an sos program (as far as we know). This is because, in that setting, products of decision varibles arise from multiplying polymomials $f$ and $q$, whose coefficients are decision variables.
%This theorem in particular implies that if we fix the degree of the multiplier $q$ and if we are given a fixed polynomial $f$, nonnegativity of $f$ can be tested using sos programs which will search for $q$. By progressively pushing up the degree of $q$, one can hope to recover a certificate of nonnegativity of $f$. If $f$ is not given though, both the coefficients of $f$ and the coefficients of $q$ are decision variables and the problem becomes nonconvex. 

By adding further assumptions on $f$, Reznick showed in \cite{Reznick_Unif_denominator} that one could further pick $q$ to be a power of $\sum_i x_i^2$.

\begin{theorem}[Reznick \cite{Reznick_Unif_denominator}]\label{th:reznick}
	Let $f$ be a positive definite form of degree $d$ in $n$ variables and define $$\epsilon(f)\mathrel{\mathop{:}}=\frac{\min\{f(u)~|~ u \in S^{n-1} \}}{\max\{f(u)~|~ u \in S^{n-1}\}}.$$ If $r\geq  \frac{nd(d-1)}{4 \log(2)\epsilon(f)}-\frac{n+d}{2}$, then $(\sum_{i=1}^n x_i^2)^r \cdot f$ is a sum of squares.
\end{theorem}
Motivated by this theorem, the notion of $r$-sos polynomials can be defined: a polynomial $f$ is said to be $r$-sos if $(\sum_i x_i^2)^r \cdot f$ is sos. Note that it is clear that any $r$-sos polynomial is nonnegative and that the set of $r$-sos polynomials is included in the set of $(r+1)$-sos polynomials. The Motzkin polynomial in (\ref{eq:motzkin}) for example is $1$-sos although not sos.

To end our review, we briefly touch upon the concept of sum of squares-convexity (sos-convexity), which we will build upon in the rest of the section. Let $H_f$ denote the Hessian matrix of a polynomial $f$. We say that $f$ is \emph{sos-convex} if $y^TH_f(x)y$ is a sum of squares (as a polynomial in $x$ and $y$). As before, optimizing over the set of sos-convex polynomials can be cast as a semidefinite program. Sum of squares-convexity is obviously a sufficient condition for convexity via the second-order characterization of convexity. However, there are convex polynomials which are not sos-convex (see, e.g., \cite{AAA_PP_not_sos_convex_journal}). For a more detailed overview of sos-convexity including equivalent characterizations and settings in which sos-convexity and convexity are equivalent, refer to \cite{ahmadi2013complete}.

\subsubsection{Notation}

Throughout, we will use the notation $H_{n,d}$ (resp. $P_{n,d}$) to denote the set of forms (resp. positive semidefinite, aka nonnegative, forms) in $n$ variables and of degree $d$. We will futhermore use the falling factorial notation $(t)_0=1$ and $(t)_k=t(t-1)\ldots(t-(k-1))$ for a positive integer $k$.

\subsection{A test for validity of polynomial norms} \label{sec:test}

In this subsection, we assume that we are given a form $f$ of degree $d$ and we would like to test whether $f^{1/d}$ is a norm using semidefinite programming. 

\begin{theorem} \label{th:test.poly.norm}
	Let $f$ be a degree-$d$ form. Then $f^{1/d}$ is a polynomial norm if and only if there exist $c>0$, $r \in \mathbb{N}$, and an sos form $q(x,y)$ such that $q(x,y) \cdot y^T H_f(x,y) y$ is sos and $\left(f(x)-c(\sum_i x_i^2)^{d/2}\right) (\sum_i x_i^2)^r$ is sos. Furthermore, this condition can be checked using semidefinite programming.
\end{theorem}

\begin{proof}
	It is immediate to see that if there exist such a $c$, $r$, and $q$, then $f$ is convex and positive definite. From Theorem \ref{th:norm.conv.pd}, this means that $f^{1/d}$ is a polynomial norm.
	
	Conversely, if $f^{1/d}$ is a polynomial norm, then, by Theorem \ref{th:norm.conv.pd}, $f$ is convex and positive definite. As $f$ is convex, the polynomial $y^TH_f(x)y$ is nonnegative. Using Theorem \ref{th:artin}, we conclude that there exists an sos polynomial $q(x,y)$ such that $q(x,y) \cdot y^TH_f(x)y$ is sos. We now show that, as $f$ is positive definite, there exist $c>0$ and $r \in \mathbb{N}$ such that  $\left(f(x)-c(\sum_i x_i^2)^{d/2}\right) (\sum_i x_i^2)^r$ is sos. Let $f_{min}$ denote the minimum of $f$ on the sphere. As $f$ is positive definite, $f_{min}>0.$ We take $c\mathrel{\mathop{:}}=\frac{f_{min}}{2}$ and consider $g(x)\mathrel{\mathop{:}}=f(x)-c(\sum_i x_i^2)^{d/2}$. We have that $g$ is a positive definite form: indeed, if $x$ is a nonzero vector in $\mathbb{R}^n$, then $$\frac{g(x)}{||x||^d}=\frac{f(x)}{||x||^d}-c=f \left( \frac{x}{||x||}\right)-c>0, $$
	by homogeneity of $f$ and definition of $c$. Using Theorem \ref{th:reznick}, $\exists r \in \mathbb{N}$ such that $g(x)(\sum_i x_i^2)^r$ is sos.
	
	For fixed $r$, a given form $f$, and a fixed degree $d$, one can search for $c>0$ and an sos form $q$ of degree $d$ such that $q(x,y) \cdot y^T H_f(x,y) y$ is sos and $\left(f(x)-c(\sum_i x_i^2)^{d/2}\right) (\sum_i x_i^2)^r$ is sos using semidefinite programming. This is done by solving the following semidefinite feasibility problem:
	\begin{equation}\label{eq:SDP.test.poly.norm}
	\begin{aligned}
	&q(x,y) \text{ sos}\\
	&c \geq 0 \\
	& q(x,y) \cdot y^TH_f(x,y)y \text{ sos}\\
	& \left(f(x)-c \left(\sum_i x_i^2\right)^{d/2}\right) \left(\sum_i x_i^2\right)^r \text{ sos},
	\end{aligned}
	\end{equation} 
	where the unknowns are the coefficients of $q$ and the real number $c$.
\end{proof}

\begin{remark}
	We remark that we are not imposing $c>0$ in the semidefinite program above. This is because, in practice, especially if the semidefinite program is solved with interior point methods, the solution returned by the solver will be in the interior of the feasible set, and hence $c$ will automatically be positive.
	One can slightly modify (\ref{eq:SDP.test.poly.norm}) however to take the constraint $c>0$ into consideration explicitely. Indeed, consider the following semidefinite feasibility problem where both the degree of $q$ and the integer $r$ are fixed:
	\begin{align}
	&q(x,y) \text{ sos} \nonumber\\
	& \gamma \geq 0 \nonumber \\
	& q(x,y) \cdot y^TH_f(x,y)y \text{ sos} \nonumber \\
	& \left(\gamma f(x)-\left(\sum_i x_i^2\right)^{d/2}\right) \left(\sum_i x_i^2\right)^r \text{ sos}. \label{eq:remark4.5}\\ \nonumber
	\end{align} 
	It is easy to check that (\ref{eq:remark4.5}) is feasible with $\gamma \geq 0$ if and only if the last constraint of (\ref{eq:SDP.test.poly.norm}) is feasible with $c>0$. To see this, take $c=1/\gamma$ and note that $\gamma$ can never be zero.
\end{remark}

To the best of our knowledge, we cannot use the approach described in Theorem \ref{th:test.poly.norm} to optimize over the set of polynomial norms with a semidefinite program. This is because of the product of decision variables in the coefficients of $f$ and $q$. The next subsection will address this issue. 

\subsection{Optimizing over the set of polynomial norms}\label{sec:opt}

In this subsection, we consider the problem of optimizing over the set of polynomial norms. To do this, we introduce the concept of $r$-sos-convexity. Recall that the notation $H_f$ references the Hessian matrix of a form $f$.

\subsubsection{Positive definite biforms and r-sos-convexity}

\begin{definition}\label{def:r.sos.convex}
	For an integer $r$, we say that a polynomial $f$ is $r$-sos-convex if $y^TH_f(x)y \cdot (\sum_i x_i^2)^r$ is sos.
\end{definition}

Observe that, for fixed $r$, the property of $r$-sos-convexity can be checked using semidefinite programming (though the size of this SDP gets larger as $r$ increases). Any polynomial that is $r$-sos-convex is convex. Note that the set of $r$-sos-convex polynomials is a subset of the set of $(r+1)$-sos-convex polynomials and that the case $r=0$ corresponds to the set of sos-convex polynomials.

It is natural to ask whether any convex polynomial is $r$-sos-convex for some $r$. Our next theorem shows that this is the case under a mild assumption. 

\begin{theorem}\label{th:r.sos.convex}
	Let $f$ be a form of degree $d$ such that $y^TH_f(x)y > 0$ for $(x,y) \in S^{n-1} \times S^{n-1}$. Let $$\eta(f)\mathrel{\mathop{:}}= \frac{\min\{y^TH_f(x)y~|~ (x,y) \in S^{n-1} \times S^{n-1} \}}{\max\{y^TH_f(x)y~|~ (x,y) \in S^{n-1} \times S^{n-1}\}}.$$
	If $r\geq  \frac{n(d-2)(d-3)}{4 \log(2)\eta(f)}-\frac{n+d-2}{2}-d$, then $f$ is $r$-sos-convex.
\end{theorem}

\begin{remark}
	Note that $\eta(f)$ can also be interpreted as $$\eta(f)  = \frac{\min_{x \in S^{n-1}} \lambda_{\min} (H_f(x)) }{\max_{x \in S^{n-1}} \lambda_{\max}(H_f(x))}=\frac{1}{ \max_{x \in S^{n-1}} \|H^{-1}_f(x)\|_2 \cdot \max_{x \in S^{n-1}} \|H_f(x)\|_2 }.$$

\end{remark}

\begin{remark}
	Theorem \ref{th:r.sos.convex} is a generalization of Theorem \ref{th:reznick} by Reznick. Note though that this is not an immediate generalization. First, $y^TH_f(x)y$ is not a positive definite form (consider, e.g., $y=0$ and any nonzero $x$). Secondly, note that the multiplier is $(\sum_i x_i^2)^r$ and does not involve the $y$ variables. (As we will see in the proof, this is essentially because $y^TH_f(x)y$ is quadratic in $y$.)
\end{remark}

\begin{remark}
	Theorem \ref{th:r.sos.convex} can easily be adapted to biforms of the type $\sum_j f_j(x)g_j(y)$ where $f_j$'s are forms of degree $d$ in $x$ and $g_j$'s are forms of degree $\tilde{d}$ in $y$. In this case, there exist integers $s,r$ such that $$\sum_j f_j(x)g_j(y) \cdot (\sum_i x_i^2)^r \cdot (\sum_i y_i^2)^s$$ is sos. For the purposes of this chapter however and the connection to polynomial norms, we will show the result in the particular case where the biform of interest is $y^TH_f(x)y.$
\end{remark}

We associate to any form $f \in H_{n,d}$, the $d$-th order differential operator $f(D)$, defined by replacing each occurence of $x_j$ with $\frac{\partial }{\partial x_j}$. For example, if $f(x_1,\ldots,x_n)\mathrel{\mathop{:}}=\sum_i c_i x_1^{a_1^i}\ldots x_{n}^{a_i^n}$ where $c_i \in \mathbb{R}$ and $a_{j}^i \in \mathbb{N}$, then its differential operator will be $$f(D)=\sum_{i} c_i \frac{\partial^{a_1^i}}{\partial x_1^{a_1^i}}\ldots \frac{\partial^{a_n^i}}{\partial x_n^{a_n^i}}.$$

Our proof will follow the structure of the proof of Theorem \ref{th:reznick} given in \cite{Reznick_Unif_denominator} and reutilize some of the results given in the chapter which we quote here for clarity of exposition.

\begin{prop}[\cite{Reznick_Unif_denominator}, see Proposition 2.6]\label{prop:sq2lin}
	For any nonnegative integer $r$, there exist nonnegative rationals $\lambda_k$ and integers $\alpha_{kl}$ such that
	$$(x_1^2+\ldots+x_n^2)^r=\sum_k \lambda_k(\alpha_{k1}x_1+\ldots+\alpha_{kn}x_n)^{2r}.$$
\end{prop}
For simplicity of notation, we will let $\alpha_k\mathrel{\mathop{:}}=(\alpha_{k1},\ldots,\alpha_{kl})^T$ and $x\mathrel{\mathop{:}}=(x_1,\ldots,x_n)^T$. Hence, we will write $\sum_k \lambda_k(\alpha_k^Tx)^{2r}$ to mean $\sum_k \lambda_k (a_{k1}x_1+\ldots+a_{kn}x_n)^{2r}$.

\begin{prop}[\cite{Reznick_Unif_denominator}, see Proposition 2.8]\label{prop:diff2sum}
	If $g \in H_{n,e}$ and $h=\sum_k \lambda_k (\alpha_k^Tx)^{d+e} \in H_{n,d+e}$, then $$g(D)h=(d+e)_e \sum_k \lambda_k g(\alpha_k) (\alpha_k^Tx)^d.$$
\end{prop}

\begin{prop}[\cite{Reznick_Unif_denominator}, see Theorem 3.7 and 3.9]\label{prop:def.Phi} For $f \in H_{n,d}$ and $s \geq d$, we define $\Phi_s(f) \in H_{n,d}$ by 
	\begin{align}\label{eq:def.Phi}
	f(D) (x_1^2+\ldots+x_n^2)^s=\mathrel{\mathop{:}} \Phi_s(f) (x_1^2+\ldots+x_n^2)^{s-d}.
	\end{align}
	The inverse $\Phi_s^{-1}(f)$ of $\Phi_s(f)$ exists and this is a map verifying $\Phi_s(\Phi_s^{-1}(f))=f.$
\end{prop}

\begin{prop}[\cite{Reznick_Unif_denominator}, see Theorem 3.12 ]\label{prop:pd.Phi}
	Suppose $f$ is a positive definite form in $n$ variables and of degree $d$ and let $$\epsilon(f)=\frac{\min\{f(u)~|~u \in S^{n-1}\}}{\max\{f(u)~|~u \in S^{n-1}\}}.$$ If $s\geq \frac{nd(d-1)}{4\log(2)\epsilon(f)}-\frac{n-d}{2}$, then $\Phi^{-1}_s(f) \in P_{n,d}.$
	
\end{prop}

We will focus throughout the proof on biforms of the following structure 
\begin{align} \label{eq:def.F}
F(x;y)\mathrel{\mathop{:}}=\sum_{1\leq i,j \leq n} y_iy_jp_{ij}(x),
\end{align}
where $p_{ij}(x) \in H_{n,d}$, for all $i,j$, and some even integer $d$. Note that the polynomial $y^TH_f(x)y$ (where $f$ is some form) has this structure. We next present three lemmas which we will then build on to give the proof of Theorem \ref{th:r.sos.convex}.

\begin{lemma}\label{lem:F.sos}
	For a biform $F(x;y)$ of the structure in (\ref{eq:def.F}), define the operator $F(D;y)$ as $$F(D;y)=\sum_{ij} y_iy_jp_{ij}(D).$$ If $F(x;y)$ is positive semidefinite (i.e., $F(x;y)\geq 0,~ \forall x,y$), then, for any $s \geq 0$, the biform $$F(D;y)(x_1^2+\ldots+x_n^2)^s$$ is a sum of squares.
\end{lemma}
\begin{proof}
	Using Proposition \ref{prop:sq2lin}, we have $$(x_1^2+\ldots+x_n^2)^s=\sum_l \lambda_l(\alpha_{l1}x_1+\ldots \alpha_{ln}x_n)^{2s},$$ where $\lambda_l \geq 0$ and $\alpha_l \in \mathbb{Z}^n.$ Hence, applying Proposition \ref{prop:diff2sum}, we get 
	\begin{align}
	F(D;y)(x_1^2+\ldots+x_n^2)^s &= \sum_{i,j} y_iy_j(p_{ij}(D)(x_1^2+\ldots+x_n^2)^s) \nonumber\\
	&= \sum_{i,j} y_iy_j \left( (2s)_d \sum_l \lambda_l p_{ij}(\alpha_l) (\alpha_l ^Tx)^{2s-d} \right) \nonumber\\
	&=(2s)_d \sum_l \lambda_l (\alpha_l ^Tx)^{2s-d} \sum_{i,j} y_iy_jp_{ij}(\alpha_l). \label{eq:proof.operator.sos}
	\end{align}
	Notice that $\sum_{i,j} y_iy_jp_{ij}(\alpha_l)$ is a quadratic form in $y$ which is positive semidefinite by assumption, which implies that it is a sum of squares (as a polynomial in $y$). Furthermore, as $\lambda_l \geq 0 ~\forall l$ and $(\alpha_l^Tx)^{2s-d}$ is an even power of a linear form, we have that $\lambda_l (\alpha_l^Tx)^{2s-d}$ is a sum of squares (as a polynomial in $x$). Combining both results, we get that (\ref{eq:proof.operator.sos}) is a sum of squares. 
\end{proof}

We now extend the concept introduced by Reznick in Proposition \ref{prop:def.Phi} to biforms. 
\begin{lemma} \label{lem:phi.def}
	For a biform $F(x;y)$ of the structure as in (\ref{eq:def.F}), we define the biform $\Psi_{s,x}(F(x;y))$ as
	\begin{align*}
	\Psi_{s,x}(F(x;y))\mathrel{\mathop{:}}=\sum_{i,j} y_iy_j \Phi_{s}(p_{ij}(x)),
	\end{align*}
	where $\Phi_s$ is as in (\ref{eq:def.Phi}). Define
	\begin{align*}
	\Psi_{s,x}^{-1}(F(x;y))\mathrel{\mathop{:}}=\sum_{i,j} y_iy_j \Phi_{s}^{-1}(p_{ij}(x)),
	\end{align*}
	where $\Phi_s^{-1}$ is the inverse of $\Phi_s$. Then, we have 
	\begin{align}
	F(D;y)(x_1^2+\ldots+x_n^2)^s=\Psi_{s,x} (F)(x_1^2+\ldots+x_n^2)^{s-d} \label{eq:def.Phi.s.x}
	\end{align}
	and 
	\begin{align}
	\Psi_{s,x}(\Psi_{s,x}^{-1}(F))=F. \label{eq:Phi.s.x.inv}
	\end{align}
\end{lemma}
\begin{proof}
	We start by showing that (\ref{eq:def.Phi.s.x}) holds:
	\begin{align}
	F(D;y)(x_1^2+\ldots+x_n^2)^s &= \sum_{i,j}y_i y_jp_{ij}(D)(x_1^2+\ldots+x_n^2)^s \nonumber \\
	&\underset{\text{using (\ref{eq:def.Phi})}}{=} \sum_{i,j} y_iy_j\Phi_s(p_{ij}(x))(x_1^2+\ldots x_n^2)^{s-d} \nonumber \\
	&=\Psi_{s,x}(F)(x_1^2+\ldots+x_n^2)^{s-d}. \nonumber
	\end{align}
	We now show that (\ref{eq:Phi.s.x.inv}) holds: $$\Psi_{s,x}(\Psi_{s,x}^{-1}(F))=\Psi_{s,x} \left( \sum_{i,j}y_iy_j \Phi_{s}^{-1}(p_{ij}(x))\right)=\sum_{i,j}y_iy_j \Phi_s \Phi_{s}^{-1}(p_{ij})=\sum_{i,j} y_iy_j p_{ij}=F.$$
\end{proof}

\begin{lemma} \label{lem:Phi.psd}
	For a biform $F(x;y)$ of the structure in (\ref{eq:def.F}), which is positive on the bisphere, let $$\eta(F)\mathrel{\mathop{:}}=\frac{\min\{F(x;y)~|~(x,y) \in S^{n-1} \times S^{n-1}\}}{\max\{F(x;y)~|~(x,y) \in S^{n-1} \times S^{n-1}\}}.$$
	If $s \geq \frac{nd(d-1)}{4\log(2)\eta(F)}-\frac{n-d}{2}$, then $\Psi_{s,x}^{-1}(F)$ is positive semidefinite.
\end{lemma}

\begin{proof} 
	Fix $y \in S^{n-1}$ and consider $F_y(x)=F(x;y)$, which is a positive definite form in $x$ of degree $d$. From Proposition \ref{prop:pd.Phi}, if $$s\geq \frac{nd(d-1)}{4\log(2)\epsilon(F_y)}-\frac{n-d}{2},$$ then $\Phi^{-1}_s(F_y)$ is positive semidefinite. As $\eta(F) \leq \epsilon(F_{y})$ for any $y \in S^{n-1}$, we have that if $$s \geq \frac{nd(d-1)}{4\log(2)\eta(F)}-\frac{n-d}{2},$$ then $\Phi_s^{-1}(F_y)$ is positive semidefinite, regardless of the choice of $y.$ Hence, $\Psi_{s,x}^{-1}(F)$ is positive semidefinite (as a function of $x$ and $y$).
	
\end{proof}

\begin{proof}[Proof of Theorem \ref{th:r.sos.convex}]
	Let $F(x;y)=y^TH_f(x)y$, let $r\geq \frac{n(d-2)(d-3)}{4\log(2)\eta(f)}-\frac{n+d-2}{2}-d$, and let $$G(x;y)=\Psi_{r+d,x}^{-1}(F).$$ We know by Lemma \ref{lem:Phi.psd} that $G(x;y)$ is positive semidefinite. Hence, using Lemma \ref{lem:F.sos}, we get that $$G(D,y)(x_1^2+\ldots+x_n^2)^{r+d}$$ is sos. Lemma \ref{lem:phi.def} then gives us: 
	\begin{align*}
	G(D;y)(x_1^2+\ldots+x_n^2)^{r+d}&\underset{\text{using } (\ref{eq:def.Phi.s.x})}{=}\Psi_{r+d,x}(G)
	(x_1^2+\ldots+x_n^2)^r\\
	&\underset{\text{using } (\ref{eq:Phi.s.x.inv})}{=}F(x;y)(x_1^2+\ldots+x_n^2)^r.
	\end{align*}
	
	As a consequence, $F(x;y)(x_1^2+\ldots+x_n^2)^r$ is sos.
	
\end{proof}
The last theorem of this section shows that one cannot bound the integer $r$ in Theorem \ref{th:r.sos.convex} as a function of $n$ and $d$ only.
\begin{theorem}\label{th:unif.bnd}
	For any integer $r \geq 0$, there exists a form $f$ in 3 variables and of degree 8 such that $H_f(x) \succ 0, \forall x \neq 0$, but $f$ is not $r$-sos-convex.
\end{theorem}
\begin{proof}
	Consider the trivariate octic:
	\begin{align*}
	f(x_1,x_2,x_3)&= 32x_1^8+118x_1^6x_2^2+40x_1^6x_3^2+25x_1^4x_2^2x_3^2-35x_1^4x_3^4+3x_1^2x_2^4x_3^2-16x_1^2x_2^2x_3^4+24x_1^2x_3^6\\
	&+16x_2^8+44x_2^6x_3^2+70x_2^4x_3^4+60x_2^2x_3^6+30x_3^8.
	\end{align*}
	It is shown in \cite{AAA_PP_not_sos_convex_journal} that $f$ has positive definite Hessian, and that the $(1,1)$ entry of $H_f(x)$, which we will denote by $H_f^{(1,1)}(x)$, is 1-sos but not sos. We will show that for any $r \in \mathbb{N}$, one can find $s \in \mathbb{N} \backslash \{0\}$ such that $$g_s(x_1,x_2,x_3)=f(x_1,sx_2,sx_3)$$ satisfies the conditions of the theorem.
	
	We start by showing that for any $s$, $g_s$ has positive definite Hessian. To see this, note that for any $(x_1,x_2,x_3) \neq 0, (y_1,y_2,y_3) \neq 0$, we have:
	$$(y_1, y_2,y_3) H_{g_s}(x_1,x_2,x_3)(y_1,y_2,y_3)^T=(y_1,sy_2,sy_3) H_f(x_1,sx_2,sx_3)(y_1,sy_2,sy_3)^T.$$
	As $y^TH_f(x)y>0$ for any $x \neq 0, y\neq 0$, this is in particular true when $x=(x_1, sx_2,sx_3)$ and when $y=(y_1, sy_2,sy_3)$, which gives us that the Hessian of $g_s$ is positive definite for any $s \in \mathbb{N} \backslash \{0\}.$
	
	We now show that for a given $r\in \mathbb{N}$, there exists $s \in \mathbb{N}$ such that $(x_1^2+x_2^2+x_3^2)^r y^TH_{g_s}(x)y$ is not sos. We use the following result from \cite[Theorem 1]{Reznick1}: for any positive semidefinite form $p$ which is not sos, and any $r\in \mathbb{N}$, there exists $s \in \mathbb{N} \backslash \{0\}$ such that $(\sum_{i=1}^n x_i^2)^r\cdot p(x_1,sx_2,\ldots,sx_n)$ is not sos. As $H_f^{(1,1)}(x)$ is 1-sos but not sos, we can apply the previous result. Hence, there exists a positive integer $s$ such that $$(x_1^2+x_2^2+x_3^2)^r \cdot H_f^{(1,1)}(x_1,sx_2,sx_3)=(x_1^2+x_2^2+x_3^2)^r \cdot H_{g_s}^{(1,1)}(x_1,x_2,x_3)$$
	is not sos. This implies that $(x_1^2+x_2^2+x_3^2)^r \cdot y^TH_{g_s}(x)y$ is not sos. Indeed, if $(x_1^2+x_2^2+x_3^2)^r \cdot y^TH_{g_s}(x)y$ was sos, then $(x_1^2+x_2^2+x_3^2)^r \cdot y^TH_{g_s}(x)y$ would be sos with $y=(1,0,0)^T.$ But, we have $$(x_1^2+x_2^2+x_3^2)^r \cdot (1,0,0)H_{g_s}(x)(1,0,0)^T=(x_1^2+x_2^2+x_3^2)^r \cdot H_{g_s}^{(1,1)}(x),$$
	which is not sos. Hence, $(x_1^2+x_2^2+x_3^2)^r \cdot y^TH_{g_s}(x)y$ is not sos, and $g$ is not $r$-sos-convex.
\end{proof}

\begin{remark}\label{rem:pd.hess.str.conv}
	Any form $f$ with $H_f(x) \succ 0, \forall x \neq 0$ is strictly convex but the converse is not true.
	
	To see this, note that any form $f$ of degree $d$ with a positive definite Hessian is convex (as $H_f(x)\succeq 0, \forall x$) and positive definite (as, from a recursive application of Euler's theorem on homogeneous functions, $f(x)=\frac{1}{d(d-1)}x^TH_f(x)x$). From the proof of Theorem \ref{th:norm.str.conv}, this implies that $f$ is strictly convex.
	
	To see that the converse statement is not true, consider the strictly convex form $f(x_1,x_2)\mathrel{\mathop{:}}=x_1^4+x_2^4$. We have $$H_f(x)=12\cdot \begin{bmatrix} x_1^2 & 0 \\ 0 & x_2^2 \end{bmatrix}$$ which is not positive definite e.g., when $x=(1,0)^T$.
\end{remark}

\subsubsection{Optimizing over a subset of polynomial norms with $r$-sos-convexity}

In the following theorem, we show how one can efficiently optimize over the set of forms $f$ with $H_f(x)\succ 0$, $\forall x \neq 0.$ Comparatively to Theorem \ref{th:test.poly.norm}, this theorem allows us to impose as a constraint that the $d^{th}$ root of a form be a norm, rather than simply testing whether it is. This comes at a cost however: in view of Remark \ref{rem:pd.hess.str.conv} and Theorem \ref{th:norm.str.conv}, we are no longer considering all polynomial norms, but a subset of them whose $d^{th}$ power has a positive definite Hessian.

\begin{theorem}\label{th:opt.poly.norms}
	Let $f$ be a degree-$d$ form. Then $H_f(x) \succ 0, \forall x\neq 0$ if and only if $\exists c>0, r \in \mathbb{N}$ such that $f(x)-c(\sum_i x_i^2)^{d/2}$ is $r$-sos-convex. Furthermore, this condition can be imposed using semidefinite programming.
\end{theorem}

\begin{proof}
	
	If there exist $c>0, r\in \mathbb{N}$ such that $g(x)=f(x)-c(\sum_i x_i^2)^{d/2}$ is $r$-sos-convex, then $y^TH_g(x)y \geq 0$, $\forall x,y.$ As the Hessian of $(\sum_i x_i^2)^{d/2}$ is positive definite for any nonzero $x$ and as $c>0$, we get $H_f(x)\succ 0$, $\forall x\neq 0.$

	Conversely, if $H_f(x)\succ 0$, $\forall x\neq 0$, then $y^TH_f(x)y>0$ on the bisphere (and conversely). Let $$f_{\min}\mathrel{\mathop{:}}= \min_{||x||=||y||=1} y^TH_f(x)y.$$ We know that $f_{\min}$ is attained and is positive. Take $c\mathrel{\mathop{:}}=\frac{f_{\min}}{2d(d-1)}$ and consider $$g(x)\mathrel{\mathop{:}}= f(x)-c(\sum_i x_i^2)^{d/2}.$$
	Then $$y^TH_g(x)y=y^TH_f(x)y-c \cdot \left( d(d-2)(\sum_i x_i^2)^{d/2-2} (\sum_i x_iy_i)^2 + d\sum_i (x_i^2)^{d/2-1} (\sum_i y_i^2) \right).$$
	Note that, by Cauchy-Schwarz, we have $(\sum_i x_i y_i)^2 \leq ||x||^2||y||^2$. If $||x||=||y||=1$, we get $$y^TH_g(x)y\geq y^TH_f(x)y-c(d(d-1))>0.$$
	Hence, $H_g(x)\succ 0, \forall x\neq 0$ and there exists $r$ such that $g$ is $r$-sos-convex from Theorem \ref{th:r.sos.convex}.

	For fixed $r$, the condition that there be $c>0$ such that $f(x)-c(\sum_i x_i^2)^{d/2}$ is $r$-sos-convex can be imposed using semidefinite programming. This is done by searching for coefficients of a polynomial $f$ and a real number $c$ such that 
	\begin{equation}\label{eq:SDP.opt.poly.norm}
	\begin{aligned}
	&y^TH_{f-c (\sum_i x_i^2)^{d/2}}y \cdot (\sum_i x_i^2)^r \text{ sos}\\
	&c\geq 0.
	\end{aligned} 
	\end{equation}
	Note that both of these conditions can be imposed using semidefinite programming.
\end{proof}

\begin{remark}
	Note that we are not imposing $c>0$ in the above semidefinite program. As mentioned in Section~\ref{sec:test}, this is because in practice the solution returned by interior point solvers will be in the interior of the feasible set.
	
	In the special case where $f$ is completely free\footnote{This is the case of our two applications in Section \ref{sec:apps}.} (i.e., when there are no additional affine conditions on the coefficients of $f$), one can take $c \geq 1$ in (\ref{eq:SDP.opt.poly.norm}) instead of $c \geq 0$. Indeed, if there exists $c>0$, an integer $r$, and a polynomial $f$ such that $f -c(\sum_i x_i^2)^{d/2}$ is $r$-sos-convex, then $\frac1c f$ will be a solution to (\ref{eq:SDP.opt.poly.norm}) with $c \geq 1$ replacing $c \geq 0$.
\end{remark}

%\begin{openprob}
%	A set is set to be SDP-representable if it can be written as the projection of a Linear Matrix Inequality representation \cite{Helton_Nie_SDP_repres}. Is the 1-sublevel set of any polynomial norm SDP-representable?
%\end{openprob}
%
%It is clear that the 1-sublevel set of any polynomial norm where the Hessian of the $d^{th}$ power is positive definite is SDP-representable. In \cite{Helton_Nie_SDP_repres_2}, Helton and Nie further show that any set with nonsingular positively curved boundary has an SDP representation. This is a weaker statement than requiring that the set be strictly convex (which corresponds to the case at hand) and to our knowledge, the previous problem remains open.

%======================================================
\section{Applications} \label{sec:apps}
%===================================================

\subsection{Norm approximation and regression}\label{sec:norm.reg}

In this section, we study the problem of approximating a (non-polynomial) norm by a polynomial norm. We consider two different types of norms: $p$-norms with $p$ noneven (and greater than 1) and gauge norms with a polytopic unit ball. For $p$-norms, we use as an example $||(x_1,x_2)^T||=(|x_1|^{7.5}+|x_2|^{7.5})^{1/7.5}$. For our polytopic gauge norm, we randomly generate an origin-symmetric polytope and produce a norm whose 1-sublevel corresponds to that polytope. This allows us to determine the value of the norm at any other point by homogeneity (see \cite[Exercise 3.34]{BoydBook} for more information on gauge norms, i.e., norms defined by convex, full-dimensional, origin-symmetric sets). To obtain our approximations, we proceed in the same way in both cases. We first sample $N=200$ points $x_1,\ldots,x_N$ on the sphere $S$ that we denote by $x_1,\ldots,x_{N}$. We then solve the following optimization problem with $d$ fixed:
\begin{equation}\label{eq:opt.norm.reg}
\begin{aligned}
&\min_{f\in H_{2,d}} &&\sum_{i=1}^{N} (||x_i||^d-f(x_i))^2\\
&\text{s.t. } &&f \text{ sos-convex}.
\end{aligned}
\end{equation}
Problem (\ref{eq:opt.norm.reg}) can be written as a semidefinite program as the objective is a convex quadratic in the coefficients of $f$ and the constraint has a semidefinite representation as discussed in Section \ref{sec:sos.review}. The solution $f$ returned is guaranteed to be convex. Moreover, any sos-convex form is sos (see \cite[Lemma 8]{helton2010}), which implies that $f$ is nonnegative. One can numerically check to see if the optimal polynomial is in fact positive definite (for example, by checking the eigenvalues of the Gram matrix of a sum of squares decomposition of $f$). If that is the case, then, by Theorem \ref{th:norm.conv.pd}, $f^{1/d}$ is a norm. Futhermore, note that we have
\begin{align*}
\left(\sum_{i=1}^N (||x_i||^d-f(x_i))^2 \right)^{1/d} &\geq \frac{N^{1/d}}{N} \sum_{i=1}^N (||x_i||^d-f(x_i))^{2/d}\\
&\geq \frac{N^{1/d}}{N} \sum_{i=1}^N (||x_i||-f^{1/d}(x_i))^2,
\end{align*}
where the first inequality is a consequence of concavity of $z \mapsto z^{1/d}$ and the second is a consequence of the inequality $|x-y|^{1/d} \geq ||x|^{1/d}-|y|^{1/d}|$. This implies that if the optimal value of (\ref{eq:opt.norm.reg}) is equal to $\epsilon$, then the sum of the squared differences between $||x_i||$ and $f^{1/d}(x_i)$ over the sample is less than or equal to $N \cdot (\frac{\epsilon}{N})^{1/d}$.

It is worth noting that in our example, we are actually searching over the entire space of polynomial norms of a given degree. Indeed, as $f$ is bivariate, it is convex if and only if it is sos-convex \cite{ahmadi2013complete}. In Figure \ref{fig:norm.approx}, we have drawn the 1-level sets of the initial norm (either the $p$-norm or the polytopic gauge norm) and the optimal polynomial norm obtained via (\ref{eq:opt.norm.reg}) with varying degrees $d$. Note that when $d$ increases, the approximation improves.

\begin{figure}[h]
	\begin{center}
		\mbox{
			\subfigure[p-norm approximation]

			{\label{fig:p.norm.approx}\scalebox{0.25}{\includegraphics{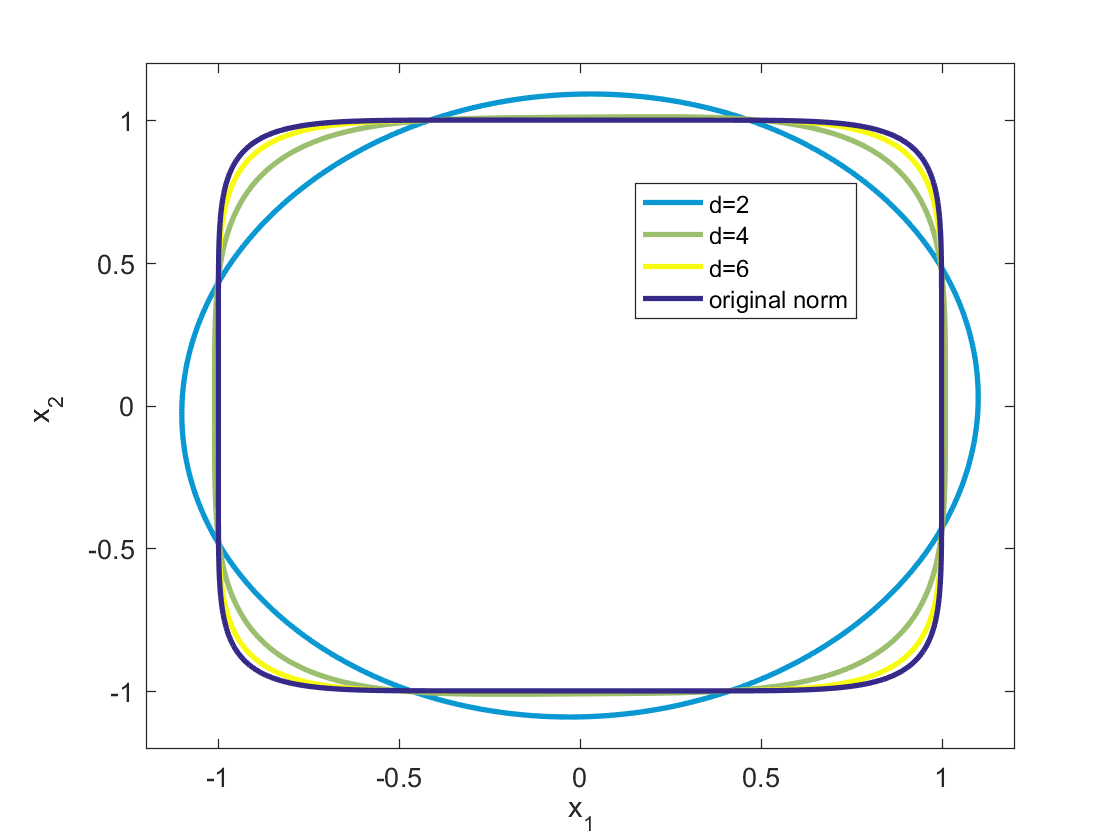}}}}
		\mbox{
			\subfigure[Polytopic norm approximation]
			{\label{fig:polytop.approx}\scalebox{0.25}{\includegraphics{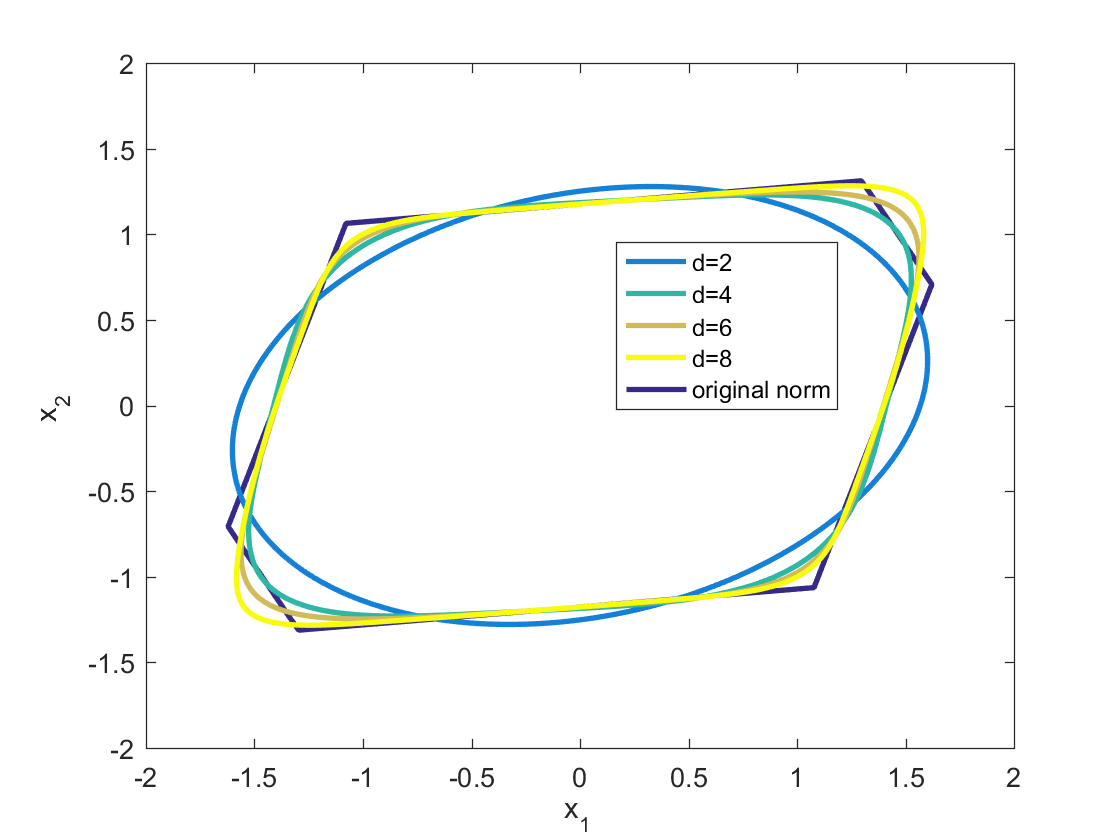}}}
		}
		\caption{Approximation of non-polynomial norms by polynomial norms}
		\label{fig:norm.approx}
	\end{center}
\end{figure} 

%\begin{figure}[h]
%	\centering
%	\begin{subfigure}{0.49\textwidth}
%		\includegraphics[width=\textwidth]{ch-polynorms/pnorm_regression_higher_degree_v2}
%		\caption{p-norm approximation}
%	\end{subfigure}
%	~ %add desired spacing between images, e. g. ~, \quad, \qquad, \hfill etc. 
%	%(or a blank line to force the subfigure onto a new line)
%	\begin{subfigure}{0.49\textwidth}
%		\includegraphics[width=\textwidth]{ch-polynorms/polytopic_norm_ex_2_v1}
%		\caption{Polytopic norm approximation}
%	\end{subfigure}
%	\caption{Approximation of non-polynomial norms by polynomial norms}\label{fig:norm.approx}
%\end{figure}

A similar method could be used for \emph{norm regression}. In this case, we would have access to data points $x_1,\ldots,x_N$ corresponding to noisy measurements of an underlying unknown norm function. We would then solve the same optimization problem as the one given in (\ref{eq:opt.norm.reg}) to obtain a polynomial norm that most closely approximates the noisy data.

\subsection{Joint spectral radius and stability of linear switched systems}\label{sec:JSR.comp}

As a second application, we revisit a result from Ahmadi and Jungers from \cite{sosconvex_Lyap_cdc,Ahmadi_Jungers} on upperbounding the joint spectral radius of a finite set of matrices. We first review a few notions relating to dynamical systems and linear algebra.
The spectral radius $\rho$ of a matrix $A$ is defined as $$\rho(A)=\lim_{k \rightarrow \infty} ||A^k||^{1/k}.$$ The spectral radius happens to coincide with the eigenvalue of $A$ of largest magnitude. Consider now the discrete-time linear system $x_{k+1}=Ax_k$, where $x_k$ is the $n \times 1$ state vector of the system at time $k$. This system is said to be \emph{asymptotically stable} if for any initial starting state $x_0 \in \mathbb{R}^n$, $x_k \rightarrow 0,$ when $k \rightarrow \infty.$ A well-known result connecting the spectral radius of a matrix to the stability of a linear system states that the system $x_{k+1}=Ax_k$ is asymptotically stable if and only if $\rho(A)<1$. 

In 1960, Rota and Strang introduced a generalization of the spectral radius to a \emph{set} of matrices. The \emph{joint spectral radius (JSR)} of a set of matrices $\mathcal{A} \mathrel{\mathop{:}}=\{A_1,\ldots,A_m\}$ is defined as 
\begin{align}\label{eq:JSR.def}
\rho(\mathcal{A})\mathrel{\mathop{:}}=\lim_{k\rightarrow \infty} \max_{\sigma \in \{1,\ldots,m\}^k}||A_{\sigma_k} \ldots A_{\sigma_1}||^{1/k}.
\end{align}
Analogously to the case where we have just one matrix, the value of the joint spectral radius can be used to determine stability of a certain type of system, called \emph{a switched linear system.} A switched linear system models an uncertain and time-varying linear system, i.e., a system described by the dynamics
$$x_{k+1}=A_{k}x_k,$$
where the matrix $A_k$ varies at each iteration within the set $\mathcal{A}$. As done previously, we say that a switched linear system is asymptotically stable if $x_k \rightarrow \infty$ when $k \rightarrow \infty$, for any starting state $x_0 \in \mathbb{R}^n$ and any sequence of products of matrices in $\mathcal{A}$. One can establish that the switched linear system $x_{k+1}=A_{k}x_k$ is asymtotically stable if and only if $\rho(\mathcal{A})<1$ \cite{Raphael_Book}.

Though they may seem similar on many points, a key difference between the spectral radius and the joint spectral radius lies in difficulty of computation: testing whether the spectral radius of a matrix $A$ is less than equal (or strictly less) than $1$ can be done in polynomial time. However, already when $m=2$, the problem of testing whether $\rho(A_1,A_2)\leq 1$ is undecidable \cite{BlTi2}. An active area of research has consequently been to obtain sufficient conditions for the JSR to be strictly less than one, which, for example, can be checked using semidefinite programming. The theorem that we revisit below is a result of this type. We start first by recalling a Theorem linked to stability of a linear system. 
\begin{theorem}[see, e.g., Theorem 8.4 in \cite{hespanha2009linear}]\label{th:AAA.Jungers}
	Let $A \in \mathbb{R}^{n \times n}$. Then, $\rho(A)<1$ if and only if there exists a contracting quadratic norm; i.e., a function $V:\mathbb{R}^n \rightarrow \mathbb{R}$ of the form $V(x)=\sqrt{x^TQx}$ with $Q\succ 0$, such that $V(Ax)<V(x), \forall x\neq 0.$
\end{theorem}

The next theorem (from \cite{sosconvex_Lyap_cdc,Ahmadi_Jungers}) can be viewed as an extension of Theorem \ref{th:AAA.Jungers} to the joint spectral radius of a finite set of matrices. It is known that the existence of a contracting quadratic norm is no longer necessary for stability in this case. This theorem show however that the existence of a contracting polynomial norm is.

\begin{theorem}[adapted from \cite{sosconvex_Lyap_cdc,Ahmadi_Jungers}, Theorem 3.2 ]\label{th:poly.norms.JSR}
	Let $\mathcal{A}\mathrel{\mathop{:}}=\{A_1,\ldots,A_m\}$ be a family of $n \times n$ matrices. Then, $\rho(A_1,\ldots,A_m)<1$ if and only if there exists a contracting polynomial norm; i.e., a function $V(x)=f^{1/d}(x)$, where $f$ is an n-variate convex and positive definite form of degree $d$, such that $V(A_ix)<V(x),~\forall x\neq 0$ and $\forall i=1,\ldots,m.$
\end{theorem}

We remark that in \cite{ahmadi2016lower}, Ahmadi and Jungers show that the degree of $f$ cannot be bounded as a function of $m$ and $n$. This is expected from the undecidability result mentioned before.

\begin{example}
	We consider a modification of Example 5.4. in \cite{ahmadi2011analysis} as an illustration of the previous theorem. We would like to show that the joint spectral radius of the two matrices $$A_1=\frac{1}{3.924}\begin{bmatrix} -1 & -1 \\ 4 & 0 \end{bmatrix}, \quad A_2=\frac{1}{3.924}\begin{bmatrix} 3 & 3 \\ -2 & 1 \end{bmatrix}$$ is strictly less that one.
	
	To do this, we search for a nonzero form $f$ of degree $d$ such that 
	\begin{equation} \label{eq:JSR.opt.prob}
	\begin{aligned}
	&f -(\sum_{i=1}^n x_i^2)^{d/2}\text{ sos-convex}\\
	&f(x)-f(A_ix)- (\sum_{i=1}^n x_i^2)^{d/2} \text{ sos}, \text{ for } i=1,2.
	\end{aligned}
	\end{equation}
	If problem (\ref{eq:JSR.opt.prob}) is feasible for some $d$, then $\rho(A_1,A_2)<1$. A quick computation using the software package YALMIP \cite{yalmip} and the SDP solver MOSEK \cite{mosek} reveals that, when $d=2$ or $d=4$, problem (\ref{eq:JSR.opt.prob}) is infeasible. When $d=6$ however, the problem is feasible and we obtain a polynomial norm $V=f^{1/d}$ whose 1-sublevel set is the outer set plotted in Figure \ref{fig:level.sets}. We also plot on Figure \ref{fig:level.sets} the images of this 1-sublevel set under $A_1$ and $A_2$. Note that both sets are included in the 1-sublevel set of $V$ as expected. From Theorem \ref{th:poly.norms.JSR}, the existence of a polynomial norm implies that $\rho(A_1,A_2)<1$ and hence, the pair $\{A_1,A_2\}$ is asymptotically stable.

	\begin{figure}[h]
		\centering
		\includegraphics[scale=0.3]{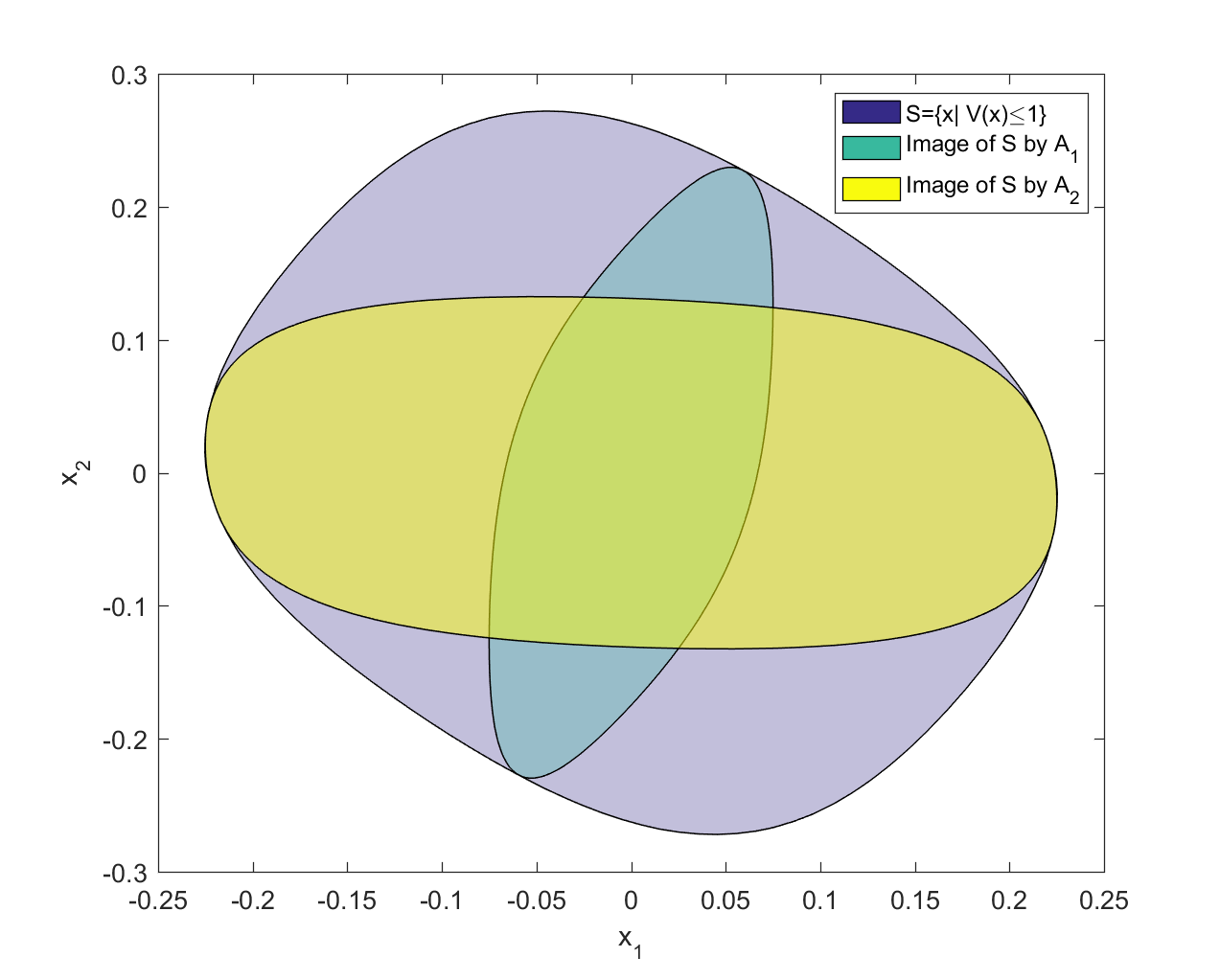}
		\caption{Image of the sublevel set of $V$ under $A_1$ and $A_2$}
		\label{fig:level.sets}
	\end{figure}

\end{example}

\begin{remark}
	As mentioned previously, problem (\ref{eq:JSR.opt.prob}) is infeasible for $d=4$. Instead of pushing the degree of $f$ up to 6, one could wonder whether the problem would have been feasible if we had asked that $f$ of degree $d=4$ be $r$-sos-convex for some fixed $r \geq 1$. As mentioned before, in the particular case where $n=2$ (which is the case at hand here), the notions of convexity and sos-convexity coincide; see \cite{ahmadi2013complete}. As a consequence, one can only hope to make problem (\ref{eq:JSR.opt.prob}) feasible by increasing the degree of $f$.
\end{remark}

%==============================================
\section{Future directions}\label{sec:open.problem}
%=============================================

In this chapter, we provided semidefinite programming-based conditions under which we could test whether the $d^{th}$ root of a degree-$d$ form is a polynomial norm (Section \ref{sec:test}), and semidefinite programming-based conditions under which we could optimize over the set of forms with positive definite Hessians (Section \ref{sec:opt}). A clear gap emerged between forms which are strictly convex and those which have a positive definite Hessian, the latter being a sufficient (but not necessary) condition for the former. This leads us to consider the following two open problems.

\begin{openprob}
	We have given a semidefinite programming hierarchy for optimizing over a subset of polynomial norms. Is there a semidefinite programming hierarchy that optimizes over all polynomial norms? 
\end{openprob}

\begin{openprob}
	Helton and Nie have shown in \cite{helton2010} that sublevel sets of forms that have positive definite Hessians are SDP-representable. This means that we can optimize linear functions over these sets using semidefinite programming. Is the same true for sublevel sets of all polynomial norms?
\end{openprob}

%Another open problem that we think would be of interest relates to the quality of approximation of norms by polynomial norms.
%
%\begin{openprob}
%	As mentioned previously, Barvinok has shown in \cite{barvinok2003} that for any norm $||.||$, there exists a nonnegative form $f$ of degree $2d$ such that 
%	$$f^{1/2d}(x) \leq ||x|| \leq \binom{n+d-1}{d}^{1/2d} f^{1/2d}(x).$$ We have shown in Section \ref{sec:approx.norms} that for any $\epsilon>0$, there exists an sos-convex and positive definite polynomial $f$ of degree $d$ such that $$f^{1/2d}(x) \leq ||x|| \leq (1+\epsilon) f^{1/2d}(x).$$ Is it possible to quantify the degree $d$ needed to obtain an approximation of precision $\epsilon$ as a function of $\epsilon$ and $n$ only? How would the degree be impacted if we searched for an $r$-sos-convex polynomial instead? In this case, could we express $d$ as a function of $\epsilon$, $n$, and $r$ only?
%\end{openprob}
On the application side, it might be interesting to investigate how one can use polynomial norms to design \emph{regularizers} in machine learning applications. Indeed, a very popular use of norms in optimization is as regularizers, with the goal of imposing additional structure (e.g., sparsity or low-rankness) on optimal solutions. One could imagine using polynomial norms to design regularizers that are based on the data at hand in place of more generic regularizers such as the 1-norm. Regularizer design is a problem that has already been considered (see, e.g., \cite{bach2012optimization,venkat}) but not using polynomial norms. This can be worth exploring as we have shown that polynomial norms can approximate any norm with arbitrary accuracy, while remaining differentiable everywhere (except at the origin), which can be beneficial for optimization purposes.

\chapter{Geometry of 3D Environments and Sum of Squares Polynomials}\label{ch:vikas}
\begin{figure}
		%\begin{figure*}[!ht]
		\centering
		\includegraphics[width=0.3\linewidth]{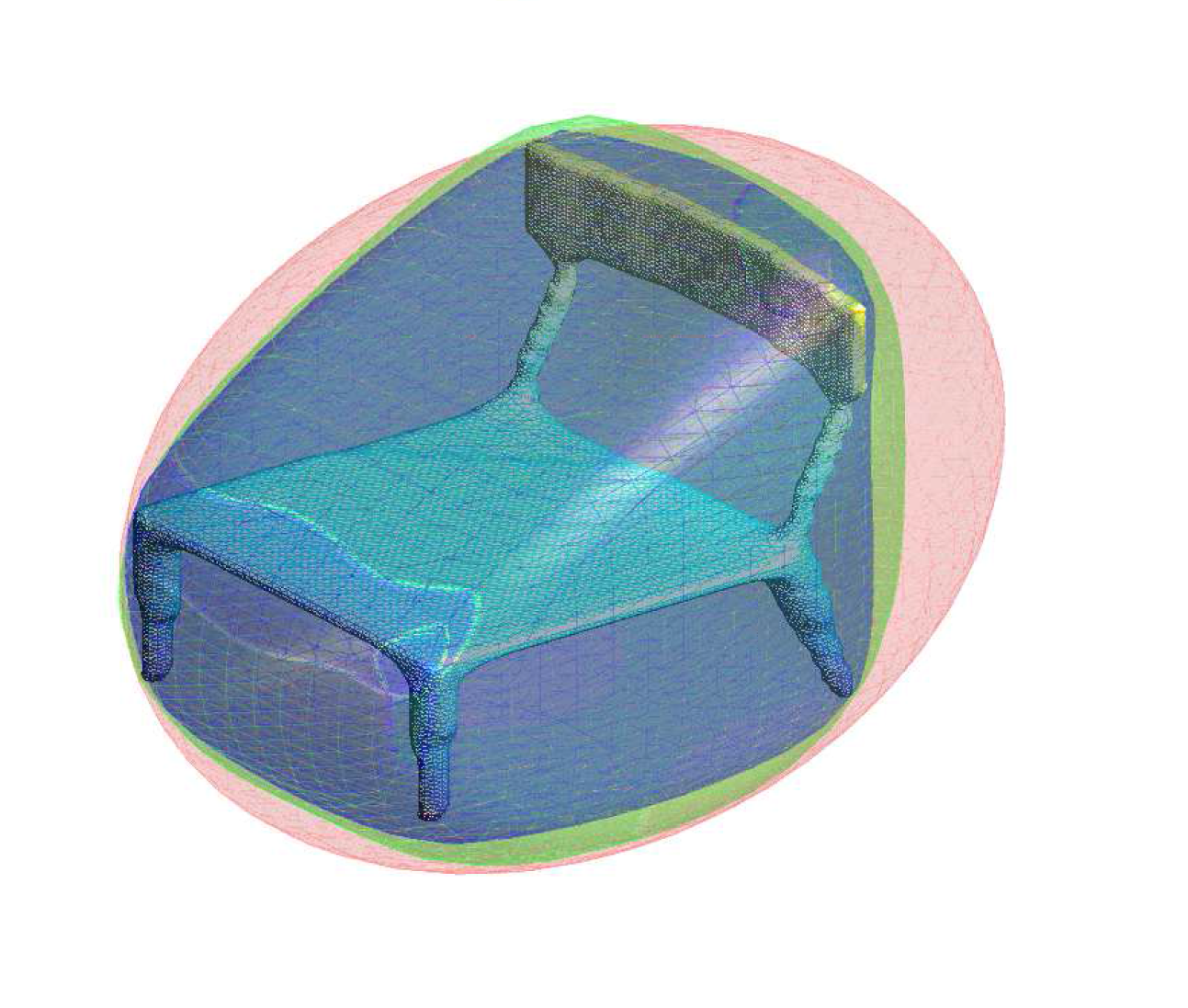}
		\includegraphics[width=0.3\linewidth]{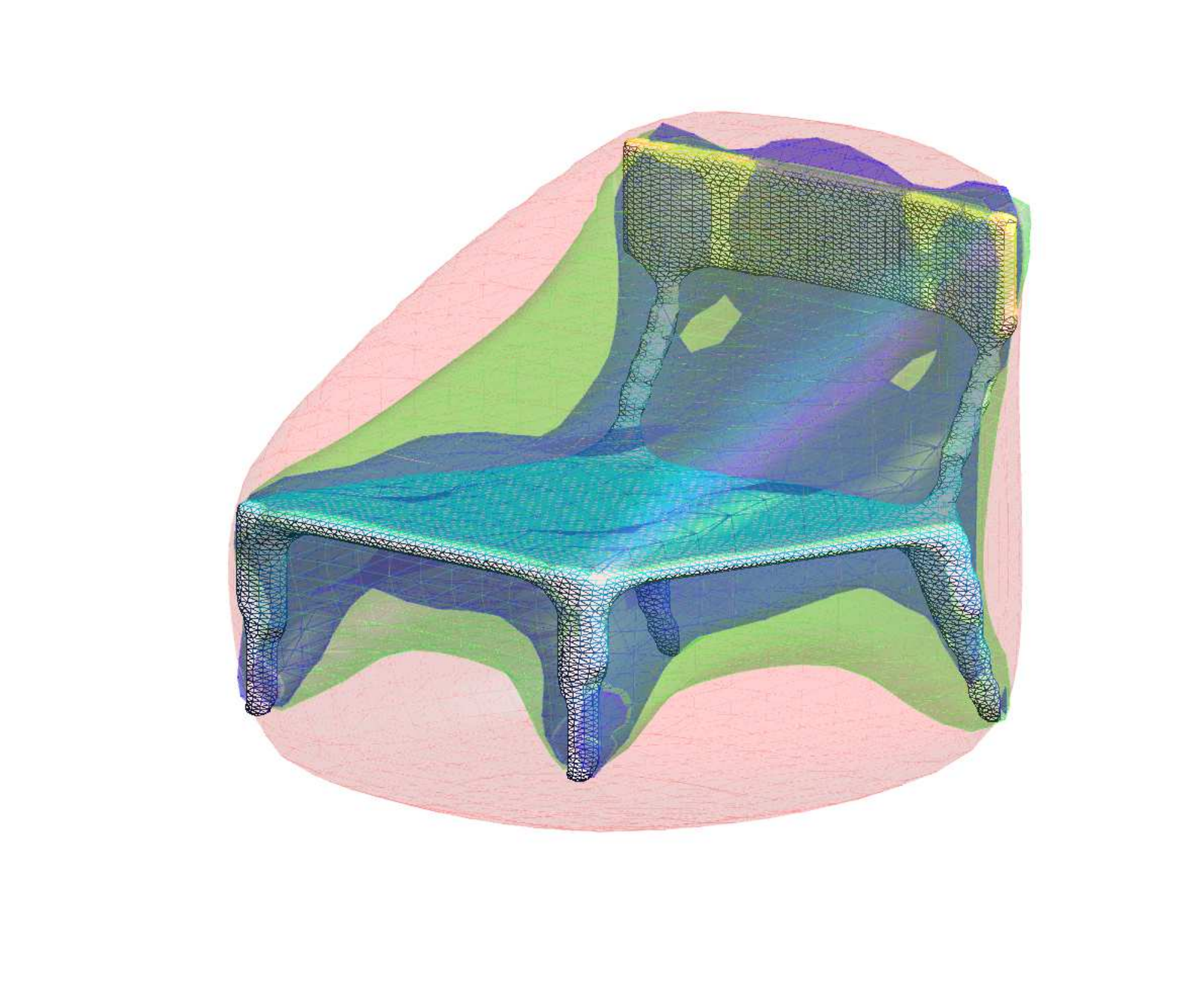}
		\includegraphics[width=0.3\linewidth]{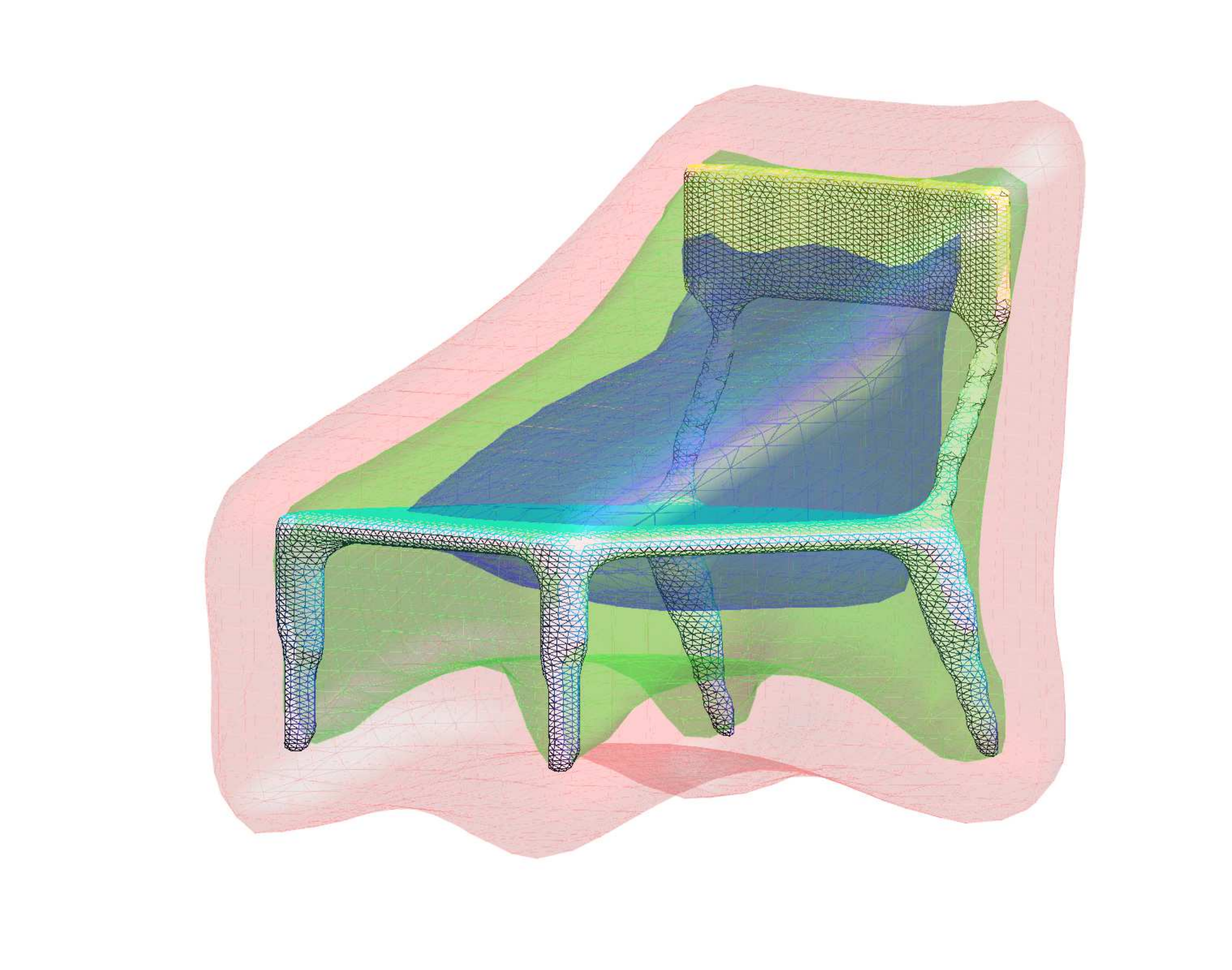}
		\caption{Sublevel sets of sos-convex polynomials of increasing degree (left); sublevel sets of sos polynomials of increasing nonconvexity (middle); growth and shrinkage of an sos-body with sublevel sets (right)}
		\label{fig:intro_pic}
		%\end{figure*}
	\end{figure}%
%}% ... an image
%
%
%%}%]

\section{Introduction}
A central problem in robotics, computer graphics, virtual and augmented reality (VR/AR), and many applications involving complex physics simulations is the  accurate, real-time determination of proximity relationships between three-dimensional objects~\cite{EricsonBook} situated in a cluttered environment. In robot navigation and manipulation tasks, path planners need to compute a dynamically feasible trajectory connecting an initial state to a goal configuration while avoiding obstacles in the environment. In VR/AR applications, a human immersed in a virtual world may wish to touch computer generated objects that must respond to contacts in physically realistic ways. Likewise, when collisions are detected, 3D gaming engines and physics simulators (e.g., for molecular dynamics) need to activate appropriate directional forces on interacting entities. All of these applications require geometric notions of separation and penetration between  representations of three-dimensional objects to be continuously monitored.

A rich class of computational geometry problems arises in this context, when 3D objects are outer approximated by convex or nonconvex bounding volumes~\cite{gottschalk1996obbtree,manocha2004collision,larsen2000fast}. In the case where the bounding volumes are convex, the Euclidean distance between them can be computed very precisely, providing a reliable certificate of safety for the objects they enclose. In the case where the bounding volumes are nonconvex, distance computation can be done either approximately via convex decomposition heuristics~\cite{ConvexDecomposition1,ConvexDecomposition2} which cover the volumes by a finite union of convex shapes, or exactly by using more elaborate algebraic optimization hierarchies that we discuss in this chapter. When 3D objects overlap, quantitative measures of degree of penetration are needed in order to optimally resolve collisions, e.g., by a gradient-based trajectory optimizer. Multiple such measures have been proposed in the literature. The {\it penetration depth} is the minimum magnitude translation that brings the overlapping objects out of collision.  The {\it growth distance}~\cite{GrowthDistance} is the minimum shrinkage of the two bodies required to reduce volume penetration down to merely surface touching. Efficient computation of penetration measures is also a problem of interest to this chapter.

\subsection{Contributions and organization of the chapter}

In this work, we propose to represent the geometry of a given 3D environment comprising multiple static or dynamic rigid bodies using sublevel sets of  polynomials. The chapter is organized as follows: In Section \ref{sec:sos.convex}, we provide an overview of the algebraic concepts of sum of squares (sos) and sum of squares-convex (sos-convex) polynomials as well as their relation to semidefinite programming and polynomial optimization. In Section \ref{sec:3D.point.cont}, we consider the problem of containing a cloud of 3D points with tight-fitting convex or nearly convex sublevel sets of polynomials. In particular, we propose and justify a new volume minimization heuristic for these sublevel sets which empirically results in tighter fitting polynomials than previous proposals~\cite{Magnani},~\cite{lasserre2016_inverse_moment}. Additionally, we give a procedure for explicitly tuning the extent of convexity imposed on these sublevel set bounding volumes using sum of squares optimization techniques. If convexity is imposed, we refer to them as {\it sos-convex bodies}; if it is not, we term them simply as {\it sos-bodies}. (See Section~\ref{sec:sos.convex} for a more formal definition.) We show that the bounding volumes we obtain are highly compact and adapt to the shape of the data in more flexible ways than canned convex primitives typically used in standard bounding volume hierarchies; see Table~\ref{tab:bounding.volumes}. The construction of our bounding volumes involves small-scale semidefinite programs (SDPs) that can fit, in an offline preprocessing phase, 3D meshes with tens of thousands of data points in a few seconds. In Section \ref{sec:distance}, we give sum of squares algorithms for measuring notions of separation or penetration, including  Euclidean distance and growth distance~\cite{GrowthDistance}, of two bounding volumes representing obstacles. We show that even when convexity is lacking, we can efficiently compute (often tight) lower bounds on these measures.
%In particular, one can naturally define a ``growth distance"~\cite{GrowthDistance} for these objects.
In Section~\ref{sec:cont.poly.sub}, we consider the problem of grouping several obstacles (i.e., bounding volumes) within one, with the idea of making a map of the 3D environment with a lower level of resolution. A semidefinite programming based algorithm for this purpose is proposed and demonstrated via an example. 
%We end in Section~\ref{sec:conclusions} with some future directions. 

\subsection{Preview of some experiments} Figure \ref{fig:intro_pic} gives a preview of some of the methods developed in this chapter using as an example a 3D chair point cloud. On the left, we enclose the chair within the 1-sublevel set of three sos-convex polynomials with increasing degree ($2$, $4$ and $6$) leading to correspondingly tighter fits. The middle plot presents the 1-sublevel set of three degree-6 sos polynomials with increasing nonconvexity showing how tighter representations can be obtained by relaxing convexity. The right plot shows the 2, 1, and 0.75 sublevel sets of a single degree-6 sos polynomial; the 1-sublevel set colored green encloses the chair, while greater or lower values of the level set define grown and shrunk versions of the object. The computation of Euclidean distances and sublevel-based measures of separation and penetration can be done in a matter of milliseconds with techniques described in this chapter. % between such bodies is a tiny convex optimization problem that can be solved in a matter of milliseconds.

%Figure 1 shows bounding volumes enclosing a 3D chair point cloud. On the left, the 1-sublevel set of three sos-convex bodies is shown, enclosing the chair more tightly with increasing degree ($2$, $4$ and $6$). The middle plot shows the 1-sublevel set of three sos-bodies with increasing non-convexity showing how tighter representations are estimated by relaxing convexity.  The right plot shows the 2, 1 and 0.75 sublevel sets of a single sos-body; the 1-sublevel set colored green encloses the chair, while greater or lower values of the level set define grown and shrunk versions of the object. The computation of Euclidean distances, and sublevel-based measures of separation and penetration between such bodies is a tiny convex optimization problem that can be solved in a matter of milliseconds.

% We are given as input, an environment comprising of multiple static or dynamic three-dimensional rigid bodies. The moving objects represent either obstacles whose movement cannot be influenced, or a body (i.e., robots) whose motion can be explicitly controlled. The initial configuration of each object is described as a raw point cloud in $\reals^3$, which can be processed offline. The goal is to compute at each time point whether a pair of objects is colliding or not, and report associated separation and penetration distances, which may be fed into a trajectory optimizer for finding collision-free paths, or used by a physics/gaming engine to switch between interaction modes. 

\section{Sum of squares and sos-convexity}\label{sec:sos.convex}

In this section, we briefly review the notions of \emph{sum of squares polynomials}, \emph{sum of squares-convexity,} and \emph{polynomial optimization} which will all be central to the geometric problems we discuss later. We refer the reader to the recent monograph~\cite{lasserreBook} for a more detailed overview of the subject. 

Throughout, we will denote the set of $n \times n$ symmetric matrices by $S^{n \times n}$ and the set of degree-$2d$ polynomials with real coefficients by $\mathbb{R}_{2d}[x]$. We say that a polynomial $p(x_1,\ldots,x_n) \in \mathbb{R}_{2d}[x]$ is \emph{nonnegative} if $p(x_1,\ldots,x_n)\geq 0, \forall x\in\mathbb{R}^n$. In many applications (including polynomial optimization that we will cover later), one would like to constrain certain coefficients of a polynomial so as to make it nonnegative. Unfortunately, even testing whether a given polynomial (of degree $2d\geq 4$) is nonnegative is NP-hard. As a consequence, we would like to replace the intractable condition that $p$ be nonnegative by a sufficient condition for it that is more tractable. One such condition is for the polynomial to have a sum of squares decomposition. We say that a polynomial $p$ is a \emph{sum of squares (sos)} if there exist polynomials $q_i$ such that $p=\sum_{i} q_i^2$. From this definition, it is clear that any sos polynomial is nonnegative, though not all nonnegative polynomials are sos; see, e.g., \cite{Reznick},\cite{laurent2009sums} for some counterexamples. Furthermore, requiring that a polynomial $p$ be sos is a computationally tractable condition as a consequence of the following characterization: A polynomial $p$ of degree $2d$ is sos if and only if there exists a positive semidefinite matrix $Q$ such that $p(x)=z(x)^TQz(x),$ where $z(x)$ is the vector of all monomials of degree up to $d$ \cite{PhD:Parrilo}. The matrix $Q$ is sometimes called the Gram matrix of the sos decomposition and is of size $\binom{n+d}{d}\times \binom{n+d}{d}$. (Throughout the chapter, we let $N\mathcal{\mathop{:}}=\binom{n+d}{d}.$) The task of finding a positive semidefinite matrix $Q$ that makes the coefficients of $p$ all equal to the coefficients of $z(x)^TQz(x)$ is a semidefinite programming problem, which can be solved in polynomial time to arbitrary accuracy~\cite{vandenberghe1996semidefinite}.

% {\gh As semidefinite programs can be slow to solve when $n$ and $d$ get big, stronger (but cheaper) sufficient conditions for nonnegativity of polynomials have been given which can be tested using linear programming or second order programming instead of semidefinite programming. In \cite{iSOS_journal}, Ahmadi and Majumdar introduced the concept of \emph{dsos} (resp. \emph{sdsos}) polynomials. A polynomial $p$ is said to be dsos (resp. sdsos) if it can be written as $p(x)=z(x)^TQz(x),$ where $Q$ is diagonally dominant (resp. scaled diagonally dominant). Enforcing diagonal dominance (resp. scaled diagonal dominance) of a matrix can de done using linear programming (LP) (resp. second order programming (SOCP)), thus ensuring that optimizing over the set of dsos (resp. sdsos) polynomials is a linear program (resp. second order cone program). }

%This lemma states that if we wish to obtain a sum of squares decomposition of a fixed polynomial $p$, we need to search for a positive semidefinite matrix $Q$ whose entries are such that the coefficients of the polynomial $p$ and the polynomial $z(x)^TQz(x)$ are equal -- this is a semidefinite program and can be solved in polynomial time up to arbitrary accuracy~\cite{vandenberghe1996semidefinite}.

The concept of sum of squares can also be used to define a sufficient condition for convexity of polynomials known as \emph{sos-convexity}. We say that a polynomial $p$ is sos-convex if the polynomial $y^T \nabla^2 p(x)y$ in $2n$ variables $x$ and $y$ is a sum of squares. Here, $\nabla^2 p(x)$ denotes the Hessian of $p$, which is a symmetric matrix with polynomial entries.
%is a sum of squares in $x$ and $y$, where $\nabla^2 p(x)$ is the Hessian of $p$.
For a polynomial of degree $2d$ in $n$ variables, one can check that the dimension of the Gram matrix associated to the sos-convexity condition is $\tilde{N}\mathcal{\mathop{:}}=n \cdot \binom{n+d-1}{d-1}$. It follows from the second order characterization of convexity that any sos-convex polynomial is convex, as $y^T\nabla^2 p(x)y$ being sos implies that $\nabla^2 p(x) \succeq 0, ~\forall x.$ The converse however is not true, though convex but not sos-convex polynomials are hard to find in practice; see \cite{ahmadi2013complete}. Through its link to sum of squares, it is easy to see that testing whether a given polynomial is sos-convex is a semidefinite program. By contrast, testing whether a polynomial of degree $2d \geq 4$ is convex is NP-hard \cite{NPhard_Convexity_MathProg}. 

%{\gh Note that much as one could define LP and SOCP-based analogs to sum of squares polynomials using dsos and sdsos polynomials, one can define dsos-convex and sdsos-convex polynomials which are cheaper (but less exact) inner approximations of the set of convex polynomials (see \cite{DCP}, Section 3, for a more detailed overview).}

A \emph{polynomial optimization problem} is a problem of the form
\begin{align}
\min_{x \in K} p(x), \label{eq:basic.opt}
\end{align}
where the objective $p$ is a (multivariate) polynomial and the feasible set $K$ is a basic semialgebraic set; i.e., a set defined by polynomial inequalities: $$K:=\{x~|~g_i(x)\geq 0, i=1,\ldots,m\}.$$

%One area where the concepts of sum of squares and sos-convexity are widely used is \emph{polynomial optimization.} Let $K:=\{x~|~g_i(x)\geq 0, i=1,\ldots,m\},$ where $p$ and $g_i$ are multivariate polynomials, be a basic semi algebraic set, i.e., a set defined by polynomial inequalities.  A polynomial optimization problem is then an optimization problem of the form 
%\begin{align}
%   \min_{x \in K} p(x). \label{eq:basic.opt}
%\end{align}

It is straightforward to see that problem (\ref{eq:basic.opt}) can be equivalently formulated as that of finding the largest constant $\gamma$ such that  $p(x)-\gamma\geq 0,\forall x\in K.$ It is known that, under mild conditions (specifically, under the assumption that $K$ is Archimedean \cite{laurent2009sums}), the condition $p(x)-\gamma > 0, \forall x \in K$, is equivalent to the existence of sos polynomials $\sigma_i(x)$ such that $p(x)-\gamma=\sigma_0(x)+\sum_{i=1}^m \sigma_i(x) g_i(x)$. Indeed, it is at least clear that if $x \in K$, i.e., $g_i(x)\geq 0$, then $\sigma_0(x)+\sum_{i=1}^m \sigma_i(x)g_i(x) \geq 0$ which means that $p(x)-\gamma \geq 0$. The converse is less trivial and is a consequence of the Putinar Positivstellensatz \cite{putinar1993positive}. Using this result, problem (\ref{eq:basic.opt}) can be rewritten as
\begin{align} 
&\underset{\gamma, \sigma_i}{\max}~ \gamma \nonumber \\
&\text{s.t. } p(x)-\gamma=\sigma_0+\sum_{i=1}^m \sigma_i(x)g_i(x),\label{eq:basic.opt.sos}\\
&\sigma_i \text{ sos, } i=0,\ldots,m. \nonumber
\end{align}

For any fixed upper bound on the degrees of the polynomials $\sigma_i$, this is a semidefinite programming problem which produces a lower bound on the optimal value of (\ref{eq:basic.opt}). As the degrees of $\sigma_i$ increase, these lower bounds are guaranteed to converge to the true optimal value of (\ref{eq:basic.opt}). Note that we are making \emph{no convexity assumptions} about the polynomial optimization problem and yet solving it \emph{globally} through a sequence of semidefinite programs.

\textbf{Sum of squares and polynomial optimization in robotics.} We remark that sum of squares techniques have recently found increasing applications to a whole host of problems in robotics, including constructing Lyapunov functions \cite{ahmadi2014towards}, locomotion planning \cite{kuindersma2016optimization}, design and verification of provably safe controllers \cite{majumdar2013control,majumdar2014control}, grasping and manipulation \cite{dai2015synthesis,posa2016stability, zhou2016convex}, robot-world calibration \cite{heller2014hand}, and inverse optimal control \cite{pauwels2014inverse}, among others. 

We also remark that a different use of sum of squares optimization for finding minimum bounding volumes that contain semialgebraic sets has been considered in \cite{Henrion,Henrion1} along with some interesting control applications (see Section~\ref{sec:cont.poly.sub} for a brief description).

%Nothing is known a priori about the degree of the polynomials $\sigma_i$ we are looking for and one has to resort to bounding the degree of $\sigma_i$ by some integer $l$ to search over a finitely parametrized space. Bounding the degree of $\sigma_i$ leads to a relaxation of problem (\ref{eq:basic.opt}) and lower bounds on the optimal solution. When we increase the value of the degree of $\sigma_i$, we define a hierarchy of improving lower bounds on the problem. 

\section{3D point cloud containment} \label{sec:3D.point.cont}

%Let $\{x_1,\ldots,x_m\}$ be points in $\mathbb{R}^3$. 
Throughout this section, we are interested in finding a body of minimum volume, parametrized as the 1-sublevel set of a polynomial of degree $2d$, which encloses a set of given points  $\{x_1,\ldots,x_m\}$ in $\mathbb{R}^n$.

\subsection{Convex sublevel sets}\label{subsec:conv.cont}

We focus first on finding a \emph{convex} bounding volume. Convexity is a common constraint in the bounding volume literature and it makes certain tasks (e.g., distance computation among the different bodies) simpler. In order to make a set of the form $\{x\in\mathbb{R}^3| \ p(x)\leq 1\}$ convex, we will require the polynomial $p$ to be convex. (Note that this is a sufficient but not necessary condition.) Furthermore, to have a tractable formulation, we will replace the convexity condition with an sos-convexity condition as described previously. Even after these relaxations, the problem of minimizing the volume of our sublevel sets remains a difficult one. The remainder of this section discusses several heuristics for this task.

\begin{comment}

Optimizing over convex bodies (as opposed to non-convex bodies) can be advantageous for a few reasons. First, if we replace the notion of convexity by sos-convexity, one can easily compute certain geometric properties of the problem such as distance between sets (see Section \ref{sec:distance}). Secondly, for purposes such as collision detection, convexity can ensure stricter guarantees than non convexity, e.g., in the case of a convex body, if two points of the body do not come into contact with the object we want to avoid, then the segment between the two points does not either. As mentioned earlier, we will restrict ourselves to bodies described as level sets of polynomials. To make these convex, we will require that our polynomials be convex. 

Finding minimum volume convex polynomial level sets containing a 3D point cloud is a hard problem however, for multiple reasons. First, as seen in Section \ref{sec:sos.convex}, optimizing over the set of convex polynomials is intractable: we replace the convexity condition by the stronger but tractable condition of being sos-convex. Second, beyond the case of degree-2 polynomials, there is no known closed form expression for volume; hence, one needs to consider heuristic measures of volume of the sublevel sets. 

\end{comment}

\subsubsection{The Hessian-based approach}\label{subsec:Boyd.method}

In \cite{Magnani}, Magnani et al. propose the following heuristic to minimize the volume of the 1-sublevel set of an sos-convex polynomial:
\begin{equation}
\begin{aligned}
& &&\min_{p \in \mathbb{R}_{2d}[x],H \in S^{\tilde{N} \times \tilde{N}}} -\log \det(H) \\
&\text{s.t. } &&p \text{ sos}, \\
& &&y^T \nabla^2 p(x)y=w(x,y)^THw(x,y),~H\succeq 0, \label{eq:Magnani.log.det}\\
& && p(x_i)\leq 1, i=1,\ldots,m,
\end{aligned}
\end{equation}
where $w(x,y)$ is a vector of monomials in $x$ and $y$ of degree $1$ in $y$ and $d-1$ in $x$.
This problem outputs a polynomial $p$ whose 1-sublevel set corresponds to the bounding volume that we are interested in. A few remarks on this formulation are in order:
\begin{itemize}
	\item The last constraint simply ensures that all the data points are within the 1-sublevel set of $p$ as required. 
	\item The second constraint imposes that $p$ be sos-convex. The matrix $H$ is the Gram matrix associated with the sos condition on $y^T\nabla^2 p(x)y$. 
	\item The first constraint requires that the polynomial $p$ be sos. This is a necessary condition for boundedness of (\ref{eq:Magnani.log.det}) when $p$ is parametrized with affine terms. To see this, note that for any given positive semidefinite matrix $Q$, one can always pick the coefficients of the affine terms in such a way that the constraint $p(x_i)\leq 1$ for $i=1,\ldots,m$ be trivially satisfied. Likewise one can pick the remaining coefficients of $p$ in such a way that the sos-convexity condition is satisfied.
	The restriction to sos polynomials, however, can be done without loss of generality. Indeed, suppose that the minimum volume sublevel set was given by $\{x~|~ p(x)\leq 1\}$ where $p$ is an sos-convex polynomial. As $p$ is convex and nonaffine, $\exists \gamma\geq 0$ such that $p(x)+\gamma\geq 0$ for all $x$. Define now $q(x)\mathrel{\mathop{:}}=\frac{p(x)+\gamma}{1+\gamma}.$ We have that $\{x~|~ p(x)\leq 1\}=\{ x~|~ q(x)\leq 1\}$, but here, $q$ is sos as it is sos-convex and nonnegative \cite[Lemma 8]{helton2010}. 
\end{itemize}

%To understand how the objective corresponds to the heuristic suggested in \cite{Magnani} to measure volume.

The objective function of the above formulation is motivated in part by the degree $2d=2$ case. Indeed, when $2d=2$, the sublevel sets of convex polynomials are ellipsoids of the form $\{x~|~ x^TPx+b^Tx+c\leq 1\}$ and their volume is given by $\frac43 \pi \cdot \sqrt{\det(P^{-1})}$. Hence, by minimizing $-\log \det(P)$, we would exactly minimize volume. As the matrix $P$ above is none other than the Hessian of the quadratic polynomial $x^TPx+b^Tx+c$ (up to a multiplicative constant), this partly justifies the formulation given in \cite{Magnani}. Another justification for this formulation is given in \cite{Magnani} itself and relates to curvature of the polynomial $p$. Indeed, the curvature of $p$ at a point $x$ along a direction $y$ is proportional to $y^T\nabla^2 p(x)y$. By imposing that $y^T\nabla^2 p(x)y=w(x,y)^THw(x,y),$ with $H \succeq 0$, and then maximizing $\log(\det(H))$, this formulation seeks to increase the curvature of $p$ along all directions so that its 1-sublevel set can get closer to the points $x_i$. Note that curvature maximization in all directions without regards to data distribution can be counterproductive in terms of tightness of fit, particularly in regions where the data geometry is flat (an example of this is given in Figure \ref{fig:comparison.with.Boyd}).

%As the Hessian is a polynomial matrix for degrees higher than 2, the generalization to higher degrees does not involve the Hessian of the polynomial itself but the Gram matrix associated to it.

A related minimum volume heuristic that we will also experiment with replaces the $\log \det$ objective with a linear one. More specifically, we introduce an extra decision variable $V \in S^{\tilde{N}\times \tilde{N}}$ and minimize $\mbox{trace}(V)$ while adding an additional constraint $\begin{bmatrix} V & I \\ I & H \end{bmatrix} \succeq 0.$
%\begin{align}
%&\min_{p \in \mathbb{R}_{2d}[x],H \in S^{\tilde{N} \times \tilde{N}},V \in S^{\tilde{N} \times \tilde{N}}} \mbox{trace}(V) \nonumber \\
%&\text{s.t. } \nonumber \\
%&p \text{ sos}, \nonumber \\
%&y^T \nabla^2 p(x) y=w(x,y)^THw(x,y), H\succeq 0, \label{eq:Magnani.trace}\\
%& p(x_i)\leq 1, i=1,\ldots,m, \nonumber \\
%&\begin{bmatrix} V & I \\ I & H \end{bmatrix} \succeq 0. \nonumber
%\end{align}
Using the Schur complement, the latter constraint can be rewritten as $V\succeq H^{-1}$. As a consequence, this trace formulation minimizes the \emph{sum} of the inverse of the eigenvalues of $H$ whereas the $\log \det$ formulation described in (\ref{eq:Magnani.log.det}) minimizes the \emph{product} of the inverse of the eigenvalues.

\subsubsection{Our approach} \label{subsec:our.approach}

We propose here an alternative heuristic for obtaining a tight-fitting convex body containing points in $\mathbb{R}^n.$ Empirically, we validate that it tends to consistently return convex bodies of smaller volume than the ones obtained with the methods described above (see Figure~\ref{fig:comparison.with.Boyd} below for an example). It also generates a relatively smaller convex optimization problem. Our formulation is as follows:
\begin{align}
&\min_{p \in \mathbb{R}_{2d}[x],P \in S^{N \times N}} -\log \det(P) \nonumber\\
&\text{s.t. } \nonumber \\
&p(x)=z(x)^TP z(x), P\succeq 0, \nonumber\\
&p \text{ sos-convex},\label{eq:VAG.log.det}\\
& p(x_i)\leq 1, i=1,\ldots,m. \nonumber
\end{align}
One can also obtain a trace formulation of this problem by replacing the $\log \det$ objective by a trace one as it was done in the previous paragraph.
%\begin{align}
%&\min_{p \in \mathbb{R}_{2d}[x],P \in S^{N \times N}, V \in S^{N \times N}} \mbox{trace}(V) \nonumber\\
%&\text{s.t. } \nonumber \\
%&p(x)=z(x)^TP z(x), P\succeq 0, \nonumber\\
%&p \text{ sos-convex},\label{eq:VAG.trace}\\
%& p(x_i)\leq 1, i=1,\ldots,m, \nonumber \\
%&\begin{bmatrix} V & I \\ I & P \end{bmatrix} \succeq 0. \nonumber
%\end{align}

Note that the main difference between (\ref{eq:Magnani.log.det}) and (\ref{eq:VAG.log.det}) lies in the Gram matrix chosen for the objective function. In (\ref{eq:Magnani.log.det}), the Gram matrix comes from the sos-convexity constraint, whereas in (\ref{eq:VAG.log.det}), the Gram matrix is generated by the sos constraint.

In the case where the polynomial is quadratic and convex, we saw that the formulation (\ref{eq:Magnani.log.det}) is exact as it finds the minimum volume ellipsoid containing the points. It so happens that the formulation given in (\ref{eq:VAG.log.det}) is also exact in the quadratic case, and, in fact, both formulations return the same optimal ellipsoid. As a consequence, the formulation given in (\ref{eq:VAG.log.det}) can also be viewed as a natural extension of the quadratic case. 

To provide more intuition as to why this formulation performs well, we interpret the 1-sublevel set $$S\mathrel{\mathop{:}}=\{x~|~p(x)\leq 1\}$$ of $p$ as the preimage of some set whose volume is being minimized. More precisely, consider the set $$T_1=\{z(x) \in \mathbb{R}^N~|~ x \in \mathbb{R}^n\}$$ which corresponds to the image of $\mathbb{R}^n$ under the monomial map $z(x)$ and the set $$T_2=\{y \in \mathbb{R}^N ~|~ y^TPy \leq 1 \},$$ for a positive semidefinite matrix $P$ such that $p(x)=z(x)^TPz(x).$ Then, the set $S$ is simply the preimage of the intersection of $T_1$ and $T_2$ through the mapping $z$. Indeed, for any $x \in S$, we have $p(x)=z(x)^TPz(x) \leq 1$. The hope is then that by minimizing the volume of $T_2$, we will minimize volume of the intersection $T_1 \cap T_2$ and hence that of its preimage through $z$, i.e., the set $S.$ 

\begin{figure}[h]
	\centering
	\includegraphics[scale=0.8]{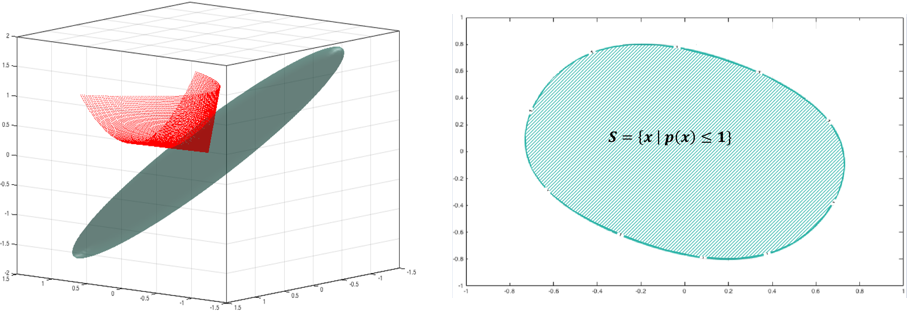}
	\caption{An illustration of the intuition behind the approach in Section \ref{subsec:our.approach}: the sets $T_1$ and $T_2$ (left) and $S$ (right)}
	\label{fig:illustration.proof}
\end{figure}

We illustrate this idea in Figure \ref{fig:illustration.proof}. Here, we have generated a random $3\times 3$ positive semidefinite matrix $P$ and a corresponding bivariate degree-4 sos polynomial $p(x_1,x_2)=z(x_1,x_2)^TPz(x_1,x_2)$, where $z(x_1,x_2)=(x_1^2,x_1x_2,x_2^2)^T$ is a map from $\mathbb{R}^2$ to $\mathbb{R}^3$. We have drawn in red the image of $\mathbb{R}^2$ under $z$ and in green the ellipsoid $\{y \in \mathbb{R}^3~|~y^TPy \leq 1\}.$ The preimage of the intersection of both sets seen in Figure~\ref{fig:illustration.proof} on the right corresponds to the 1-sublevel set of $p.$

\subsection{Relaxing convexity} \label{subsec:nonconvex}

Though containing a set of points with a convex sublevel set has its advantages, it is sometimes necessary to have a tighter fit than the one provided by a convex body, particularly if the object of interest is highly nonconvex. One way of handling such scenarios is via convex decomposition methods~\cite{ConvexDecomposition1, ConvexDecomposition2}, which would enable us to represent the object as a union of sos-convex bodies. Alternatively, one can aim for problem formulations where convexity of the sublevel sets is not imposed. In the remainder of this subsection, we first review a recent approach from the literature to do this and then present our own approach which allows for controlling the level of nonconvexity of the sublevel set. 

%Removing the sos-convexity constraint however can lead to optimal solutions for (\ref{eq:VAG.sos}) which have unbounded level sets. A way of compactifying these level sets is to impose a condition on the higher order terms of $p$. Namely, if we denote by $p_{2d}$ the polynomial containing all higher order terms of $p$, we can require that the matrix $\tilde{P}$ defined as $$p_{2d}(x)= \mathrel{\mathop{:}}\tilde{z}(x)^T \tilde{P} \tilde{z}(x)$$ where $\tilde{z}(x)$ is the monomials vector of degree exactly $d$, be positive semidefinite. The polynomial $p$ is then coercive and its level sets will be compact.

\subsubsection{The inverse moment approach}\label{subsubsec:Lasserre} In very recent work~\cite{lasserre2016_inverse_moment}, Lasserre and Pauwels propose an approach for containing a cloud of points with sublevel sets of polynomials (with no convexity constraint). Given a set of data points $x_1,\ldots,x_m\in\mathbb{R}^n$, it is observed in that paper that the sublevel sets of the degree $2d$ sos polynomial \begin{equation}\label{eq:inverse.moment.poly}
p_{\mu,d}(x)\mathrel{\mathop:}=z(x)^T M_d(\mu(x_1,\ldots,x_m))^{-1} z(x),
\end{equation}
tend to take the shape of the data accurately. Here, $z(x)$ is the vector of all monomials of degree up to $d$ and $M_d(\mu(x_1,\ldots,x_m))$ is the moment matrix of degree $d$ associated with the empirical measure $\mu\mathrel{\mathop:}=\frac{1}{m}\sum_{i=1}^{m}\delta_{x_i}$ defined over the data. This is an $\binom{n+d}{d} \times \binom{n+d}{d}$ symmetric positive semidefinite matrix which can be cheaply constructed from the data $x_1,\ldots,x_m\in\mathbb{R}^n$ (see~\cite{lasserre2016_inverse_moment} for details). One very nice feature of this method is that to construct the polynomial $p_{\mu, d}$ in (\ref{eq:inverse.moment.poly}) one only needs to invert a matrix (as opposed to solving a semidefinite program as our approach would require) after a single pass over the point cloud. The approach however does not a priori provide a particular sublevel set of $p_{\mu, d}$ that is guaranteed to contain all data points. Hence, once $p_{\mu, d}$ is constructed, one could slowly increase the value of a scalar $\gamma$ and check whether the $\gamma$-sublevel set of $p_{\mu, d}$ contains all points.

\subsubsection{Our approach and controlling convexity} \label{subsubsec:conv.control.ours} An advantage of our proposed formulation (\ref{eq:VAG.log.det}) is that one can easily drop the sos-convexity assumption in the constraints and thereby obtain a sublevel set which is not necessarily convex. %The problem then becomes:
%\begin{align}
%&\min_{p \in \mathbb{R}_{2d}[x],P \in S^{N \times N}} -\log \det(P) \nonumber\\
%&\text{s.t. } \nonumber\\
%&p=z(x)^TP z(x), P\succeq 0 \label{eq:VAG.sos}\\
%& p(x_i)\leq 1, i=1,\ldots,m.\nonumber
%\end{align}
This is not an option for formulation (\ref{eq:Magnani.log.det}) as the Gram matrix associated to the sos-convexity constraint intervenes in the objective.

Note that in neither this formulation nor the inverse moment approach of Lasserre and Pauwels, does the optimizer have control over the shape of the sublevel sets produced, which may be convex or far from convex. For some applications, it is useful to control in some way the degree of convexity of the sublevel sets obtained by introducing a parameter which when increased or decreased would make the sets more or less convex. This is what our following proposed optimization problem does via the parameter $c$, which corresponds in some sense to a measure of convexity:
\begin{align}
&\min_{p \in \mathbb{R}_{2d}[x],P \in S^{N \times N}} -\log \det(P) \nonumber\\
&\text{s.t. }\nonumber \\
&p=z(x)^TP z(x), P\succeq 0 \label{eq:VAG.param.convex}\\
& p(x)+ c (\sum_i x_i^2)^d \text{ sos-convex}. \nonumber \\
& p(x_i)\leq 1, i=1,\ldots,m.\nonumber
\end{align}
Note that when $c=0$, the problem we are solving corresponds exactly to (\ref{eq:VAG.log.det}) and the sublevel set obtained is convex. When $c>0$, we allow for nonconvexity of the sublevel sets. Note that this is a consequence of $ (\sum_i x_i^2)^d$ being a strictly convex function, which can offset the nonconvexity of $p$. 
%{\gh Indeed, adding a multiple of the term $(\sum_i x_i^2)^d+(\sum_i x_i^2)$ (which is in the interior of the set of sos-convex polynomials) to $p$ (which may not be convex), counterbalances the potential nonconvexity of $p$. In fact, for any polynomial $p$, there exists $c>0$ such that $p-c((\sum_i x_i^2)^d+(\sum_i x_i^2))$ is sos-convex (see, e.g., Lemma 1 in \cite{DCP} for a proof).} 
As we decrease $c$ towards zero, we obtain sublevel sets which get progressively more and more convex.

\subsection{Bounding volume numerical experiments} 
Figure \ref{fig:intro_pic} (left) shows the 1-sublevel sets of sos-convex bodies with degrees $2$, $4$, and $6$. A degree-$6$ polynomial gives a much tighter fit than an ellipsoid (degree 2). In the middle figure, we freeze the degree to be $6$ and increase the convexity parameter $c$ in the relaxed convexity formulation of problem~(\ref{eq:VAG.param.convex}); the 1-sublevel sets of the resulting sos polynomials with $c=0, 10, 100$ are shown. It can be seen that the sublevel sets gradually bend to better adapt to the shape of the object. The right figure shows the $2, 1,$ and $0.75$ sublevel sets of a degree-$6$ polynomial obtained by fixing $c=10$ in problem~(\ref{eq:VAG.param.convex}): the shape is retained as the body is expanded or contracted. 

Figure \ref{fig:comparison.with.Boyd} shows 1-sublevel sets of two degree-6 sos-convex polynomials. In red, we have plotted the sublevel set corresponding to maximizing curvature as explained in Section \ref{subsec:Boyd.method}. In green, we have plotted the sublevel set generated by our approach as explained in Section \ref{subsec:our.approach}. Note that our method gives a tighter-fitting sublevel set, which is in part a consequence of the flat data geometry for which the maximum curvature heuristic does not work as well.

\begin{figure}[h]
	\centering
	\includegraphics[scale=0.4]{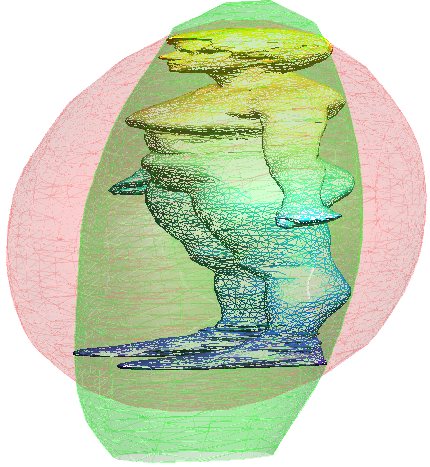}
	\caption{Comparison of degree-6 bounding volumes: our approach as described in Section \ref{subsec:our.approach} (green sublevel set) produces a tighter fitting bounding volume than the approach given in \cite{Magnani} and reviewed in Section \ref{subsec:Boyd.method} (red sublevel set). }
	\label{fig:comparison.with.Boyd}
\end{figure}

In Table \ref{tab:bounding.volumes}, we provide a comparison of various bounding volumes on Princeton Shape Benchmark datasets~\cite{ShapeData}. It can be seen that sos-convex bodies generated by higher degree polynomials provide much tighter fits than spheres or axis-aligned bounding boxes (AABB) in general. The proposed minimum volume heuristic of our formulation in (\ref{eq:VAG.log.det}) works better than that proposed in \cite{Magnani} (see (\ref{eq:Magnani.log.det})). In both formulations, typically, the log-determinant objective outperforms the trace objective. The convex hull is the tightest possible convex body. However, for smooth objects like the vase, the number of vertices describing the convex hull can be a substantial fraction of the original number of points in the point cloud. When convexity is relaxed, a degree-6 sos polynomial compactly described by just $84$ coefficients gives a tighter fit than the convex hull. For the same degree, solutions to our formulation (\ref{eq:VAG.param.convex}) with a positive value of $c$ outperform the inverse moment construction of~\cite{lasserre2016_inverse_moment}.

The bounding volume construction times are shown in Figure~\ref{fig:bv_construction_time} for sos-convex chair models. In comparison to the volume heuristics of~\cite{Magnani}, our heuristic runs noticeably faster as soon as degree exceeds $6$. We believe that this may come from the fact that the decision variable featuring in the objective in our case is a matrix of size $N \times N$, where $N=\binom{n+d}{d}$, whereas the decision variable featuring in the objective of~\cite{Magnani} is of size $\tilde{N} \times \tilde{N},$ where $\tilde{N}=n\cdot \binom{n+d-1}{d-1} > N.$
%Note that the SDP can be run on the vertices of the convex hull instead of the original point cloud: due to convexity the estimated polynomials are the same modulo numerical differences. 
Our implementation uses YALMIP~\cite{yalmip} with the splitting conic solver (SCS)~\cite{scs} as its backend SDP solver (run for 2500 iterations). Note that the inverse moment approach of~\cite{lasserre2016_inverse_moment} is the fastest as it does not involve any optimization and makes just one pass over the point cloud. However, this approach is not guaranteed to return a convex body, and for nonconvex bodies, tighter fitting polynomials can be estimated using log-determinant or trace objectives on our problem (\ref{eq:VAG.param.convex}).

%In Section *, we also show some new uses of this method in a scenario where the number of data points is large and one is interested in guaranteeing convexity of the containing sublevel set.
\begin{table*}[t!]
	\begin{center}
		\scalebox{0.68}{
		\begin{tabular}{|c|c|c|c|c|c|c|c|}
			\hline 
			& Object (id in \cite{ShapeData}) & & Human (10) & Chair (101) & Hand (181) & Vase (361) & Octopus (121)\\
			& $\#$ points/vertices in cvx hull&  & 9508/364 & 8499/320 & 7242/ 652 & 14859/1443 & 5944/414\\
			\hline
			Section & Bounding Body $\downarrow$& Objective fcn $\downarrow$ & \multicolumn{5}{c|}{Volume $\downarrow$} \\
			\hline
			&Convex-Hull &  & 0.29 & 0.66 & 0.36 & 0.91 & 0.5 \\
			& Sphere & & 3.74 & 3.73 & 3.84 & 3.91 & 4.1\\
			& AABB & & 0.59 & 1.0 & 0.81 & 1.73 & 1.28\\
			\hline 
			&\multirow{2}{*}{sos-convex ($2d=2$)} &$logdet$ & 0.58 & 1.79 & 0.82 & 1.16 & 1.30\\
			& &$trace$ & 0.97 & 1.80 & 1.40 & 1.2 &1.76\\
			%& $logdet(\vv{P}^{-1})$ & 0.57 & 1.74 & 0.81& 1.15 &1.28\\
			%& $trace(\vv{P}^{-1})$ & 0.91 & 2.67 & 1.65 & 1.16 &1.23\\
			\cline{2-8}
			& \multirow{4}{*}{sos-convex ($2d=4$)} & $logdet(\vv{H}^{-1})$ & 0.57 & 1.55 & 0.69& 1.13 & 1.04\\
			& & $trace(\vv{H}^{-1})$ & 0.56 & 2.16 & 1.28& 1.09 &3.13 \\
			& & $logdet(\vv{P}^{-1})$ & 0.44 & 1.19 & 0.53& 1.05 &0.86\\
			\rot{\rlap{~\ref{subsec:conv.cont}}} & & $trace(\vv{P}^{-1})$ & 0.57& 1.25 & 0.92 & 1.09  &1.02\\
			\cline{2-8}
			& \multirow{4}{*}{sos-convex ($2d=6$)} & $logdet(\vv{H}^{-1})$ & 0.57 & 1.27 & 0.58& 1.09& 0.93\\
			& & $trace(\vv{H}^{-1})$ & 0.56 & 1.30 & 0.57 & 1.09 & 0.87\\
			& & $logdet(\vv{P}^{-1})$ & 0.41 &  1.02 & 0.45& 0.99 &0.74\\
			& & $trace(\vv{P}^{-1})$ & 0.45 & 1.21 & 0.48 & 1.03  &0.79\\
			\hline
			\rule{0pt}{8pt}
			%			\hline
			& Inverse-Moment ($2d=2$) & & 4.02 & 1.42   & 2.14 & 1.36 &1.74\\
			& Inverse-Moment ($2d=4$) & & 1.53 & 0.95  & 0.90 & 1.25 &0.75\\
			\rot{\rlap{~\ref{subsubsec:Lasserre}}}	& Inverse-Moment ($2d=6$) & & 0.48 & 0.54   & 0.58 & 1.10 &0.57\\
			\hline
			%\multirow{2}{*}{SoS-ncvx (d=2, c=-10)} & $logdet(\vv{P}^{-1})$  & 0.57 & 1.74&0.81 &1.16 &1.28\\
			%			%& $trace(\vv{P}^{-1})$ & 0.90 & 2.72 & 1.63& 1.16 &1.23\\
			%			%\hline
			& \multirow{2}{*}{sos ($2d=4, c=10$)}  
			& $logdet(\vv{P}^{-1})$ & 0.38 & 0.72 & 0.42 & 1.05 &0.63\\
			& & $trace(\vv{P}^{-1})$ & 0.51 & 0.78& 0.48 & 1.11 &0.71\\
			\cline{2-8}
			&	\multirow{2}{*}{sos ($2d=6, c=10$)}  
			& $logdet(\vv{P}^{-1})$ & 0.35 & 0.49& 0.34 &0.92& 0.41\\
			&	& $trace(\vv{P}^{-1})$ & 0.37 & 0.56 & 0.39 & 0.99&0.54\\
			\cline{2-8}
			%\hline
			%			%\multirow{2}{*}{SoS-ncvx (d=2, c=-100)}  
			%			%& $logdet(\vv{P}^{-1})$  & 0.57& 1.74 & 0.81 &1.15&1.28\\
			%			%& $trace(\vv{P}^{-1})$ & 0.87 & 2.63 & 1.63  & 1.16&1.25\\
			%			%\hline
			&\multirow{2}{*}{sos ($2d=4, c=100$)}  
			& $logdet(\vv{P}^{-1})$  & 0.36 & 0.64 & 0.39 &1.05& 0.46\\
			\rot{\rlap{~\ref{subsubsec:conv.control.ours}}}  & & $trace(\vv{P}^{-1})$ & 0.42 & 0.74  & 0.46 & 1.10 &0.54\\
			\cline{2-8}
			& \multirow{2}{*}{sos ($2d=6, c=100$)}  
			& $logdet(\vv{P}^{-1})$  & 0.21 & 0.21 & 0.26 & 0.82 &0.28\\
			& & $trace(\vv{P}^{-1})$ & 0.22 & 0.30 & 0.29 & 0.85 &0.37\\
			\hline
		\end{tabular}
	}
		\caption{ Comparison of the volume of various bounding bodies obtained from different techniques}
		\label{tab:bounding.volumes}
	\end{center}
\end{table*}

\begin{figure}[h]
	\centering
	\includegraphics[width=0.65\linewidth]{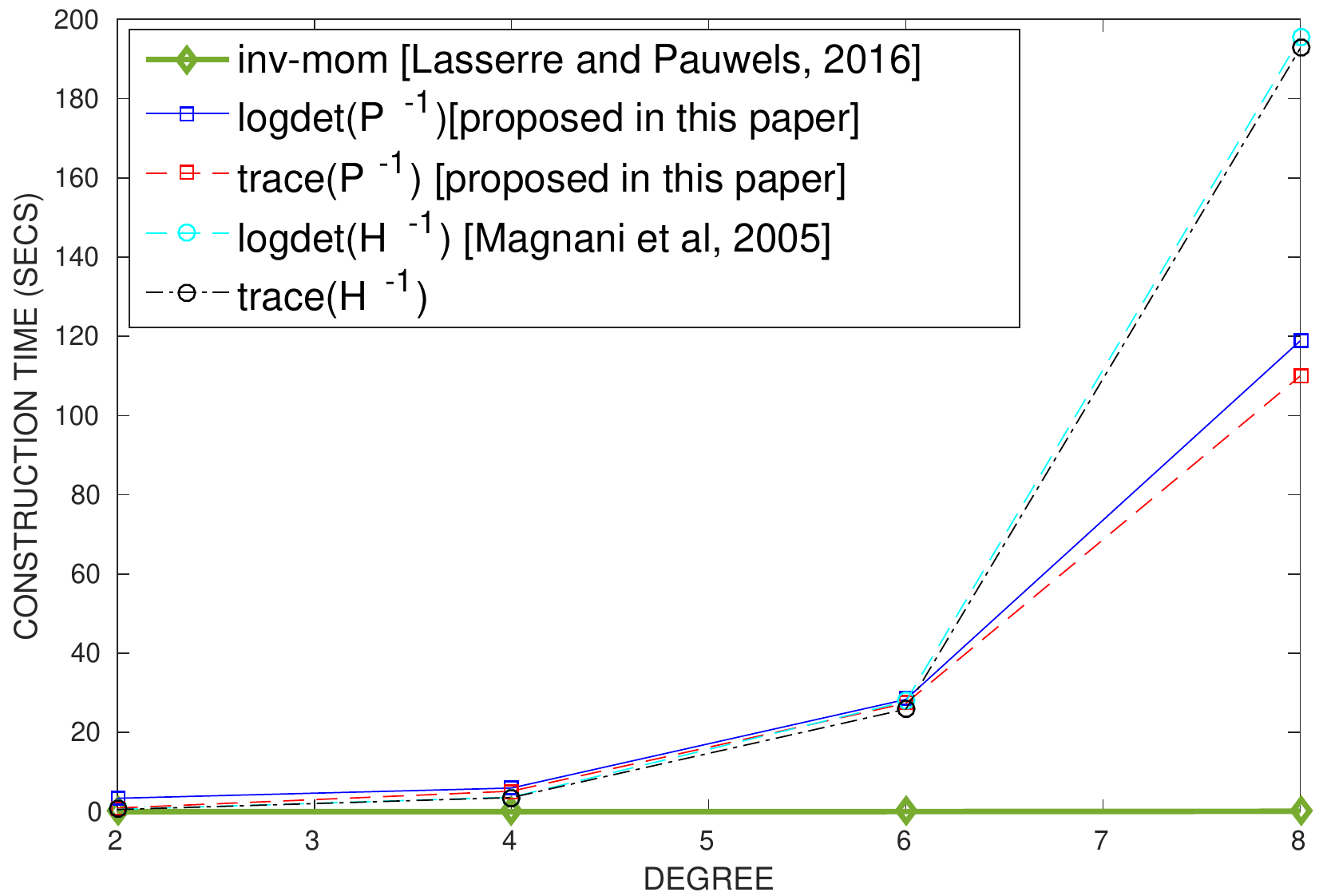}
	\caption{Bounding volume construction times}
	\label{fig:bv_construction_time}
\end{figure}

\section{Measures of separation and penetration}\label{sec:distance}

\subsection{Euclidean distance} \label{subsec:eucl.dist}
In this section, we are interested in computing the Euclidean distance between two basic semialgebraic sets $$\mathcal{S}_1 \mathcal{\mathop{:}}=\{x \in \mathbb{R}^n ~|~ g_1(x)\leq 1, \ldots, g_m \leq 1\},$$ and $$\mathcal{S}_2\mathcal{\mathop{:}}=\{x \in \mathbb{R}^n ~|~ h_1(x) \leq 1, \ldots, h_r \leq 1\}$$ (where $g_1,\ldots,g_m$ and $h_1,\ldots,h_r$ are polynomials). This can be written as the following polynomial optimization problem:
\begin{align}
&\min_{x \in \mathcal{S}_1, y \in \mathcal{S}_2} ||x-y||_2^2. \label{eq:distance}
\end{align}

We will tackle this problem by applying the sos hierarchy described at the end of Section \ref{sec:sos.convex}. This will take the form of the following hierarchy of semidefinite programs
\begin{equation} \label{eq:distance.lasserre}
\begin{aligned}
&\max_{\gamma \in \mathbb{R},\tau_i, \xi_j} \gamma\\
&||x-y||_2^2-\gamma-\sum_{i=1}^m \tau_i(x,y) (1-g_i(x))  \\
& \hspace{20mm} -\sum_{j=1}^r \xi_j(x,y)(1-h_j(y)) \text{ sos},\\
&\tau_i(x,y), ~\xi_j(x,y) \text{ sos }, \forall i,\forall j,
\end{aligned}
\end{equation}
where in the $d$-th level of the hierarchy, the degree of all polynomials $\tau_i$ and $\xi_j$ is upper bounded by $d$. Observe that the optimal value of each SDP produces a \emph{lower bound} on (\ref{eq:distance}) and that when $d$ increases, this lower bound can only improve.

Amazingly, in all examples we tried (independently of convexity of $\mathcal{S}_1$ and $\mathcal{S}_2$), the 0-th level of the hierarchy was already exact (though we were unable to prove this). By this we mean that the optimal value of (\ref{eq:distance.lasserre}) exactly matched that of (\ref{eq:distance}), already when the degree of the polynomials $\tau_i$ and $\xi_j$ was zero; i.e., when $\tau_i$ and $\xi_j$ were nonnegative scalars. An example of this phenomenon is given in Figure~\ref{fig:ex.distance.nonconvex} where the green bodies are each a (highly nonconvex) sublevel set of a quartic polynomial.

When our SDP relaxation is exact, we can recover the points $x^*$ and $y^*$ where the minimum distance between sets is achieved from the eigenvector corresponding to the zero eigenvalue of the Gram matrix associated with the first sos constraint in (\ref{eq:distance.lasserre}). This is what is done in Figure~\ref{fig:ex.distance.nonconvex}.

\begin{figure}[h] 
	\centering
	\includegraphics[scale=0.35]{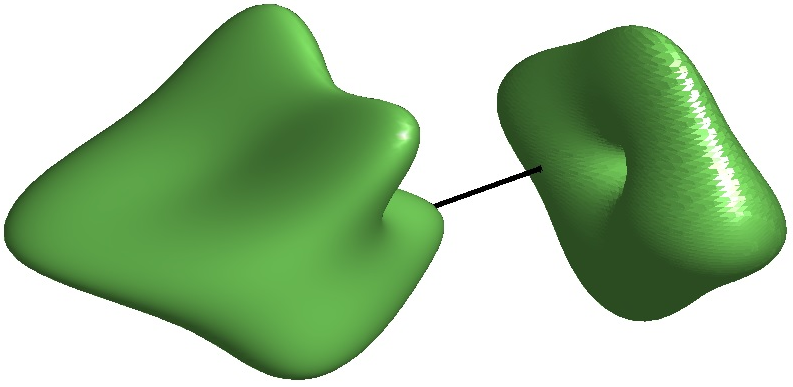}
	\caption{Minimum distance between two (nonconvex) sublevel sets of degree-4 polynomials}
	\label{fig:ex.distance.nonconvex}
\end{figure}

%\textbf{The sos-convex case.} In the particular case where all polynomials involved in the expressions of $S_1$ and $S_2$ are sos-convex, problem (\ref{eq:distance}) can be solved exactly using the first level of the sos hierarchy given in (\ref{eq:distance.lasserre}). In other words, if we take $\tau_i$ and $\xi_j$ to be polynomials of degree $0$ (i.e., constants) and the degree of $\sigma$ to be the maximum degree of all polynomials $\{||x-y||_2^2; ~1-g_i(x), ~i=1,\ldots,m; 1-h_j(y), j=1,\ldots,r\}$, then the optimal value of the above semidefinite program will provide the exact solution to problem (\ref{eq:distance}). This is a corollary of a more general result due to Lasserre \cite{lasserre2009convexity} to which our problem conforms.

\textbf{The sos-convex case.} One important special case where we know that the 0-th level of the sos hierarchy in (\ref{eq:distance.lasserre}) is \emph{guaranteed} to be exact is when the defining polynomials $g_i$ and $h_i$ of $\mathcal{S}_1$ and $\mathcal{S}_2$ are \emph{sos-convex}. This is a corollary of the fact that the 0-th level sos relaxation is known to be tight for the general polynomial optimization problem in (\ref{eq:basic.opt}) if the polynomials $p$ and $-g_i$ involved in the description of $K$ there are sos-convex; see~\cite{lasserre2009convexity}. An example of the computation of the minimum distance between two degree-6 sos-convex bodies enclosing human and chair 3D point clouds is given below, together with the points achieving the minimum distance.
\begin{figure}[h] 
	\centering
	\includegraphics[scale=0.35]{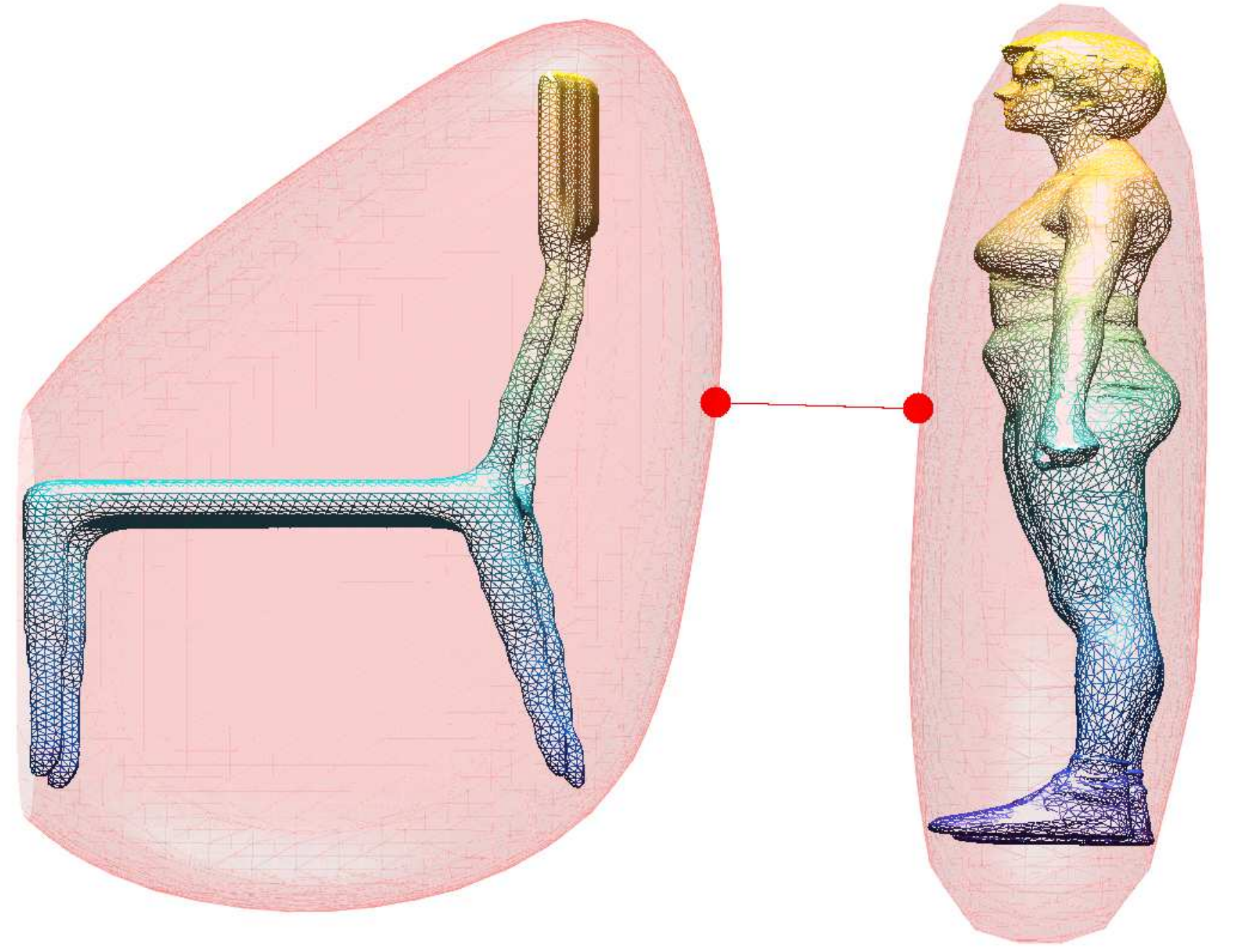}
	\caption{Minimum distance between two convex sublevel sets of degree-6 polynomials}
	\label{fig:ex.distance}
\end{figure}

Using MATLAB's fmincon active-set solver, the time required to compute the distance between two sos-convex bodies ranges from around 80 milliseconds to 340 milliseconds seconds as the degree is increased from $2$ to $8$; see Table~\ref{tab:distance.times}. We believe that the execution time can be improved by an order of magnitude with more efficient polynomial representations, warm starts for repeated queries, and reduced convergence tolerance for lower-precision results.
\begin{table}[H]
	\begin{center}
		\begin{tabular}{ccccc}
			\hline
			degree & 2 & 4 & 6 & 8\\ 
			time (secs) & 0.08 & 0.083 & 0.13 & 0.34 \\
			\hline
		\end{tabular}
		\caption{Euclidean distance query times for sos-convex sets.}
		\label{tab:distance.times}
	\end{center}
\end{table}

\subsection{Penetration measures for overlapping bodies}
\begin{figure*}[t] 
	\includegraphics[height=4cm,width=0.24\linewidth]{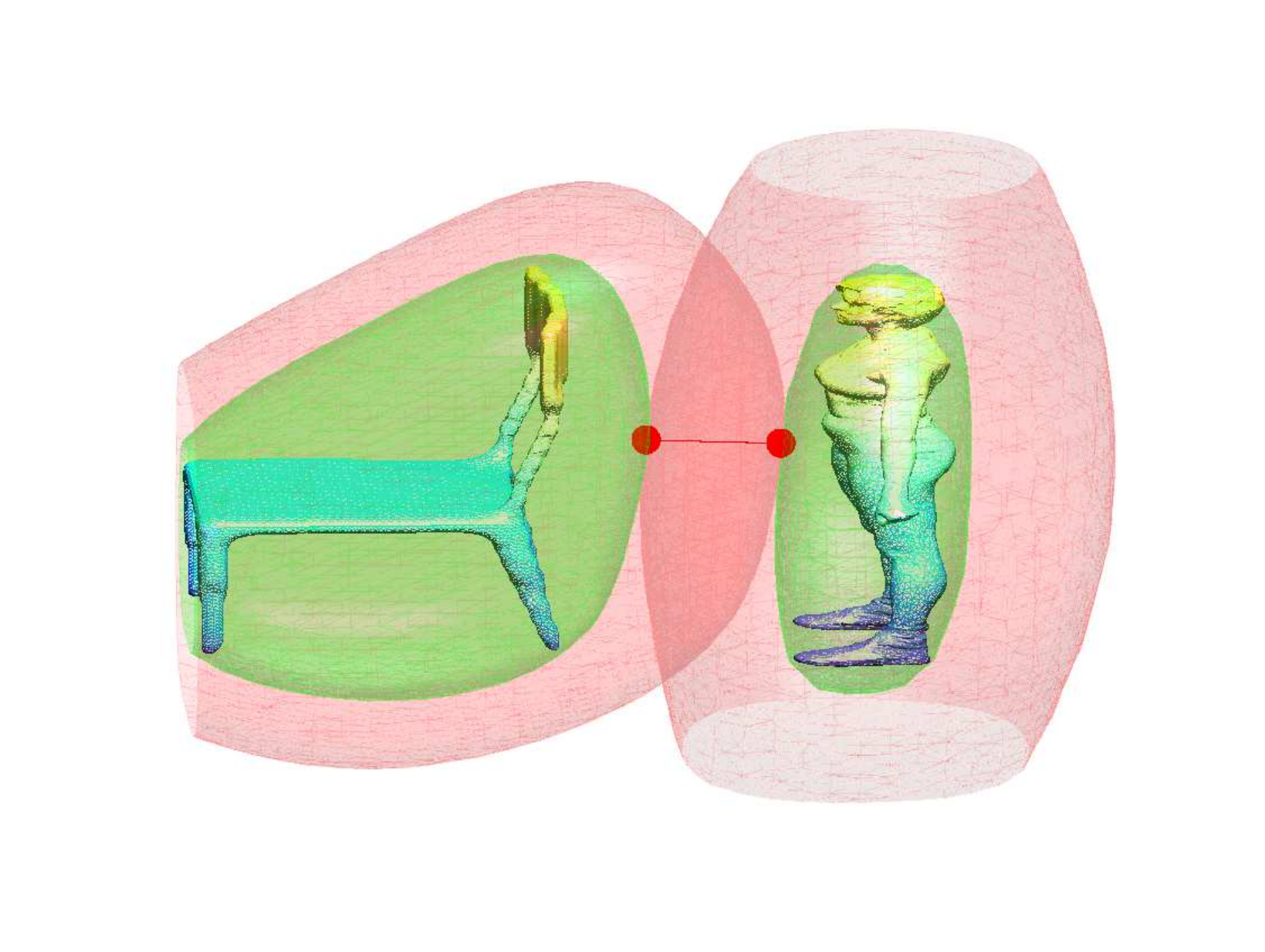}
	\includegraphics[height=4cm,width=0.24\linewidth]{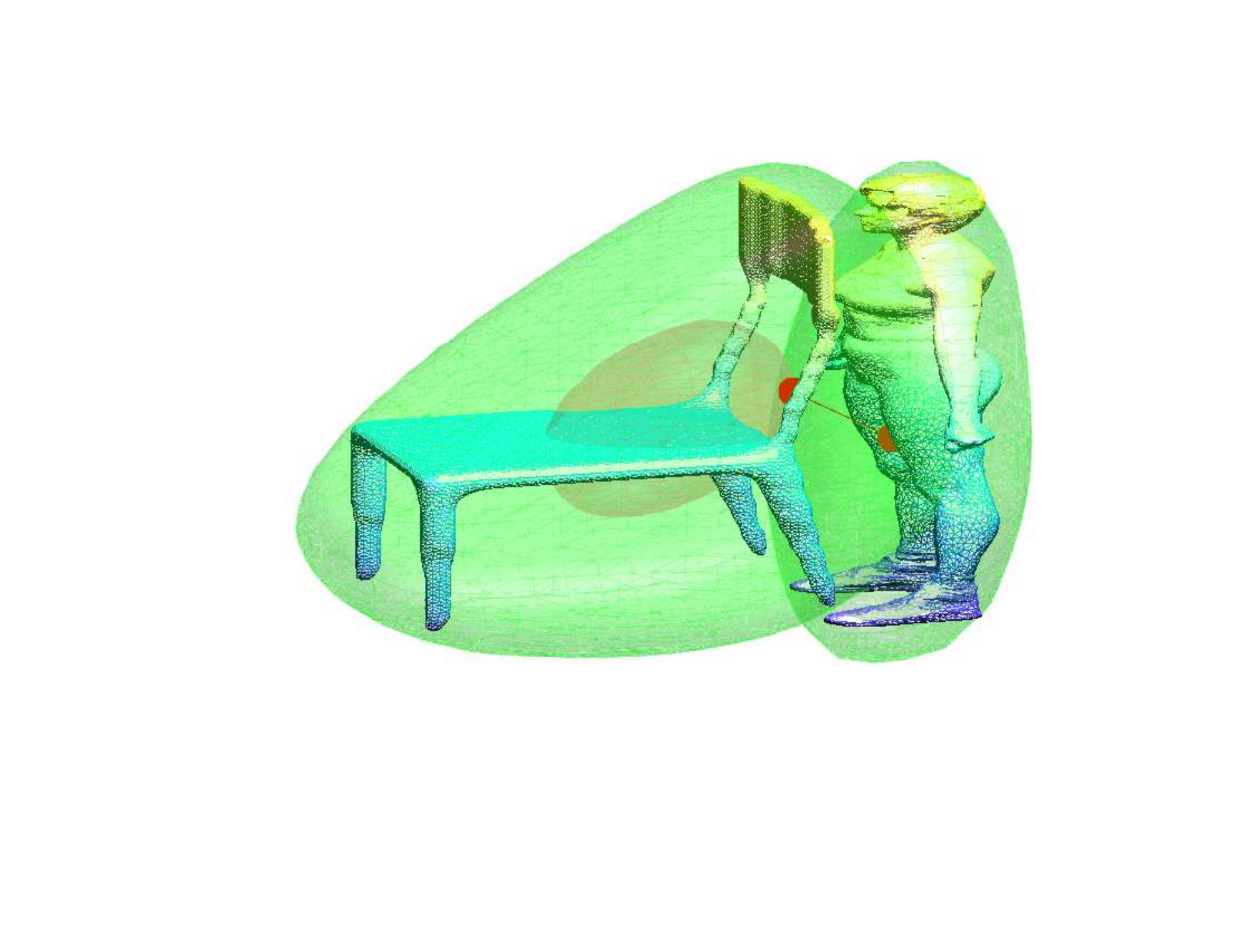}
	\includegraphics[height=4cm, width=0.24\linewidth]{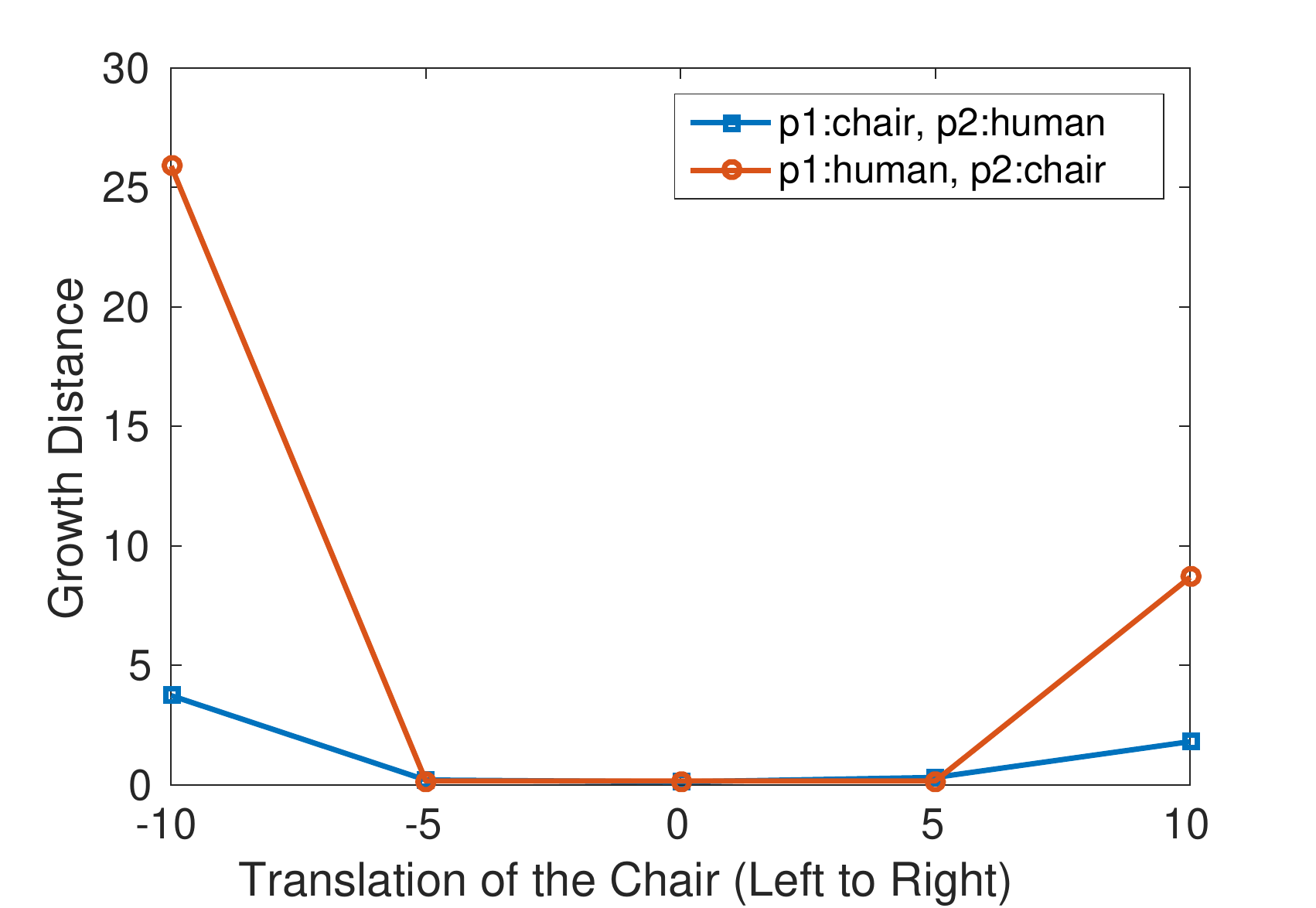}
	\includegraphics[height=4cm, width=0.24\linewidth]{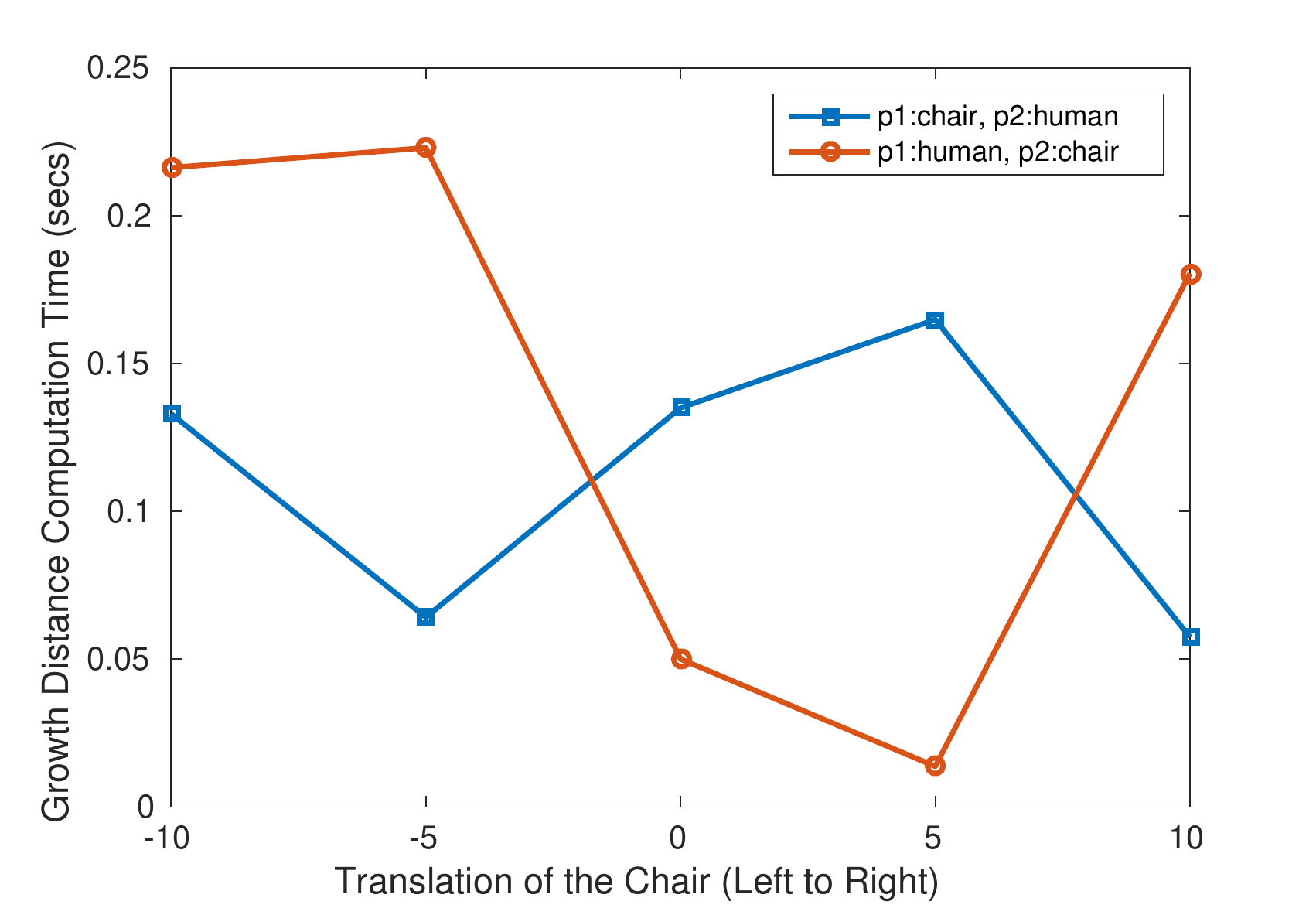}
	\caption{Growth distances for separated (left) or overlapping (second-left) sos-convex bodies; growth distance as a function of the position of the chair (second-right); time taken to solve (\ref{eq:overlap}) with warm-start (right)}
	\label{fig:penetration}
\end{figure*}
As another application of sos-convex polynomial optimization problems, we discuss a problem relevant to collision avoidance. Here, we assume that our two bodies $\mathcal{S}_1$, $\mathcal{S}_2$ are of the form $\mathcal{S}_1\mathrel{\mathop{:}}=\{x~|~ p_1(x)\leq 1\}$ and $\mathcal{S}_2\mathrel{\mathop{:}}=\{x~|~ p_2(x)\leq 1\},$ where $p_1,p_2$ are sos-convex. As shown in Figure \ref{fig:intro_pic} (right), by varying the sublevel value, we can grow or shrink the sos representation of an object. The following convex optimization problem, with optimal value denoted by $d(p_1||p_2)$, provides a measure of separation or penetration between the two bodies:
\begin{align}
&d(p_1||p_2) = \min p_1(x) \nonumber \\
&\text{s.t. } p_2(x) \leq 1. \label{eq:overlap}
\end{align} Note that the measure is asymmetric, i.e., $d(p_1||p_2)\neq d(p_2||p_1)$. It is clear that $$p_2(x) \leq 1 \Rightarrow p_1(x) \geq d(p_1||p_2).$$ In other words, the sets $\{x~|~ p_2(x) \leq 1\}$ and $\{x~|~ p_1(x) \leq d(p_1||p_2)\}$ do not overlap. As a consequence, the optimal value of (\ref{eq:overlap}) gives us a measure of how much we need to shrink the level set defined by $p_1$ to eventually move out of contact of the set $\mathcal{S}_2$ assuming that the ``seed point", i.e., the minimum of $p_1$, is outside ${\cal S}_2$. It is clear that,
\setlength{\parindent}{0em}
\begin{itemize}
	\item if $d(p_1||p_2) > 1$, the bounding volumes are separated.
	\item if $d(p_1||p_2) = 1$, the bounding volumes touch.
	\item if $d(p_1||p_2) < 1$, the bounding volumes overlap.
\end{itemize}

These measures are closely related to the notion of growth models and growth distances~\cite{GrowthDistance}. Note that similarly to what is described for the sos-convex case in Section \ref{subsec:eucl.dist},
%as a consequence of Theorem \ref{th:lasserre}, 
the optimal solution $d(p_1||p_2)$ to (\ref{eq:overlap}) can be computed exactly using semidefinite programming, or using a generic convex optimizer. The two leftmost subfigures of Figure~\ref{fig:penetration} show a chair and a human bounded by 1-sublevel sets of degree 6 sos-convex polynomials (in green). In both cases, we compute $d(p_1||p_2)$ and $d(p_2||p_1)$ and plot the corresponding minimizers. In the first subfigure, the level set of the chair needs to grow in order to touch the human and vice-versa, certifying separation. In the second subfigure, we translate the chair across the volume occupied by the human so that they overlap. In this case, the level sets need to contract. In the third subfigure, we plot the optimal value of the problem in (\ref{eq:overlap}) as the chair is translated from left to right, showing how the growth distances dip upon penetration and rise upon separation. The final subfigure shows the time taken to solve (\ref{eq:overlap}) when warm started from the previous solution. The time taken is of the order of 150 milliseconds without warm starts to 10 milliseconds with warm starts.

%\subsection{Separation and penetration under rigid body motion}
%Suppose $\{x ~|~ p({x})\leq 1\}$ is a minimum-volume sos-convex body enclosing a rigid 3D object. If the object is rotated by ${R}\in SO(3)$ and translated by ${t}\in \reals^3$, then the polynomial $p'({x}) = p({R}^T {x} - {R}^T {t})$ encloses the transformed object. This is because, if $p({x}) \leq 1$, then $p'({R}{x} + {t}) \leq 1$. For continuous motion, the optimization for Euclidean distance or sublevel-based separation/penetration distances can be warm started from the previous solution. The computation of the gradient of these measures %with respect to pose 
%using parametric convex optimization, and exploring the potential of this idea for motion planning is left for future work.

\section{Containment of polynomial sublevel sets} \label{sec:cont.poly.sub}

In this section, we show how the sum of squares machinery can be used in a straightforward manner to contain polynomial sublevel sets (as opposed to point clouds) with a convex polynomial level set. More specifically, we are interested in the following problem: Given a basic semialgebraic set 

\begin{equation}\label{eq:basic.semialgebraic.set.S}
\mathcal{S}\mathrel{\mathop:}=\{x\in\mathbb{R}^n| \ g_1(x)\leq 1,\ldots, g_m(x)\leq 1   \},
\end{equation}
find a convex polynomial $p$ of degree $2d$ such that 
\begin{equation}\label{eq:S.in.sublevelset.p}
\mathcal{S}\subseteq \{x\in\mathbb{R}^n|\ p(x)\leq 1\}.
\end{equation} 
Moreover, we typically want the unit sublevel set of $p$ to have small volume. Note that if we could address this question, then we could also handle a scenario where the unit sublevel set of $p$ is required to contain the union of several basic semialgebraic sets (simply by containing each set separately). For the 3D geometric problems under our consideration, we have two applications of this task in mind:

\begin{itemize}
	\item {\bf Convexification:} In some scenarios, one may have a nonconvex outer approximation of an obstacle (e.g., obtained by the computationally inexpensive inverse moment approach of Lasserre and Pauwels as described in Section \ref{subsec:nonconvex}) and be interested in containing it with a convex set. This would e.g. make the problem of computing distances among obstacles more tractable; cf. Section \ref{sec:distance}. 
	
	\item {\bf Grouping multiple obstacles:} For various navigational tasks involving autonomous agents, one may want to have a mapping of the obstacles in the environment in varying levels of resolution. A relevant problem here is therefore to group obstacles: this would lead to the problem of containing several polynomial sublevel sets with one.
\end{itemize}

In order to solve the problem laid out above, we propose the following sos program:
\begin{align}
&\min_{p \in \mathbb{R}_{2d}[x],\tau_i \in \mathbb{R}_{2\hat{d}}[x],P \in S^{N \times N}} -\log \det(P) \nonumber\\
\text{s.t. } &p(x)=z(x)^TPz(x),  P\succeq 0, \nonumber\\
&p(x) \quad \text{sos-convex,} \label{eq:p.sosconvex}\\
& 1-p(x)-\sum_{i=1}^m \tau_i(x)(1-g_i(x)) \quad \text{sos,} \label{eq:s.procedure}\\ 
& \tau_i(x) \quad \text{sos,} \label{eq:tau.sos}\quad i=1,\ldots,m.
\end{align}

%$\tilde{N}=\binom{n+d-1}{d-1}$

It is straightforward to see that constraints (\ref{eq:s.procedure}) and (\ref{eq:tau.sos}) imply the required set containment criterion in (\ref{eq:S.in.sublevelset.p}). As usual, the constraint in (\ref{eq:p.sosconvex}) ensures convexity of the unit sublevel set of $p$. The objective function attempts to minimize the volume of this set. A natural choice for the degree $2\hat{d}$ of the polynomials $\tau_i$ is $2\hat{d}=2d-\min_i deg(g_i)$, though better results can be obtained by increasing this parameter.

An analoguous problem is discussed in recent work by Dabbene, Henrion, and Lagoa~\cite{Henrion1, Henrion}. In the paper, the authors want to find a polynomial $p$ of degree $d$ whose 1-superlevel set $\{x~|~ p(x)\geq 1\}$ contains a semialgebraic set $\mathcal{S}$ and has minimum volume. Assuming that one is given a set $B$ containing $\mathcal{S}$ and over which the integrals of polynomials can be efficiently computed, their method involves searching for a polynomial $p$ of degree $d$ which minimizes $\int_B p(x)dx$ while respecting the constraints $p(x)\geq 1$ on $\mathcal{S}$ and $p(x) \geq 0$ on $B$. Note that the objective is linear in the coefficients of $p$ and that these last two nonnegativity conditions can be made computationally tractable by using the sum of squares relaxation. The advantage of such a formulation lies in the fact that when the degree of the polynomial $p$ increases, the objective value of the problem converges to the true volume of the set $\mathcal{S}$.

%However, when the degree of the polynomial is fixed, one is not guaranteed to obtain the minimum-volume sub/super-level set (among all sub/super-level sets generated by polynomials of that degree) enclosing $\mathcal{S}$. In the quadratic case for example, this method may not return the minimum-volume ellipsoid containing $\mathcal{S}$ whereas our approach will.}

\textbf{Example.} In Figure~\ref{fig:mult.cont}, we have drawn in black three random ellipsoids and a degree-4 convex polynomial sublevel set (in yellow) containing the ellipsoids. This degree-4 polynomial was the output of the optimization problem described above where the sos multipliers $\tau_i(x)$ were chosen to have degree $2$.
\begin{figure}[ht]
	\begin{center}
		\includegraphics[scale=0.5]{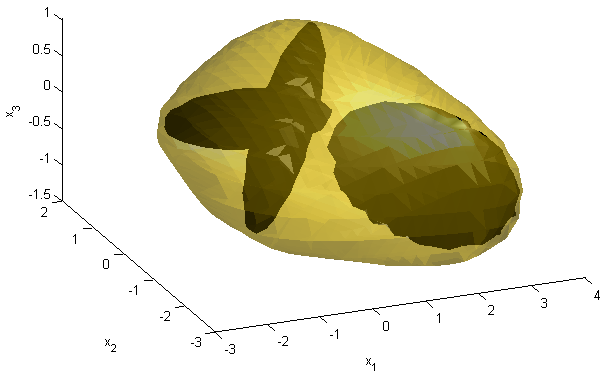}
		\caption{Containment of 3 ellipsoids using a sublevel set of a convex degree-4 polynomial}
		\label{fig:mult.cont}
	\end{center}
\end{figure}

We end by noting that the formulation proposed here is backed up theoretically by the following converse result.

\begin{theorem}
	Suppose the set $\mathcal{S}$ in (\ref{eq:basic.semialgebraic.set.S}) is Archimedean
	%\footnote{See, e.g., \cite{laurent2009sums} for a formal definition. This is a mild assumption that can always be met, e.g., if the set $\mathcal{S}$ is compact and the radius of a ball containing it is known.}
	and that $\mathcal{S}\subset \{x\in\mathbb{R}^n|\ p(x)\leq 1\}.$ Then there exists an integer $\hat{d}$ and sum of squares polynomials $\tau_1,\ldots,\tau_m$ of degree at most $\hat{d}$ such that 
	\begin{equation}
	1-p(x)-\sum_{i=1}^m \tau_i(x)(1-g_i(x)) 
	\end{equation}
	is a sum of squares.
\end{theorem}

\begin{proof}
	The proof follows from a standard application of Putinar's Positivstellensatz \cite{putinar1993positive} and is omitted.
\end{proof}

\chapter{Nonnegative polynomials and shape-constrained regression}\label{ch:mihaela}

Unlike the other chapters in this thesis, the paper on which this chapter is based is still in preparation. We recommend that future readers read the submitted version of this chapter if it is available at the time of reading.

\section{Introduction}

Regression is a key problem in statistics and machine learning. Its goal is to estimate relationships between an \emph{explained variable} (e.g., the price of a second-hand car) and a \emph{vector of explanatory variables} (e.g., the make, brand, mileage, power, or age of this car). In many applications, one can observe a monotonous dependency between the explained variable and the explanatory variables. Examples arise in many different areas, including medicine, e.g., loss of hippocampus gray matter with respect to age \cite{hippo} or survival rate with respect to white blood cell count in patients fighting leukemia \cite{leukemia}; biology and environmental engineering, e.g., frequency of occurrence of a specific plant as a function of environment pollution~\cite{plants}; electrical and computer engineering, e.g., failure rate of software as a function of number of bugs \cite{software_failure}; economics, e.g., production output of a competitive firm as a function of its inputs, \cite{pricing}; and civil engineering, e.g., total shaded area on the floor of a room as a function of length of a blind over the window in that room \cite{shade}, to name a few.

In addition or in parallel to monotonicity, one may also wish to impose convexity or concavity constraints on the regressor. Examples where such a need arises are given, e.g., in \cite{xu2016faithful}. They include geometric programming \cite{BoydBook}, computed tomography \cite{prince1990reconstructing}, target reconstruction \cite{lele1992convex}, circuit design \cite{hannah2012ensemble}, queuing theory \cite{chen2016generalized}, and utility function estimation in economics \cite{meyer1968consistent}.

In the following, we refer to the problem of fitting a convex or monotonous regressor to data as \emph{shape-constrained regression}. As evidenced above, this problem appears ubiquitously in applications and has consequently been widely studied. We review prior literature on both monotone and convex regression below. We focus on \emph{polynomial} regression as this will be the subject of interest throughout this chapter.

\paragraph{Prior work on monotone regression.} Past work on monotonically-constrained polynomial regression has by and large focused on univariate polynomials. Methods that enforce monotonicity include specific parametrizations of polynomial families (see \cite{elphinstone1983target} and \cite{mckay2011variable}) or iterative algorithms that leverage geometric properties of univariate polynomials (in \cite{hawkins1994fitting} for example, the derivative of the polynomial is constrained to be zero at inflection points). Extensions to multivariate polynomials involve adding univariate polynomials together to get a (separable) multivariate polynomial, which ignores interactions between explanatory variables (see \cite{mckay2011variable}). Furthermore, all the methods considered in this paragraph impose monotonicity of the regressor globally, as opposed to over a given set, which may be too restrictive.

Another way of obtaining monotonous (but not necessarily polynomial) predictive models is via the use of artificial neural networks (ANNs). The easiest way to guarantee that an ANN outputs an increasing function with respect to all features is to keep the edge weights in the neural net nonnegative, see \cite{wang1994neural,kay2000estimating,dugas2001incorporating,dugas2009incorporating,zhang1999feedforward}. However, it has been shown in \cite{daniels2010monotone} that in order for a neural network with nonnegative weights to approximate any monotonically increasing function in $n$ features arbitrarily well, the ANN must have $n$ fully connected hidden layers, which can lead to computational limitations and requires a large training dataset.

\emph{Interpolated look-up tables} are another popular approach to monotone regression (see, e.g., \cite{gupta2016monotonic}). Here, the feature space is discretized into different cells, and each point in the feature space $x$ is associated to a vector of linear interpolation weights $\phi(x)$, which reflects the distance of $x$ to each vertex of the specific cell it belongs to. The function we wish to learn is then given by a linear combination of $\phi(x)$, i.e., $f(x)=\theta^T\phi(x)$, and the parameter $\theta$ is obtained by solving $\min_{\theta} l(y_i, \theta^T\phi(x_i))$, where $l$ is a convex loss function. If the entries of $\theta$ satisfy some pairwise constraints, then the function $f$ is guaranteed to be monotonous. We remark that in this approach, the size of $\theta$, and so the number of variables, is exponential in the number of features.

Finally, we mention two other research directions which also involve breaking down the feature domain into smaller subsets. These are \emph{regression trees} and \emph{isotonic regression}. In the first, the feature domain is recursively partitioned into smaller subdomains, where interactions between features are more manageable. On each subdomain, a fit to the data is computed, and to obtain a function over the whole domain, the subdomain fits are aggregated, via, e.g., gradient boosting; see \cite{breiman1984classification,freund1999short,ridgeway1999state,friedman2001greedy}. To obtain monotone regressors, one enforces monotonicity on each subregion, as aggregation maintains this structural property \cite{chipman2016high,hofner2011boosting}. In the second method, a piecewise constant function $f$ is fitted to the data in such a way that $f(x_i)\leq f(x_j)$ if $x_i$ and $x_j$ are breakpoints of the function and $x_i \succeq x_j$, where $\succeq$ is some partial or total ordering. Both of these methods present some computational challenges in the sense that, much like interpolated look-up tables, they scale poorly in the number of features. In the case of the second method, the function produced also lacks some desirable analytic properties, such as smoothness and differentiability.

\paragraph{Prior work on convex regression.} The work by Magnani, Lall, and Boyd in \cite{Magnani} is the closest to what is presented in this chapter. Similarly to what is done here, a sum of squares approach to impose convexity of their polynomial regressor is used in that reference. However, contrarily to us, convexity is imposed globally, and not locally. Furthermore, our focus in this chapter is on approximation results and computational complexity analysis, which is not a focus of their work. Other methods for computationally efficient convex regression involve fitting a piecewise linear model to data. This is done, e.g., in \cite{hannah2013multivariate,magnani2009convex}. Other related work in the area consider convex regression from a more statistical viewpoint. The reference in \cite{guntuboyina2015global} for example, studies maximum likelihood estimation for univariate convex regression whereas \cite{seijo2011nonparametric}, \cite{lim2012consistency}, and more recently \cite{mazumder2017computational} study the multivariate case. In particular, the first two papers show consistency of the maximum likelihood estimator whereas the latter paper provides a more efficient and scalable framework for its computation.

\subsection{Outline}
The outline of the chapter is as follows. In Section \ref{sec:problem.formulation}, we specify our problem formulation in more detail. In particular, we define the notion of monotonicity profile (which encodes how the polynomial regressor varies depending on each variable) in Definition \ref{def:mon.prof}. In Section \ref{sec:np.hard}, we show that both the problem of testing whether a polynomial has a certain monotonicity profile over a box and the problem of testing whether a polynomial is convex over a box are NP-hard already for cubic polynomials (Theorems \ref{th:mon.np.hard} and \ref{th:np.hardness.convex}). This motivates our semidefinite programming-based relaxations for fitting a polynomial that is constrained to be monotone or convex to data. These are presented in Section \ref{sec:sdp}. Among other things, we show that any monotone (resp. convex) function can be approximated to arbitrary accuracy by monotone (resp. convex) polynomials, with sum of squares certificates of these properties (Theorems \ref{th:mono.existence} and \ref{th:convex.existence}). In Section \ref{sec:experiments}, we show how our methods perform on synthetic regression problems as well as real-world problems (namely predicting interest rates for personal loans and predicting weekly wages). In particular, we show that in both real-world problems, the shape-constrained regressor provides a lower root mean squared error on testing data than the unconstrained regressor. 

\subsection{Notation}
We briefly introduce some notation that will be used in the rest of the chapter. A matrix $M$ is said to be positive semidefinite (psd) if $x^TMx \geq 0$ for all $x \in \mathbb{R}^n.$ We write $M \succeq 0$ to signify that $M$ is psd. We will denote by $\lambda_{\max}(M)$ (resp. $\lambda_{\min}(M)$) the largest (resp. smallest) eigenvalue of $M$. Given positive integers $m$ and $n$, we let $0_{m \times n}$ and $1_{m \times n}$ be the matrices of dimension $m \times n$ which contain respectively all zeros, or all ones. We will write $I$ for the identity matrix and $e_j$ for the $j^{th}$ basis vector, i.e., a vector in $\mathbb{R}^n$ of all zeros, except for the $j^{th}$ component which is equal to 1. Finally, we denote the Hessian matrix of a twice continuously differentiable function $f$ by $H_f$.

\section{Problem formulation}\label{sec:problem.formulation}

In this chapter, we consider the problem of \emph{polynomial regression}, i.e., the problem of fitting a polynomial function $p:\mathbb{R}^n \rightarrow \mathbb{R}$ to data points $(x_i,y_i)$, $i=1,\ldots,m$. Here, $x_i$ is a vector in $\mathbb{R}^n$, often called the \emph{feature vector} or \emph{vector of explanatory variables}, and $y_i$ is a scalar corresponding to the response. To obtain our regressor $p$, we fix its degree and search for its coefficients such that $p$ minimizes some convex loss function. This could be, e.g., the \emph{least squares error}, $$\min_{p} \sum_{i=1}^m (y_i-p(x_i))^2,$$ or, the \emph{least absolute deviation error}, $$\min_{p} \sum_{i=1}^m |y_i-p(x_i)|.$$
In our setting, we would additionally like to add shape constraints to our regressor, such as monotonicity or convexity. More specifically, we consider the model we outline next. We assume that $y_i$ is a measurement of an underlying unknown (not necessarily polynomial) function $f:\mathbb{R}^n \mapsto \mathbb{R}$ at point $x_i$ corrupted by some noise $\epsilon_i$. In other words, we have
\begin{align}\label{eq:f.noise}
y_i=f(x_i)+\epsilon_i,~\forall i=1,\ldots,m.
\end{align}
We further assume that we possess prior knowledge regarding the shape of $f$, e.g., increasing in variable $j$ or convex over a certain region. We would then like our regressor $p$ to have the same attributes. This is a very natural problem when considering applications such as those discussed in the introduction of this chapter.

Throughout the chapter, we assume that our feature vectors $x_i$ belong to a \emph{box}
\begin{align}\label{def:box}
B\mathrel{\mathop{:}}=[b_1^-,b_1^+]\times \cdots \times [b_n^-,b_n^+],
\end{align} 
where $b_1^-,b_1^+,\ldots,b_n^-,b_n^+$ are real numbers satisfying $b_i^-\leq b_i^+, \forall i=1,\ldots,n$. In practice, this is often, if not always, the case, as features are generally known to lie within certain ranges. We would like to mention nevertheless that the techniques presented in this chapter can be extended to any feature domain that is \emph{basic semialgebraic}, i.e., defined by a finite number of polynomial equalities and inequalities. The shape constraints we define next are assumed to hold over this box: they are, respectively, monotonicity over $B$ with respect to a feature and convexity over $B$.

\begin{definition}[Monotonicity over a box with respect to a variable]
	We say that a function $f:\mathbb{R}^n \rightarrow \mathbb{R}$ is monotonically increasing\footnote{Throughout this chapter, we will use the terminology \emph{increasing} (resp. decreasing) to describe a property which is perhaps more commonly referred to as \emph{nondecreasing} (resp. nonincreasing). This is to avoid potential confusion arising from the use of a negation.}  over a box $B$ with respect to a variable $x_j$ if $$f(x_1,\ldots,x_j,\ldots,x_n)\leq f(x_1,\ldots,y_j,\ldots,x_n)$$ for any fixed $(x_1,\ldots,x_j,\ldots,x_n)$, $(x_1,\ldots,y_j,\ldots, x_n) \in B$ with $x_j \leq y_j$.
	Similarly, we say that $f$ is monotonically decreasing over $B$ with respect to variable $x_j$ if $$f(x_1,\ldots,x_j,\ldots,x_n)\geq f(x_1,\ldots,y_j,\ldots,x_n)$$ for any fixed $(x_1,\ldots,x_j,\ldots,x_n)$, $(x_1,\ldots,y_j,\ldots, x_n) \in B$ with $x_j \leq y_j$.
\end{definition}

For differentiable functions, an equivalent definition of monotonicity with respect to a variable---and one we will use more frequently---is given below.

\begin{lemma}
	A differentiable function $f$ is monotonically increasing (resp. decreasing) over a box $B$ with respect to a variable $x_j$ if and only if $\frac{\partial f(x)}{\partial x_j}\geq 0$ (resp. $\frac{\partial f(x)}{\partial x_j} \leq 0$) for all $x \in B.$
\end{lemma}

\begin{proof}
	We prove the increasing version of the theorem as the decreasing version is analogous. 
	Suppose that $f$ is monotonically increasing with respect to variable $x_j$. This implies that for any fixed $(x_1,\ldots, x_n) \in B$ and for any $\epsilon>0$ with $x_j+\epsilon e_j \in B$, we have $$f(x_1,\ldots,x_j,\ldots,x_n) \leq f(x_1,\ldots,x_j+\epsilon e_j,\ldots,x_n),$$ 
 which is equivalent to
	$$\frac{1}{\epsilon}\left(f(x_1,\ldots,x_j+\epsilon e_j,\ldots,x_n) - f(x_1,\ldots,x_j,\ldots,x_n)\right)\geq 0.$$ 
	By taking the limit as $\epsilon \rightarrow 0$, we obtain that $\frac{\partial f(x)}{\partial x_j} \geq 0$ for all $x \in B.$ 
	
	Suppose now that $\frac{\partial f(x)}{\partial x_j}\geq 0$ for all $x \in B.$ 
	%First note that for any point in $B$ of the form $(x_1,\ldots,x_{j-1},b_j^+,\ldots,x_n)$, we have $$f(x_1,\ldots,b_j^+,\ldots,x_n)\leq f(x_1,\ldots,y_j,\ldots,x_n)$$ for any $(x_1,\ldots,y_j,\ldots,x_n) \in B$ such that $b_j^+ \leq y_j$. Indeed, the only possible value $y_j$ can take is $b_j^+$ for which this inequality trivially holds. 
	Fix any point $x=(x_1,\ldots,x_n)\in B$. There exists $\epsilon\geq 0$ such that $x_\epsilon=x+\epsilon e_j\in B.$ By Taylor's formula with an integral remainder, we have
	$$f(x_{\epsilon})=f(x)+\int_{t=0}^1 \nabla f(x+t(x_\epsilon-x))^T (x_\epsilon-x)dt,$$ which is equivalent to $$f(x_{\epsilon})=f(x)+ \epsilon \int_{t=0}^1 \frac{ \partial f(x+\epsilon t e_j)}{\partial x_j} dt.$$
	Since $x_\epsilon \in B$ and $B$ is a box, $x+\epsilon t e_j \in B$ for any $t \in [0,1]$. Hence $$\frac{\partial f(x+\epsilon t e_j)}{\partial x_j} \geq 0, \forall t\in [0,1].$$ As we are integrating a nonnegative integrand, we get that $f(x_\epsilon)-f(x) \geq 0$ for any nonnegative $\epsilon$. This concludes our proof as $y_j \geq x_j$ implies that $y_j=x_j+\epsilon$ for some $\epsilon\geq0$.
\end{proof}

We now define a notion that encapsulates how a differentiable function varies with respect to each of its variables.

\begin{definition}[Monotonicity profile] \label{def:mon.prof}
	For any differentiable function $f$, its \emph{ monotonicity profile} over a box $B$ is a vector in $\mathbb{R}^n$ with entries defined as follows: 
	\begin{align*}
	\rho_j=\begin{cases} 1 \text{ if $f$ is monotonically increasing over $B$ with respect to $x_j$}\\ -1 \text{ if $f$ is monotonically decreasing over $B$ with respect to $x_j$}\\ 0 \text{ if there are no monotonicity requirements on $f$ with respect to $x_j$.}  \end{cases}
	\end{align*}
\end{definition}

When we assume that we have prior knowledge with respect to monotonicity of our underlying function $f$ in (\ref{eq:f.noise}), we in fact mean that we have access to the monotonicity profile of $f$. 

We now consider another type of shape constraint that we are interested in: convexity over a box.

\begin{definition}[Convexity over a box]
	We say that a function $f:\mathbb{R}^n \rightarrow \mathbb{R}$ is convex over a box $B$ if $$f(\lambda x+(1-\lambda)y) \leq \lambda f(x)+(1-\lambda) f(y), \forall x,y \in B, \forall \lambda \in (0,1).$$	
\end{definition}

\begin{proposition}
	A twice-differentiable function $f$ is convex over a box $B$ if and only if $H_f(x)\succeq 0$ for all $x \in B$.
\end{proposition}

The proof of this proposition readily follows from the proof of the analogous proposition for global convexity; see, e.g., Theorem 22.5 in \cite{chong.zak}.

%\begin{definition}[Convexity over a box]
%	A twice-differentiable function $f$ is said to be \emph{convex over a box $B\subseteq \mathbb{R}^n$} if $H_f(x) \succeq 0$ for all $x \in B.$
%\end{definition}
%Note that this notion can easily be extended to convexity over an arbitrary set $S$.

\section{Computational complexity results}\label{sec:np.hard}

As mentioned previously, we would like to optimize some convex loss function over the set of polynomial regressors constrained to be convex or monotonous over a box $B$. In this section, we show that, unless P=NP, one has no hope of doing this in a tractable fashion as even the problem of testing if a given polynomial has these properties is NP-hard.
\begin{theorem}\label{th:mon.np.hard}
	Given a cubic polynomial $p$, a box $B$, and a monotonicity profile $\rho$, it is NP-hard to test whether $p$ has profile $\rho$ over $B$.
\end{theorem}

\begin{proof}
	We provide a reduction from the MAX-CUT problem, which is well known to be NP-hard~\cite{GareyJohnson_Book}. Consider an unweighted undirected graph $G=(V,E)$ with no self-loops. A cut in $G$ is a partition of the $n$ nodes of the graph into two sets, $S$ and $\bar{S}$. The size of the cut is the number of edges connecting a node in $S$ to a node in $\bar{S}$. MAX-CUT is the following decision problem: given a graph $G$ and an integer $k$, test whether $G$ has a cut of size at least $k$. We denote the adjacency matrix of the graph $G$ by $A$, i.e., $A$ is an $n \times n$ matrix such that $A_{ij}=1$ if $\{i,j\} \in E$ and $A_{ij}=0$ otherwise. We let $$\gamma\mathrel{\mathop{:}}=\max(0,\max_{i} \{A_{ii}+\sum_{j\neq i} |A_{ij}|\}).$$ Note that $\gamma$ is an integer and an upper bound on the largest eigenvalue of $A$ from Gershgorin's circle theorem \cite{gersh}.
	
	We will show that testing whether $G$ has a cut of size at least $k$ is equivalent to testing whether the polynomial 
	\begin{align*}
	p(x_1,\ldots,x_n)=&\frac{1}{4} \sum_{j=2}^n x_1^2A_{1j}x_j+\frac{1}{2}x_1\cdot (\sum_{1<i<j\leq n} x_i A_{ij}x_j)-\frac{\gamma}{12} x_1^3-\frac{\gamma}{4} x_1\sum_{i=2}^n x_i^2\\
	&+x_1 \cdot \left( k+\frac{n\gamma}{4}-\frac14 e^TAe\right)
	\end{align*}
	has monotonicity profile $\rho=(1,0,\ldots,0)^T$ over $B=[-1,1]^n$.
	
	First, note that $p$ has profile $\rho$ over $B$ if and only if $$\frac{\partial{p}(x)}{\partial x_1} \geq 0,~ \forall x\in B.$$ We have
	\begin{align*}
	\frac{\partial{p}(x)}{\partial x_1}&=\frac{1}{2}\sum_{j=2}^n x_1A_{1j}x_j+\frac{1}{2}\sum_{1<i<j \leq n} x_iA_{ij}x_j-\frac{\gamma}{4} x_1^2-\frac{\gamma}{4} \sum_{i=2}^n x_i^2 +(k+\frac{n\gamma}{4}-\frac14 e^TAe)\\
	&=\frac{1}{4}\sum_{i,j} x_iA_{ij}x_j-\frac{\gamma}{4} \sum_{i=1}^n x_i^2+(k+\frac{n\gamma}{4}-\frac14 e^TAe)\\
	&=\frac{1}{4} x^T(A-\gamma I)x+(k+\frac{n\gamma}{4}-\frac14 e^TAe).
	\end{align*}
	Hence, testing whether $p$ has profile $\rho$ over $B$ is equivalent to testing whether the optimal value of the quadratic program 
	\begin{equation}\label{eq:testing.profile}
	\begin{aligned}
	&\min_{x\in \mathbb{R}^n} &&\frac{1}{4}x^T(A-\gamma I)x\\
	&\text{s.t. } &&x \in [-1,1]^n
	\end{aligned}
	\end{equation}
	is greater or equal to $\frac{1}{4}e^TAe-k-\frac{n\gamma}{4}$. As $\gamma$ is an upperbound on the maximum eigenvalue of $A$, we have $A-\gamma I \preceq 0$, which implies that $x^T(A-\gamma I)x$ is a concave function. It is straightforward to show (see, e.g., \cite[Property 12]{bensonconcave}) that the minimum of a concave function over a compact set is attained at an extreme point of the set. As a consequence, one can rewrite (\ref{eq:testing.profile}) as
	\begin{equation*}
	\begin{aligned}
	&\min_{x\in \mathbb{R}^n} &&\frac{1}{4}x^T(A-\gamma I)x\\
	&\text{s.t. } &&x_i \in \{-1,1\}, \forall i=1,\ldots,n.
	\end{aligned}
	\end{equation*}
	As $x^Tx=n$ when $x\in \{-1,1\}^n$, testing whether $p$ has profile $\rho$ over $B$ is in fact equivalent to testing whether 
	\begin{equation}\label{eq:opt.pb.max.cut}
	\begin{aligned}
	p^*\mathrel{\mathop{:}}=&\min_{x\in \mathbb{R}^n} &&\frac{1}{4}x^TAx\\
	&\text{s.t. } &&x_i \in \{-1,1\}, \forall i=1,\ldots,n,
	\end{aligned}
	\end{equation}
	is greater or equal to $\frac{1}{4}e^TAe-k-\frac{n\gamma}{4}+\frac{n\gamma}{4}=\frac{1}{4}e^TAe-k$.

	It is easy to check that the size of the maximum cut in $G$ is equal to $\frac{1}{4}e^TAe-p^*$. Testing whether $G$ contains a cut of size at least $k$ is hence equivalent to testing whether $$p^* \geq \frac{1}{4}e^TAe-k.$$ As shown above, this is exactly equivalent to testing whether $p$ has profile $\rho$ over $B$, which concludes the proof.	
	
\end{proof}

We remark that this theorem is minimal in the degree of the polynomial in the sense that testing whether a quadratic polynomial $q$ has a given monotonicity profile $\rho$ over a box $B$ is a computationally tractable problem. Indeed, this simply amounts to testing whether the linear function $\rho_i \frac{\partial q(x)}{\partial x_i}$ is nonnegative over $B$ for all $i=1,\ldots,n.$ This can be done by solving a sequence of linear programs (in polynomial time) indexed by $i$---where the objective is $\rho_i \frac{\partial q(x)}{\partial x_i}$ and the constraints are given by the box---and testing if the optimal value is negative for some $i$.

\begin{theorem}\label{th:np.hardness.convex}
	Given a cubic polynomial $p$ and a box $B$, it is NP-hard to test whether $p$ is convex over $B$.
\end{theorem}

The proof of this theorem will use the following result of Nemirovskii~\cite{Nemirovskii_interval_matrix_NPhard}.

\begin{lemma}[cf. proof of Proposition 2.1. in \cite{Nemirovskii_interval_matrix_NPhard}]\label{lemma:nem}
	Given a positive integer $m$ and an $m$-dimensional vector $a$ with rational positive entries and with $||a||_2\leq 0.1$, let $A=(I_m-aa^T)^{-1}$ and $\mu=m-d(a)^{-2},$ where $d(a)$ is the smallest common denominator of all entries of $a$. Then, it is NP-hard to decide whether 
	\begin{align}\label{def:L}
	L(x)\mathrel{\mathop{:}}=\begin{bmatrix} A & x \\ x^T & \mu \end{bmatrix}\succeq 0, \text{ for all } x \in \mathbb{R}^m \text{ with } ||x||_{\infty} \leq 1.
	\end{align} 
	Furthermore, for any vector $a$, either (\ref{def:L}) holds (i.e., $x^TA^{-1}x \leq m-d^{-2}(a)$, $\forall ||x||_{\infty}\leq 1$) or there exists $x \in \mathbb{R}^m$ with $||x||_{\infty}\leq 1$ such that $x^TA^{-1}x \geq m.$
	%	Let $n$ be a positive integer and let $Q$ be a set of rational numbers $\{\hat{q}_{ij}, \bar{q}_{ij}, i,j=1,\ldots,n\}$ such that $\hat{q}_{ij}\leq \bar{q}_{ij}$, $\hat{q}_{ij}=\hat{q}_{ji}$, and $\bar{q}_{ij}=\bar{q}_{ji}$, for all $i,j.$ Define an interval matrix $A[Q]$ to be the family of symmetric $n \times n$ matrices with real entries $A_{ij}$ belonging to the segments $[\hat{q}_{ij}, \bar{q}_{ij}],$ for $i, j = 1,\ldots, n$. 
	%	
	%	Given $n$ and $Q$, it is NP-hard to test whether all matrices in $A[Q]$ are positive semidefinite.		
\end{lemma}

\begin{proof}[Proof of Theorem \ref{th:np.hardness.convex}]
	
	We show this result via a reduction from the NP-hard problem given in Lemma \ref{lemma:nem}. Let $m$ be a positive integer and $a$ be an $m$-dimensional vector with rational entries. Let $L(x)$ to be an $(m+1) \times (m+1)$ matrix as defined in (\ref{def:L}) and set $H(y)$ to be the $m \times (m+1)$ matrix of mixed partial derivatives of the cubic polynomial $y^TL(x)y$, i.e.,  $$H_{ij}(y)=\frac{\partial^2 y^TL(x)y}{\partial x_i \partial y_j},~\forall i=1,\ldots,m,~\forall j=1,\ldots,m+1.$$ Note that the entries of $H(y)$ are linear in $y$. Consider now the matrix $H(y)^TH(y)$, which is a symmetric matrix with entries quadratic in $y$, i.e., its $(i,j)$-entry is given by $y^TQ_{ij}y$, where $Q_{ij}$, for all $i,j$, is an $(m+1) \times (m+1)$ matrix. Denote by $q_{ij}$ the maximum entry in absolute value of the matrix $Q_{ij}$ and set 
	\begin{equation}\label{def:n.alpha.gamma}
	\begin{aligned} 
	&n\mathrel{\mathop{:}}=m+1, \\
	&\alpha\mathrel{\mathop{:}}=
	\max_i \{q_{ii}+\sum_{j \neq i}q_{ij}\} \cdot 4d^2(a)(1+m)n\\
	&\gamma\mathrel{\mathop{:}}=\frac{1}{2}\cdot \frac{1}{d^2(a)(1+m)}.
	\end{aligned} 
	\end{equation} Consider the cubic polynomial $$f(x,y)=\frac{1}{2}y^TL(x)y+\frac{n^2 \alpha}{2} x^Tx+\frac{\gamma}{6} \sum_{i=1}^{n} y_i^3,$$ and the box $$B\mathrel{\mathop{:}}= \left\{(x,y) \in \mathbb{R}^m \times \mathbb{R}^{n}~|~ x \in [-1,1]^m, y \in \left[\frac{1}{n},1\right]^{n}\right\}.$$ We will show that $f(x,y)$ is convex over $B$ if and only if $L(x)\succeq 0$ for all $x \in \tilde{B}\mathrel{\mathop{:}}=[-1,1]^m.$ Note that the Hessian of $f$ is given by $$H_f(x,y)=\begin{bmatrix}n^2 \alpha I_m  & H(y) \\ H(y)^T & L(x)+\gamma \cdot \text{diag}(y)\end{bmatrix}, $$ where $\text{diag}(y)$ is an $n\times n$ diagonal matrix with the vector $y$ on its diagonal. Hence, we need to show that $H_f(x,y) \succeq 0$ over $B$ if and only if $L(x)\succeq 0$ over $\tilde{B}.$
	
	We start by showing that if $L(x)$ is not positive semidefinite over $\tilde{B}$, then $H_f(x,y)$ is not positive semidefinite over $B$. As $L(x)$ is not positive semidefinite over $\tilde{B}$, from Lemma \ref{lemma:nem}, there exists $x_0 \in \tilde{B}$ such that $x_0^TA^{-1}x_0 \geq m.$ Let $y_0=1_{n\times 1}$ and observe that $(x_0,y_0) \in B.$ Let $$z=\frac{1}{||(-A^{-1}x_0,1)^T||_2}\begin{pmatrix} 0_{m \times 1}\\ -A^{-1}x_0\\ 1
	\end{pmatrix}.$$ We have 
	\begin{equation}\label{eq:expand.hessian}
	\begin{aligned}
	z^TH_f(x_0,y_0)z&= \frac{1}{1+||A^{-1}x_0||_2^2} \cdot\begin{bmatrix} -A^{-1}x_0\\ 1 \end{bmatrix}^T (L(x_0)+\gamma \cdot \text{diag}(y_0))\begin{bmatrix} -A^{-1}x_0\\ 1 \end{bmatrix} \\
	&= \frac{1}{1+||A^{-1}x_0||_2^2} \left( \mu -x_0^TA^{-1}x_0 +\gamma (1+||A^{-1}x_0||_2^2) \right)\\
	&=\frac{\mu-x_0^TA^{-1}x_0}{1+||A^{-1}x_0||_2^2}+\gamma.
	\end{aligned}
	\end{equation}
	Since $x_0^TA^{-1}x_0 \geq m$ and $\mu=m-d^{-2}(a)$, we have 
	\begin{align}\label{eq:ub.schur}
	-\frac{1}{d^2(a)}\geq \mu -x_0^T A^{-1}x_0.
	\end{align} Furthermore, 
	\begin{align*}
	1+||A^{-1}x_0||_2^2&\leq 1+x_0^T(I_m-aa^T)^2 x_0\\
	&=1+x_0^T (I_m-2aa^T+||a||_2^2aa^T)x_0\\
	&=1+||x_0||^2-2||a^Tx_0||_2^2+||a||_2^2||a^Tx_0||^2\\
	&\leq 1+m||x_0||_\infty^2\\
	&\leq 1+m,
	\end{align*}
	where we have used in the last two inequalities the facts that $||a||_2\leq 0.1,$ $||x_0||_{\infty}\leq 1$, and $||x_0||_2\leq \sqrt{m} ||x_0||_{\infty}$. Combining this with (\ref{eq:expand.hessian}) and (\ref{eq:ub.schur}), we get
	$$z^TH_f(x_0,y_0)z\leq -\frac{1}{d^2(a)\cdot (1+m)}+\gamma.$$
	Replacing $\gamma$ by its expression in (\ref{def:n.alpha.gamma}), we obtain:
	$$z^TH_f(x_0,y_0)z \leq -\frac{1}{2} \cdot \frac{1}{d^2(a)\cdot (1+m)}<0,$$
	and conclude that $H_f(x,y)$ is not positive semidefinite over $B$.
	
	Suppose now that $L(x) \succeq 0$ for all $x \in \tilde{B}.$ We will show that $H_f(x,y) \succeq 0$ for all $(x,y) \in B.$ As $\alpha>0$, we equivalently show, using the Schur complement, that $$L(x)+\gamma \cdot\text{diag}(y) -\frac{1}{n^2 \alpha} H(y)^TH(y) \succeq 0, \text{ for all } (x,y) \in B.$$ As $L(x) \succeq 0$ for any $x \in \tilde{B}$, it remains to show that $$\gamma \cdot \text{diag}(y)-\frac{1}{n^2 \alpha} H(y)^TH(y) \succeq 0$$ for all $y \in [\frac{1}{n},1]^n.$ Fix $y\in [\frac{1}{n},1]^n$. Note that $$\gamma \cdot \text{diag}(y) \succeq  \gamma \cdot \frac{1}{n}I=\frac{1}{2d^2(a)(1+m)n}I$$ and that $H(y)^TH(y)\succeq 0$. Recall that entry $(i,j)$ of $H(y)^TH(y)$ is given by $y^TQ_{ij}y$ and that $q_{ij}$ is the maximum entry in absolute value of $Q_{ij}$. Simple algebra and the fact that $y \in [1/n,1]^n$ show that $|y^TQ_{ij}y| \leq q_{ij}n^2$. We then have 
	\begin{align*}
	\lambda_{\max} (H(y)^TH(y)) &\leq \max_i \{y^TQ_{ii}y+\sum_{j \neq i} y^TQ_{ij}y \}\\
	&\leq  \max_i \{q_{ii}n^2+\sum_{j \neq i} q_{ij}n^2 \}\\
	&=n^2 \cdot \max_i \{q_{ii} +\sum_{j \neq i}q_{ij}\},
	\end{align*}
	where the first inequality is a consequence of Gershgorin's circle theorem \cite{gersh}.
	We deduce that 
	\begin{align*}
	\gamma \cdot \text{diag}(y)-\frac{1}{n^2\alpha} H(y)^TH(y) \succeq \left(\frac{1}{2d^2(a)(1+m)n}-\frac{1}{\alpha} \cdot \max_i \{q_{ii}+\sum_{j \neq i} q_{ij}\}\right)I.
	\end{align*}
	Replacing $\alpha$ by its expression given in (\ref{def:n.alpha.gamma}), we get that $$\gamma \cdot \text{diag}(y)-\frac{1}{n^2\alpha} H(y)^TH(y) \succeq \frac{1}{4d^2(a)(1+m)n}I \succeq 0,$$
	which concludes the proof.
\end{proof}

Note that again this theorem is minimal in the degree of the polynomial. Indeed testing whether a quadratic polynomial is convex over a box is equivalent to testing if a quadratic form is convex over $\mathbb{R}^n$ (this is a consequence of the Hessian being constant). The latter condition can be tested by checking if its (constant) Hessian matrix is positive semidefinite. This can be done in polynomial time~\cite{nonnegativity_NP_hard}.

Independently of shape-constrained regression, we would like to remark that Theorem \ref{th:np.hardness.convex} is interesting in its own right. It has been shown in \cite{NPhard_Convexity_MathProg} that testing whether a quartic polynomial is convex over $\mathbb{R}^n$ is an NP-hard problem. One could wonder if this problem would get any easier over a region. This theorem answers the question in the negative, and shows that this problem is hard even for lower-degree polynomials. This is particularly relevant as subroutines of some optimization software (e.g., BARON \cite{baron}) involve testing convexity of functions over a set (typically a box). The result presented here shows that efficient algorithms for testing convexity over a box are very unlikely to always return the correct answer.

\section{Semidefinite programming-based relaxations}\label{sec:sdp}

In light of the previous hardness results, we provide tractable relaxations of the previous concepts, i.e., monotonocity over a box and convexity over a box, involving semidefinite programming. These relaxations are based on the notion of \emph{sum of squares} polynomials, which we provide a brief exposition of below.

\subsection{Review of sum of squares polynomials}

A polynomial $p$ is a sum of squares (sos) if it can be written as a sum of squares of other polynomials, i.e., $p(x)=\sum_i q_i(x)^2,$ where $q_i(x)$ are some polynomials. Being a sum of squares is obviously a sufficient condition for nonnegativity. It is not however necessary, as the Motzkin polynomial (which is nonnegative but not sos) can attest to~\cite{MotzkinSOS}. Sum of squares polynomials are widely used as a surrogate for nonnegative polynomials as one can optimize over the set of sos polynomials using semidefinite programming (SDP) contrarily to nonnegative polynomials, which form an intractable set to optimize over. The fact that one can optimize over the set of sos polynomials using semidefinite programming is the consequence of the following theorem: a polynomial $p$ of degree $2d$ is a sum of squares if and only if there exists a positive semidefinite matrix $Q$ such that $p(x)=z(x)^TQz(x)$, where $z(x)=(1,x_1,\ldots,x_n,x_1x_2\ldots,x_n^d)$ is the vector of standard monomials of degree $\leq d$. We say that an $m \times m$ polynomial matrix $M(x)$ is an \emph{sos-matrix} if there exists a polynomial matrix $V(x)$ of size $q \times m$, where $q$ is some integer, such that $M(x)=V(x)^TV(x)$. This is equivalent to the polynomial $y^TM(x)y$ in $2n$ variables $(x,y)$ being a sum of squares.

\subsection{Relaxations and approximation results}

In this section, we revisit the task of fitting a polynomial function $p$ to data $(x_i,y_i) \in \mathbb{R}^n \times \mathbb{R}$ generated from noisy measurements of a function $f$:
\begin{align}\label{eq:f.def.2}
y_i=f(x_i)+\epsilon_i, ~i=1,\ldots,m.
\end{align}
With no constraints on the regressor, this fit can be obtained by minimizing some convex loss function such as the least squares error $\sum_{i=1}^m (p(x_i)-y_i)^2.$ Here, we consider two different cases of constrained regression, corresponding to two shape constraints on the function $f$ in (\ref{eq:f.def.2}) that generates the data. For concreteness, we will throughout use the least squares error as our convex loss function, though our algorithms can be extended to hold for other convex loss functions such as the least absolute deviation function or any sos-convex polynomial loss function.

\subsubsection{Monotonically-constrained polynomial regression }
We assume that the monotonicity profile $\rho$ of $f$ in (\ref{eq:f.def.2}) as well as a box $B$ which contains the feature vectors are given. We wish to fit a polynomial $p$ to data $(x_i,y_i),~i=1,\ldots,m$ generated using $f$, such that $p$ also has monotonicity profile $\rho$ over $B$. In other words, we are interested in solving the following optimization problem:
\begin{equation}\label{eq:opt.solve.monoton.hard}
\begin{aligned}
f_{mon}\mathrel{\mathop{:}}=&\inf_{p \text{ of degree } d} &&\sum_{i=1}^m (p(x_i)-y_i)^2\\
&\text{s.t. } &&\rho_j \frac{\partial p(x)}{\partial x_j}\geq 0,~\forall j=1,\ldots,n \text{ and for all } x \in B.
\end{aligned}
\end{equation}

Theorem \ref{th:mon.np.hard} suggests that this problem cannot be solved efficiently unless $P=NP$. We present here a relaxation of this problem with some formal guarantees.

\begin{theorem}\label{th:mono.existence}
	Let $f$ be a $C^1$ function with monotonicity profile $\rho$ over a box 
	\begin{align}\label{def:box.mono.ex}
	B=\{(u_1,\ldots,u_n) \in \mathbb{R}^n ~|~ (b_i^+-u_i)(u_i-b_i^-)\geq 0, \forall i=1,\ldots,n\}.
	\end{align} 
	For any $\epsilon>0$, there exists an integer $d$ and a polynomial $p$ of degree $d$ such that $$\max_{x \in B} |f(x)-p(x)|<\epsilon$$
	and such that $p$ has same monotonicity profile $\rho$ over $B$.	Furthermore, this monotonicity profile can be certified using sums of squares certificates. 
\end{theorem}

Note that the definition of the box given here is slightly different to the one given in (\ref{def:box}). The way the feature box is described actually comes into play in the structure of the certificate of nonnegativity of the derivative of $p$ which we wish to obtain for Theorem \ref{th:mono.existence}. A similar result to the one given above can be obtained using the definition of the box given in (\ref{def:box}). We will discuss this distinction further in Remark \ref{rem:box.rep}.

The proof of this theorem uses Putinar's Positivstellensatz \cite{putinar1993positive}, which we repeat for completeness after the following lemma.

\begin{lemma}\label{lem:approx.poly}
	
	Let $m$ be a nonnegative integer and assume that $f \in C^m(\mathbb{R}^n)$, i.e., $f$ has continuous derivatives of order up to $m$. Let $k=(k_1,\ldots,k_n)$ be a multi-index such that $\sum_{i=1}^n |k_i| \leq m$ and let $$\partial^k f \mathrel{\mathop{:}}=\frac{\partial^k f(x)}{\partial x_1^{k_1}\ldots \partial x^{k_n}}.$$
	
	Then, for any $\epsilon>0$, there exists a positive integer $d$ and a polynomial $p$ of degree $d$ such that for any $k$, with $\sum_{i=1}^n |k_i|\leq m$, we have $$\max_{x \in B} |\partial^k p(x)-\partial^k f(x)|<\epsilon.$$
\end{lemma}

\begin{proof}
	This lemma is a straightforward consequence of Theorem 6.7 in \cite{fellhauer2016approximation},or equivalently, Theorem 4 in \cite{veretennikov2016partial}. These theorems state that under the assumptions of the lemma, for any $k$ such that $\sum_{i=1}^n |k_i|\leq m$, we have	
	$$\max_{x\in [0,1]^n}|\partial^k B_{f,d}(x)-\partial^k f(x)| \underset{d \rightarrow \infty}{\rightarrow} 0,$$ where $B_{f,d}(x)$ is the Bernstein polynomial approximation to $f$ of order $d$, defined over $[0,1]^n$. This is the following polynomial
	$$B_{f,d}(x)=\sum_{j_1,\ldots,j_n=0}^n f\left(\frac{j_1}{d},\ldots,\frac{j_n}{d}\right) C_d^{j_1}\ldots C_d^{j_n}x_1^{j_1}(1-x_1)^{d-j_1}\ldots x_n^{j_n}(1-x_n)^{d-j_n},$$
	where $$C_d^{j_i}=\frac{d!}{j_i!(d-j_i)!}.$$
	
	To obtain the lemma, we let $d_0=\max_i d_i$, where $nd_i$ is the degree of the Bernstein polynomial needed to obtain $\max_{x \in [0,1]^n}|\partial^{k^i} B_{f,d}(x)-f(x)|<\epsilon$ when $k^i$ is a given set of indices. The result then follows by translating and scaling the variables that define the multivariate Bernstein polynomial so that it is defined over the box $B$ rather than $[0,1]^n$.	
%	
%	Let $\epsilon>0$ and let $k^i, i=1,\ldots,N$ be the set of multi-indexes such that $\sum_{j=1}^n |k^i_j|\leq m.$ For each fixed $k^i$, by definition of a limit and using Lemma \ref{lem:Bernstein}, there exists $d_i$ such that for any $d \geq d_i$, $$\max_{x \in [0,1]^n}|\partial^{k^i} B_{d}(f,x)-f(x)|<\epsilon.$$
%	Let $d_0=\max_i d_i$. It follows that there exists $d_0$ and a polynomial $p$ of degree $nd_0$  such that for all $k$ where $\sum_i |k_i| \leq m,$ we have $$\max_{x \in [0,1]^n}|\partial^{k^i} B_{d}(f,x)-f(x)|<\epsilon.$$
%	To extend the result over to any box, one can simply translate and scale the variables that define the multivariate Bernstein polynomial.
\end{proof}

\begin{theorem}[Putinar's Positivstellensatz \cite{putinar1993positive}]\label{lem:putinar}
	Let $$S=\{x \in \mathbb{R}^n ~|~ g_1(x)\geq 0, \ldots, g_m(x)\geq 0\}$$ and define by $$M(g)=\{\sigma_0(x)+\sum_i \sigma_i(x)g_i(x)~|~ \sigma_i \text{ sos},~\forall i=0,\ldots,m\}.$$ Assume that $\{g_1,\ldots,g_m\}$ satisfy the Archimedean property, i.e., that there exists $N \in~\mathbb{N}$ such that $$N-\sum_i x_i^2\in M(g).$$ If a polynomial $p$ is positive on $S$, then $p(x) \in M(g).$
\end{theorem}

%\begin{corollary}\label{lem:putinar.box}
%	If a polynomial $p$ is positive on a box 
%	\begin{align} \label{def:box.arch}
%B=\{x \in \mathbb{R}^n~| (b_i^+-x_i)(x_i-b_i^-)\geq 0\},
%	\end{align} then there exist sum of squares polynomials $s_0(x),\ldots,s_{n}(x)$ such that 
%	\begin{align}\label{eq:box.odd}
%	p(x)=s_0(x)+\sum_{i=1}^n s_i(x)(b_i^+-x_i)(x_i-b_i^-).
%	\end{align}
%\end{corollary}
%
%\begin{proof}
%	This lemma is a direct consequence of Lemma \ref{lem:putinar}.
%	We simply need to show that $B$ as defined in (\ref{def:box.arch}) is Archimedean. As
%	$$(b_i^+)^2+(b_i^-)^2-x_i^2=(x_i-(b_i^++b_i^-))^2+2(b_i^+-x_i)(x_i-b_i^-) \in M(g),$$
%	if $N=\lceil \sum_i (b_i^+)^2+\sum_i (b_i^-)^2 \rceil$, we get 
%	$$N-\sum_i x_i^2=N-\sum_i (b_i^+)^2-\sum_i (b_i^-)^2+\sum_i ((b_i^+)^2+(b_i^-)^2-x_i^2)\in M(g).$$
%	
%	
%\end{proof}

We now prove Theorem~\ref{th:mono.existence} using these results.

\begin{proof}[Proof of Theorem \ref{th:mono.existence}]
	Let $f$ be a function in $C^1$, $B$ be a box as in (\ref{def:box.mono.ex}), and $\epsilon>0.$ Without loss of generality, we will assume that $\rho=(1,0,\ldots,0)^T$, i.e., $$\frac{\partial f(x)}{\partial x_1}\geq 0, \forall x \in B.$$ The same results can be obtained for any other monotonicity profile.
	
Let $C\mathrel{\mathop{:}}=\max_{x\in B} |x_1|.$ From Theorem~\ref{lem:approx.poly}, there must exist a polynomial $q$ of degree $d$ such that $$	\max_{x\in B} |f(x)-q(x)|\leq \frac{\epsilon}{2(1+2C)} $$ and $$\max_{x\in B} \left|\frac{\partial f(x)}{\partial x_i} - \frac{\partial q(x)}{\partial x_i}\right|\leq \frac{\epsilon}{2(1+2C)},~\forall i=1,\ldots,n.$$
	Let $p(x)\mathrel{\mathop{:}}=q(x)+\frac{\epsilon}{1+2C} \cdot x_1.$ 
	For all $x \in B$, we have
	\begin{align*}
	|f(x)-p(x)|&\leq |f(x)-q(x)|+ |q(x)-p(x)|\\
	&\leq \frac{\epsilon}{2(1+2C)}+\frac{\epsilon}{(1+2C)} \cdot C\\
	&=\frac{\epsilon/2+C\epsilon}{1+2C}\\
	&=\frac{\epsilon}{2}<\epsilon.
	\end{align*}
	Furthermore, as $\frac{\partial f(x)}{\partial x_1}\geq 0$, we have
	\begin{align*}
	\frac{\partial p(x)}{\partial x_1}&=\frac{\partial p(x)}{\partial x_1}-\frac{\partial q(x)}{\partial x_1}+\frac{\partial q(x)}{\partial x_1}-\frac{\partial f(x)}{\partial x_1}+\frac{\partial f(x)}{\partial x_1}\\
	&\geq \frac{\epsilon}{1+2C}-\frac{\epsilon}{2(1+2C)}\\
	&=\frac{\epsilon}{2(1+2C)}>0,
	\end{align*}
	for all $x \in B$. Hence, there exists a polynomial $p$ with the same monotonicity profile as $f$ such that $\max_{x \in B} |f(x)-p(x)|<\epsilon.$
	
	Furthermore, as $\frac{\partial p(x)}{\partial x_1}>0$ over $B$, there exists an integer $r$ and sum of squares polynomials $s_0,\ldots,s_n$ of degree $r$ such that
	\begin{align} \label{eq:box.odd}
	\frac{\partial p(x)}{\partial x_1}=s_0(x)+\sum_{i=1}^n s_i(x)(b_i^+-x_i)(x_i -b_i^-).
	\end{align}
	This is a consequence of Theorem \ref{lem:putinar} as $B$ as defined is Archimedean. Indeed,
	$$(b_i^+)^2+(b_i^-)^2-x_i^2=(x_i-(b_i^++b_i^-))^2+2(b_i^+-x_i)(x_i-b_i^-) \in M(g),$$ hence, if $N=\lceil \sum_i (b_i^+)^2+\sum_i (b_i^-)^2 \rceil$, we get that
	$$N-\sum_i x_i^2=N-\sum_i (b_i^+)^2-\sum_i (b_i^-)^2+\sum_i ((b_i^+)^2+(b_i^-)^2-x_i^2)\in M(g).$$

	$$\frac{\partial p(x)}{\partial x_1}=s_0(x)+\sum_{i=1}^n s_i(x)(b_i^+-x_i)(x_i -b_i^-).$$ 
\end{proof}

\begin{remark}\label{rem:box.rep}
	The format of the sum of squares certificate of positivity of $p$ over the box $B$ depends on the representation that one uses to represent the box. 
	If the box had been defined instead as $$B=\{x \in \mathbb{R}^n~| b_i^-\leq x_i \leq b_i^+\},$$ then we would have had 	\begin{align}\label{eq:box.even}
	p(x)=s_0(x)+\sum_{i=1}^n s_i(x)(b_i^+-x_i)+\sum_{i=1}^n t_i(x)(x_i-b_i^-),
	\end{align}
	where $s_1(x),\ldots,s_{n}(x)$ and $t_1(x),\ldots,t_n(x)$ are sos. Indeed the set of polynomials $$\{x_i-b_i^-, b_i^+-x_i\}$$ satisfy the Archimdean property as well. To see this assume wlog that $b_i^+\geq b_i^-\geq 0$ and note that:
	$$b_i^++x_i=(x_i-b_i^-)+(b_i^++b_i^-) \in M(g)$$
	and hence $$(b_i^+)^2-x_i^2=(b_i^+-x_i)\frac{(b_i^++x_i)^2}{2b_i^+}+(b_i^++x_i)\frac{(b_i^+-x_i)^2}{2b_i^+} \in M(g).$$
	If $N=\lceil \sum_i (b_i^+)^2 \rceil$, we then have $$N-\sum_i x_i^2=N-\sum_i (b_i^+)^2+\sum_i ((b_i^+)^2- x_i^2) \in M(g).$$
	We have chosen to use the formulation given in (\ref{eq:box.odd}) rather than the one in (\ref{eq:box.even}) as one need only search for $n$ sos polynomials in (\ref{eq:box.odd}) rather than $2n$, in (\ref{eq:box.even}).
	
	%	We would like to note however that in the univariate case, formulation (\ref{eq:box.even}) is known to be exact for even polynomials and formulation (\ref{eq:box.odd}) is exact for odd polynomials. As a consequence, in our applications, we tend to change representations depending on the evenness of the degree of the polynomial at hand.
\end{remark} 

\begin{corollary}\label{cor:prob.solve.monoton}
	Recall the definition of $f_{mon}$ as given in (\ref{eq:opt.solve.monoton.hard}). Consider the following hierarchy of semidefinite programs indexed by $r$:
	\begin{equation}\label{eq:opt.prob.sos}
	\begin{aligned}
	f_{mon}^r\mathrel{\mathop{:}}=&\inf_{p \text{ of degree } d, ~\sigma_j^k} \sum_{i=1}^m (p(x_i)-y_i)^2\\
	&\text{s.t. } \rho_j\frac{\partial p(x)}{\partial x_j}=\sigma_j^0(x)+\sum_{k=1}^{n}\sigma_{j}^k(x)(b^+_k-x_k)(x_k-b^-_k), \forall j=1,\ldots,n \\
	&\sigma_j^k \text{ are sos and have degree }r.
	\end{aligned}
	\end{equation}
	We have $$f_{mon}^r \rightarrow f_{mon} \text{ as } r \rightarrow \infty.$$
\end{corollary}
\begin{proof}
	As $\{f^r_{mon}\}_r$ is decreasing and lower bounded by $f_{mon}$, it converges to some constant $c \geq f_{mon}$. Suppose by way of contradiction that $c=f_{mon}+\epsilon$ for some $\epsilon>0.$ By definition of $f_{mon}$ in (\ref{eq:opt.solve.monoton.hard}), there exists $q$ of degree $d$ such that $$f_{mon} \leq \sum_{i=1}^m (q(x_i)-y_i)^2 < f_{mon}+\epsilon/2. $$
	Consider the continuous function $$g(\alpha)=\alpha^2 (\sum_{i=1}^m \rho^Tx_i)^2+2\alpha \sum_{i=1}^m \rho^Tx_i \cdot (q(x_i)-y_i) $$ and note that $g(0)=0$ and $\lim_{\alpha \rightarrow \infty}g(\alpha)=+\infty$. Hence there exists $\alpha_0>0$ such that $g(\alpha_0)=\epsilon/2$. Similarly to the proof of Theorem \ref{th:mono.existence}, we now define $$p(x)=q(x)+\alpha_0 \rho^Tx.$$
	We have, for all $j=1,\ldots,m$,
	$$\rho_j \frac{\partial p(x)}{\partial x_j}=\rho_j \frac{\partial q(x)}{\partial x_j}+\alpha_0\rho_j^2 >0, \forall x \in B,$$ which implies from Theorem \ref{lem:putinar} that $p$ is feasible for (\ref{eq:opt.prob.sos}). But we have 
	\begin{align*}
	\sum_{i=1}^m (p(x_i)-y_i)^2&=\sum_{i=1}^m (q(x_i)-y_i)^2+\alpha_0^2 (\sum_{i=1}^m \rho^Tx_i)^2+2\alpha_0 \sum_{i=1}^m \rho^Tx_i \cdot (q(x_i)-y_i)\\
	&< f_{mon}+\epsilon/2+\epsilon/2=f_{mon}+\epsilon,
	\end{align*}
which contradicts the fact that $c=f_{mon}+\epsilon$.
\end{proof}

\subsubsection{Polynomial regressors constrained to be convex}

In this section, we assume that it is known that $f$ is convex over a box $B$, which is given to us. The goal is then to fit a polynomial $p$ to the data $(x_i,y_i), i=1,\ldots,m$ such that $p$ is also convex over $B$. In other words, we wish to solve the following optimization problem:
\begin{equation}\label{eq:conv.hard.problem}
\begin{aligned}
f_c\mathrel{\mathop{:}}=&\inf_{p \text{ of degree } d} &&\sum_{i=1}^m (p(x_i)-y_i)^2\\
&\text{s.t. } &&H_p(x)\succeq 0, \forall x \in B.
\end{aligned}
\end{equation}
Again, Theorem \ref{th:np.hardness.convex} suggests that this problem cannot be solved efficiently unless $P=NP.$

\begin{theorem}\label{th:convex.existence}
	Let $f$ be a $C^2$ function which is convex over a box 
	\begin{align}\label{def:box.convex.ex}
	B=\{(u_1,\ldots,u_n) \in \mathbb{R}^n ~|~ (b_i^+-u_i)(u_i-b_i^-)\geq 0, \forall i=1,\ldots,n\}.
	\end{align}
	For any $\epsilon>0$, there exists an integer $d$ and a polynomial $p$ of degree $d$ such that $$\max_{x \in B} |f(x)-p(x)|<\epsilon$$
	and such that $p$ is also convex over $B$.	Furthermore, convexity of $p$ over $B$ can be certified using a sum of squares certificate. 
\end{theorem}

This proof uses the following lemma, which is a generalization of Putinar's Positivstellensatz for matrices.

\begin{lemma}[Theorem 2 in \cite{scherer2006matrix}]\label{lem:matrix.approx}
	Let $$S=\{x \in \mathbb{R}^n ~|~ g_1(x)\geq 0, \ldots, g_m(x)\geq 0\}$$ and assume that $\{g_1,\ldots,g_m\}$ satisfy the Archimedean property (see Theorem \ref{lem:putinar}). If the symmetric-valued polynomial matrix $H(x)$ is positive definite on $S$, then there exist sos-matrices $S_0(x),\ldots, S_m(x)$ such that $$H(x)=S_0(x)+\sum_{i=1}^m S_i(x)g_i(x).$$
\end{lemma}

%\begin{remark}
%	It follows immediately from Lemma \ref{lem:putinar.box} that if $$B=\{x \in \mathbb{R}^n ~|~ (b_i^+-x_i)(x_i-b_i^-)\geq 0\}$$
%	and $H(x)\succ 0$ on $B$ then there exist sos-matrices $S_0(x),\ldots,S_n(x)$ such that $$H(x)=S_0(x)+\sum_{i=1}^n S_i(x)(b_i^+-x_i)(x_i-b_i^-).$$
%\end{remark}

%\begin{theorem}\label{th:convex.existence}
%Let $f$ be a $C^2$ function, convex over a box 
%$$B=\{(u_1,\ldots,u_n)\in \mathbb{R}^n ~|~ (b_i^+-u_i)(u_i-b_i^-)\geq 0, \forall i=1,\ldots,n\}.$$
%Consider a noisy model of generating data 
%$$y_i=f(x_i)+\epsilon_i,~\forall i=1,\ldots,m,$$
%	where $x_i \in \mathbb{R}^n$ is a feature vector and $\epsilon_i, i=1,\ldots,n$ are independently sampled, and follow a normal distribution with mean zero and standard deviation $\frac{\sigma}{m}$. Let $y \in \mathbb{R}^n$ and consider \begin{equation}\label{eq:opt.conv.sos}
%\begin{aligned}
%s_d\mathrel{\mathop{:}}=&\inf_p \sum_{i=1}^p \mathbb{E}|y_i-p(x_i)|\\
%&\text{s.t. } y^TH_p(x)y=\sigma_0(x,y)+\sum_{k=1}^{n}\sigma_k(x,y)(b^+_k-x_k)(x_k-b^-_k) \\
%&\sigma_k \text{ are sos of degree }d-2\text{ in x and quadratic in y}.
%\end{aligned}
%\end{equation}
% For any $\epsilon>0$, there exists a polynomial $p$ of degree $d$, solution to $(\ref{eq:opt.conv.sos})$, such that $$\max_{x\in B}|f(x)-p(x)|<\epsilon.$$
%This implies that  $\lim_{d\rightarrow \infty} s_d \leq \sigma \sqrt{\frac{2}{\pi}}$. In particular $s_d\rightarrow 0$ if we consider the noiseless case, i.e., $\epsilon_i=0, \forall i=1,\ldots,m.$
%\end{theorem}

\begin{proof}[Proof of Theorem \ref{th:convex.existence}]
	Let $f$ be a function in $C^2$, $B$ be a box as in (\ref{def:box.convex.ex}), and $\epsilon>0.$ Assume that $$H_f(x)\succeq 0, \forall x\in B$$ 
	and let $C\mathrel{\mathop{:}}=\max_{x \in B} \frac12 \sum_{i=1}^n x_i^2.$ From Lemma \ref{lem:approx.poly}, we know that there exists a polynomial $q$ of degree $d$ such that $$\max_{x \in B} |f(x)-q(x)|\leq \frac{\epsilon}{2(1+2nC)}$$ 
	and 
	\begin{align}\label{eq:second.derivative.bound}
	\max_{x \in B} \left|\frac{\partial^2 f(x)}{\partial x_i \partial x_j}- \frac{\partial^2 q(x)}{\partial x_i \partial x_j}\right| \leq \frac{\epsilon}{2(1+2nC)},\forall i,j=1,\ldots,n.
	\end{align}
	We denote by $M(x)=H_q(x)-H_f(x).$ As $f$ and $q$ are in $C^2$, the entries of $M(x)$ are continuous in $x$. This implies that $$x \mapsto \lambda_{\min}(M(x))$$ is continuous since the minimum eigenvalue of a matrix is continuous with respect to its entries \cite[Corollary VI.1.6]{bhatia2013matrix}. Let $$\Lambda\mathrel{\mathop{:}}=\min_{x\in B} \lambda_{\min}(M(x))$$ and note that $M(x)\succeq \Lambda I$, for all $x \in B.$ As the minimum is attained over $B$, there exists $x_0 \in B$ such that $\Lambda=\lambda_{\min} M(x_0).$ From (\ref{eq:second.derivative.bound}), we know that the absolute value of each entry of $M(x_0)$ is upperbounded by $\frac{\epsilon}{2(1+2nC)}.$ Recalling that for a matrix $A$ with entries $a_{ij}$, $||A||_{\max}=\max_{i,j} |a_{ij}|$, this implies that $$||M(x_0)||_{\max} \leq \frac{\epsilon}{2(1+2nC)}.$$ 
	By equivalence of norms, we have 
	$$||M(x_0)||_2\leq n ||M(x_0)||_{\max} \leq  \frac{n\epsilon}{2(1+2nC)},$$
	and as $||M(x_0)||_2=\max \{|\lambda_{\min}(M(x_0))|, |\lambda_{\max}(M(x_0))| \},$ we deduce that $$\max \{|\lambda_{\min}(M(x_0))|, |\lambda_{\max}(M(x_0))| \} \leq \frac{n\epsilon}{2(1+2nC)}.$$ This implies that $$-\frac{n\epsilon}{2(1+2nC)}\leq \Lambda \leq \frac{n\epsilon}{2(1+2nC)}$$ and so $M(x)\succeq -\frac{n\epsilon}{2(1+2nC)}$ for all $x \in B.$ Let $$p(x)=q(x)+\frac{n\epsilon}{2(1+2nC)}x^Tx.$$
	We have, for any $x \in B$,
	\begin{align*}
	|f(x)-p(x)|&\leq |f(x)-q(x)|+|q(x)-p(x)|\\
	&\leq \frac{\epsilon}{2(1+2nC)}+\frac{n\epsilon}{2(1+2nC)}\cdot 2C\\
	&=\frac{\epsilon}{2}<\epsilon.
	\end{align*}
	As $M(x) \succeq \Lambda I$, $\Lambda \geq -\frac{n\epsilon}{2(1+2nC)}$, and $H_f(x)\succeq 0$, we also have
	\begin{align*}
	H_p(x)&=H_p(x)-H_q(x)+H_q(x)-H_f(x)+H_f(x)\\
	&\succeq \frac{2n\epsilon}{2(1+2nC)}I-\frac{n\epsilon}{2(1+2nC)}I\\
	&\succeq \frac{n \epsilon}{2(1+2nC)} \succ 0.
	\end{align*}
	We conclude that there exists a polynomial $p$ which is convex over $B$ and such that $$\max_{x \in B} |f(x)-p(x)|<\epsilon.$$
	
	Furthermore, from Lemma \ref{lem:matrix.approx}, this implies that there exist sum of squares polynomials $\sigma_k(x,y)$, $k=1,\ldots,n$ of degree $r$ in $x$ and quadratic in $y$ such that
	$$y^TH_p(x)y=\sigma_0(x,y)+\sum_{k=1}^n \sigma_k(x,y)(b_k^+-x_k)(x_k-b_k^-).$$
\end{proof}

\begin{corollary}\label{cor:prob.solve.convex} Recall the definition of $f_c$ as given in (\ref{eq:conv.hard.problem}). Consider the following hierarchy of semidefinite programs indexed by $r$:
	\begin{equation}\label{eq:opt.prob.conv}
	\begin{aligned}
	f_c^r\mathrel{\mathop{:}}=&\inf_{p \text{ of degree } d, ~\sigma_k} \sum_{i=1}^m (p(x_i)-y_i)^2)\\
	&\text{s.t. } y^TH_p(x)y=\sigma_0(x,y)+\sum_{k=1}^n \sigma_k(x,y)(b_k^+-x_k)(x_k-b_k^-) \\
	&\sigma_k \text{ are sos, and of degree }\leq r \text{ in } x \text{ and  } 2 \text{ in } y.
	\end{aligned}
	\end{equation} We have $$f_c^r \rightarrow f_c \text{ as } r \rightarrow \infty.$$
\end{corollary}

\begin{proof}
	The proof of this theorem is analogous to that of Corollary \ref{cor:prob.solve.monoton} and hence left to the reader.
\end{proof}

\subsection{Cases where the semidefinite programming-based relaxations are exact}

In Corollaries \ref{cor:prob.solve.monoton} and  \ref{cor:prob.solve.convex}, we have replaced the original problem of finding polynomial regressors which are convex or monotone over $B$ with sum of squares-based relaxations. In both cases, we have asymptotic guarantees on the quality of these relaxations, i.e., we are guaranteed to recover the solutions of (\ref{eq:opt.solve.monoton.hard}) and (\ref{eq:conv.hard.problem}) if the degree of the sos polynomials involved is arbitrarily high. (We remark that no explicit bound on this degree can be given as a function of the number of variables and the degree only~\cite{Reznick_Unif_denominator}.) In two particular cases (which we cover below), one can in fact come up with semidefinite programming-based relaxations which are \emph{exact}: this means that the degree of the sum of squares polynomials needed to recover the true solution is explicitly known. Hence, one can write a semidefinite program that exactly solves (\ref{eq:opt.solve.monoton.hard}) and (\ref{eq:conv.hard.problem}). We review these two cases below.

\subsubsection{The quadratic case}

In this particular case, we wish to solve (\ref{eq:opt.solve.monoton.hard}) and (\ref{eq:conv.hard.problem}) with $d=2$. 

We first consider the case where we would like to constrain $p$ to have a certain monotonicity profile, i.e., we would like to solve (\ref{eq:opt.solve.monoton.hard}). As $p$ is quadratic, each of its partial derivatives is a linear function. Requiring that a linear function be nonnegative over a box can be done using the following lemma, which is a variant of the Farkas lemma.
\begin{lemma}[See, e.g., Proposition I.1 in \cite{handelman}]
	Let $K$ be a bounded polyhedron with nonempty interior defined by $\beta_i \geq 0, i=1,\ldots,s$, where $\beta_i=\alpha_i^Tx+\gamma_i$ are linear forms ($\alpha_i \in \mathbb{R}^n$ and $\gamma_i \in \mathbb{R}$). If $\beta$ is a linear form, nonnegative over $K$, then there exist nonnegative scalars $\lambda_1,\ldots,\lambda_s$ such that $$\beta=\sum_{i=1}^s \lambda_i \beta_i.$$
\end{lemma}

From this lemma, it follows that, when $p$ is quadratic, solving (\ref{eq:opt.solve.monoton.hard}) is exactly equivalent to solving
\begin{equation*}
\begin{aligned}
f_m^r\mathrel{\mathop{:}}=&\inf_{p \text{ of degree } 2, ~\lambda_j^k,~\tau_j^k} \sum_{i=1}^m (p(x_i)-y_i)^2\\
&\text{s.t. } \rho_j\frac{\partial p(x)}{\partial x_j}=\lambda_j^0+\sum_{k=1}^{n} \lambda_j^k(b^+_k-x_k)+\sum_{k=1}^n \tau_{j}^k(x_k-b^-_k), \forall j=1,\ldots,n, \forall x\in \mathbb{R}^n, \\
&\lambda_j^k \geq 0, k=0,\ldots,n,~ \tau_j^k \geq 0, k=1,\ldots,n,
\end{aligned}
\end{equation*}
%\begin{lemma}[Farkas' Lemma, see Corollary 7.1d. in \cite{SchrijverFarkas}]
%	Let $A$ be a matrix and let $b$ be a vector. There exists a vector $x \geq 0$ with $Ax=b$ if and only if $yb\geq 0$ for each row vector $y$ with $yA\geq 0.$
%\end{lemma}
which is a convex quadratic program.

In the case where we would like to solve (\ref{eq:conv.hard.problem}), note that the Hessian of any quadratic function is constant. Hence, as written, problem (\ref{eq:conv.hard.problem}) is a semidefinite program.

\subsubsection{The separable case}

Recall that a function $f:\mathbb{R}^n \rightarrow \mathbb{R}$ is said to be separable if $$f(x)=\sum_{i=1}^{n}f_i(x_i)$$ for some univariate functions $f_i:\mathbb{R}\mapsto \mathbb{R}.$

We first consider the case where we would like to solve (\ref{eq:opt.solve.monoton.hard}), assuming that $p$ is separable, i.e., $p(x)=\sum_{i=1}^n p_i(x_i)$. Note that we have $$\frac{\partial p(x)}{\partial x_j}=p_j'(x_j).$$ In other words, one can replace (\ref{eq:opt.solve.monoton.hard}) by
\begin{equation*}
\begin{aligned}
f_{mon}\mathrel{\mathop{:}}=&\inf_{p \text{ separable of degree } d} &&\sum_{i=1}^m (p(x_i)-y_i)^2\\
&\text{s.t. } &&\rho_j p_j'(x_j)\geq 0, \forall x_j \in [b_j^-,b_j^+], j=1,\ldots,n.
\end{aligned}
\end{equation*}
where $x_j\mapsto p_j'(x_j)$ is a univariate polynomial.
We then use the following lemma.
\begin{lemma}[Theorem 3.72 in \cite{blekherman2012semidefinite}]\label{lem:pablo.univ}
	Let $a<b$. Then the univariate polynomial $p(x)$ is nonnegative over $[a,b]$ if and only if it can be written as 
	$$\begin{cases}
	&p(x)=s(x)+(x-a)\cdot (b-x)\cdot t(x), \text{ if $deg(p)$ is even}\\
	&p(x)=s(x)\cdot(x-a)+t(x)\cdot(b-x), \text{ if $deg(p)$ is odd},
	\end{cases}$$
	where $t(x),s(x)$ are sum of squares polynomials. In the first case, we have $deg(p)=2d$ and $deg(t)\leq 2d-2$ and $deg(s)\leq 2d$. In the second case, we have $deg(p)=2d+1$ and $deg(t)\leq 2d$ and $deg(s)\leq 2d.$
\end{lemma}

Depending on the degrees of $p_i$, we use Lemma \ref{lem:pablo.univ} to rewrite the previous optimization problem as a semidefinite program. For example, in the case where the degrees of $p_i$ are all odd and equal to $d=2d'+1$, we would get: 
\begin{equation*}
\begin{aligned}
f_m\mathrel{\mathop{:}}=&\inf_{p \text{ separable and of degree } d} \sum_{i=1}^m (p(x_i)-y_i)^2\\
&\text{s.t. } \rho_jp_j'(x_j)=s_j(x)+(x_j-b_j^-)(b_j^+-x_)\cdot t_j(x),~ j=1,\ldots,n,\\
&s_j \text{ sos and of degree } \leq 2d',~t_j \text{ sos and of degree } \leq 2d'-2.
\end{aligned}
\end{equation*}

To illustrate this, we have generated data $(x_i,y_i)\in [-2,2]\times \mathbb{R}$ with $ i=1,\ldots, 40$ which we would like to fit a univariate polynomial $p$ of degree $3$ to. (Note that the univariate case is a special case of the separable case.) For visualization purposes, we restrict ourselves to a parametric family of polynomials whose coefficients are indexed by $a$ and $b$: 
\begin{align}\label{def:pab}
p_{a,b}(x)=a\cdot x^3+b \cdot x^2+(a+2b)\cdot x+\frac12.
\end{align}

We have plotted in Figure \ref{fig:monotone.sep} the values of $a$ and $b$ for which:
\begin{enumerate}[(i)]
	\item $\sum_{i=1}^{40} (p_{a,b}(x_i)-y_i)^2 \leq 250$ in dark gray,
	\item $\sum_{i=1}^{40} (p_{a,b}(x_i)-y_i)^2 \leq 250$ and $p_{a,b}$ is nondecreasing over $[-2,2]$ in light gray.
\end{enumerate}

\begin{figure}[h]
	\centering
	\includegraphics[scale=0.25]{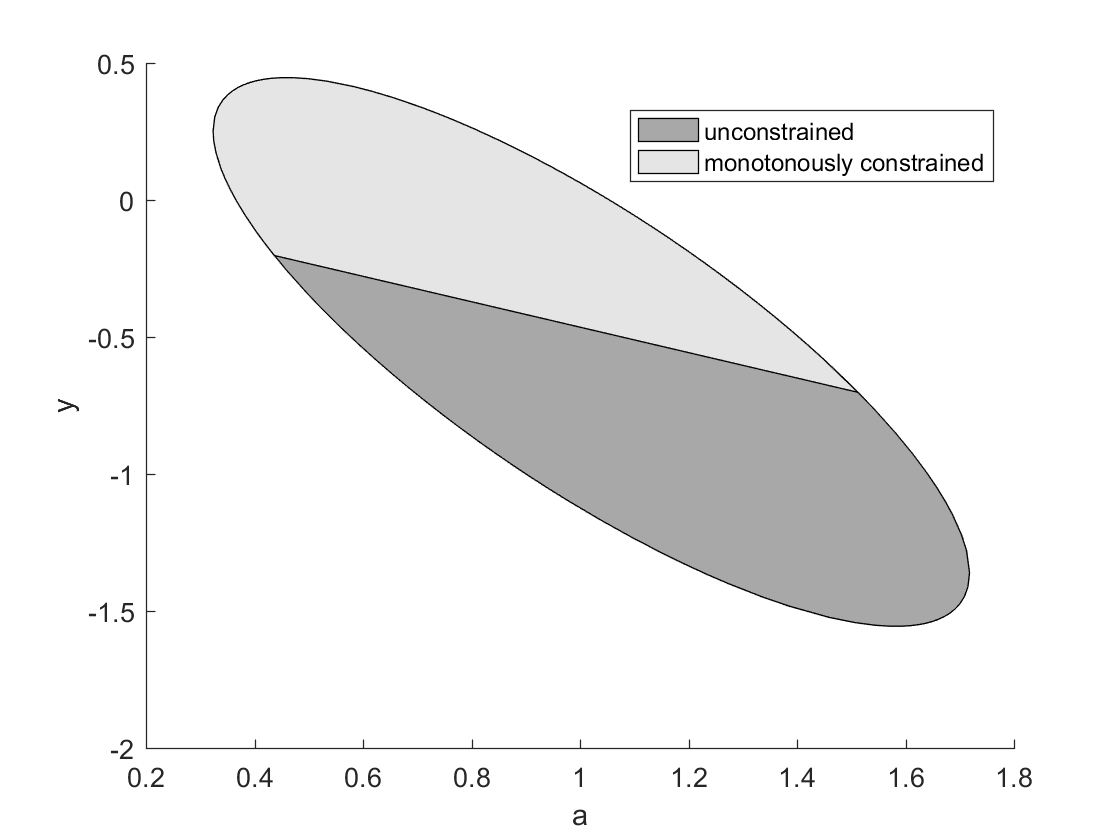}
	\caption{Values of $a$ and $b$ for which $p_{a,b}(x)$ in (\ref{def:pab}) has mean squared error less than 250 in the unconstrained and the monotonically-constrained settings}
	\label{fig:monotone.sep}
\end{figure}
As a sanity check, we plot in Figure \ref{fig:monotone.check} the fits that we obtain when $(a,b)=(1.6,-1.5)$ and when $(a,b)=(0.6,0)$. Note that the first fit is not monotonous, whereas the second one is, which is what we expect from Figure \ref{fig:monotone.sep}.

\begin{figure}[h!]
	\centering
	\subfigure[Fit for $(a,b)=(1.6,-1.5)$]{\includegraphics[width = 0.49\textwidth]{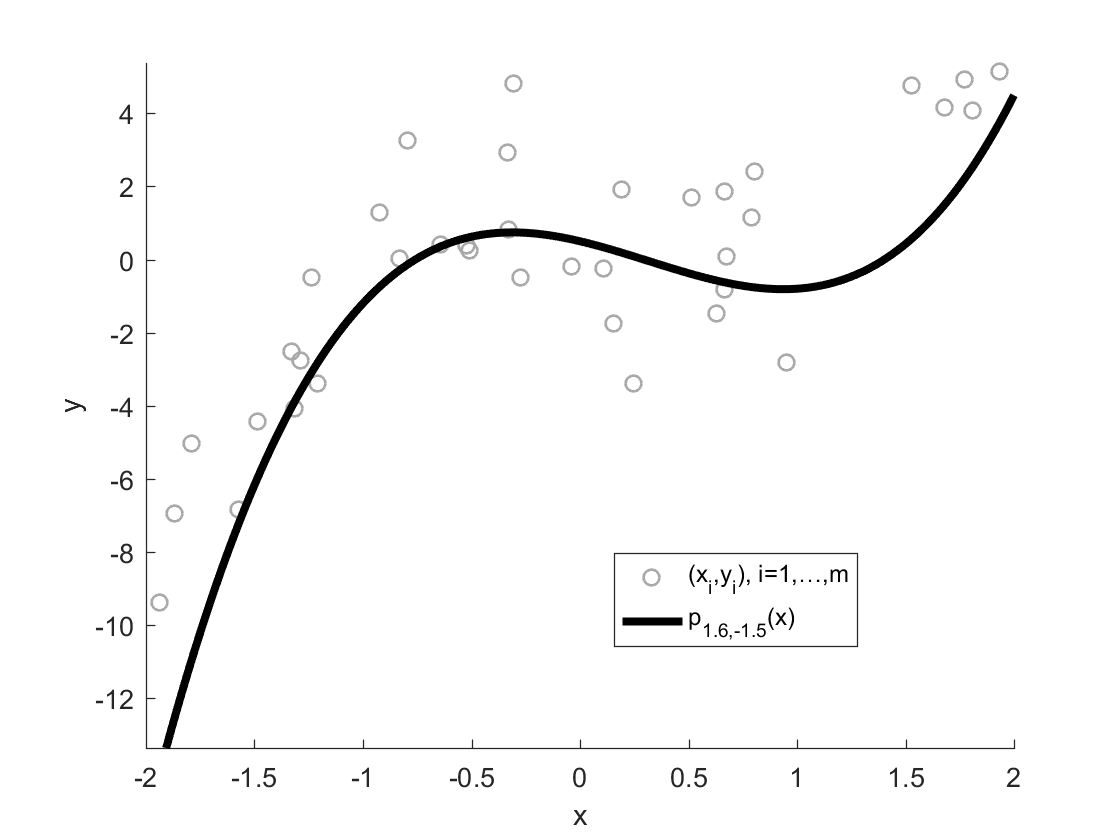}}
	\subfigure[Fit for $(a,b)=(0.6,0)$]{\includegraphics[width=0.49\textwidth]{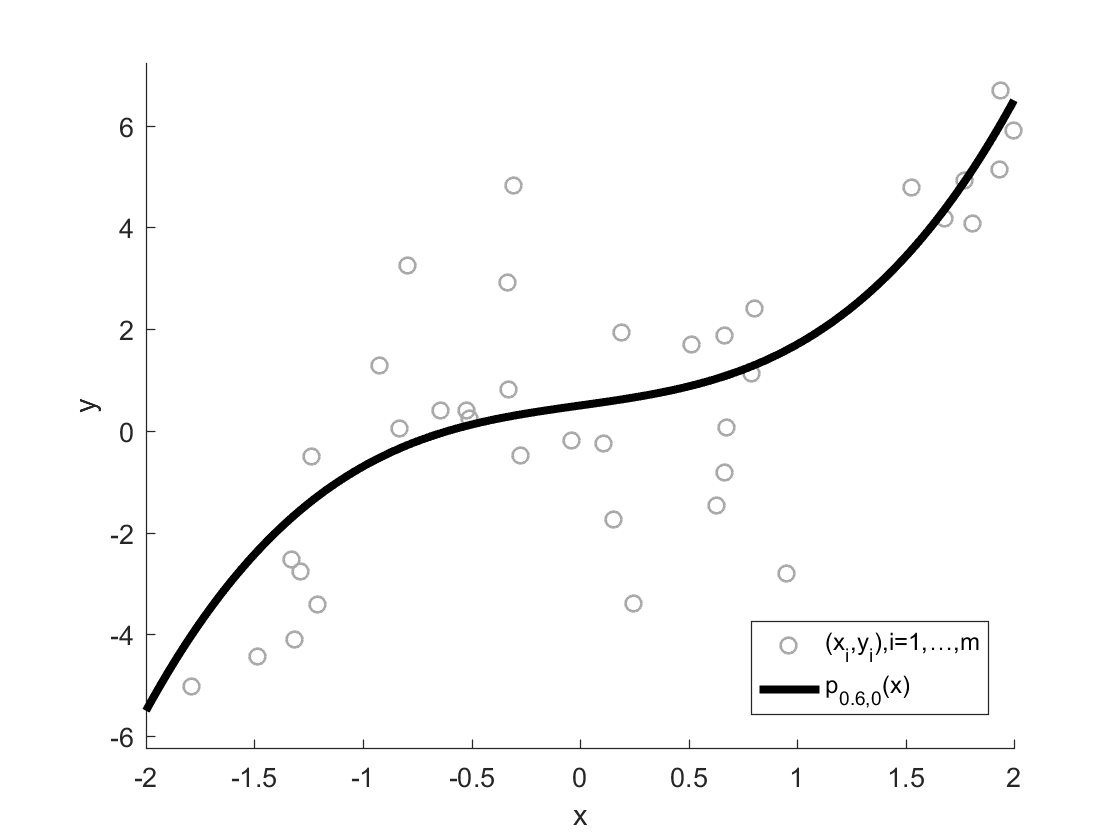}} 
	\caption{Plots of the polynomial $p_{a,b}$ in (\ref{def:pab}) for different values of $(a,b)$ in the monotonous case}
	\label{fig:monotone.check}
\end{figure}

We now consider the case where we would like to solve (\ref{eq:conv.hard.problem}), i.e., where we constrain $p$ to be convex over $B$. We assume that we are searching over the set of separable polynomials of degree $d$. Note here that if $p(x)=\sum_{i=1}^n p_i(x_i)$, then $H_p(x)$ is a diagonal matrix with diagonal entry $i$ corresponding to $p_i''(x_i).$ Hence, the condition $H_p(x)\succeq 0, \forall x \in B$ is equivalent to requiring that $p_i''(x_i)\geq 0$, for all $x_i \in [b_i^-,b_i^+]$, $i=1,\ldots,n$. Once again, we use Lemma \ref{lem:pablo.univ} to rewrite the previous optimization problem as a semidefinite program. For example, in the case where the degrees of $p_i$ are all even and equal to $d=2d'$, we get 
\begin{equation*}
\begin{aligned}
f_m\mathrel{\mathop{:}}=&\inf_{p \text{ separable of degree } d} \sum_{i=1}^m (p(x_i)-y_i)^2\\
&\text{s.t. } p_j''(x_j)= s_j(x)+(x_j-b_j^-)(b_j^+-x_)\cdot t_j(x),~ j=1,\ldots,n,\\
&s_j \text{ sos and of degree } \leq 2d'-2,~t_j \text{ sos and of degree } \leq 2d'-4.
\end{aligned}
\end{equation*}

To illustrate these results, we have generated data $(x_i,y_i)\in [-2,2]\times \mathbb{R}$ with $ i=1,\ldots, 40$ which we would like to fit a univariate polynomial $p$ of degree $4$ to. (Note again that the univariate case is a special case of the separable case.) For visualization purposes, we restrict ourselves again to a parametric family of polynomials whose coefficients are indexed by $a$ and $b$: 
\begin{align}\label{def:pab.2}
p_{a,b}(x)=\frac{1}{10}x^4+a\cdot x^3+b\cdot x^2-(a+b)\cdot x-\frac{2}{10}.
\end{align}

We have plotted in Figure \ref{fig:convex.sep} the values of $a$ and $b$ for which:
\begin{enumerate}[(i)]
	\item $\sum_{i=1}^{40} (p_{a,b}(x_i)-y_i)^2 \leq 7$ in dark gray,
	\item $\sum_{i=1}^{40} (p_{a,b}(x_i)-y_i)^2 \leq 7$ and $p_{a,b}$ is convex over $[-2,2]$ in light gray.
\end{enumerate}

\begin{figure}[h]
	\centering
	\includegraphics[scale=0.25]{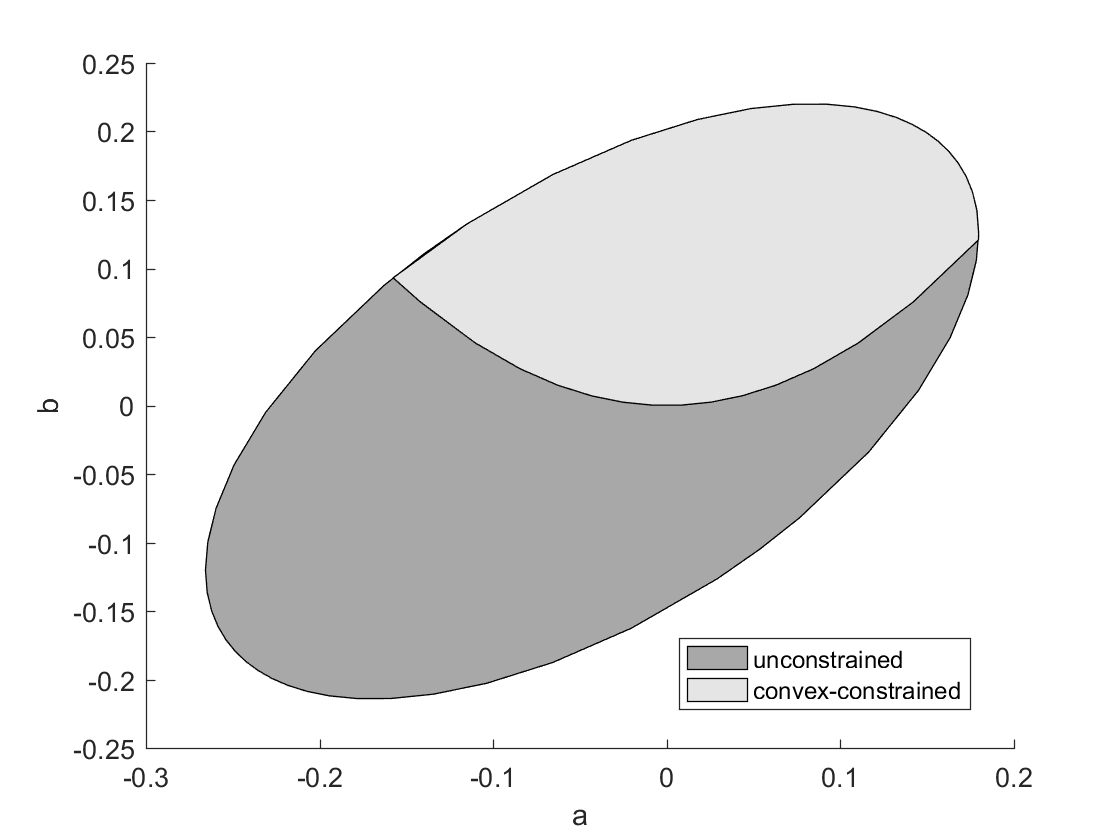}
	\caption{Values of $a$ and $b$ for which $p_{a,b}(x)$ in (\ref{def:pab.2}) has mean squared error less than 7 in the unconstrained and the convexity-constrained settings}
	\label{fig:convex.sep}
\end{figure}
As a sanity check, we plot in Figure \ref{fig:convex.check} the fits that we obtain when $(a,b)=(-0.2,-0.1)$ and when $(a,b)=(0.15,0.1)$. Note that the first fit is not convex, whereas the second one is, which is what we expect from Figure \ref{fig:convex.sep}.

\begin{figure}[h!]
	\centering
	\subfigure[Fit for $(a,b)=(-0.2,-0.1)$]{\includegraphics[width = 0.49\textwidth]{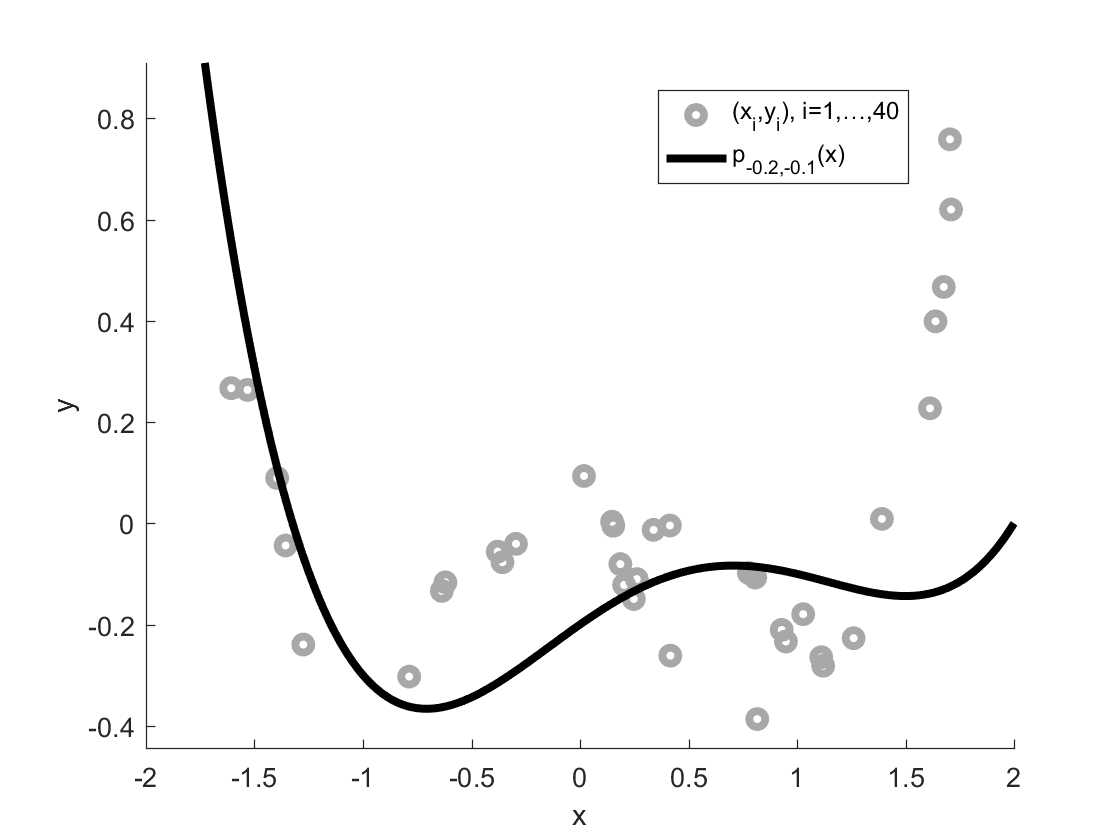}}
	\subfigure[Fit for $(a,b)=(0.15,0.1)$]{\includegraphics[width=0.49\textwidth]{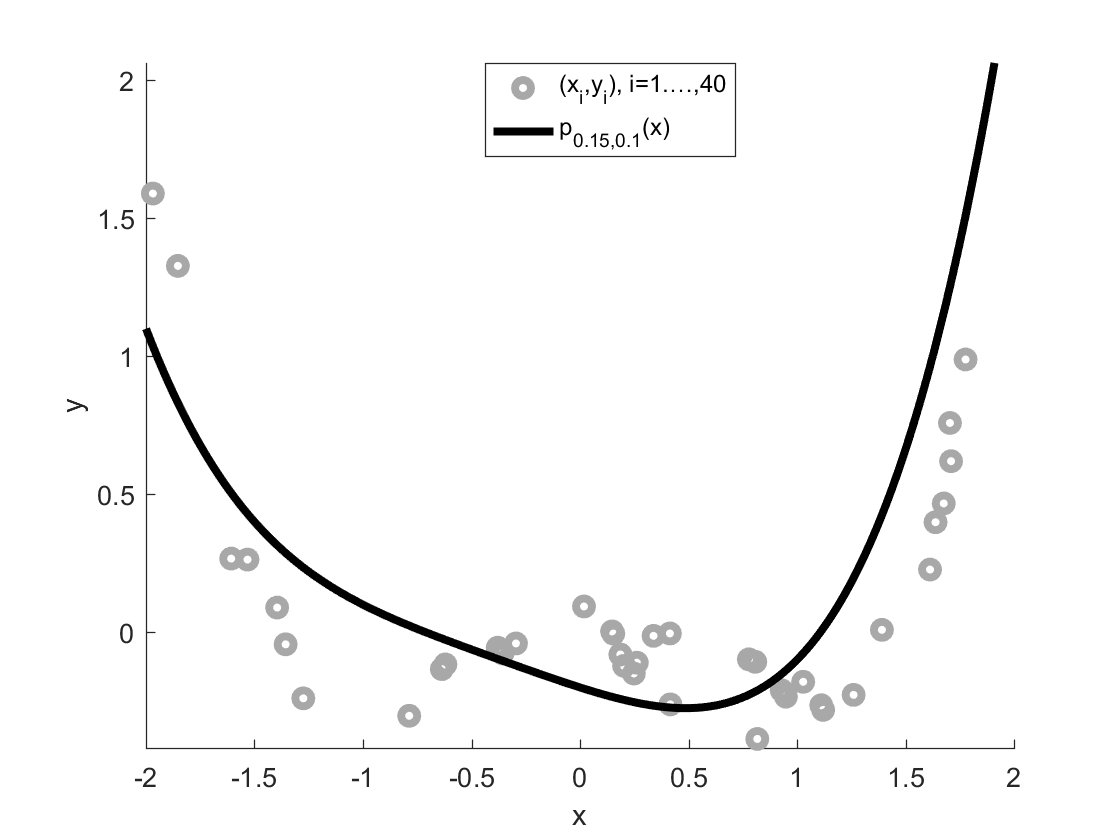}} 
	\caption{Plot of the polynomial $p_{a,b}$ in (\ref{def:pab.2}) for different values of $(a,b)$}
	\label{fig:convex.check}
\end{figure}

\section{Experimental results}\label{sec:experiments}

We now provide some illustrations of our methods on different datasets. In the first part of this section, we consider synthetic datasets. This will enable us to compare the advantages and limitations of our relaxations in terms of performance metrics such as training and testing accuracy, robustness, flexibility and scalability. In the second part of this section, we look at how our methods perform on real-life datasets.

\subsection{Synthetic regression problems}
For the synthetic experiments, we analyze the performance of 4 different algorithms: UPR, which corresponds to unconstrained polynomial regression, MCPR, which corresponds to polynomial regression with monotonicity constraints, CCPR which corresponds to polynomial regression with convexity constraints, and MCPR+CCPR, which corresponds to polynomial regression with both monotonicity and convexity constraints. The underlying function for this experiment as described in (\ref{eq:f.def.2}) is a multivariate exponential:
$$f(x) = e^{\left\Vert x \right\Vert_2}$$
The function $f:\mathbb{R}^n \rightarrow \mathbb{R}$ is monotonically increasing in all directions, thus, it has a monotonicity profile $\rho_i = 1, \forall i$. Furthermore, $f$ is convex.
\subsubsection{Data Generation}
We denote by $X$ the feature matrix, i.e., the matrix obtained by concatenating the $m$ feature vectors $x_i$ of length $n$. Each column or $X$ corresponds to a feature and each row is an observation of all the $n$ features. Hence, $X$ is an $m \times n$ matrix. For our synthetic datasets, we generate each entry of $X$ uniformly at random in an interval $[b^-,b^+]$, where $b^-=0.5$ and $b^+=2$. The feature domain in this case is taken to be $$B = \{x \in \mathbb{R}^n~|~b^{-} \leq x_i \leq b^+,\ i =1,\ldots,n\}.$$ We compute the response variable $y_i$ by evaluating $f$ at each column $x_i$ of $X$, which we corrupt by some noise, whose scaling $\epsilon$ we vary in order to test for robustness. As a consequence, if we denote by $y$ the $m \times 1$ vector containing $y_1,\ldots,y_m$ and by $f(X)$ the $m\times 1$ vector obtained by applying $f$ to each row of $X$, we have $y = f(X) + \epsilon,$
where $\epsilon$ is a vector with each entry taken to be iid and Gaussian of mean zero and standard deviation $\alpha \sqrt{var(f(X))}$. Here $var(f(X))$ is the variance of the set of random points obtained when varying the input $X$ to $f$ and $\alpha$ is a fixed constant, which we use to parametrize noise (e.g., $\alpha=1$ is low noise, whereas $\alpha=10$ is high noise).

In the following, we wish to fit a polynomial $p$ of degree $d$ to the data, such that the mean squared error (which is a normalization of the least squared error) $$\frac{1}{m}||p(X)-y||^2$$ is minimized.

\subsubsection{Comparative performance}
One of the biggest drawbacks of unconstrained polynomial regression is the algorithmic instability to noise. Here we want to compare the four algorithms listed above with respect to robustness to noise. To do this, we fit polynomials of varying degrees to the data in both high-noise ($\alpha=10$ as described previously) and low-noise ($\alpha=1$) settings. We then compare the Root Mean Squared Error (RMSE)
$$RMSE(X,y)=\sqrt{\frac{||p(X)-y||^2}{m}}$$
on the testing and training samples. The results are given in Figure \ref{compare_by_degree}. Note that the thin light blue constant line listed as ``Reference'' is the reference RMSE, i.e., the value obtained when one computes the RMSE for the function $f$ itself. 
%$$\mathcal{L}= RMSE(X,Y) =\sqrt{\frac{\left \Vert p(X) - Y \right \Vert ^2}{N}}$$

%In Fig. \ref{compare_by_degree} we vary the degree of the decision polynomials for all algorithms and compare their performance in low noise and  MCPR and high noise settings.
\begin{figure}[h!]
	\centering
	\subfigure[Comparison of RMSE on the training set in a low noise setting]{\includegraphics[width = 0.49\textwidth]{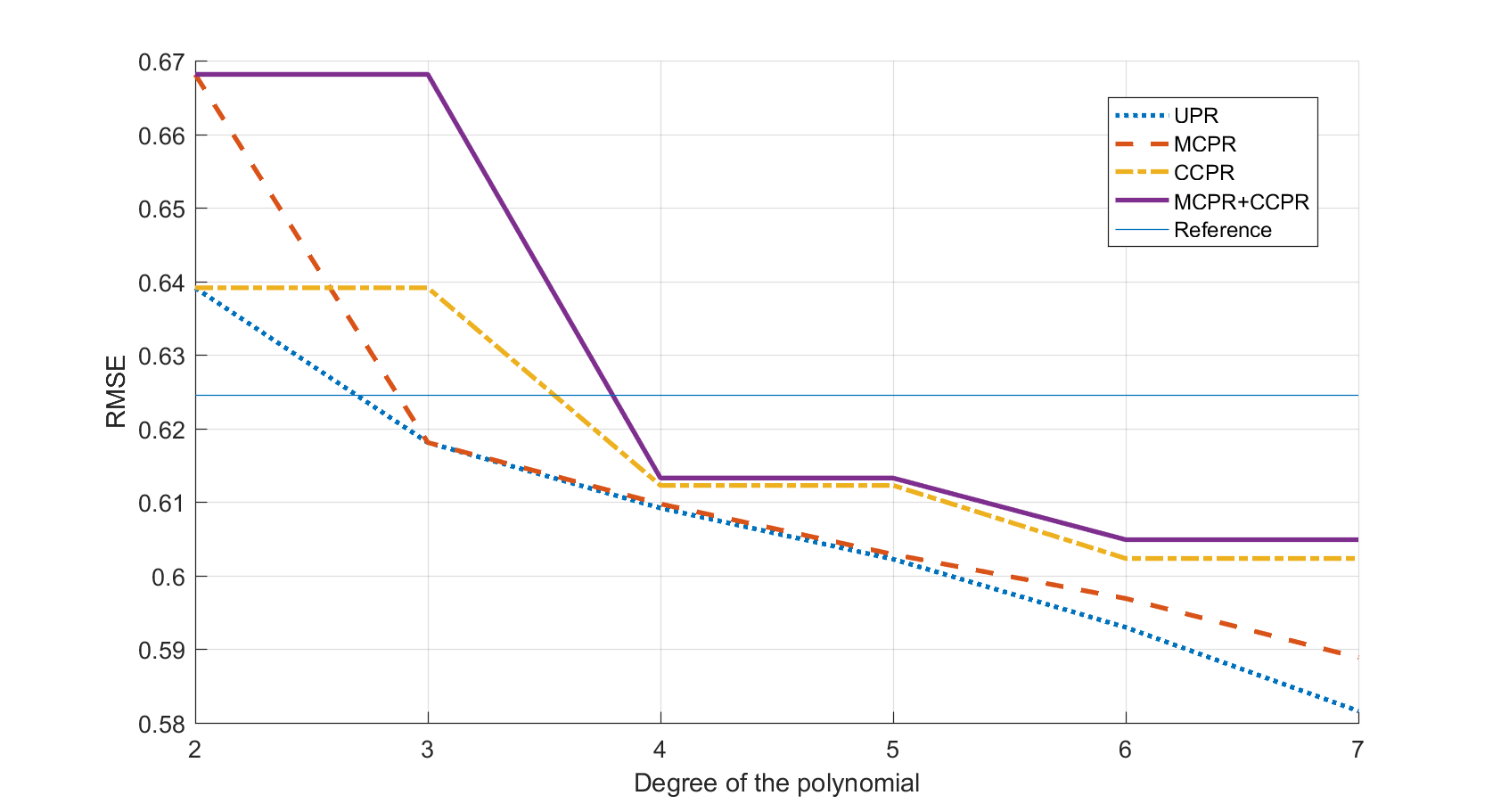}}
	\subfigure[Comparison of RMSE on the testing set in a low noise setting]{\includegraphics[width = 0.49\textwidth]{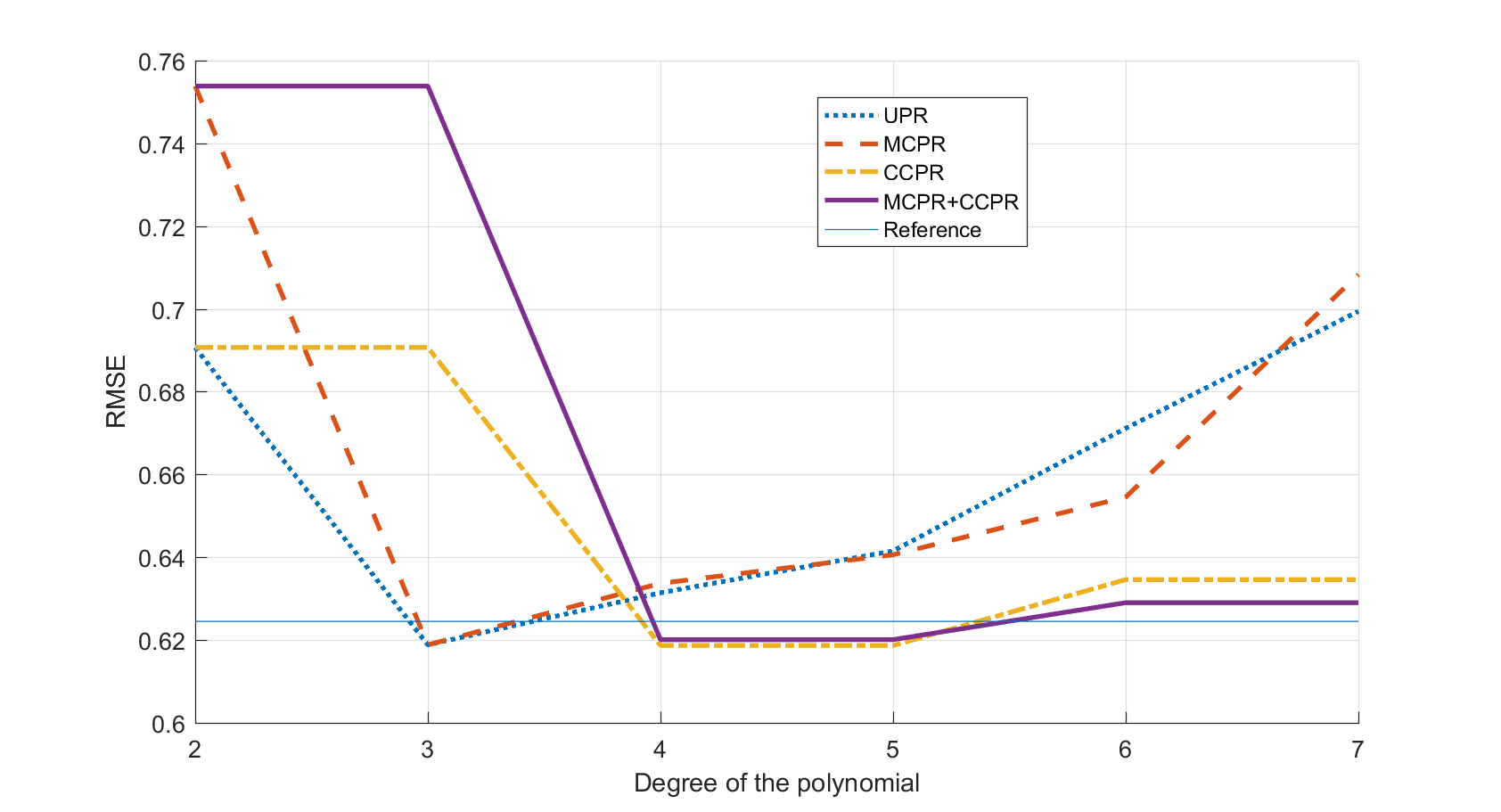}} \\
	\subfigure[Comparison of RMSE on the training set in a high noise setting]{\includegraphics[width = 0.49\textwidth]{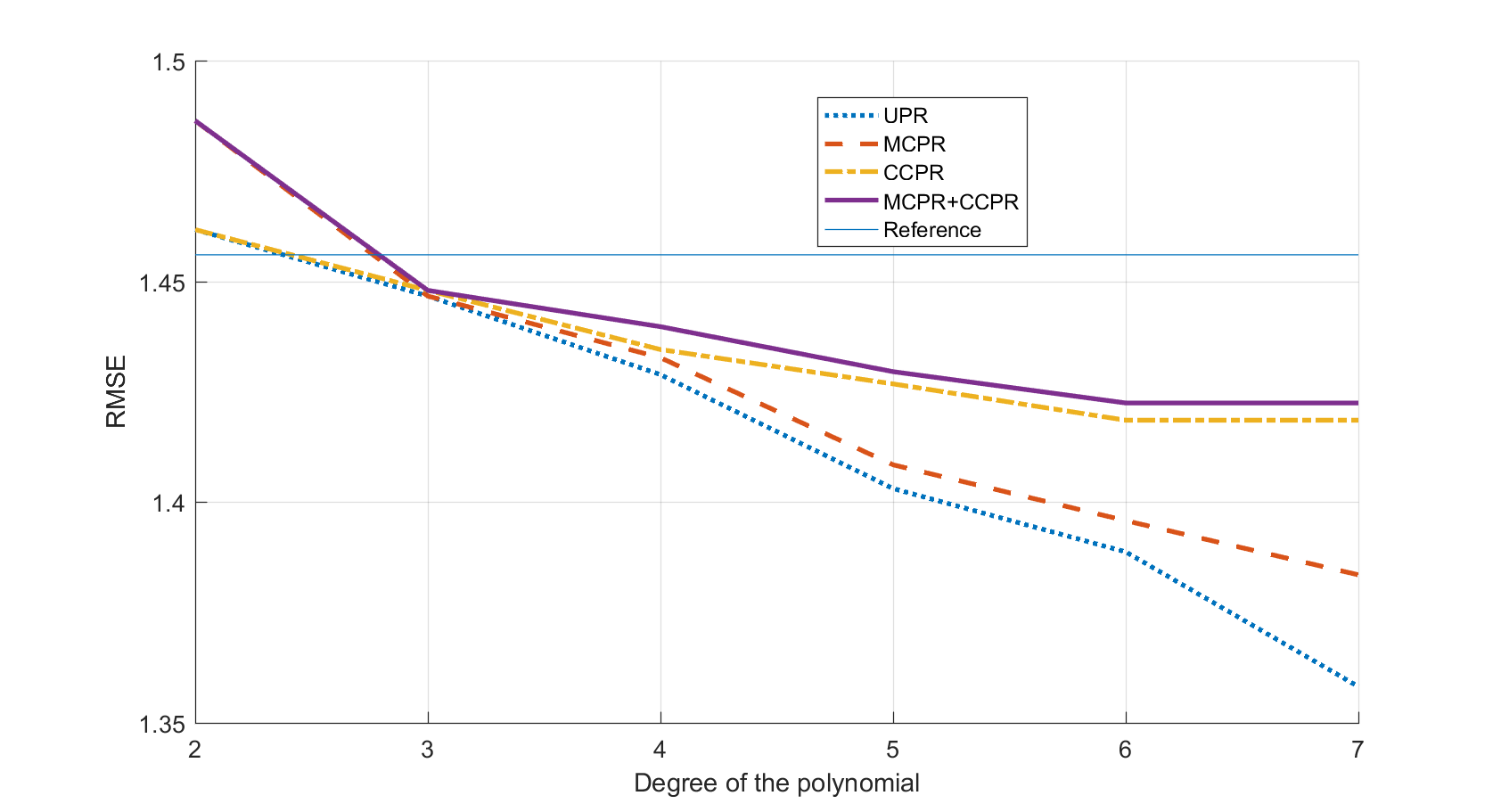}}
	\subfigure[Comparison of RMSE on the testing set in a high noise setting]{\includegraphics[width = 0.49\textwidth]{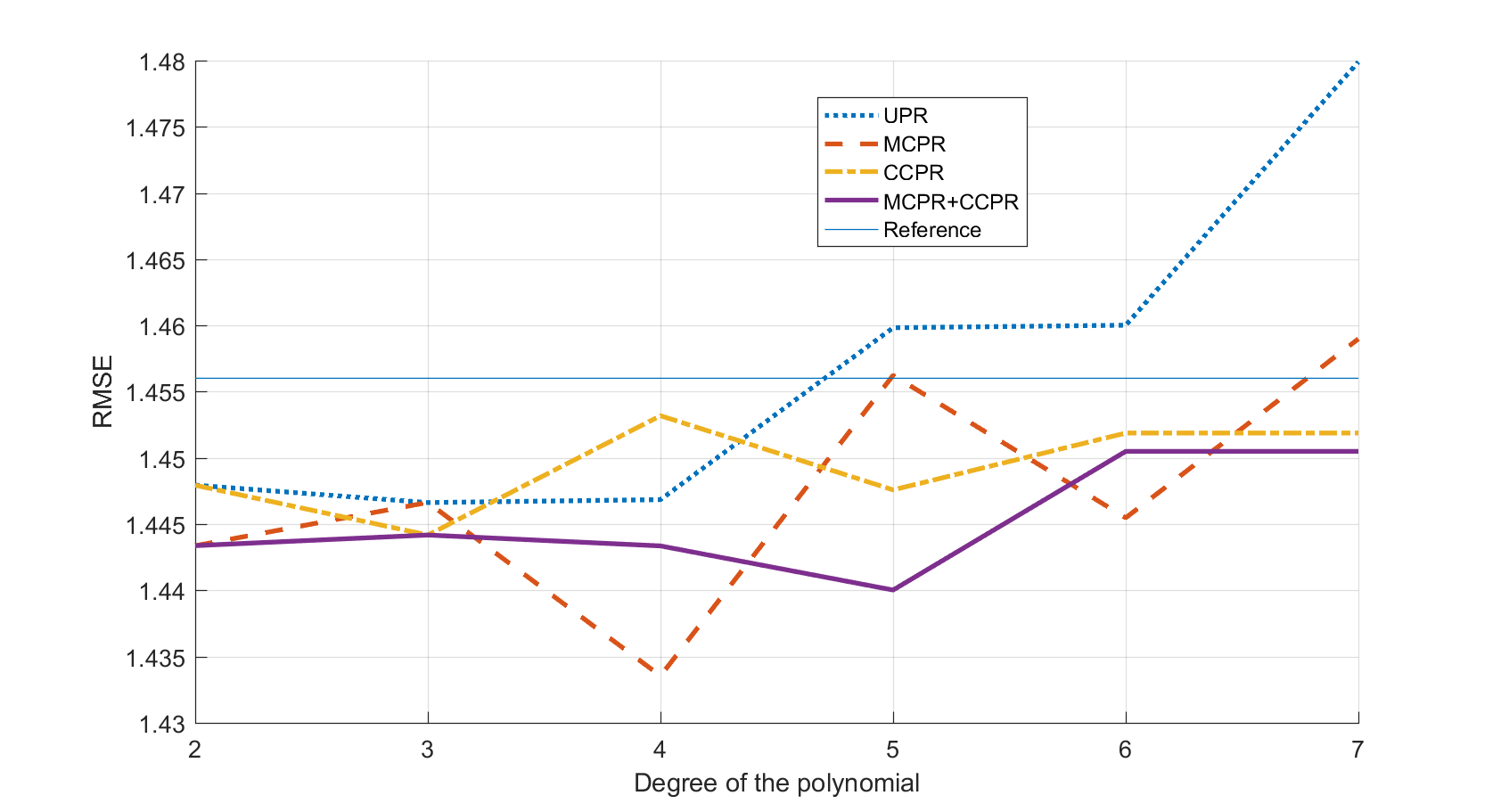}} 
	\caption{RMSEs of the fitted functions for different noise scaling factors and polynomial degrees}
	\label{compare_by_degree}
\end{figure}
\par As expected, from Figure \ref{compare_by_degree}, we see that UPR tends to overfit. This can be observed by comparing the RMSE of UPR to the Reference RMSE: anything below the reference can be considered to be overfitting. Note that for both training sets, and particularly when the degree of the polynomials is high, the data points corresponding to UPR are well below those given by the Reference. Introducing monotonicity or convexity constraints improves both the accuracy on the test data as well as robustness to noise, in the sense that the RMSE of these algorithms remains moderate, even in high noise environments. When both monotonicity and convexity are imposed, the benefits compound. Indeed, MCPR+CCPR has similar performance for both the testing and the training data, and the RMSE obtained with this algorithm is the closest to the reference line. Note that MCPR+CCPR performs well both in low noise as well as high noise settings, which indicates the ability to robustly learn the true underlying distribution.

Lastly we compare qualitatively the robustness of UCR, MCPR, CCPR, and MCPR+CCPR with respect to the true underlying function. The plots in Figure \ref{projections} are obtained by projecting the 4 fitted functions and the underlying function onto one of the features (this is done by fixing all the other features to some arbitrary values in their range). We consider the case where the polynomials are of degree $4$ and of degree $7$. 

\begin{figure}[h!]
	\centering
	\subfigure[Projections of degree 4 fits and the underlying function in a low noise setting]{\includegraphics[width = 0.49\textwidth]{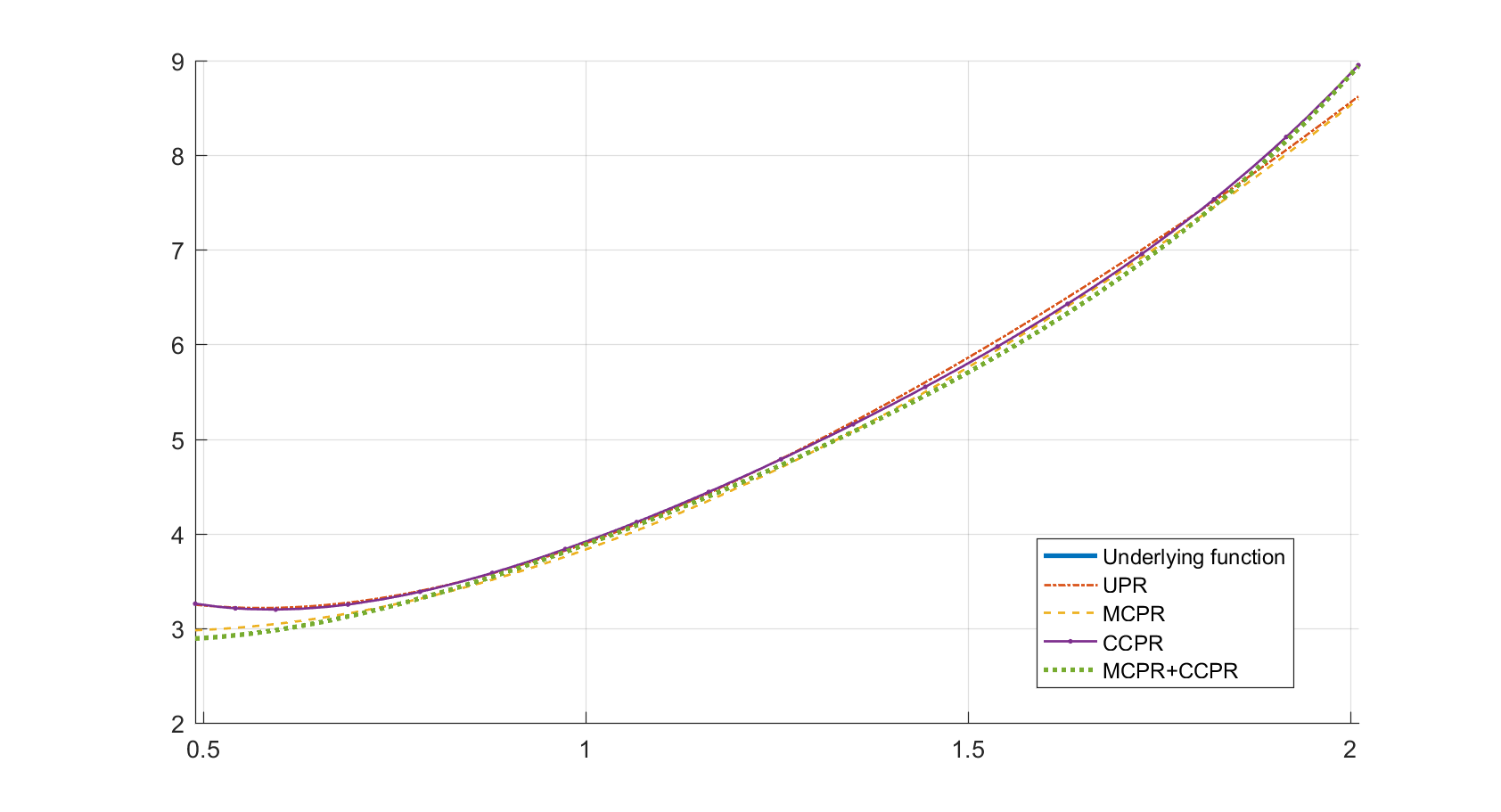}}
	\subfigure[Projections of degree 7 fits and the underlying function in a low noise setting]{\includegraphics[width = 0.49\textwidth]{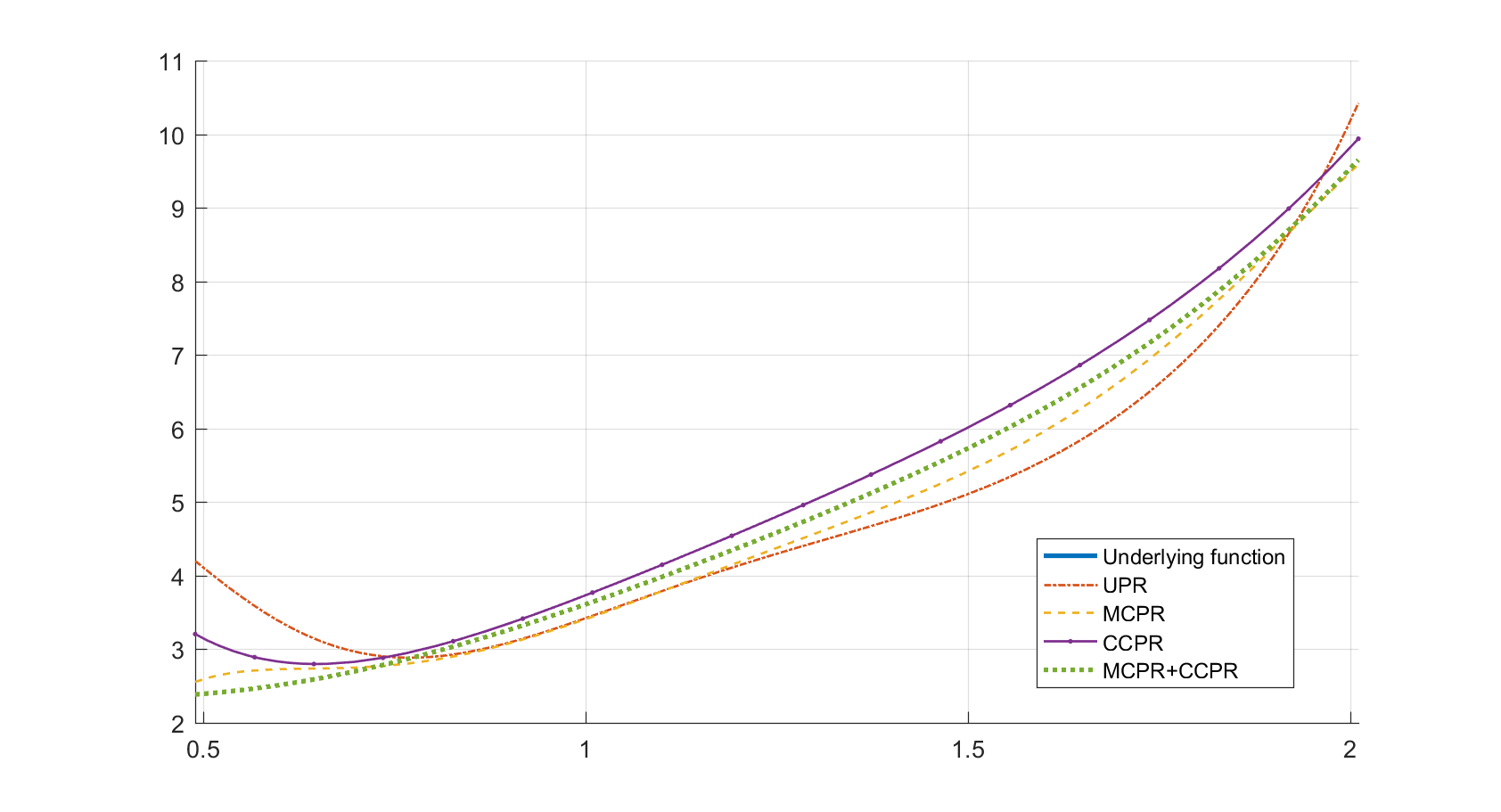}} \\
	\subfigure[Projections of degree 4 fits and the underlying function in a high noise setting]{\includegraphics[width = 0.49\textwidth]{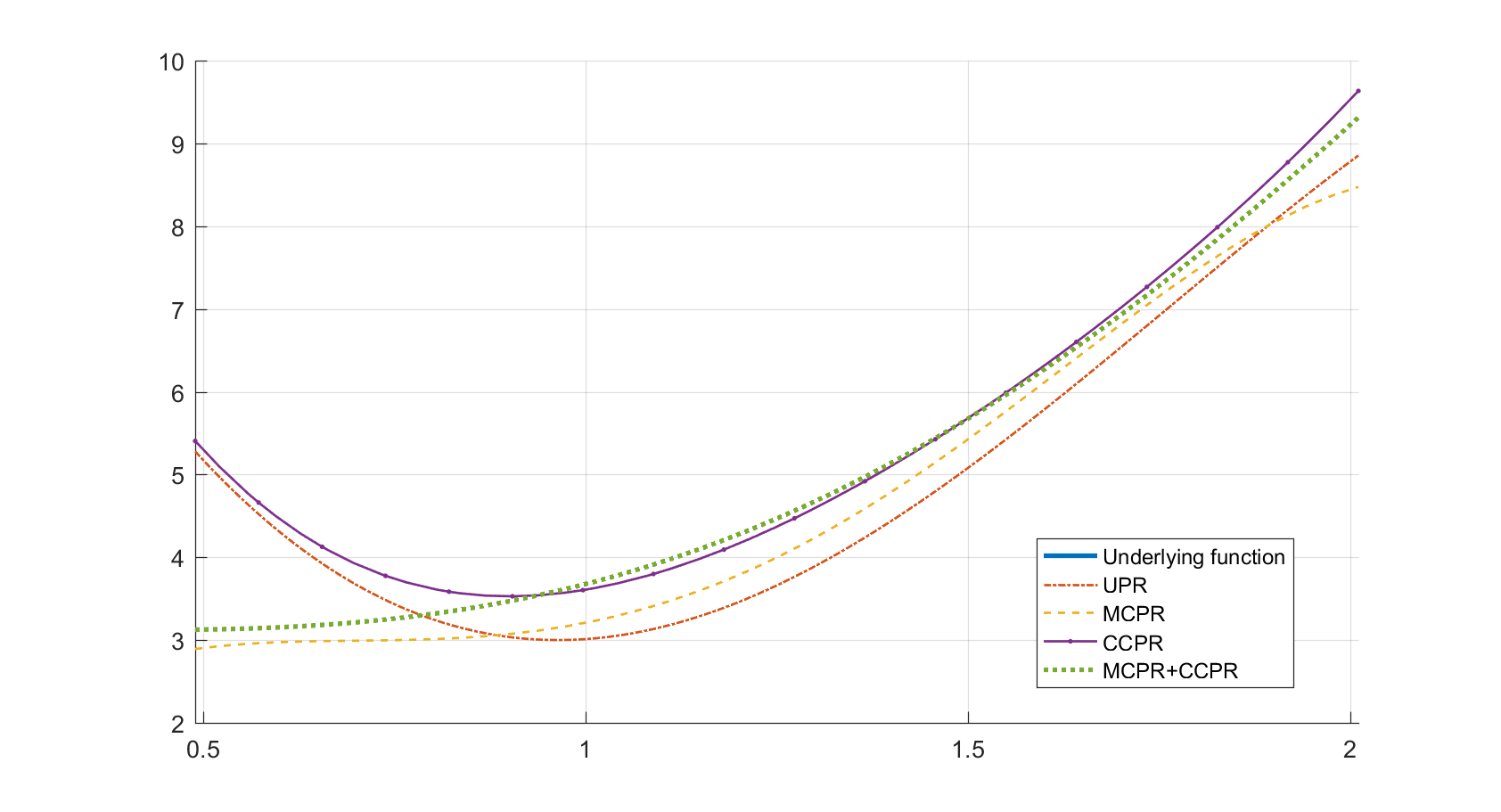}}
	\subfigure[Projections of degree 7 fits and the underlying function in a high noise setting]{\includegraphics[width = 0.49\textwidth]{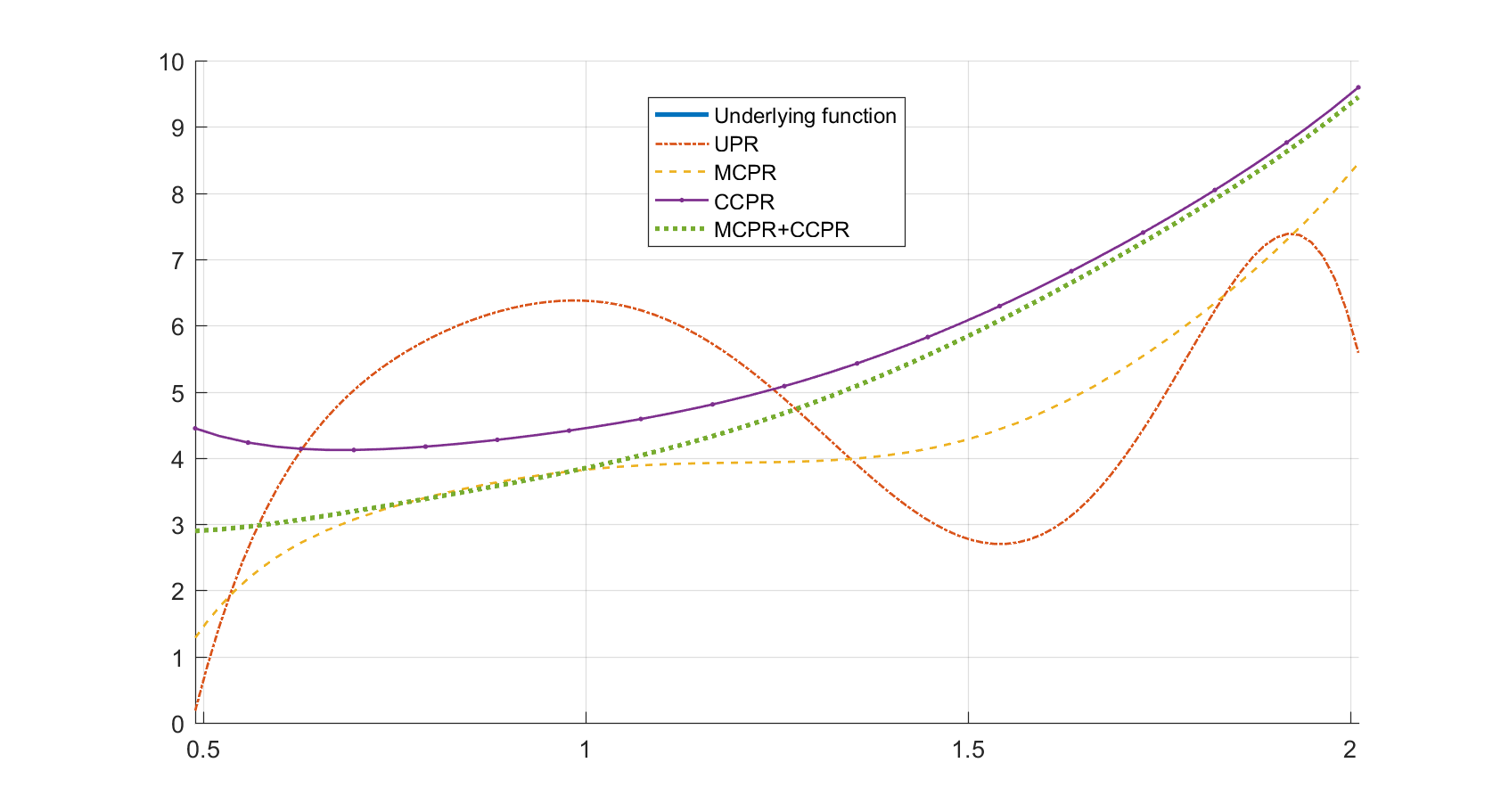}} 
	\caption{The projection of the fitted functions for different noise scaling factors and polynomial degrees}
	\label{projections}
\end{figure}

The results obtained confirm our previous observations. First, UPR tends to overfit, particularly when the noise scaling factor is high and when the degree of the polynomial fit is large (this is because, as the degree increases, the polynomials gain in expressiveness). Having monotonicity and convexity constraints proves to be a very efficient way of regularizing the polynomial fit, even in high noise settings: the fits obtained are very close to the true function. Furthermore, though their performance does deteriorate slightly in the high noise and high degree regime, the overall shape of the projection stays close to that of the underlying function, and that of lower degrees. This in contrast to the unconstrained fit whose shape is very unstable when the degree and the noise varies. 

%Finally, we conclude:
%
%\begin{itemize}
%	\item \textit{MCPR+CCPR is more robust:}
%	As the noise scaling factor increases, the training dataset becomes increasingly noisy. Polynomials are a very expressive family of function, therefore, as expected, the unconstrained polynomial overfits the noise. Having monotonicity and convexity constraints proves to be a great way to regularize a polynomial, thus leading to a much more robust outcome. Moreover the performance of the algorithm is deteriorating "gracefully", whereas the unconstrained polynomial displays a very unstable behavior.
%	\item \textit{Performance of MCPR+CCPR improves for higher degree without overfitting:} In contrast, due to overfitting, the performance of the uncostrained case deteriorates fast for higher degree polynomials. 
%\end{itemize}

\subsection{Applications to real regression problems}

In this section we present two applications of our methods to real datasets. Our first example uses monotonically constrained polynomial regression (MPCR) to predict interest rates for personal loans. The second example is a hybrid regression setting with a mixture of monotonicity and convexity constraints which is used to predict weekly wages from a set of features.

\subsection{Predicting interest rates for personal loans}
In this subsection, we study data for loans issued between the years 2007-2011 by Lending Club \cite{lendingclub}. We decided to focus on the particular category of home loans so as to avoid having to deal with categorical variables such as loan type. The updated dataset has $N=3707$ observations and 32 numerical features. Though the MCPR algorithm has run time polynomial in the number of features, we encounter issues with memory for too large a number of features. Hence, some data preprocessing is necessary to reduce the number of features. This was done by eliminating highly correlated covariates and running some canonical feature selection procedures. In the end, we consider six features. The response variable in this case is the interest rate on home loans. The features along with their monotonicity signs and their descriptions are presented below: 
\begin{itemize}
	\item \texttt{dti:+1} - Ratio of the borrower's total monthly
	debt payments an the self-reported monthly income.  A borrower with high dti is perceived to be riskier, which typically corresponds to higher interest rates. 
	\item \texttt{delinq\_2yrs:+1} - The number of past-due delinquencies in the past 2 years. The interest rate is monotonically increasing with respect to the number of delinquencies.
	\item \texttt{pub\_rec:+1}  - Number of derogatory public records. The  interest rate is monotonically increasing with respect to this feature.
	\item \texttt{out\_prncp:+1} - Remaining outstanding principal. This feature has a monotonically increasing relationship with the interest rate. 
	\item \texttt{total\_rec\_prncp:-1}  - Principal received to date with a monotonically decreasing dependency.
	\item \texttt{total\_rec\_int:-1} - Interest received to date. The interest rate is monotonically decreasing with respect to this feature.
\end{itemize}

We compute the average RMSE for testing and training sets through a 10-fold cross validation. We compare in Figure \ref{rmse_credit} the results for fitting polynomials of different degrees in both the unconstrained and monotonically constrained settings.
\begin{figure}[h!]
	\centering
	\subfigure[Values taken by the RMSE on training data]{\includegraphics[width = 0.49\textwidth]{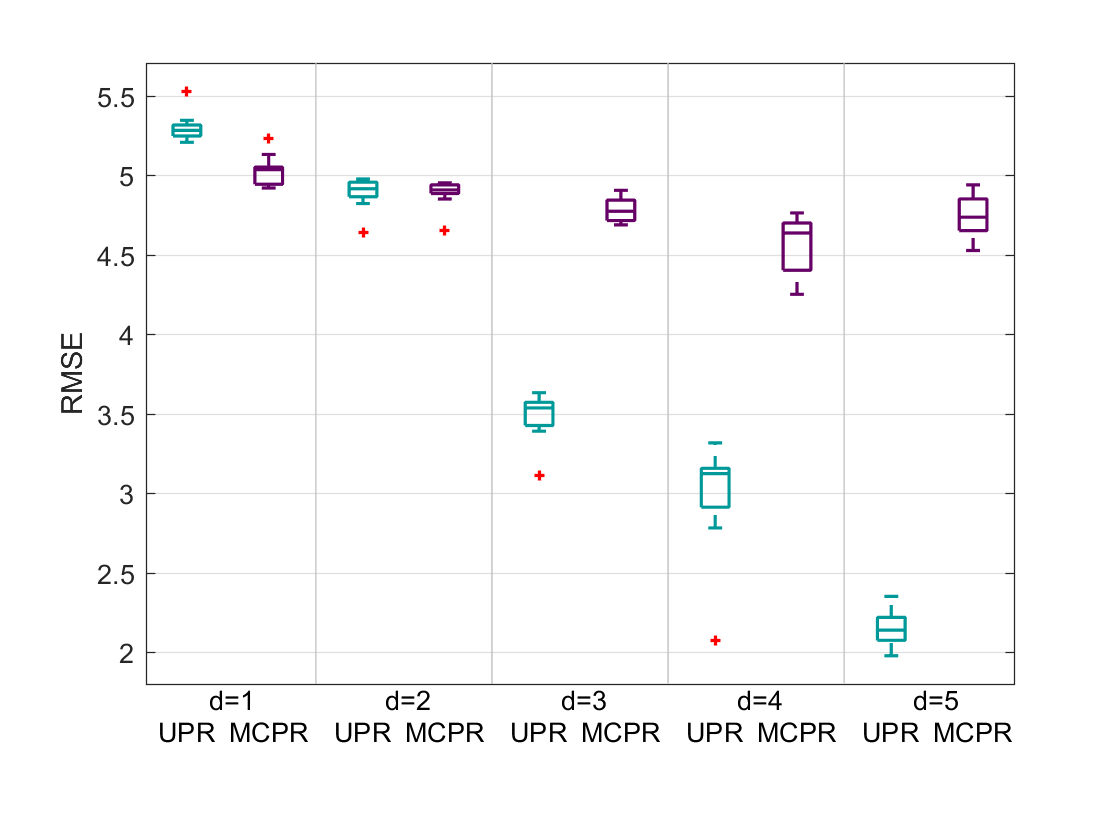}}
	\subfigure[Values taken by the RMSE on testing data]{\includegraphics[width = 0.49\textwidth]{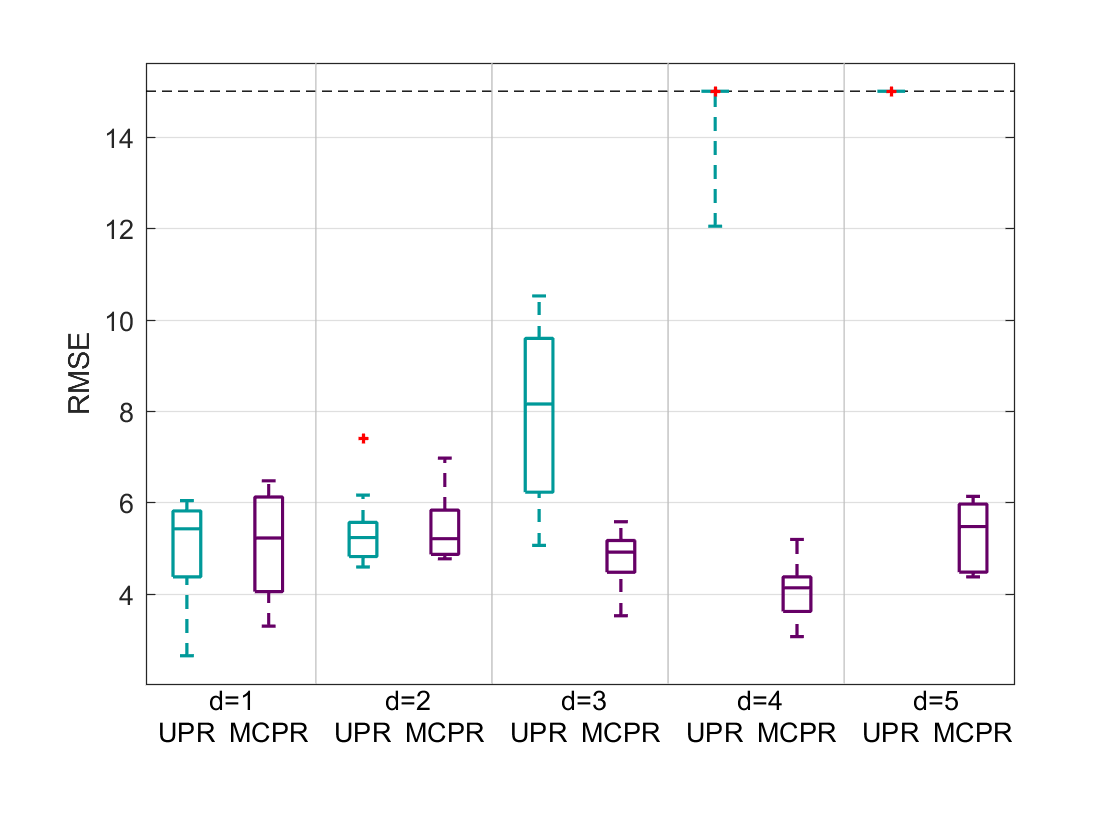}} 
	\caption{Comparative performance of testing and training sets for 10 fold cross validation.}
	\label{rmse_credit}
\end{figure}
The best performance was achieved by a degree 4, monotonically constrained polynomial regression with average RMSE of $4.09$ and standard error of $0.20$. Already for degree $5$, the unconstrained regression runs into numerical problems as it becomes rank deficient, i.e., the number of coefficients that needs to be determined is larger than the number of data points. Therefore, monotonicity constraints can be an efficient way of ensuring robustness in settings where the number of datapoints is relatively small, but the relationship between the covariates is complex.

\subsection{Predicting weekly wages}
In this section, we analyze data from the 1988 Current Population Survey. This data is freely available under the name \texttt{ex1029} in the Sleuth2 R package \cite{sleuth}. The data contains $N=25361$ observations and 2 numerical features: years of experience and years of education. We expect wages to increase with respect to years of education and be concave with respect to years of experience. We compare the performance of this hybrid constrained regression problem with the unconstrained case, as well as the CAP algorithm proposed by Hannah \cite{hannah2013multivariate}. Similarly to the previous example we compute the RMSEs with 10-fold cross validation. In addition we time our algorithm in order to compare the runtimes with the CAP algorithm.
The results are presented in Figure \ref{rmse_wage}.
\begin{figure}[h!]
	\centering
	\subfigure[Values taken by the RMSE on training data]{\includegraphics[width = 0.49\textwidth]{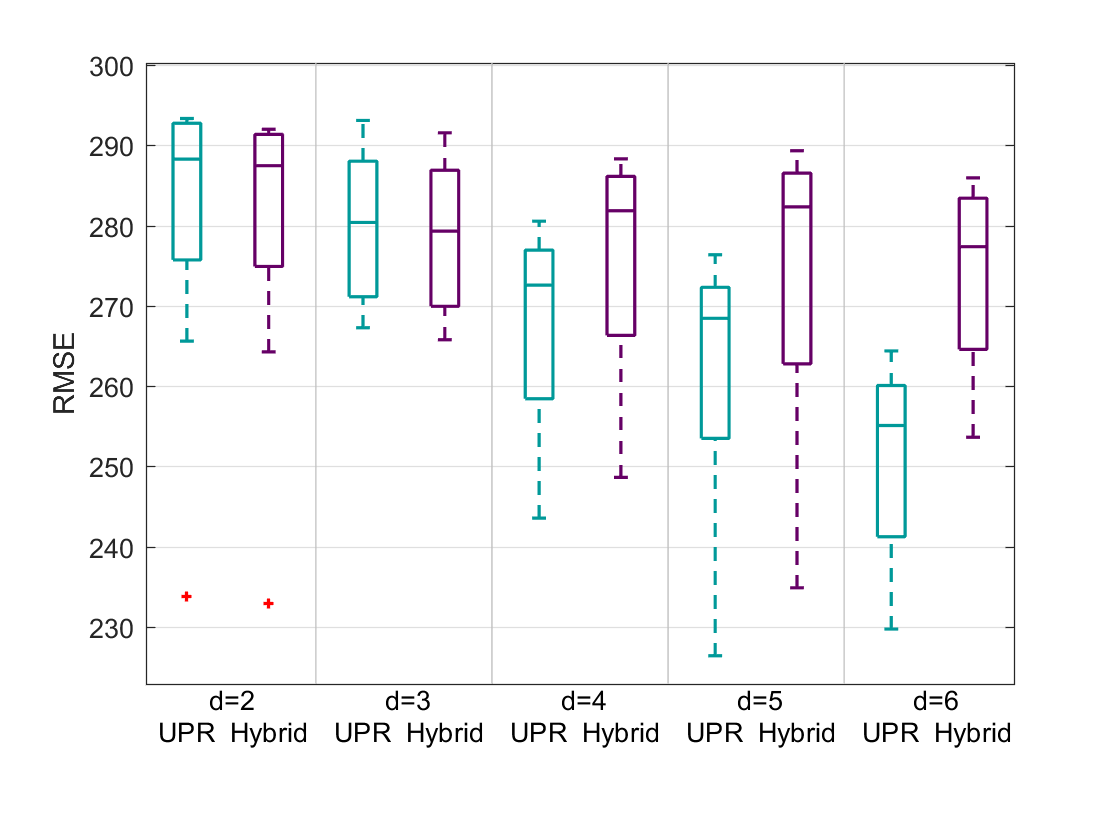}}
	\subfigure[Values taken by the RMSE on testing data]{\includegraphics[width = 0.49\textwidth]{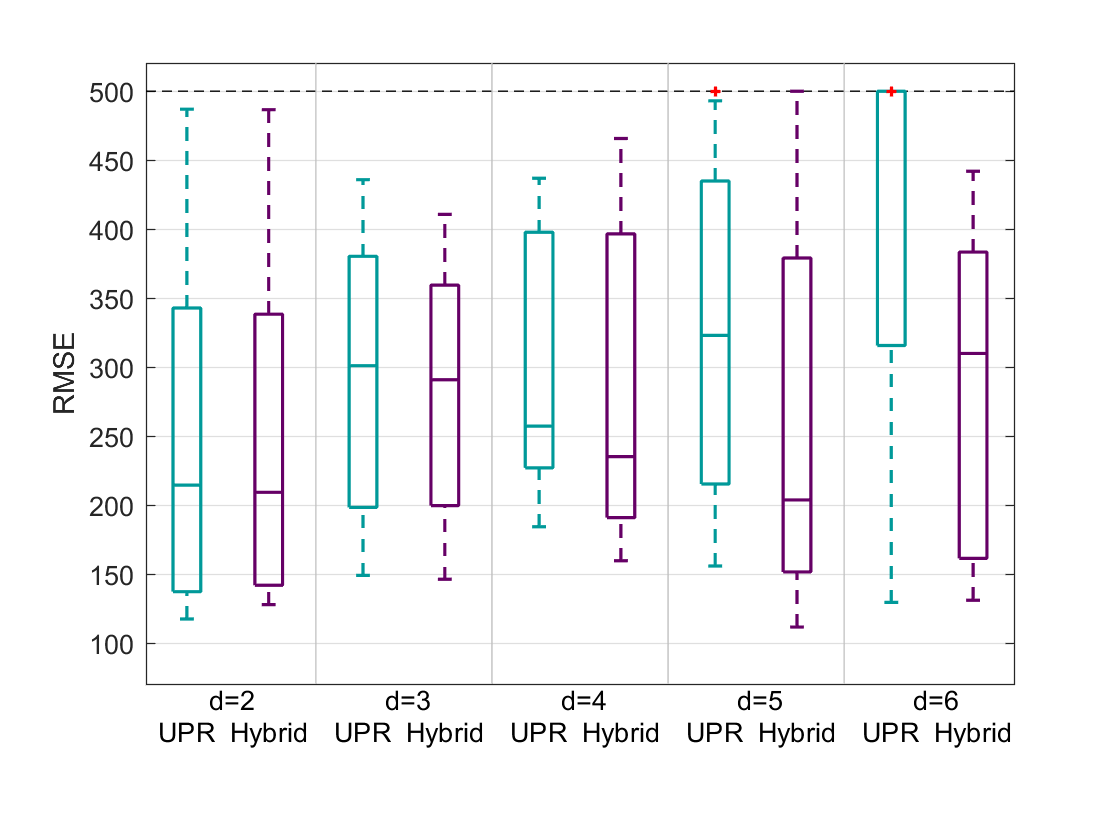}} 
	\caption{Comparative performance of testing and training sets for 10 fold cross validation.}
	\label{rmse_wage}
\end{figure}

The best performing algorithm is the monotonically and convexly constrained degree 2 polynomial with average test RMSE: $250.0$ and standard error $39.2$. The algorithm with the smallest standard error, therefore the one with the most consistent performance is the degree 3 hybrid polynomial with test RMSE: $285.0 \pm 29.9$. In comparison, the CAP and Fast CAP algorithm have test RMSE: $385.7 \pm 20.8$. Our algorithm does not only perform better in terms of RMSE, it also has a better runtime performance. For the degree 2 hybrid regression, the run time is $0.24\pm0.01$ seconds, and for degree 3, the hybrid regresion runtime is $0.26 \pm 0.01$ seconds. In contrast, the CAP algorithm takes $12.8 \pm 0.8$ seconds and the Fast CAP algorithm takes $1.9 \pm 0.2$ seconds.

%
%%\include{ch-usage/chapter-usage}
%%\include{ch-conclusion/chapter-conclusion}
%%\appendix % all chapters following will be labeled as appendices
%%\include{ch-appendicies/implementation}
%%\include{ch-appendicies/printing}

% Make the bibliography single spaced

\bibliographystyle{plain}

% add the Bibliography to the Table of Contents
\cleardoublepage
\ifdefined\phantomsection
  \phantomsection  % makes hyperref recognize this section properly for pdf link
\else
\fi
\addcontentsline{toc}{chapter}{Bibliography}

% include your .bib file
\bibliography{thesis}

\end{document}